\documentclass[11pt]{amsart}
\usepackage{preamble}

\title{Arithmetic Siegel--Weil formula on the product of modular curves}

\date{October 13, 2026}

\author[Ryan C. Chen]{Ryan C. Chen}
\address{Department of Mathematics,
	Fine Hall - Washington Road.
        Princeton, NJ 08544, USA}
\email{rcchen.math@gmail.com}

\author[Baiqing Zhu]{Baiqing Zhu}
\address{Department of Mathematics,
	Fine Hall - Washington Road.
        Princeton, NJ 08544, USA}
\email{bz6439@princeton.edu}

\begin{document}

\begin{abstract}
We prove the full generating series arithmetic Siegel--Weil formula on the product of two modular curves,
with arbitrarily singular coefficients. This is the case $\operatorname{GSpin}(2,2)$
in Kudla's program.
\end{abstract}
\maketitle

\pagenumbering{roman}
\tableofcontents
\newpage
\pagenumbering{arabic}

\section{Introduction}

\subsection{Background}

The theory of modular correspondences is
one of the earliest settings in which intersection theory
was related to modular forms.
Let $\mc{E}ll \ra \Spec \Z$ denote the moduli stack of elliptic curves.
The coarse moduli scheme \(Y_0(1)\) of $\mc{E}ll_\mathbb{Q}$ is isomorphic to \(\A^1_{\Q} \) by the $j$-invariant map,
with compactification \(X_0(1) \simeq \mathbb P^1_{\Q}\).
For positive integers \(m\),
the modular correspondence \(T(m)\subset Y_0(1) \times Y_0(1) \) parametrizes
pairs of elliptic curves admitting a degree $m$ isogeny.
The closure $\overline{T(m)}$ in \(X_0(1) \times X_0(1) \) has divisor class
\[
[\overline{T(m)}]
=
\mathcal O\bigl(\sigma_1(m),\sigma_1(m)\bigr),
\qquad
\sigma_1(m)=\sum_{d\mid m}d.
\]
The numbers \(\sigma_1(m)\) are also the positive Fourier coefficients of the nonholomorphic weight \(2\) Eisenstein series
\[
G_2(\tau)
=
\frac{1}{8\pi y}-\frac{1}{24}
+\sum_{m\geq 1}\sigma_1(m)q^m,
\qquad
\tau=x+iy,\quad q=e^{2\pi i\tau}.
\]
Thus, the divisor classes of modular correspondences are already encoded by an Eisenstein series.

Hurwitz further studied the intersection of two modular correspondences
\(T(m_1)\) and \(T(m_2)\); see \cite[Proposition 2.4]{GK93}.
In the case of proper intersection, the
intersection is supported on the locus of elliptic curves with complex multiplication,
and its degree is expressed in terms of Hurwitz class numbers. This gives
a remarkable relation between the proper intersections of modular correspondences and
the nonsingular Fourier coefficients of a genus \(2\) Siegel Eisenstein series.

Gross and Keating \cite{GK93} studied the arithmetic analogue by considering the proper
intersections of three modular correspondences on the Deligne--Mumford stack $\mc{M}\coloneq\mc{E}ll\times_{\Spec\mathbb{Z}}\mc{E}ll$.
The new contributions occur at primes where the
correspondences meet in the supersingular locus. By studying deformations
of endomorphisms of supersingular formal groups, they computed the
corresponding local intersection multiplicities and related them to
derivatives of nonsingular Fourier coefficients of a genus \(3\) Siegel Eisenstein series. The Kronecker--Hurwitz and
Gross--Keating formulas treat only the case of proper intersection, and
therefore account only for the nonsingular Fourier coefficients. For
non-proper intersections, the resulting cycles have excess dimension and
do not directly define the expected geometric or arithmetic degrees; the
arithmetic-geometric analogues of singular Fourier coefficients were thus left undefined and unknown.

Kudla's program provides a common framework for these phenomena.
Special cycles on Shimura varieties associated with orthogonal groups
are organized into generating series
indexed by symmetric matrices $T$,
and are expected to be related to Siegel Eisenstein series on symplectic groups.
In the geometric setting, degrees and intersection numbers of special cycles are expected to
equal values of Eisenstein series; in the arithmetic setting, arithmetic degrees and height pairings
are expected to equal their central derivatives \cite{Kud97Ann,KudlaProgram2004}.
The product of two modular curves corresponds to $\operatorname{GSpin}(2,2)$
for the standard split quadratic space.
In this language, the formulas of Kronecker and Hurwitz provide the comparison with nonsingular coefficients of the genus \(1\) and genus \(2\) geometric Siegel--Weil formulas,
while the Gross--Keating formula provides the comparison with nonsingular coefficients of the genus \(3\) arithmetic Siegel--Weil formula.
The unresolved non-proper intersections on $\mc{M}$
correspond precisely to the singular Fourier coefficients.

For quaternionic Shimura curves,
arithmetic Siegel--Weil formulas were established by Kudla, Rapoport, and Yang \cite{KRYbook}.
For modular curves \(X_0(N)\), related formulas were proved by Du and Yang
and by Sankaran, Shi, and Yang at square-free $N$ \cite{DY1,SSY},
and by one of the authors for arbitrary $N$ \cite{Zhu23,zhusingular}.

In higher dimensions, the nonsingular coefficients are by now
better understood.
Kudla and Rapoport formulated precise non-Archimedean identities relating intersections
of integral models of special cycles to derivatives of local representation densities \cite{KR1,KR2}.
At finite places, Li and Zhang related local arithmetic intersections to derivatives of local representation densities \cite{LZ,LiZhang-orthogonalKR}.
At the Archimedean place, Garc\'ia and Sankaran constructed Green forms in arbitrary codimension
and related the resulting heights to derivatives of local Whittaker functions \cite{GS19}
for compact Shimura varieties;
see also Bruinier--Yang \cite{BY20} and the corresponding unitary results of Liu \cite{Liu2011}.

The singular coefficients remain subtler.
For singular matrices $T$, the associated special cycle
has codimension \(\operatorname{rank}(T)\)
in the generic fiber if nonempty, which is
smaller than the codimension required for a top-degree arithmetic class.
One then supplies the missing codimension using powers of a metrized tautological bundle,
together with suitably normalized Archimedean terms.
For compact Shimura varieties, the Archimedean aspects are handled
in \cite{GS19}, including an ``Archimedean part'' of the arithmetic Siegel--Weil formula
for arbitrarily singular $T$.
In \cite[{\S 4.4}]{Chen2024}, it is explained how to combine the construction
of \cite[{\S 5.4}]{GS19} with derived modifications in positive characteristic
to define arithmetic cycle classes for possibly singular $T$, though some technical issues remain
with non-compact moduli spaces and with possible modifications of Green currents.
For singular terms in arithmetic Siegel--Weil, progress includes work by one of the authors on the corank \(1\) formula for unitary Shimura varieties
of arbitrarily large dimension \cite{Chen2024},
and the work of Bruinier and Howard on arithmetic volumes relevant for the most degenerate Fourier coefficients \cite{BruinierHoward2026}.

In this article, we establish a complete central-derivative
arithmetic Siegel--Weil formula for $\mc{M}$,
for the entire generating series and all arbitrarily singular
Fourier coefficients.
One motivation for this work was to illustrate
a full-generating-series arithmetic Siegel--Weil formula in higher dimension:
our moduli stack $\mc{M}$ has generic fiber which is a surface,
in contrast with the previously mentioned curve cases
in \cite{KRYbook} etc..
To our knowledge, ours is the first such example
with generic fiber of dimension $> 1$.

Just as the classical Siegel mass formula involves a weighted count
(``stacky count''), we expect that arithmetic Siegel--Weil formulas
are best formulated in stacky language. We thus work with the stack
$\mc{M}$ rather than its coarse moduli space $\A^1 \times \A^1$.

\subsection{Main results}

Our arithmetic Siegel--Weil formula is, roughly,
an equality between arithmetic
degrees of top degree cycle classes
on the arithmetic $3$-fold $\mc{M}$
and the central derivative of a genus $3$ Siegel Eisenstein series.
We begin by describing the arithmetic degrees.

Let \(T\in\operatorname{Sym}_3(\mathbb Q)\).
There is a special cycle \(\mathcal{Z}(T)\) over \( \mathcal{M} \)
defined using modular correspondences. Over the open modular surface, \(\mathcal{Z}(T)\) parametrizes two elliptic curves \(E_1\) and \(E_2\), together with a triple of homomorphisms
\[
\boldsymbol{x}=(x_1,x_2,x_3)\in\operatorname{Hom}(E_1,E_2)^3
\]
whose moment matrix is \(T\). More precisely, the quadratic form on \(\operatorname{Hom}(E_1,E_2)\) is given by
\[
Q(x)=\deg(x),
\]
and the moment matrix \(T(\boldsymbol{x}) \in \opn{Sym}_3(\Q)\) is characterized by
\[
Q(a_1x_1+a_2x_2+a_3x_3)
=
{}^t\!a\,T(\boldsymbol{x})\,a,
\qquad a=(a_1,a_2,a_3)\in\mathbb Z^3.
\]
The cycle \(\mathcal{Z}(T)\) is defined by the condition \(T(\boldsymbol{x})=T\).
In particular, the usual modular correspondence \(T(m)\) is roughly the special cycle attached to the \(1\times1\) matrix \((m)\).
In our case, $\mc{Z}(T)$ is flat over $\Spec \Z$
whenever $\opn{rk}(T) < 3$,
and is $0$-dimensional whenever $\opn{rk}(T) = 3$ (Lemma \ref{correct-codim}).

We now describe an arithmetic degree
which is a sum of two terms: the height of $\mc{Z}(T)$
plus a Green current integral.
More precisely,
there is a metrized tautological line bundle \(\widehat{\mathcal{E}}\)
on $\mc{M}$
given by tensoring the pullbacks of the metrized Hodge bundle
from each copy of $\mc{E}ll$
(Section \ref{ssec:moduli_problems}).
Take the first Chern form $\Omega \coloneqq c_1(\widehat{\mc{E}}^{\vee})$
of the dual metrized bundle.
Then we have $\dim \mc{Z}(T) = 3 - \opn{rk}(T)$
if $\mc{Z}(T)$ is nonempty
(Lemma \ref{correct-codim}).
For any $\msf{y} \in \opn{Sym}_3(\R)_{>0}$,
we define the \emph{arithmetic degree}
    \begin{equation}\label{equation:intro:arith_deg}
        \widehat{\deg} [\widehat{\mathcal{Z}}(T,\mathsf y)]
        \coloneqq \widehat{\deg}\left((\widehat{\mc{E}}^{\vee})^{3 - \opn{rk}(T)}\big|_{\mc{Z}(T)}\right)
        +
        \frac{1}{2}\int_{\overline{\mc{M}}(\mathbb{C})}\mathfrak{g}(T,\msf{y})
    \end{equation}
where the first term is the height of $\mc{Z}(T)$
with respect to $\widehat{\mc{E}}^{\vee}$,
and where the second term
is an integral involving a current $\mf{g}(T,\msf{y})$
whose construction on $\mc{M}(\C)$
is a small modification of the one considered in \cite{GS19}.
More precisely, see Section \ref{ssec: height of arithmetic_special_cycles}.
In Appendix \ref{sec append: genus 3 convergence},
we explain how $\mf{g}(T,\msf{y})$
extends to a current on the Deligne--Rapoport
compactification $\overline{\mc{M}}(\C)$,
so that the integral in \eqref{equation:intro:arith_deg}
is well-defined.

If $T$ is nonsingular, we have $\dim \mc{Z}(T) = 3 - \opn{rk}(T) = 0$
if $\mc{Z}(T)$ is nonempty,
and the height term in \eqref{equation:intro:arith_deg}
collapses to $\log p$ times the usual algebro-geometric degree of the $0$-cycle $\mc{Z}(T)$,
where $p$ is the unique prime over which $\mc{Z}(T)$ is supported.
This is the case of proper intersection studied by Gross and Keating.
The case of singular $T$, or $\dim \mc{Z}(T) > 0$
corresponds exactly to the ``excess intersection''
cases not covered by Gross and Keating.
The appearance of $\widehat{\mc{E}}^{\vee}$
in \eqref{equation:intro:arith_deg}
corrects for this excess intersection,
and can be used to define a top-codimension arithmetic
cycle class $[\widehat{\mc{Z}}(T, \msf{y})] \in \widehat{\opn{Ch}}^3(\mc{M})$
for arbitrary $(T, \msf{y})$,
as explained in \eqref{full-definition-ari-cycle}.

We next describe the analytic side of the formula.
Let $E(\msf{z}, s)$ be the classical Siegel Eisenstein series of genus $3$ and weight $2$,
given by
    \begin{equation}
    E(\msf{z}, s) \coloneqq
        \sum_{\gamma \in P(\Z) \backslash \Sp_{6}(\Z)}
            \frac{ \det(\msf{y})^{s/2}}{\det(c \msf{z} + d)^{2} |\det(c \msf{z} + d)|^{s}}
        \quad\quad
        \gamma = \begin{pmatrix} a & b \\ c & d \end{pmatrix}
    \end{equation}
for $P \subseteq \opn{Sp}_{6}$ the Siegel parabolic,
for $\msf{z} = \mathsf x + i\mathsf y\in\mathbb H_3$ and $s \in \C$
with $\opn{Re}(s) \gg 0$, with meromorphic continuation to $s \in \C$.
As explained in \eqref{global-fe},
the Eisenstein series $E(\msf{z}, s)$
satisfies a functional equation centered at $s = 0$.
There is a Fourier expansion
    \begin{equation}
    E(\msf{z}, s) = \sum_{T \in \opn{Sym}_3(\Q)} C_T(\msf{y}, s) q^T \quad\quad q^T = e^{2 \pi i \opn{tr}(T \msf{z})}.
    \end{equation}
Our main theorem, in generating series form, is as follows.

\begin{theorem}[Theorem \ref{thm: first time}]
    \label{theorem:intro}
We have
    \begin{equation}\label{eq: main thm intro}
        \sum_{T \in \opn{Sym}_3(\Q)}
        \widehat{\deg} [ \widehat{\mc{Z}}(T,\msf{y}) ]
        \cdot
        e^{2\pi i\operatorname{tr}(T\mathsf z)}
        =
        - \frac{\opn{vol}(\mc{M}(\C))}{2} \frac{\mrm{d}}{\mrm{d}s} \bigg|_{s = 0} E(\msf{z},s).
    \end{equation}
\end{theorem}
More precisely, Theorem \ref{theorem:intro}
asserts an equality between $\widehat{\deg}[\widehat{\mc{Z}}(T, \msf{y})]$
and the central derivative of the $T$-th Fourier coefficient $C_T(\msf{y}, s)$,
for any $T \in \opn{Sym}_3(\Q)$ and $\msf{y} \in \opn{Sym}_3(\R)_{>0}$,
and up to the factor $-\opn{vol}(\mc{M}(\C))/2$.
Here $\opn{vol}(\mc{M}(\C)) = 1/288$ denotes the orbifold
volume of $\mc{M}(\C)$
with respect to the volume form $\Omega^2$.

In Theorem \ref{theorem:intro}, the terms with positive definite $T$
correspond exactly to the situation studied by Gross and Keating.
The terms of ranks \(2,1,\) and \(0\) are the singular coefficients arising from non-proper intersections.
They are expressed, respectively, through the heights of certain CM cycles (rank 2),
the heights of modular correspondences (rank 1),
and the arithmetic volume of $\mc{M}$ with respect to the dual metrized tautological bundle (rank 0).
Thus, a single genus \(3\) Eisenstein series simultaneously encodes the proper arithmetic intersections
of modular correspondences and all of their excess-intersection contributions,
along with Green current integrals.
We emphasize that the correct arithmetic Siegel--Weil formula must not forget 
the Green currents: for example, the constant term
of $-\frac{1}{2} \opn{vol}(\mc{M}(\C)) \frac{\dd}{\dd s}\big |_{s = 0} E(\msf{z},s)$
is the arithmetic volume of $\mc{M}$
\emph{plus} an additional Green current integral.

\subsection{Strategy of the proof}
As already explained, the arithmetic degrees \eqref{equation:intro:arith_deg} have two contributions: heights
of special cycles and integrals of the currents
\(\mathfrak g(T,\mathsf y)\). We discuss their evaluation separately.

\subsubsection{Heights of special cycles}

Let \(T\) be positive semidefinite and put \(r=\operatorname{rank}(T)\).
For \(r=0\) and \(r = 1\), the required calculations reduce to the arithmetic
volume of \(\mathcal M\) and the heights of the Hecke correspondences
\(T(m)\), respectively. These quantities were computed explicitly by
Burgos Gil, Kramer, and K\"uhn \cite{Burgos2007}.
The computation is based on explicit \emph{global sections}
of line bundles, e.g. the Ramanujan discriminant function,
and on Rohrlich’s modular Jensen’s formula \cite{Rohrlich1984}.

In contrast, the method for $r = 3$ and $r = 2$ is \emph{local},
without explicit sections of line bundles.
We first recall the work of Gross--Keating for the case $r = 3$.
When \(r=3\), the special cycle \(\mathcal Z(T)\) is an Artinian
Deligne--Mumford stack supported on the supersingular locus. Its height
is a sum of stacky lengths of local rings, multiplied by
\(\log p\). The required local lengths are given by the deformation-theoretic
Gross--Keating formulas \cite{GK93}.
On the analytic side, there is a factorization of the Fourier coefficients
$C_T$ into local Whittaker functions. More precisely,
after normalizing $C_T^3(\msf{y}, s) \coloneqq (s + 1)(s + 2)\Lambda(s + 2) \Lambda(2s + 2)C_T(\msf{y},s)$
to produce the symmetric functional equation $C_T^3(\msf{y}, -s) = -C_T^3(\msf{y}, s)$
(Section \ref{ssec:normalized_Eisenstein})
we have a factorization
    \begin{equation}
    -(32 \pi^2)^{-1} C_T^3(\msf{y}, s) = W^{\star}_{T,\infty}(\msf{y}, s) \prod_p W^{\star}_{T,p}(s)
    \end{equation}
into normalized local Whittaker functions, as in \eqref{equation:normalized_global_Whittaker:m=4}.
Gross--Keating's result is a comparison of the stacky length in characteristic $p$
with $\frac{\mrm{d}}{\mrm{d}s} \big|_{s = 0} W^{\star}_{T,p}(s)$.
We have $\opn{vol}(\mc{M}(\C))/2 = (32 \pi^2)^{-1} (s + 1)(s + 2) \Lambda(s + 2) \Lambda(2s+2)\big |_{s = 0}$,
consistent with the constant in Theorem \ref{theorem:intro}.

The remaining case is \(r=2\), which is the main new contribution
of the heights part of our work.
Unlike Gross--Keating's case, the cycle $\mc{Z}(T)$ is flat
over $\Spec \Z$ and hence its height does not admit an obvious canonical
local decomposition. We now describe our local decomposition procedure.
The generic fiber of
\(\mathcal Z(T)\) consists of CM points parametrizing pairs of elliptic
curves with complex multiplication by orders
in an explicit imaginary quadratic field depending on 
$T$ (Lemma \ref{correct-codim}).
The height of \(\mathcal Z(T)\) is 
then given by weighted averages of the Faltings heights of
these elliptic curves, with weights determined by the cycle
multiplicities and automorphism groups.

The Faltings height of an elliptic curve with complex multiplication
by the full ring of integers \(\mathcal O_K\) is known explicitly
by the formula of Lerch and Chowla--Selberg
as recalled in \eqref{equation:Faltings_height_CM:Chowla--Selberg}
(also \cite[p.~29]{Deligne85}).
The height in the case of CM by a non-maximal order can be found in the work
of Nakkajima and Taguchi \cite{NT91}.

Rather than evaluate these heights individually, we compute the
difference between the relevant average height and the maximal-order CM
height. This reduces the rank \(2\) height calculation to a ``relative
height correction'', with the maximal-order formula providing the
reference value. It turns out that this ``average relative height correction''
for $\mc{Z}(T)$ decomposes locally in a canonical way as a sum (see \eqref{eq: local height change})
    \begin{equation}
    \sum_{p} \frac{\widehat{\deg}_{\widehat{\mc{E}}^{\vee}}(\mc{Z}(T))_p}
    {\deg_{\mc{E}^{\vee}_{\Q}}(\mc{Z}(T)_{\Q})}
    \end{equation}
over all the finite primes whenever the denominator is nonzero,
where $\deg_{\mc{E}^{\vee}_{\Q}}(\mc{Z}(T)_{\Q})$
denotes the degree of the $0$-cycle $\mc{Z}(T)_{\Q}$
in the generic fiber (Section \ref{ssec:height_decomp}).
We then make a global-to-local reduction step and explain how the
``average relative height correction'' can be calculated
via Rapoport--Zink uniformization.
More precisely, we show
    \begin{equation}\label{equation:intro:average_height_diff_at_p}
    \frac{\widehat{\deg}_{\widehat{\mc{E}}^{\vee}}(\mc{Z}(T))_p}
    {\deg_{\mc{E}^{\vee}_{\Q}}(\mc{Z}(T)_{\Q})}
    =
    -\frac{\delta_{\mrm{Fal},p}(T)}{\deg_p(T)} \cdot \log p
    \end{equation}
in \eqref{equation:RZ_uniformization:heights:normalized_height},
whenever the denominators are nonzero,
where the quantities $\delta_{\mrm{Fal},p}(T)$
and $\deg_p(T)$ are formulated purely in terms of
local special cycles on Rapoport--Zink spaces
\eqref{equation:delta_T_deg_T}.
Thus, the average height calculation for $r = 2$ reduces
to a key local comparison of ``local height corrections''
with derivatives of local Whittaker functions
appearing in the Eisenstein series Fourier coefficient.
More precisely, let $T^\flat\in\mrm{Sym}_2(\mathbb{Q})$ be a 
$2\times 2$ nondegenerate matrix associated to the rank 2 
matrix $T$ as in \eqref{reduce-to-diagonal},
for example if $T = \opn{diag}(0, T^{\flat})$.
We then have
\begin{equation}
    \delta_{\mrm{Fal},p}(T) \cdot \log p = - \frac{\mrm{d}}{\mrm{d}s} \bigg|_{s = -1/2} W_{T^\flat,p}^{\star}(s)
    \quad\quad
    \deg_p(T) = W_{T^\flat,p}^{\star}(-1/2)
\end{equation}
where $W_{T^\flat,p}^{\star}(s)$ is a normalized local Whittaker function of genus 2.
The derivative and special value are taken at the non-central point $s = -1/2$,
and the central point is at $s = 0$.
This local comparison is proved in Section \ref{sec:proof_non-Arch}.
There, Theorem \ref{corank-1-local}
states the result in terms of local densities, whose relation
with local Whittaker functions is explained
in Remark \ref{remark:Whittaker_interpolation_and_Den}.

Lastly, extracting the height calculation for $\mc{Z}(T)$
from the average height calculation
amounts to extracting the numerator of the left-hand side of
\eqref{equation:intro:average_height_diff_at_p} from the displayed ratio.
For this purpose, the left-hand side denominator can be calculated
e.g. by complex uniformization and compared with the special value of a genus $2$
Eisenstein series Fourier coefficient;
this is Lemma \ref{lemma:complex_degree:corank_1}.

\subsubsection{Integrals of currents}
For $T\in\mrm{Sym}_3(\mathbb{Q})$ and
$\msf{y}\in\mrm{Sym}_3(\mathbb{R})_{>0}$,
Garc\'ia and Sankaran \cite{GS19} construct a smooth $(2,2)$-form
$\nu^\circ(T,\msf{y})$ on $\mc{M}(\mathbb{C})$.
Its Mellin transform defines a family of currents
\begin{equation}
    \mathfrak{g}^\circ(T,\msf{y},\rho)
    =
    \int_1^\infty
    \nu^\circ(T,t\msf{y})\frac{\mrm{d}t}{t^{1+\rho}},
\end{equation}
initially for $\mrm{Re}(\rho)$ sufficiently large; see
\eqref{int-first-sum}.

In Appendix \ref{sec append: genus 3 convergence}, we show that
this family admits a meromorphic continuation to
$\mrm{Re}(\rho)>-1/2$ with values in currents on the compactification
$\overline{\mc{M}}(\mathbb{C})$, extending the corresponding
family on $\mc{M}(\mathbb{C})$.
Taking the constant term at $\rho=0$ gives an extension of
$\mf{g}^\circ(T,\msf{y})$ to the compactification.
Its integral, understood as the pairing of this extension
with the constant function $1$, is therefore well-defined.
Moreover, the convergence estimates in the appendix justify
interchanging the two integrations in their common domain
of absolute convergence. By meromorphic continuation, we obtain
\begin{equation}\label{eq: integral intro}
\begin{aligned}
    \int_{\mc{M}(\mathbb{C})}\mathfrak{g}^\circ(T,\msf{y})
    =
    \ct\int_{\mc{M}(\mathbb{C})}
    \left(
        \int_1^\infty
        \nu^\circ(T,t\msf{y})\frac{\mrm{d}t}{t^{1+\rho}}
    \right)
    =
    \ct\int_1^\infty
    \left(
        \int_{\mc{M}(\mathbb{C})}\nu^\circ(T,t\msf{y})
    \right)\frac{\mrm{d}t}{t^{1+\rho}},
\end{aligned}
\end{equation}
where $\ct$ denotes the constant term at $\rho=0$.
The correction terms in \eqref{eq: corrected currents}
contribute geometric degrees of special cycles, which are
computed explicitly in $\S$\ref{ssec:complex_fiber_degree}.
Thus the evaluation of the Archimedean contribution in
\eqref{equation:intro:arith_deg} reduces to computing
\eqref{eq: integral intro}.

Following the strategy of \cite{GS19}, we express
$\nu^\circ(T,t\msf{y})$ in terms of the $T$-th Fourier
coefficient of an explicit theta series.
More precisely, \eqref{wedge-omega} gives
\begin{equation}\label{eq: nu - theta}
    \nu^\circ(T,t\msf{y})_z
    =
    C_{\theta,T}\left(
        t\msf{y},h_z,
        \tilde{\nu}_3\otimes
        \boldsymbol{1}_{\textup{M}_2(\hat{\mathbb{Z}})^3}
    \right)\Omega_z^2
\end{equation}
for $z\in\mc{M}(\mathbb{C})$,
where $C_{\theta,T}$ is a $T$-th Fourier coefficient of a theta series
for the Schwartz function $\tilde{\nu}_3\otimes
\boldsymbol{1}_{\textup{M}_2(\hat{\mathbb{Z}})^3}$,
where $\tilde{\nu}_3$ arises from a Schwartz form in Kudla--Millson theory.
The integral of $\nu^\circ(T,t\msf{y})$ over $\mc{M}(\C)$
can then be identified, up to
an explicit normalization factor,
with a theta integral over $O(V)(\mathbb{Q})\backslash O(V)(\mathbb{A})$.
Our setting lies outside Weil's convergence range
\cite{weil1965siegel}, so the ordinary Siegel--Weil formula
does not apply directly.
Instead, we use the regularized Siegel--Weil formula
\cite{KudRallis94}; see also \cite{HarrisKudla1991}.
Together with the absolute convergence of the relevant
Fourier-coefficient integrals, this yields the identity
predicted by the ordinary Siegel--Weil formula and expresses
\eqref{eq: integral intro} in terms of Fourier coefficients
of Eisenstein series.
The relation with the derivative of the classical Siegel Eisenstein series
(formed using a different Schwartz function)
then appears via a suitable differential operator,
as in \cite[{(5.139)}]{GS19}
and \cite[{Lemma 3.5}]{GS19}.

\subsection{Structure of the paper}

In Part \ref{part:analytic},
we set up the precise normalizations of Eisenstein series
and local Whittaker functions, explain explicit
formulas relating singular Fourier coefficients
of $E(\msf{z}, s)$ to Fourier coefficients of lower genus
Eisenstein series, and record the relevant Siegel--Weil formula
for coefficient-by-coefficient convergent theta integrals.

Sections \ref{sec:KM_upper_half}
through \ref{sec:Green_and_arith}
explain the construction of arithmetic cycle classes.
Sections \ref{sec:Faltings_heights_rk2}
through \ref{sec:RZ_uniformization}
explain how the global height term for $T$ of corank $1$
can be reduced to a collection of local quantities
calculated on Rapoport--Zink spaces.
This is an analogue of the procedure
from \cite{Chen2024} with some simplifications,
e.g. we work with the ratio \eqref{equation:RZ_uniformization:heights:normalized_height}
which is eventually compared with a log derivative in \eqref{eq: whit-der-height-change},
so that the groupoid cardinality $\# [ \cdots ]$
from \eqref{equation:RZ_uniformization:heights:height:local-to-global} does not need to be computed
explicitly, unlike in loc. cit..
Once the problem is reduced to a local calculation
on Rapoport--Zink spaces,
the comparison with local Whittaker functions
(via local densities) is accomplished in Section \ref{sec:proof_non-Arch}.

The main theorem is proved in Section \ref{sec:proof},
by assembling our main results
with the existing pieces in the literature
and aligning the constants.

\subsection{Acknowledgements} 

This research was partially conducted during the period
RC served as a Clay Research Fellow, when RC
was hosted at Princeton and at the Max Planck Institute
for Mathematics in Bonn.
BZ carried out part of this work during
his time at Columbia University and part during his time at Princeton University.

An initial version of Appendix \ref{sec append: genus 3 convergence}
was produced by GPT-6 Astra.
The current version has been rewritten from scratch.

\part{Analytic side}
\label{part:analytic}

\section{Local densities of quadratic lattices}
Let $p$ be any prime. Let $F$ be a $p$-adic local field with ring of integers $\Of$, uniformizer $\varpi\in\Of$ and residue field $\Of/(\varpi)\simeq \mathbb{F}_{q}$ of size $q$.
Let $v_F: F\rightarrow\mathbb{Z}\cup\{\infty\}$ be the valuation with $v_F(\varpi) = 1$.

\subsection{Notations on quadratic lattices}\label{not-ql}
A quadratic lattice $(M,q)$ is a finite free $\mathcal{O}_{F}$-module equipped with a quadratic form $q: M\rightarrow F$.
Let $M_F\coloneqq M\otimes_{\Of}F$ be the $F$-vector space associated to $M$.
The quadratic form $q$ induces a symmetric bilinear form
    \begin{equation}
    M_F\times M_F\stackrel{(\cdot,\cdot)}\longrightarrow F
    \end{equation}
given by $(x,y)=q(x+y)-q(x)-q(y)$.
Let $M^{\vee}=\{x\in M\otimes_{\mathcal{O}_{F}}F:(x,M)\subset \mathcal{O}_{F}\}$.
We say a quadratic lattice $M$ is \emph{integral} if $q(x)\in\mathcal{O}_{F}$ for all $x\in M$,
and is \emph{self-dual} if it is integral and $M$ = $M^{\vee}$.
Given a tuple $\bsm{x}$ of elements in $M$, our convention is that the \emph{normalized Gram matrix}
or \emph{moment matrix}
is the matrix $\frac{1}{2}(\bsm{x}, \bsm{x})$.
\par
Suppose the symmetric bilinear form $(\cdot,\cdot)$ on $M_F$ is nondegenerate.
With $n = \textup{rank}_{\Of} M$,
let $\{x_{i}\}_{i=1}^{n}$ be any $\Of$ basis of $M$, with $t_{ij}=\frac{1}{2}(x_{i},x_{j})$.
The \emph{discriminants} of $M$ and $M_{F}$ are
\begin{equation*}
    \opn{disc}(M) = (-1)^{n(n-1)/2}\textup{det}\left((t_{ij})\right)\in F^{\times}/(\mc{O}_F^{\times})^{2}
    \quad\quad
    \opn{disc}(M_F) = (-1)^{n(n-1)/2}\textup{det}\left((t_{ij})\right)\in F^{\times} / (F^{\times})^2.
\end{equation*}
Given a nondegenerate matrix  $T \in \opn{Sym}_n(F)$,
we set $\disc(T) \coloneqq
    (-1)^{n(n-1)/2} \det(T)$.

Let $\chi_F: F^{\times}/(F^{\times})^{2}\rightarrow\{\pm1,0\}$ be the function
    \begin{equation*}
    \chi_F(a) \coloneqq
    \begin{cases}
    +1 & \text{if $F[X]/(X^2 - a)$ is split over $F$}
    \\
    -1 & \text{if $F[X]/(X^2 - a)$ is inert over $F$}
    \\
    0 & \text{if $F[X]/(X^2 - a)$ is ramified over $F$}.
    \end{cases}
    \end{equation*}
We define $\chi_F(M) = \chi_F(\opn{disc}(M))$.
For nondegenerate $T \in \opn{Sym}_n(F)$,
we set $\chi_F(T) \coloneqq \chi_F(\disc(T))$.
When $F=\mathbb{Q}_p$, we set $\chi_p \coloneqq \chi_F$.
\par
Let $\mathrm{Hilb}(\cdot,\cdot)_{F}$ be the Hilbert symbol on the local field $F$.
If moreover $\{x_{i}\}_{i=1}^{n}$ is any orthogonal basis of the quadratic space $M_F$, i.e. $t_{ij}=0$ if $i\neq j$ and $t_{ii}\neq 0$, the \emph{Hasse invariant} of the quadratic space $M_F$ is
\begin{equation}
    \label{equation:Hasse-invariant}
    \epsilon(M_F) = \prod\limits_{1 \leq i<j \leq n} \mathrm{Hilb}(t_{ii},t_{jj})_{F}\in\{\pm1\}.
\end{equation}
The invariants $\opn{disc}(M_F)$ and $\epsilon(M_F)$ determine $M_F$ up to isomorphism.

The lattice $M$ has an associated sequence $\opn{GK}(M) = (a_1, \ldots, a_n) \in \Z^n$
of \emph{Gross--Keating invariants}, with $a_1 \leq \cdots \leq a_n$.
These invariants $\opn{GK}(M)$ are useful for uniform statements
in all residue characteristics, particularly when including $p = 2$.
The invariants $\opn{GK}(M)$ were introduced in \cite[{\S 4}]{GK93}
for $F = \Q_p$ and $n = 3$, and the analogous general case
is defined in \cite[{Definition 0.1}]{IK18}.
See also Appendix \ref{sec:local_den_and_p=2} for the definition,
there with $F = \Q_p$.
From the definition, one finds that $M$ is integral if and only if $a_i \geq 0$ for all $i$.
It follows from this that $a_1 = \min_{x \in M} v_F(q(x))$
for any $M$, not necessarily integral.
We set $\nu(M) \coloneq a_1 + \cdots + a_n$ for the sum of the Gross--Keating
invariants.
For the reader's convenience, we recall \cite[{Theorem 0.1}]{IK18}
which reads
    \begin{equation}\label{equation:total_GK}
    \nu(M) =
        \begin{cases}
        v_F(4^{\lfloor n / 2 \rfloor} \disc(M)) & \text{if $n$ is odd}
        \\
        v_F(4^{\lfloor n / 2 \rfloor} \disc(M)) & \text{if $n$ is even and $\chi_F(M) \neq 0$}
        \\
        v_F(4^{\lfloor n / 2 \rfloor} \disc(M)) - v_F(\Delta_M) + 1 & \text{if $n$ is even and $\chi_F(M) = 0$,}
        \end{cases}
    \end{equation}
where $\Delta_M \subseteq \mc{O}_F$ denotes the discriminant ideal
of the extension $F[X]/(X^2 - \disc(M))$ over $F$.
In particular, if $n$ is even, we find that $\chi_F(M) = 0$ if and only if
$\nu(M)$ is odd. The previous sentence also holds for $n$ odd when $p \neq 2$,
but fails for $n$ odd and $p = 2$.

\subsection{Local densities}
\label{ssec:local_den}

Set
    \begin{align}
    \opn{Sym}_n(\mc{O}_F)^*
        & \coloneqq
            \{ b \in \opn{Sym}_n(F) :
                \opn{tr}(bc) \in \mc{O}_{F} \text{ for all } c \in \opn{Sym}_n(\mc{O}_{F}) \}.
    \end{align}
An element $T \in \opn{Sym}_n(F)$ lies in $\opn{Sym}_n(\mc{O}_F)^*$
if and only if $T_{i,i} \in \mc{O}_F$ and $T_{i,j} \in \frac{1}{2} \mc{O}_F$ for all $i \neq j$.

Given nonsingular matrices $S \in \opn{Sym}_m(F)$ and $T \in \opn{Sym}_n(F)$,
we recall the \emph{local representation density} (or just \emph{local density})
    \begin{equation}
    \opn{Den}(S,T) \coloneqq
        \lim_{k \ra \infty}
        \frac{\opn{vol}(\{x \in M_{m,n}(\mc{O}_{F}) :
            {}^t x S x - T \in \varpi^k \opn{Sym}_n(\mc{O}_{F})^* \})}{q^{-k \cdot n (n + 1)/2}},
    \end{equation}
where $M_{m,n}(\mc{O}_{F})$ is given the Haar measure of total volume $1$.
The limit argument stabilizes for $k \gg 0$.
This depends only on the isomorphism classes
of the quadratic lattices defined by $S$ and $T$,
so we also write $\opn{Den}(M,L)$ if $M$
and $L$ are quadratic lattices admitting bases with normalized Gram matrices
$S$ and $T$, respectively.
If $n > m$, we have $\opn{Den}(S,T) = 0$.
If $S \in \opn{Sym}_n(\mc{O}_F)^*$
but $T \not \in \opn{Sym}_n(\mc{O}_F)^*$, we have $\opn{Den}(S,T) = 0$.

\begin{remark}\label{remark:Den_via_Rep}
Let $L, M$ be two integral non-degenerate quadratic $\mathcal{O}_{F}$-lattices. Let $\textup{Rep}_{M,L}$ be the scheme of integral representations, an $\mathcal{O}_{F}$-scheme such that for any $\mathcal{O}_{F}$-algebra $R$,
\begin{equation*}
    \textup{Rep}_{M,L}(R)=\textup{QHom}(L\otimes_{\mathcal{O}_{F}}R, M\otimes_{\mathcal{O}_{F}}R),
\end{equation*}
where \textup{QHom} denotes the set of module homomorphisms which preserve the quadratic forms. The local density of integral representations is
\begin{equation*}
    \textup{Den}(M,L)=\lim\limits_{d\rightarrow\infty}\frac{\#\textup{Rep}_{M,L}(\mathcal{O}_{F}/\varpi^{d})}{q^{d\cdot \textup{dim}(\textup{Rep}_{M,L})_{F}}}.
\end{equation*}
Equivalently, $\opn{Den}(M,L)$ is the volume of $\opn{Rep}_{M,L}(\mc{O}_F)$
with respect to the fibral measure associated to the map $M_{m,n}(\mc{O}_F) \ra \opn{Sym}_n(\mc{O}_F)^*$
given by $x \mapsto {}^t x S x - T$ in the previous notation,
if the source and target are taken to have total volume $1$.
The limit in the construction of $\opn{Den}(M,L)$ corresponds to a shrinking tubular neighborhood
of $\opn{Rep}_{M,L}(\mc{O}_F)$.
\end{remark}
\begin{remark}
In Remark \ref{remark:Den_via_Rep},
if $L,M$ have ranks $n,m$ respectively, then $(\textup{Rep}_{M,L})_{F}\neq\varnothing$
if and only if $n\leq m$, in which case $\textup{dim}(\textup{Rep}_{M,L})_{F}=mn-\frac{n(n+1)}{2}$.
If $n > m$, then $\textup{Rep}_{M,L}(\mathcal{O}_{F}/\varpi^{d}) = \emptyset$ for $d \gg 0$.
\end{remark}

\subsection{Normalized local density polynomial}
\label{nor-local-density}
Let $H$ be the standard split rank $2$ quadratic $\mc{O}_F$-lattice,
with quadratic form given by $q(x,y)=xy$.
For a nonnegative integer $r$, define $H^{r}$ to be the orthogonal sum of $r$ copies of $H$.

Let $L$ be a quadratic $\mc{O}_F$-lattice of rank $n \geq 0$.
We set
    \begin{equation}
    \opn{Nor}(X,L)
        \coloneqq
        \begin{cases}
        (1 - q^{-(n+1)/2} X) \prod_{i = 1}^{\lfloor n / 2 \rfloor} (1 - q^{2i - (n + 1)} X^2) & \text{if $n$ is odd}
        \\
        (1 - q^{-1/2} \chi_F(L) X)^{-1} (1 - q^{-(n+1)/2} X) \prod_{i = 1}^{\lfloor n / 2 \rfloor} (1 - q^{2i - (n + 1)} X^2) & \text{if $n$ is even,}
        \end{cases}
    \end{equation}
which is taken from \cite[{\S 4}]{Ikeda17},
up to a shift.

Consider an even integer $m \geq 0$, and set $s_0 = (m - (n+1))/2$.
With $m$ suppressed from notation,
the discussion in \cite[{\S 4}]{Ikeda17}
(see also Remark \ref{remark:Whittaker_interpolation_and_Den})
implies that there is a (unique)
Laurent polynomial $\opn{Den}^{\star}(X,L) \in \Z[X, X^{-1/2}]$ such that
    \begin{equation}\label{equ: denstar}
    \opn{Den}^{\star}(X,L)|_{X = q^{-k}} =
        \begin{cases}
        \left ( X^{-\nu(L) / 2} \frac{\opn{Den}(H^{m/2 + k - 2 s_0},L)}{\opn{Nor}(q^{s_0} X, L)}
            \right ) \bigg|_{X = q^{-k}}
            & \text{if $n$ is odd}
        \\
        \left ( X^{-\lfloor \nu(L) / 2 \rfloor} \frac{\opn{Den}(H^{m/2 + k - 2 s_0},L)}{\opn{Nor}(q^{s_0} X, L)}
            \right ) \bigg|_{X = q^{-k}}
            & \text{if $n$ is even}
        \end{cases}
    \end{equation}
for all integers $k$ with $m /2 + k - 2 s_0 \geq 0$
and $\opn{Nor}(q^{s_0} X, L)|_{X = q^{-k}} \neq 0$.
We moreover have $X^{\nu(L) / 2} \opn{Den}^{\star}(X,L) \in \Z[X]$
if $n$ is odd
and $X^{\lfloor \nu(L) / 2 \rfloor} \opn{Den}^{\star}(X,L) \in \Z[X]$
if $n$ is even,
and these elements of $\Z[X]$ have constant term $1$.
By \cite[{Theorem 4.1}]{Ikeda17} combined with
the formula for $\nu(L)$ recalled in \eqref{equation:total_GK},
we have the functional equation
    \begin{equation}\label{equation:local_Den_functional_equation}
    \begin{cases}
        \opn{Den}^{\star}( \frac{1}{q^{2 s_0} X}, L )
            = \eta_L \opn{Den}^{\star} ( X , L )
                & \text{if $n$ is odd}
        \\
        \opn{Den}^{\star} ( \frac{1}{q^{2 s_0} X}, L )
            = \opn{Den}^{\star} (X, L ) & \text{if $n$ is even}
    \end{cases}
    \end{equation}
where
    \begin{equation}
    \eta_L =
    \begin{cases}
    +1 & \text{if $L$ is split}
    \\
    -1 & \text{if $L$ is nonsplit}
    \end{cases}
    \end{equation}
via the explanation after \cite[{Definition 0.4}]{IK18}.
We then set
\begin{equation}\label{nor-der-local-density}
        \den^{\star}(L)\coloneqq\den^{\star}(1,L), \quad \partial\den^{\star}(L)\coloneqq\frac{\textup{d}}{\textup{d}X}\bigg\vert_{X=1}\den^{\star}(X,L).
\end{equation}
For $n$ odd, we are abusing notation as in \cite[{Remark 15.5.2}]{Chen2024}
to handle expressions involving $X^{-1/2}$,
e.g. $\frac{d}{dX} X^{-1/2} \coloneqq -\frac{1}{2} X^{-3/2}$ as one expects.

If $\opn{GK}(L) = (0, \ldots, 0)$,
then we have $\opn{Den}^{\star}(X,L) = 1$,
a constant function:
we know $\opn{Den}^{\star}(X,L) \in \Z[X]$ must be a constant
because of the functional equation \eqref{equation:local_Den_functional_equation},
and $\opn{Den}^{\star}(X,L)$ has constant term $1$.

\begin{example}\label{mod-local-density}
    For this paper, we are interested in the case $F = \Q_p$ with $m = 4$
    and $0 \leq n \leq m - 1$.
    In these cases, we have
    \begin{equation}
        \den^{\star}(p^{-k},L)=\begin{cases}
            p^{k ( \nu(L)/2 )} \frac{\den(H^{k+2},L)}{\left(1-p^{-k-2}\right)\left(1-p^{-2k-2}\right)} 
                &\textup{if $\textup{rank}_{\zp}(L)=3$}\\
            p^{k\lfloor \nu(L)/2 \rfloor} \frac{\den(H^{k+1},L)}{\left(1-p^{-k-1}\right)\left(1-p^{-2k}\right)} \cdot \left(1-\chi_p(L)p^{-k}\right)
                &\textup{if $\textup{rank}_{\zp}(L)=2$}\\
            p^{k ( \nu(L)/2 )} \frac{\den(H^{k}, L)}{1 - p^{-k}} 
                &\text{if $\rank_{\Z_p}(L) = 1$}\\
            1
                & \text{if $\rank_{\Z_p}(L) = 0$}
        \end{cases}
    \end{equation}
    for all integers $k \geq 0$ if $n = 3$,
    for $k \geq 1$ if $n = 2$ or $n = 1$,
    and for all $k$ if $n = 0$.    
\end{example}

\section{Weil representations, Eisenstein series and Whittaker functions}
    \subsection{Notations on Siegel groups}
    \label{ssec:Eisenstein:Siegel:groups}
For a positive integer $r$, let $\mathsf{W}_{r}$ be the standard symplectic space over $\mathbb{Q}$
of dimension $2 r$, equipped with a basis with Gram matrix
    \begin{equation}
    w_r \coloneqq
        \begin{pmatrix}
        0 & 1_r\\
        -1_r & 0
        \end{pmatrix}.
    \end{equation}
Let $P \subseteq \opn{Sp}_{2r} \coloneqq \opn{Sp}(\msf{W}_r)$ be the Siegel parabolic,
with Levi $M \subseteq P$ and unipotent radical $N \subseteq P$ given explicitly as
    \begin{align*}
    P & = \left \{ \begin{pmatrix} * & * \\ 0 & * \end{pmatrix} \in \opn{Sp}(\msf{W}_r) \right \}
    \\
    M & = \left\{m(a)=\begin{pmatrix}a & 0\\0 & ^{t}a^{-1}\end{pmatrix}: a\in \textup{GL}_{r} \right\}
    \quad\quad
    N = \left\{n(b)=\begin{pmatrix}1_{r} & b\\0 & 1_{r}\end{pmatrix}: b\in \textup{Sym}_{r} \right\}.
\end{align*}
We sometimes write $P_r, M_r, N_r$ to emphasize $r$-dependence.
The Hermitian symmetric domain associated to the group $\textup{Sp}(\mathsf{W}_r)$ is the genus $r$ upper half-plane
\begin{equation*}
    \mathbb{H}_r=\{ \msf{z} = \msf{x} + i \msf{y} : \msf{x}, \msf{y} \in\textup{Sym}_r(\mathbb{R}), \msf{y} >0 \}.
\end{equation*}
There is the usual transitive action $\opn{Sp}_{2r}(\R) \acts \mbb{H}_r$
given by
    \begin{equation}
    \msf{z} \mapsto (a \msf{z} + b)(c \msf{z} + d)^{-1}
    \quad\quad
    \text{for}
    \quad\quad
    \begin{pmatrix}
    a & b 
    \\
    c & d
    \end{pmatrix} \in \opn{Sp}_{2r}(\R)
    \quad
    \text{and}
    \quad
    \msf{z} \in \mbb{H}_r,
    \end{equation}
and the stabilizer of $i 1_r \in \mbb{H}_r$ is the standard maximal compact
    \begin{equation}\label{maximalcompact}
        K_{r,\infty}=\bigg\{\begin{pmatrix}
            a & b\\
            -b & a
        \end{pmatrix}:\boldsymbol{\textup{k}}=a+ib\in U(r)\bigg\}\simeq U(r).
    \end{equation}
Given an element $\msf{z}=\msf{x}+i\msf{y}\in\mathbb{H}_r$,
we write $g_{\msf{z},\infty} = n(\msf{x}) m(a) \in P(\R)$
for any element with the property $\msf{z} = g_{\msf{z},\infty} \cdot i 1_r$,
i.e. $a \cdot {}^t a = \msf{y}$, such that moreover $\det(a) > 0$.
We then write $g_{\msf{z}} \in P(\A)$ for an element
with
    \begin{equation}
    \label{def_gy}
    \text{$g_{\msf{z},v} = 1_{2r}$ for all $v < \infty$ and $g_{\msf{z},\infty}$ as above.}
    \end{equation}

\subsection{Degenerate principal series}
Fix the standard additive character $\psi:\mathbb{A}/\mathbb{Q}\rightarrow\mathbb{C}^{\times}$ such that $\psi_{\infty}(x)=e^{2\pi ix}$. We define the degenerate principal series to be the unnormalized smooth induction
\begin{equation}\label{equation:I_r,v_global}
    I_r(s)\coloneqq\textup{Ind}_{P_r(\mathbb{A})}^{\textup{Sp}(\mathsf{W}_{r})(\mathbb{A})}\left(\vert\textup{det}(\cdot)\vert_{\mathbb{A}}^{s+(r+1)/2}\right),\,\,\,s\in\mathbb{C},
\end{equation}
where $\det(\cdot) \colon P_r(\A) \ra \A^{\times}$ is abuse of notation
for $m(a) \mapsto \det(a)$.
This is the restricted tensor product of the following local representations of $\textup{Sp}(\mathsf{W}_{r})(\mathbb{Q}_v)$:
\begin{equation}\label{equation:I_r,v_local}
    I_{r,v}(s)\coloneqq\textup{Ind}_{P_{r}(\mathbb{Q}_v)}^{\textup{Sp}(\mathsf{W}_{r})(\mathbb{Q}_v)}\left(\vert\textup{det}(\cdot)\vert_{v}^{s+(r+1)/2}\right),\,\,\,s\in\mathbb{C}.
\end{equation}
We will use $\mathsf{r}$ (resp. $\mathsf{r}_v$) to denote the right regular representation of $\textup{Sp}(\mathsf{W}_{r})(\mathbb{A})$ (resp. $\textup{Sp}(\mathsf{W}_{r})(\mathbb{Q}_v)$) on $I_r(s)$ (resp. $I_{r,v}(s)$).

\subsection{Representations of $K_{r,\infty}$}
    \label{ssec:Eisenstein:K-reps}
The irreducible representations of the group $K_{r,\infty}\simeq U(r)$ (\ref{maximalcompact}) are denoted by $\pi_{\lambda}$ with
\begin{equation*}
    \lambda=(l_1,\cdots,l_r)\in\mathbb{Z}^{r},\,\,\,\,l_1\geq \cdots \geq l_r.
\end{equation*}
The $K_{r,\infty}$-types appearing in $I_{r,\infty}(s)$ were determined by Kudla and Rallis \cite[(1.3)]{KR90DegeneratePS} (see also \cite[Proposition 2.8]{KudRallis94}) to be precisely those irreducible representations $\pi_\lambda$ of $K_{r,\infty}$ with highest weight $\lambda = (l_1,\cdots,l_r)$ such that
\begin{equation}\label{infty-ps}
    l_i\in 2\mathbb{Z},\,\,\,1\leq i\leq r.
\end{equation}
\par
Moreover, these $K_{r,\infty}$-types appear with multiplicity one in $I_{r,\infty}(s)$. If $\Phi_{\infty}^\lambda(\cdot, s) \in I_{r,\infty}(s)$ is a non-zero highest weight vector of weight $\lambda$, then $\Phi_{\infty}^\lambda(e,s) \neq 0$ (\cite[Proposition 1.1]{KR90DegeneratePS}), hence from now on we normalize all such highest weight vectors so that $\Phi_{\infty}^\lambda(e,s) = 1$. Note that the restriction of $\Phi_{\infty}^\lambda$ to $K_{r,\infty}$ is independent of $s$. For scalar weights $\lambda = l(1,\cdots,1)$ with $l\in2\mathbb{Z}$, we have
\begin{equation*}
    \Phi_{\infty}^l(k,s)\coloneqq\Phi_{\infty}^{l(1,\cdots,1)}(k,s)=\det(k)^{l},\,\,\,k\in K_{r,\infty}.
\end{equation*}

\subsection{Notations on Eisenstein series}
For a standard section $\Phi(-,s)\in I_r(s)$ (i.e., its restriction to the standard maximal compact subgroup of $\textup{Sp}(\mathsf{W}_{r})(\mathbb{A})$ is independent of $s$),
define the associated Siegel Eisenstein series
\begin{equation*}
    E(g,s,\Phi)=\sum\limits_{\gamma\in P(\Q) \backslash \textup{Sp}(\mathsf{W}_r)(\mathbb{Q})}\Phi(\gamma g,s),
\end{equation*}
which converges for $\textup{Re}(s)\gg 0$ and admits meromorphic continuation to $s\in\mathbb{C}$.

We have a Fourier expansion
\begin{equation}\label{fourier=decomposition}
    E(g,s,\Phi)=\sum\limits_{T\in\textup{Sym}_{r}(\mathbb{Q})}E_{T}(g,s,\Phi),
\end{equation}
where
\begin{equation*}
    E_{T}(g,s,\Phi)=\int_{\textup{Sym}_{r}(\mathbb{Q})\backslash \textup{Sym}_{r}(\mathbb{A})}E(n(b)g,s,\Phi)\psi(-\textup{tr}(Tb)) \dd n(b),
\end{equation*}
for $\opn{Re}(s) \gg 0$,
where the Haar measure $ \dd n(b)$ is normalized to be self-dual with respect to $(b,c) \mapsto \psi(\opn{tr}(bc))$.

When $T$ is nonsingular, for factorizable $\Phi = \otimes_v\Phi_v$, we have a factorization of the Fourier coefficient into a product
\begin{equation}\label{nonsingular-factorization}
    E_T(g,s,\Phi)=\prod\limits_{v}W_{T,v}(g_v,s,\Phi_v)
\end{equation}
where the local (generalized) Whittaker function is defined by
\begin{equation}\label{local-whittaker}
    W_{T,v}(g_v,s,\Phi_v)=\int_{N_r(\mathbb{Q}_v)}\Phi_v(w_r^{-1}n(b)g,s)\psi_v(-\textup{tr}(Tb)) \dd n(b) \end{equation}
for $\opn{Re}(s) \gg 0$,
where the Haar measure is self-dual for $(b,c) \mapsto \psi_v(\opn{tr}(bc))$.
For an arbitrary matrix $T\in\textup{Sym}_r(\mathbb{Q})$,
we also have the global Whittaker function:
\begin{equation}\label{global-whittaker}
    W_{T}(g,s,\Phi)=\int_{N_r(\mathbb{A})}\Phi(w_r^{-1}n(b)g,s)\psi(-\textup{tr}(Tb)) \dd n(b) \end{equation}
for $\opn{Re}(s) \gg 0$,
again for the self-dual Haar measure.
When the section $\Phi$ is factorizable,
the above global Whittaker function is the product of all the local Whittaker functions.
It is exactly the $T$-th Fourier coefficient of $E(g,s,\Phi)$ when $T$ is nonsingular.

In this article, we focus on standard sections $\Phi\in I_r(s)$ of the following form:
\begin{equation}
    \Phi_{r}^{\lambda}(-,s)=\Phi^{\lambda}_{r,\infty}(-,s)\otimes \bigotimes_{v<\infty}\Phi_{r,v}(-,s),
\end{equation}
where the weight vector $\lambda=(l_1,\cdots,l_r)$ satisfies (\ref{infty-ps}) and the function $\Phi_{r,v}(-,s)\in I_v(s)$ is the unique spherical function such that $\Phi_{r,v}(e,s)=1$ for all $s\in\mathbb{C}$.
If $\lambda=l(1,\cdots,1)$ for some $l \in 2 \Z$, we use $\Phi_{r}^{l}(-,s)$ to denote this section.

For $\msf{z}=\msf{x}+i\msf{y}\in\mathbb{H}_r$,
and for $m = 2 l$, we define the classical weight $l$ Eisenstein series to be
\begin{equation}\label{classical-eis}
    E(\msf{z},s,\Phi^l_r ) = (\det \msf{y})^{-m / 4} E(g_\msf{z},s,\Phi^l_r)
\end{equation}
for any $g_{\msf{z}}$ as in Section \ref{ssec:Eisenstein:Siegel:groups}. This is independent of the choice of $g_{\msf{z}}$ because different choices differ by an element in $\mrm{SO}_r\subset U(r)$, whose determinant is $1$.
A computation yields the more classical form
    \begin{equation}
    E(\msf{z}, s, \Phi^l_r) =
        \sum_{\gamma \in P(\Z) \backslash \Sp_{2r}(\Z)}
        \frac{ \det(\msf{y})^{(s - s_0)/2}}{\det(c \msf{z} + d)^{m / 2} |\det(c \msf{z} + d)|^{s - s_0}}
    \quad\quad
    \gamma = \begin{pmatrix} a & b \\ c & d \end{pmatrix}
    \end{equation}
for $\opn{Re}(s) \gg 0$, where $s_0 \coloneqq (m - (r + 1))/2$,
and where $P(\mathbb{Z})$ and $\Sp_{2r}(\Z)$
are defined as in Section \ref{ssec:Eisenstein:Siegel:groups}
but with $\Z$ instead of $\Q$.
For
\begin{equation}\label{classical-coefficient}
    C_T(\msf{y},s,\Phi)\coloneqq(\det \msf{y})^{-m/4}E_T(g_{i\msf{y}},s,\Phi)\cdot e^{2\pi \textup{tr}(T\msf{y})},
\end{equation} 
we have the following Fourier expansion of the classical Eisenstein series (\ref{classical-eis}):
\begin{equation}\label{fouexpan-classical-eis}
    E(\msf{z},s,\Phi)=\sum\limits_{T \in \opn{Sym}_r(\Q)}C_T(\msf{y},s,\Phi)\cdot q^{T},\,\,\,\, q^{T}\coloneqq e^{2\pi i\textup{tr}(T\msf{z})}.
\end{equation}

\subsection{Siegel--Weil section}

Let $\mbb{V}$ be a rank $m$ quadratic space over $\A$,
giving a collection of $m$-dimensional quadratic spaces
$\mbb{V}_v$ over $\Q_v$ for all places $v$ of $\Q$,
with bilinear pairing $(-,-)$ and quadratic form $x \mapsto \frac{1}{2}(x,x)$.
We moreover assume that $\mbb{V} \otimes_{\A} \A^{\mtt{S}} \cong V \otimes_{\Q} \A^{\mtt{S}}$
for some finite set of places $\mtt{S}$
and some quadratic space $V$ over $\Q$,
where $\A^{\mtt{S}}$ denotes the away-from-$\mtt{S}$ ad\`ele ring.
In this paper, we are primarily interested in the case where
\begin{equation}\label{equation:incoherent_M_2}
    V = \opn{M}_2(\Q) \quad\quad \mathbb{V}_{v}=V_v=\textup{M}_2(\mathbb{Q}_v)\,\,\textup{if $v<\infty$} \quad\quad \mathbb{V}_{\infty} \text{ is positive definite}
\end{equation}
where $V$ is equipped with the quadratic form $x \mapsto \det(x)$.
In this case, we can immediately verify that $\prod\limits_{v}\epsilon(\mathbb{V}_{v})=-1$,
where the Archimedean Hasse invariant is also defined by the formula \eqref{equation:Hasse-invariant}.

We allow general $\mbb{V}$ for the moment,
but simplify slightly by requiring $m$ to be even.
For each place $v$ of $\Q$,
consider the quadratic character $\chi_{\mbb{V}_v} \colon \Q_v^{\times} \ra \{ \pm 1 \}$,
defined by the Hilbert symbol $\chi_{\mbb{V}_v}(a_v) \coloneqq (a_v, \opn{disc}(\mbb{V}_v))_{\Q_v}$.
and define $\chi_{\mbb{V}} \colon \A^{\times} \ra \{ \pm 1 \}$
by $\chi_{\mbb{V}}(a) = \prod_v \chi_{\mbb{V}_v}(a_v)$.
Assume that $\chi_{\mbb{V}}|_{\Q^{\times}}$ is trivial.
In the situation of \eqref{equation:incoherent_M_2},
our $\chi_{\mbb{V}}$ is the trivial character.

Let $r \geq 0$ be a nonnegative integer, 
and let $\mathscr{S}(\mathbb{V}^{r})$ be the space of Schwartz functions on $\mathbb{V}^{r}$.
The fixed choice of $\psi$ gives a Weil representation $\omega_r$ of $\textup{Sp}(\mathsf{W}_{r})(\mathbb{A})\times\textup{O}(\mathbb{V})$
on $\mathscr{S}(\mathbb{V}^{r})$, which may be described explicitly
by the Schr\"odinger model
    \begin{align}
    (\omega(m(a)) \varphi)(\bsm{x})
        & = \chi_{\mbb{V}}(a) |a|^{m / 2}_{\A} \varphi(\bsm{x} \cdot a)
            && m(a) \in M(\A)
    \\
    (\omega(n(b)) \varphi)(\bsm{x})
        & = \psi(\tfrac{1}{2} \opn{tr}(b(\bsm{x}, \bsm{x}))) \varphi(\bsm{x})
            && n(b) \in N(\A)
    \\
    (\omega(w_r) \varphi)(\bsm{x})
        & = \gamma_{\psi}(\mbb{V})^r \widehat{\varphi}(\bsm{x})
    \\
    (\omega(h) \varphi)(\bsm{x})
        & = \varphi(h^{-1} \cdot \bsm{x})
            && h \in O(V(\A))
    \end{align}
as in \cite[{\S 5}]{Kudlasplitting}
and \cite[{\S I.1}]{Kud97Ann};
here $\gamma_{\psi}(\mbb{V}) = \prod_v \gamma_{\psi_v}(\mbb{V}_v)$
is a product of local Weil indices, all but finitely many of which
are equal to $1$,
where $\gamma_{\psi_v}(\mbb{V}_v)$ means the Weil index associated to
the ``non-degenerate character of the second degree''
$x \mapsto \psi_v(\frac{1}{2} (x,x))$ on $\mbb{V}_v$.
In the situation of \eqref{equation:incoherent_M_2},
we have $\gamma_{\psi_v}(\mbb{V}_v) = 1$ if $v < \infty$
by e.g. \cite[{Lemma 14.1.1}]{Chen2024},
and we have $\gamma_{\psi_{\infty}}(\mbb{V}_{\infty}) = - 1$
by \cite[{Proposition A.10(2)}]{Rao1993WeilRepresentation}.
Above, we abuse notation $\chi_{\mbb{V}}(a) = \chi_{\mbb{V}}(\det a)$
and $|a|_{\A} = |\det a|_{\A}$, as usual.

Set $s_0 \coloneqq (m - (r + 1))/2$.
For $\boldsymbol\varphi\in\mathscr{S}(\mathbb{V}^{r})$, define a function
\begin{equation*}
    \Phi_{\boldsymbol\varphi}(g,s)\coloneqq\left(\omega_r(g)\boldsymbol\varphi\right)(0)\cdot\vert\textup{det}(a(g))\vert_{\A}^{s - s_0},\,\,\,g\in\textup{Sp}(\mathsf{W}_{r})(\mathbb{A})
\end{equation*}
with $g = m(a(g)) n(b) k$ for $k$ in the standard maximal compact subgroup.
Then $\Phi_{\boldsymbol\varphi}(g, s)\in I_r(s, \chi_{\mbb{V}})$
where $I_r(s, \chi_{\mbb{V}})$ is defined as in \eqref{equation:I_r,v_global}
but with $\chi_{\mbb{V}} |-|_{\A}^{s + (r+1)/2}$ instead of $|-|_{\A}^{s + (r+1)/2}$ from loc. cit..
This $\Phi_{\boldsymbol\varphi}$ is
called the standard Siegel--Weil section associated to $\boldsymbol\varphi$.
For $\boldsymbol\varphi\in \mathscr{S}(\mathbb{V}^{r})$, we write $E(g,s,\boldsymbol\varphi)\coloneqq E(g,s,\Phi_{\boldsymbol\varphi})$.

Given $\varphi_v \in \ms{S}(\mbb{V}_v^r)$,
we similarly have $\Phi_{\varphi_v}$,
defined using the local Weil representation,
i.e. replacing $\mbb{A}$ with $\Q_v$ above.
If $v = \infty$ and $\mbb{V}_{\infty}$ is positive definite
with $m \equiv 0 \pmod{4}$,
the formula \cite[{(2.66)}]{GS19} implies that
the Gaussian $\varphi_{\infty} \in \ms{S}(\mbb{V}_{\infty}^r)$
given by
    \begin{equation}
    \varphi_{\infty}(\bsm{x}) = e^{-\pi \opn{tr}(\bsm{x}, \bsm{x})}
    \end{equation}
is such that $\Phi_{\varphi_{\infty}} = \Phi^{m / 2}_{\infty}$
with the latter as defined in Section \ref{ssec:Eisenstein:K-reps}.
If $v < \infty$, if $\mbb{V}_v$ is isomorphic to the standard split quadratic space of rank $m$,
and if $L \subseteq \mbb{V}_v$ is a lattice which
is isomorphic to the standard split quadratic lattice of rank $m$,
i.e. admitting a basis with the quadratic form
    \begin{equation}
    \frac{1}{2}
    \begin{pmatrix}
    0 & 1_{m/2}
    \\
    1_{m/2} & 0
    \end{pmatrix},
    \end{equation}
then the function $\varphi_v = \pmb{1}_{L^r}$
is such that $\Phi_{\varphi_v}$ is the unique spherical standard section
with $\Phi_{\varphi_v}(1,s) = 1$.
This applies, for example, to the situation of \eqref{equation:incoherent_M_2}.

\begin{remark}\label{remark:Whittaker_interpolation_and_Den}
    For a non-Archimedean place $v$, consider a standard section $\Phi_{\varphi_v}(-,s)\in I_{r,v}(s)$
    associated to a Schwartz function $\varphi_v=\boldsymbol{1}_{M^{r}}$
    where $M\subset\BV_v$ is a $\Z_v$-lattice of rank $m$. Let $L$ be a quadratic $\Z_v$-lattice
    of rank $r$ with normalized Gram matrix $T$, for some basis.
    We have the following
    relation between local Whittaker function and the local density for
    all integers $k\geq0$ by a generalization of the calculation
    in \cite[{Proposition A.4, Proposition A.6}]{Kud97Ann},\footnote{
    We find that a few constants should be modified in loc. cit.
    in the case of residue characteristic $p = 2$.
    In the proof of \cite[{Proposition A.4}]{Kud97Ann},
    the self-dual Haar volume for $\opn{Sym}_n(\mc{O})$
    should contribute an additional factor of $|2|_v^{n(n-1)/4}$.
    In the proof of \cite[{Proposition A.6}]{Kud97Ann},
    we find that the measure $\dd x$ should be self-dual
    for the pairing defined by $2 S_r$, for compatibility 
    with \cite[{\S 5}]{Kudlasplitting}.
    This contributes an additional factor of $|2|_v^{mn/2}$.
    Our normalization of the local density differs from
    loc. cit., i.e. we do not have the $|2|^{n(n + 1 + 2r)}$ factor.
    We also mention that $\gamma_F(\det V, \psi)$ and $\gamma_{F}(\psi)$ in \cite{Kud97Ann}
    are what would have been denoted $\gamma_F(\det V, \frac{1}{2} \psi)$ and $\gamma_{F}(\frac{1}{2} \psi)$
    in \cite{Kudlasplitting}, and that Kudla's $\gamma(V)$ in \cite[{Proposition A.4}]{Kud97Ann}
    is our $\gamma_{\psi_v}(\mbb{V}_v)^{-r}$.
    }
\begin{equation}\label{whitt-local-density}
    W_{T,v}\left(e,k + s_0,\Phi_{\varphi_v}\right)=\gamma_{\psi_v}(\mbb{V}_v)^{-r} |2|_{v}^{r(r - 1) / 4 + m r / 2}\vert\opn{disc}(M)\vert_v^{r/2}\cdot\textup{Den}(M\obot H^{k}, L).
\end{equation}
The same formula holds if $\Q_p$ is replaced with an arbitrary $p$-adic field.
The same formula also holds for possibly odd $m$,
if one replaces the symplectic group with its metaplectic double cover.
\end{remark}

\subsection{Pullback map and intertwining operator}
In this section we fix some notation which will be used repeatedly in the next section.
Given nonnegative integers $r_1\leq r_2$,
there is a standard inclusion $\mathsf{W}_{r_1}\rightarrow\mathsf{W}_{r_2}$ inducing 
an embedding $\eta_{r_1}^{r_2}:\textup{Sp}(\mathsf{W}_{r_1})(\mathbb{A})\rightarrow\textup{Sp}(\mathsf{W}_{r_2})(\mathbb{A})$
as in \cite[(5.39)]{GS19}.
Define the following pullback map for every place $v$ \cite[(5.45)]{GS19}
\begin{equation}
    \left(\eta_{r_1}^{r_2}\right)^{\ast}: I_{r_2,v}(s)\rightarrow I_{r_1,v}\left(s+\frac{r_2-r_1}{2}\right).
\end{equation}
\par
Let $\mathcal{M}_r(s)=\otimes_{v}\mathcal{M}_{r,v}(s)$ be the standard intertwining operator, where
\begin{align}
    \mathcal{M}_{r,v}(s): I_{r,v}(s)&\longrightarrow I_{r,v}(-s)\\
    \Phi&\longmapsto\left(\mathcal{M}_{r,v}(s)\Phi\right)(g)\coloneqq\notag\int_{N_{r}(\mathbb{Q}_v)}\Phi\left(w_{r}^{-1}n(b)g,s\right)\textup{d}n(b)\notag
\end{align}
is defined for $s \in \C$ with $\textup{Re}(s) \gg 0$.
Both $\mathcal{M}_{r,v}(s)$ and $\mathcal{M}_r(s)$ admit meromorphic continuation
to $s \in \mathbb{C}$.

\section{Normalization of Whittaker functions and Eisenstein series}
\subsection{Normalized Archimedean local Whittaker functions}

    Fix a nonnegative $m \in 4 \Z$,
    set $s_0 = (m - ( r + 1))/2$,
    and consider the standard section $\Phi^{m/2}_{\infty}(g,s)$
    as in Section \ref{ssec:Eisenstein:K-reps}.
    Fix $T \in \opn{Sym}_{r}(\R)$ nonsingular.
    Given $\msf{y} \in \opn{Sym}_{r}(\R)_{>0}$
    and any $a \in \GL_r(\R)$ with $a \cdot {}^t a = \msf{y}$
    and $\det a > 0$, a change of variables gives
        \begin{equation}
        W_{T,\infty}(m(a), s, \Phi^{m/2}_{\infty})
        =
(\det \msf{y})^{-s/2 + (r + 1)/4}
        W_{{}^t a T a,\infty}(1, s, \Phi^{m/2}_{\infty}).
        \end{equation}
We set
        \begin{equation}
        W_{T,\infty}(\msf{y}, s, \Phi^{m/2}_{\infty})
        \coloneqq
        (\det \msf{y})^{-m/4} W_{T,\infty}(m(a), s, \Phi^{m/2}_{\infty}),
        \end{equation}
    and this does not depend on the choice of $a$.

    For any $b \in \opn{Sym}_r(\R)$, a computation gives
        \begin{equation}
        \Phi^{m/2}_{\infty}(w_r^{-1} n(b), s)
=
        \det(i 1 + b)^{-(s - s_0 + m)/2} \cdot \det(- i 1 + b)^{-(s - s_0)/2}
        \end{equation}
with log branch convention as in \cite[{(1.11)}]{Shimura82}
    (see also \cite[{Remark 13.2.1}]{Chen2024}),
    so we have
        \begin{equation}
        W_{T,\infty}(1, s, \Phi^{m/2}_{\infty})
        =
        \int_{\opn{Sym}_r(\R)} \det(i 1 + b)^{-(s - s_0 + m)/2} \cdot \det(- i 1 + b)^{-(s - s_0)/2}
            \psi_{\infty}(-\opn{tr}(Tb)) \dd b
        \end{equation}
    for the Haar measure $\dd b$ which is self-dual for the pairing
    $(b,c) = \psi_{\infty}(\opn{tr}(bc))$,
    i.e. $2^{r (r - 1)/4}$ times the Euclidean measure.
    
    We have 
        \begin{equation}
        W_{T,\infty}(1,s,\Phi^{m/2}_{\infty})
        = 2^{r (r - 1)/4} \xi(1, T ; \alpha, \beta),
        \quad\quad
        \alpha = (s - s_0 + m)/2
        \quad
        \beta = (s - s_0)/2
        \end{equation}        
    with $\xi(-, -; -, -)$ as in \cite[{(1.25)}]{Shimura82}
    and $s_0 = (m - (r + 1))/2$.
    Recall the Archimedean local $L$-factors
        \begin{equation}
        L_{\R}(s, \pmb{1})
        =
        \pi^{-s/2} \Gamma\left ( \frac{s}{2} \right ),
        \quad\quad
        L_{\R}(s, \opn{sgn})
        =
        \pi^{-(s+1)/2} \Gamma \left ( \frac{s+1}{2} \right ).
        \end{equation}
    Let $(r_1, r_2)$ be the signature of $T$.
We set
        \begin{align}
        W^{\star}_{T,\infty}(\msf{y}, s)
        \coloneqq
        \frac{(2 \pi)^{r(r - 1)/4}}{(-2 \pi i)^{rm/2}}
        \pi^{-r (s - s_0)/2}
        \bigg ( \prod_{j = 0}^{r - 1}
        \Gamma((s - s_0 + m - j)/2) \bigg )
        |\det T|^{-(s + s_0)/2}
        e^{-2 \pi i \opn{tr}(i T \msf{y})} W_{T,\infty}(\msf{y}, s, \Phi^{m/2}_{\infty})
        \end{align}
    if $r$ is odd,
    and we set 
        \begin{align}
        W^{\star}_{T,\infty}(\msf{y}, s)
        & \coloneqq
        \frac{(2 \pi)^{r(r - 1)/4}}{(-2 \pi i)^{rm/2}}
        \pi^{-r (s - s_0)/2}
        \bigg ( \prod_{j = 0}^{r - 1}
        \Gamma((s - s_0 + m - j)/2) \bigg )
        |\det T|^{-(s + s_0)/2}
        e^{-2 \pi i \opn{tr}(i T \msf{y})} W_{T,\infty}(\msf{y}, s, \Phi^{m/2}_{\infty})
        \\
        \notag
        & \mathrel{\phantom{\coloneqq}} \cdot
        \frac{L_{\R}(s + 1/2, \opn{sgn}^{r(r-1)/2})}{L_{\R}(s + 1/2, \opn{sgn}^{r(r-1)/2 + r_2})}
        \end{align}
    if $r$ is even,
    where the normalizing factor should
    be compared with \cite[{(3.3.14)}]{GS19},
    and also Shimura \cite{Shimura82},
    at least without the additional Archimedean $L$-factors.
    For any $a \in \GL_r(\R)$ with $a \cdot {}^t a  = \msf{y}$
    and $\det a > 0$,
    we have 
        \begin{equation}
        W^{\star}_{T,\infty}(\msf{y}, s)
        =
        W^{\star}_{{}^t a T a, \infty}(s) \coloneqq W^{\star}_{{}^t a T a, \infty}(1_r, s).
        \end{equation}

    We have
        \begin{equation}\label{norW-archi-value}
        W^{\star}_{T,\infty}(\msf{y}, s_0)
        =
        \begin{cases}
        1 & \text{if $T$ is positive definite}
        \\
        0 & \text{if $r + 1 \leq m$ and $T$ is not positive definite}
        \end{cases}
        \end{equation}
    and have the functional equation
        \begin{equation}
            \label{equation:Archimedean_Whittaker:functional_equation}
        W^{\star}_{T,\infty}(\msf{y}, s)
        =
        \begin{cases}
        (-1)^{r_1 r_2 / 2} W^{\star}_{T,\infty}(\msf{y}, -s)
        & \text{if $r$ is odd}
        \\
        W^{\star}_{T,\infty}(\msf{y}, -s) & \text{if $r$ is even.}
        \end{cases}
        \end{equation}
as can be shown using
    \cite[{(3.15),(3.7)}]{Shimura82}
    in the positive definite case
    and \cite[{(4.12.K), (4.34.K)}]{Shimura82}
    in the other cases.

    \begin{lemma}\label{lemma: growth estimate}
        Let $T\in\textup{Sym}_r(\mathbb{R})$ be a nondegenerate matrix.
        \begin{itemize}
            \item [(a).] If $T$ is positive definite, for all $\msf{y}\in\textup{Sym}_r(\mathbb{R})_{>0}$ 
            \begin{equation}\label{center-point-SW-estimate}
                \frac{\textup{d}}{\textup{d}s}\bigg|_{s=s_0}W_{T,\infty}^{\star}(t\msf{y},s)=O(t^{-1})\,\,\,\,\textup{as}\,\,\,\,t\rightarrow+\infty.
            \end{equation}
            For all fixed $s^\prime$, there is a positive constant $C$ such that
            \begin{equation}\label{general-point-estimate}
                \frac{\textup{d}}{\textup{d}t}\left[\frac{\textup{d}}{\textup{d}s}\bigg|_{s=s^\prime}W_{T,\infty}^{\star}(t\msf{y},s)\right]=O(t^{-1-C})\,\,\,\,\textup{as}\,\,\,\,t\rightarrow+\infty.
            \end{equation}
            \item[(b).] If $T$ is not positive definite, for all fixed $s^\prime$ and $\msf{y}\in\textup{Sym}_r(\mathbb{R})_{>0}$, there are positive constants $C^\prime$ and $C^{\prime\prime}$ such that
            \begin{equation*}
                \frac{\textup{d}}{\textup{d}s}\bigg|_{s=s^\prime}W_{T,\infty}^{\star}(t\msf{y},s)=O(e^{-C^{\prime}t})\,\,\,\,\textup{and}\,\,\,\,\frac{\textup{d}}{\textup{d}t}\left[\frac{\textup{d}}{\textup{d}s}\bigg|_{s=s^\prime}W_{T,\infty}^{\star}(t\msf{y},s)\right]=O(e^{-C^{\prime\prime}t})\,\,\,\,\textup{as}\,\,\,\,t\rightarrow+\infty.
            \end{equation*}
        \end{itemize}
    \end{lemma}
    \begin{proof}
        We first prove (\ref{center-point-SW-estimate}).
        Set $\mc{W}_{T,\infty}(\msf{y},s)= W_{T,\infty}(\msf{y}, s, \Phi^{m/2}) e^{2\pi\mrm{tr}(T\msf{y})}$.
        By the definition of $W_{T,\infty}^{\star}(\msf{y},s)$,
        we have
        \begin{equation}\label{log-der-nor-W}
            \frac{\left(W_{T,\infty}^{\star}\right)^\prime(\msf{y},s)}{W_{T,\infty}^{\star}(\msf{y},s)}=-\frac{1}{2}\log\det(\pi T)+\frac{\Gamma_r^\prime((s-s_0+m)/2)}{2\Gamma_r((s-s_0+m)/2)}+\frac{\mc{W}_{T,\infty}^\prime(\msf{y},s)}{\mc{W}_{T,\infty}(\msf{y},s)}
        \end{equation}
        where $\Gamma_r(s) \coloneqq \pi^{r(r-1)/4} \prod_{j = 0}^{r - 1} \Gamma(s - j/2)$.
        It is proved in \cite[Proposition 3.2 (iii), (iv)]{GS19} that for $t>0$,
        \begin{equation}\label{log-der-GS-W}
            \frac{\mc{W}_{T,\infty}^\prime(t\msf{y},s)}{\mc{W}_{T,\infty}(t\msf{y},s)}=\frac{1}{2}\log\det(\pi T)-\frac{\Gamma_r^\prime((s-s_0+m)/2)}{2\Gamma_r((s-s_0+m)/2)}+O(t^{-1})\,\,\,\,\textup{as}\,\,\,\,t\rightarrow+\infty.
        \end{equation}
        Then (\ref{center-point-SW-estimate}) follows by combining (\ref{log-der-nor-W}) and (\ref{log-der-GS-W}).

        Equation (\ref{general-point-estimate}) and part ($b$) follow,
        respectively, from Proposition 3.2 (ii) and Proposition 3.3 of loc. cit..
    \end{proof}

\subsection{Normalized non-Archimedean local Whittaker functions}
\label{ssec:normalized_non-Arch_Whitt}

    Suppose $\Phi_{r,p}(-,s) \in I_{r,p}(s)$
    is the normalized spherical section, i.e. the unique
    $K_p$-invariant function with $\Phi_{r,p}(1,s) = 1$.
    Fix a nonnegative $m \in 2 \Z$
    and let $s_0 = (m - (r + 1))/2$ as usual.
    Given any nonsingular $T \in \opn{Sym}_r(\Q_p)$,
    we define the \emph{normalized local Whittaker function}
    $W^{\star}_{T,p}(s)_m \in \C[p^{s}, p^{-s/2}]$
    to be the function given by
        \begin{equation}
        \label{equation:non-Arch_Whittaker:normalized_Den}
        W^{\star}_{T,p}(s - s_0)_m 
        =
        \opn{Den}^{\star}(p^{-s},L)
        \end{equation}
for $L$ any quadratic $\Z_p$-lattice admitting
    a basis with normalized Gram matrix $T$,
    and $\opn{Den}^{\star}(X,L)$
    as defined in Section \ref{nor-local-density}.
    More explicitly, we have
        \begin{align}
        W^{\star}_{T,p}(s)_m 
        &=
        |2|_p^{-r(r-1)/4}
        |4^{\lfloor r / 2 \rfloor} \disc(T)|_p^{-(s + s_0)/2}
        \zeta_p(s + (r + 1)/2)
        \\
        \notag
        & \mathrel{\phantom{=}} \cdot \left ( \prod_{i = 1}^{\lfloor r / 2 \rfloor} \zeta_p(2s + r + 1 - 2i)
        \right )
        W_{T,p}(1,s,\Phi_{r,p})
        \end{align}
    if $r$ is odd and
        \begin{align}
        W^{\star}_{T,p}(s)_m 
        & =
        |2|_p^{-r(r-1)/4}
        |4^{\lfloor r / 2 \rfloor} \disc(T)
        \Delta_T^{-1}|_p^{-(s + s_0)/2}
        L_p(s + 1/2, \eta_T)^{-1}
        \zeta_p(s + (r + 1)/2)
        \\
        \notag
        & \mathrel{\phantom{=}} \cdot \left ( \prod_{i = 1}^{\lfloor r / 2 \rfloor}
        \zeta_p(2s + r + 1 - 2i) \right )
        W_{T,p}(1,s,\Phi_{r,p})
        \end{align}
    if $r$ is even,
    where $\Delta_T \subseteq \Z_p$ is the discriminant
    ideal for the extension $\Q_p[X]/(X^2 - \disc(T))$
    of $\Q_p$
    where $\eta_T \colon \Q_p^{\times} \ra \{ \pm 1 \}$
    is the quadratic character associated to this extension
    of $\Q_p$, and where $\zeta_p(s) \coloneqq (1 - p^{-s})^{-1}$,
    and $L_p(s, \eta_T) = (1 - \chi_p(T) p^{-s})^{-1}$ are the usual
    local $L$-factors.

    The formulas in \eqref{equation:local_Den_functional_equation}
    translate to the functional equations
        \begin{align}\label{equation:non-Archimedean_Whittaker:functional_equation}
        W^{\star}_{T,p}(-s)_m = 
        \begin{cases}
        \pm W^{\star}_{T,p}(s)_m & \text{if $r$ is odd}
        \\
        W^{\star}_{T,p}(s)_m & \text{if $r$ is even}
        \end{cases}
        \end{align}
    where the sign ``$\pm$'' is ``$+$'' if $T$ defines a split quadratic space,
    and is ``$-$'' otherwise.
    We typically lighten notation and write $W^{\star}_{T,p}(s)$
    instead of $W^{\star}_{T,p}(s)_m$, when $m$ is understood.

    If $T \not \in \opn{Sym}_r(\Z_p)^*$,
    we have $W^{\star}_{T,p}(s) = 0$ identically,
    as $\opn{Den}^{\star}(X,L) = 0$ in the above notation
    as because $\opn{Den}(H^k, L) = 0$ for all $k \geq 0$
    via the explanation in Section \ref{ssec:local_den}.

    \subsection{Normalized global Whittaker functions}

    Consider a nonsingular matrix $T \in \opn{Sym}_r(\Q)$,
    for some $r \geq 0$. Fix an integer $m \equiv 0 \pmod{4}$.
    A computation using \eqref{equation:Archimedean_Whittaker:functional_equation}
    and \eqref{equation:non-Archimedean_Whittaker:functional_equation}
    gives the functional equation
        \begin{equation}
        W^{\star}_{T,\infty}(\msf{y}, s) \prod_{p} W^{\star}_{T,p}(s)
        =
        \begin{cases}
        (-1)^{(r^2 - 1)/8} W^{\star}_{T,\infty}(\msf{y}, -s) \prod_{p} W^{\star}_{T,p}(-s) & \text{if $r$ is odd}
        \\
        W^{\star}_{T,\infty}(\msf{y}, -s) \prod_{p} W^{\star}_{T,p}(-s) & \text{if $r$ is even.}
        \end{cases}
        \end{equation}

    Consider any degree $2$ \'etale algebra $E / \Q$
    with associated quadratic character $\eta \colon \A^{\times} \ra \{ \pm 1 \}$.
    Let $\Delta \subseteq \Z$ be the discriminant ideal.
    For a place $v$ of $\Q$, we write $L_v(s,\eta)$ for the associated local $L$-factor,
    and we let $\Lambda(s, \eta) = \prod_v L_v(s, \eta)$ be the completed $L$-function,
    where the product runs over all places $v$.
    Recall that there is a functional equation
        \begin{equation}
        \Lambda(s, \eta) = \epsilon(s,\eta) \Lambda(1-s, \eta)
        \end{equation}
    where $\epsilon(s, \eta) = |\Delta|^{-s+1/2}$.
    If we set 
        \begin{equation}
        \Lambda^*(s,\eta) = |\Delta|^{s/2} \Lambda(s, \eta)
        \end{equation}
    then we have
        \begin{equation}
        \Lambda^*(s, \eta) = \Lambda^*(1 - s, \eta).
        \end{equation}
    If $\eta$ is the trivial character,
    we write $\Lambda(s) = \Lambda(s, \eta)$ for short.

    In the special case of $m = 4$, we compute
        \begin{equation}
        \label{equation:normalized_global_Whittaker:m=4}
        \frac{W^{\star}_{T,\infty}(\msf{y}, s)
        \prod_{p} W^{\star}_{T,p}(s)}{C_T(\msf{y}, s, \Phi^{m/2}_r)}
        =
        \begin{cases}
        -\frac{1}{32 \pi^2} (s + 1)(s + 2) \Lambda(s + 2) \Lambda(2s + 2)   
            & \text{if $r = 3$}
        \\
        \frac{1}{16 \pi^2} (s + 1/2)(s + 3/2) \frac{\Lambda(2s + 1) \Lambda(s + 3/2)}{\Lambda^*(s + 1/2, \eta_T)}
            & \text{if $r = 2$}
        \\
        -\frac{1}{8 \pi} (s + 1) \Lambda(s + 1)
            & \text{if $r = 1$}
\end{cases}
        \end{equation}
    where $\eta_T$ denotes the quadratic character associated to $\Q[X]/(X^2 - \disc T)$.

\subsection{Normalized Eisenstein series}
\label{ssec:normalized_Eisenstein}
The following lemma follows from \cite[Lemma 4.1, Lemma 4.6]{Kudla1988} (see also \cite[Theorem 3.10]{STEIN_2021} and \cite[(4.20)]{KudRallis94}).
\begin{lemma}\label{intertwining-r-2}
For all positive integers $r$, we have
\begin{align}\label{infinity-place}
    &\left(\mathcal{M}_{r,\infty}\Phi_{r,\infty}^{2}\right)(g_{\infty},s)\\
    &=(-1)^{r}\cdot\frac{2^{(1-s)r+r(r-1)/4}\pi^{r(r+1)/2}\Gamma_r(s)}{\Gamma_r(s/2+(r+5)/4)\Gamma_r(s/2+(r-3)/4)}\cdot\Phi_{r,\infty}^{2}(g_{\infty},-s),\,\,\,\,g_{\infty}\in\textup{Sp}(\mathsf{W}_r)(\mathbb{R}).\notag
\end{align} 
    For a non-Archimedean place $p$ of $\mathbb{Q}$, we have
    \begin{equation}\label{finite-place}
        \left(\mathcal{M}_{r,p}\Phi_{r,p}\right)(g_p,s)=\vert2\vert_p^{r(r-1)/4}\Phi_{r,p}(g_{p},-s)\cdot\begin{cases}
            \frac{\zeta_p(s-1)\zeta_p(2s-1)}{\zeta_p(s+2)\zeta_p(2s+2)}, &\textup{if $r=3$;}\\
            \frac{\zeta_p(s-1/2)\zeta_p(2s)}{\zeta_p(s+3/2)\zeta_p(2s+1)}, &\textup{if $r=2$;}\\
            \frac{\zeta_p(s)}{\zeta_p(s+1)}, &\textup{if $r=1$,}
            \end{cases}
    \end{equation}
    where $\zeta_p(s)=(1-p^{-s})^{-1}$. Therefore we have
    \begin{equation}\label{spherical-Phi}
        \left(\mathcal{M}_r\Phi^2_r\right)(g,s)=\Phi^2_r(g,-s)\cdot\begin{cases}
            -\frac{(1-s)(2-s)}{(1+s)(2+s)}\cdot\frac{\Lambda(s-1)\Lambda(2s-1)}{\Lambda(s+2)\Lambda(2s+2)}, &\textup{if $r=3$;}\\
            \frac{(1/2-s)(3/2-s)}{(1/2+s)(3/2+s)}\cdot\frac{\Lambda(s-1/2)\Lambda(2s)}{\Lambda(s+3/2)\Lambda(2s+1)}, &\textup{if $r=2$;}\\
            -\frac{s-1}{s+1}\cdot\frac{\Lambda(s)}{\Lambda(s+1)}, &\textup{if $r=1$.}
        \end{cases}
    \end{equation}
\end{lemma}
\begin{proof}
    We remark that in (\ref{infinity-place}), there is an extra factor of $2^{r(r-1)/4}$ compared to \cite[Lemma 4.6]{Kudla1988}. This factor arises because of the difference between the measure $\textup{d}n(b)$ that we are using and the standard Euclidean measure $\textup{d}x$ in loc. cit.. The extra factor $\vert2\vert_p^{r(r-1)/4}$ in (\ref{finite-place}) also arises in this way (see also Remark \ref{remark:Whittaker_interpolation_and_Den}).
   
    The formula (\ref{spherical-Phi}) follows from combining (\ref{infinity-place}), (\ref{finite-place}), \cite[(5.94)]{GS19}
    and the duplication formula
    \begin{equation*}
        \Gamma(s)\Gamma(s+1/2)=2^{1-2s}\pi^{1/2}\Gamma(2s),\,\,\,s\in\mathbb{C}.
    \end{equation*}
\end{proof}
Define the normalized Eisenstein series to be
\begin{equation}\label{normalized-eis}
    E^{r}(g,s)\coloneqq E(g,s,\Phi_r^{2})\cdot\begin{cases}
        (1+s)(2+s)\Lambda(2+s)\Lambda(2+2s), &\textup{if $r=3$;}\\
        (1/2+s)(3/2+s)\Lambda(s+3/2)\Lambda(2s+1), &\textup{if $r=2$;}\\
        (1+s)\Lambda(s+1), &\textup{if $r=1$.}
    \end{cases}
\end{equation}
By (\ref{spherical-Phi}) and the functional equation $E(g,s,\Phi)=E(g,-s,\mathcal{M}_r(s)\Phi)$ for any section $\Phi\in I_r(s)$, we have
\begin{equation}\label{global-fe}
    E^r(g,-s)=
    \pm
    E^r(g,s)
\end{equation}
where the sign ``$\pm$'' is ``$+$'' if $r = 1, 2$
and is ``$-$'' if $r = 3$.
For $\msf{z}=\msf{x}+i\msf{y}\in\mathbb{H}_r$, define $E^{r}(\msf{z},s)=(\det \msf{y})^{-1}E^{r}(g_\msf{z},s)$. We have the following Fourier expansion by (\ref{fouexpan-classical-eis}):
\begin{equation}
    E^{r}(\msf{z},s)=\sum\limits_{T \in \opn{Sym}_r(\Q)}C_{T}^{r}(\msf{y},s)\cdot q^{T}
\end{equation}
where $C_T^{r}(\msf{y},s)$ is given by normalizing $C_T(\msf{y},s,\Phi_r^{2})$ by the same factor from (\ref{normalized-eis}).
\begin{example}\label{genus1wt2example}
    For $r=1$ and $s=1$, the Eisenstein series $E^1(\msf{z},1)$ is the classical
    non-holomorphic Eisenstein series of weight 2 whose explicit Fourier expansion is
    \begin{equation}\label{genus1wt2}
        E^{1}(\msf{z},1)=\frac{\pi}{3}-\frac{1}{\textup{Im}(\msf{z})}-8\pi\sum\limits_{n=1}^{\infty}\sigma_1(n)q^{n}.
    \end{equation}
\end{example}

\section{Singular Fourier coefficients}
In this section we focus on the Fourier coefficients in the decomposition (\ref{fourier=decomposition}) of the genus $r$ Eisenstein series $E(g,s,\Phi)$ when the matrix $T$ is singular, especially for $m=4$ and $r=3$. 
\subsection{Decomposition of the singular Fourier coefficients}
Let $T \in \opn{Sym}_r(\Q)$ be a matrix
\begin{equation}\label{diagonal-T}
    T=\begin{pmatrix}
        \boldsymbol{0}_{r-t}  & \\
         & T^{\flat}
    \end{pmatrix}
\end{equation}
where $0\leq t \leq r$ and $T^{\flat}\in\textup{Sym}_{t}(\mathbb{Q})$ is nonsingular if $t\neq 0$,
and $t = \opn{rk}(T)$. By \cite[{Lemma 5.4}]{GS19}, there is a decomposition
    \begin{equation}\label{singular-Eis}
    E_T(g,s,\Phi) = \sum_{j = t}^r \sum_{a \in P_{r-t,j-t}(\Q) \backslash \GL_{r - t}(\Q)} B_T^j\left(m\left(\begin{smallmatrix}
        a & \\
         & 1_t
    \end{smallmatrix}\right) g, s,\Phi\right),
    \end{equation}
where $P_{r-t,j-t}(\Q)$ is a parabolic subgroup of $\textup{GL}_{r-t}$ whose explicit expression can be found in \cite[(5.40)]{GS19}. The functions $B_{T}^{j}$ are defined by the following integral
\begin{equation}
    B_{T}^{j}(g,s,\Phi)=\int_{N_j(\mathbb{A})}\Phi\left(\eta_{j}^{r}\left(w_j^{-1} n(b)\right)g,s\right)\psi\left(-\textup{tr}\left(T\left(\begin{smallmatrix}
        \boldsymbol{0}_{r-j} & \\
         & b
    \end{smallmatrix}\right)\right)\right)\textup{d}n(b).
\end{equation}

It turns out that the most important terms in (\ref{singular-Eis}) are the terms corresponding to $j=t$ and $j=r$. For $j=r$, we have the following identity by combining \cite[Lemma 5.4, Remark 5.6, (5.102), (5.103)]{GS19},
\begin{align}\label{j=r}
    B_{T}^r(g,s,\Phi)&=W_{T^{\flat}}\left(e,s-\frac{r-t}{2},[U_{t,r}(s)\circ\mathsf{r}(g)]\Phi\right)\\&=W_{T^{\flat}}\left(e,\frac{r-t}{2}-s,[(\eta_{t}^{r})^{\ast}\circ\mathcal{M}_r(s)\circ\mathsf{r}(g)]\Phi\right)\notag\\&=W_{T^{\flat}}\left(e,\frac{r-t}{2}-s,[(\eta_{t}^{r})^{\ast}\circ\mathsf{r}(g)\circ\mathcal{M}_r(s)]\Phi\right),\notag
\end{align}
where the definition of the operator $U_{j,r}(s)$ can be found in \cite[(5.46), Lemma 5.5]{GS19} (We omit the details because the explicit definition will not be needed.).
For $j=t$, 
\begin{align}\label{j=t}
    B_{T}^t(g,s,\Phi)=W_{T^{\flat}}\left(e,s+\frac{r-t}{2},[(\eta_{t}^{r})^{\ast}\circ\mathsf{r}(g)]\Phi\right).
\end{align}
Moreover, if $t=0$, i.e., $T=\boldsymbol{0}_r$, 
\begin{equation}\label{constant term}
    B_{\boldsymbol{0}_r}^0(g,s,\Phi)=\Phi(g,s),\,\,\,\,B_{\boldsymbol{0}_r}^r(g,s,\Phi)=\left(\mc{M}_r(s)\Phi\right)(g,-s).
\end{equation}
\begin{lemma}\label{first-term}
    Let $\msf{y}=\textup{diag}\{\msf{y}_1,\msf{y}_2\}\in\textup{Sym}_r(\mathbb{R})$
    be a positive definite symmetric matrix where $\msf{y}_1\in\textup{Sym}_{r-t}(\mathbb{R})$
    and $\msf{y}_2\in\textup{Sym}_{t}(\mathbb{R})$.
    Let $h\in\textup{Sp}(\mathsf{W}_{t})(\mathbb{A})$ be an element. For all sections $\Phi\in I_r(s)$,
    we have
    \begin{equation*}
        \left([(\eta_{t}^{r})^{\ast}\circ \mathsf{r}(g_{i\msf{y}})]\Phi\right)\left(h,s-\frac{r-t}{2}\right)=(\det\msf{y}_1)^{s/2+(t+1)/4}\cdot \left([\mathsf{r}(g_{i\msf{y}_2})\circ(\eta_{t}^{r})^{\ast}]\Phi\right)\left(h,s-\frac{r-t}{2}s\right).
    \end{equation*}
\end{lemma}
\begin{proof}
    By the definition of degenerate principal series, we have
    \begin{align*}
        \left([(\eta_{t}^{r})^{\ast}\circ \mathsf{r}(g_{i\msf{y}})]\Phi\right)\left(h,s\right)&=\Phi\left(\eta_{t}^{r}(h)g_{i\msf{y}},s-\frac{r-t}{2}\right)\\
        &=\det\msf{y}_1^{s/2+(t+1)/4}\Phi\left(\eta_{t}^{r}(hg_{i\msf{y}_2}),s-\frac{r-t}{2}\right).
    \end{align*}
    The last expression is $\det\msf{y}_1^{s/2+(t+1)/4}\cdot \left([\mathsf{r}(g_{i\msf{y}_2})\circ(\eta_{t}^{r})^{\ast}]\Phi\right)\left(h,s\right)$.
\end{proof}

\subsection{The case of corank 1}
Let $m=4$, $r=3$ and $t=2$.
We obtain the following lemma by combining the explicit formula of the intertwining operator in Lemma \ref{intertwining-r-2} and Lemma \ref{first-term}.
\begin{lemma}\label{second-term}
     Let $T^\flat\in\textup{Sym}_2(\mathbb{Q})$ be a matrix.
Let $\msf{y}=\textup{diag}\{\msf{y}_1,\msf{y}_2\}\in\textup{Sym}_3(\mathbb{R})$ be a positive definite symmetric matrix where $\msf{y}_1\in\textup{Sym}_{1}(\mathbb{R})$ and $\msf{y}_2\in\textup{Sym}_{2}(\mathbb{R})$. We have 
     \begin{align}
         &W_{T^{\flat}}\left(e,\frac{1}{2}-s,[(\eta_{2}^{3})^{\ast}\circ\mathsf{r}(g_{i\msf{y}})\circ\mathcal{M}_3(s)]\Phi_3^2\right)\\&=-\msf{y}_1^{-s/2+1}\cdot\frac{(1-s)(2-s)}{(1+s)(2+s)}\cdot \frac{\Lambda(s-1)\Lambda(2s-1)}{\Lambda(s+2)\Lambda(2s+2)}\cdot W_{T^{\flat}}\left(g_{i\msf{y}_2},\frac{1}{2}-s,\Phi_2^2\right).\notag
     \end{align}
\end{lemma}

Let $T\in\textup{Sym}_3(\mathbb{Q})$ be of the form in (\ref{diagonal-T})
and consider a positive definite $3\times3$ matrix of the
form $\msf{y}=\textup{diag}\{\msf{y}_1,\msf{y}_2\}\in\textup{Sym}_3(\mathbb{R})$
where $\msf{y}_1\in\textup{Sym}_{1}(\mathbb{R})$ and $\msf{y}_2\in\textup{Sym}_{2}(\mathbb{R})$.
Combining Lemma \ref{first-term} and \ref{second-term} gives
\begin{align}
    E_T^{3}(g_{i\msf{y}},s)=&(1+s)(2+s)\Lambda(2+s)\Lambda(2+2s)\cdot \msf{y}_1^{s/2+1}\cdot W_{T^{\flat}}\left(g_{i\msf{y}_2},\frac{1}{2}+s,\Phi_2^2\right)\\
    -&(1-s)(2-s)\Lambda(2-s)\Lambda(2-2s)\cdot \msf{y}_1^{-s/2+1}\cdot W_{T^{\flat}}\left(g_{i\msf{y}_2},\frac{1}{2}-s,\Phi_2^2\right)\notag\\
    =&\msf{y}_1^{s/2+1}E^{2}_{T^{\flat}}\left(g_{i\msf{y}_2},\frac{1}{2}+s\right)-\msf{y}_1^{-s/2+1}E^{2}_{T^{\flat}}\left(g_{i\msf{y}_2},\frac{1}{2}-s\right).\notag
\end{align}
Therefore
\begin{equation}
    \label{equation:Eisenstein:coeff_corank_1}
    C_T^{3}(\msf{y},s)=\msf{y}_1^{s/2}C_{T^{\flat}}^{2}\left(\msf{y}_2,\frac{1}{2}+s\right)-\msf{y}_1^{-s/2}C_{T^{\flat}}^{2}\left(\msf{y}_2,\frac{1}{2}-s\right).
\end{equation}

\subsection{The case of corank 2}
Let $m=4$, $r=3$ and $t=1$.
\begin{lemma}
    Let $T=\opn{diag}\{\boldsymbol{0}_2,T^\flat\}$ where $T^\flat\in\textup{Sym}_1(\mathbb{Q})\backslash\{(0)\}$, let $\msf{y}=\opn{diag}\{\msf{y}_1,\msf{y}^\flat\}\in\textup{Sym}_3(\mathbb{R})$ be a positive definite symmetric matrix such that $\msf{y}_1\in\textup{Sym}_2(\mathbb{R})$ and $\msf{y}^\flat\in\textup{Sym}_1(\mathbb{R})$. We have
    \begin{align}\label{corank2-C_T}
        C_{T}^{3}(\msf{y},s)&=(\det \msf{y}_1)^{s/2}(1+s)\Lambda(2+2s)C_{T^\flat}^{1}(\msf{y}^\flat,1+s)\\&-(\det \msf{y}_1)^{-s/2}(1-s)\Lambda(2-2s)C_{T^\flat}^{1}(\msf{y}^\flat,1-s)\notag\\&+f_T(\msf{y},s)\notag,
    \end{align}
    where $f_T(\msf{y},-s)=-f_T(\msf{y},s)$ and there exists a positive number $C$ such that
    \begin{equation}\label{estimate-middle}
        \lim\limits_{t\rightarrow+\infty}\frac{\partial f_T}{\partial s}(t \msf{y},0)=0,\,\,\,\,\lim\limits_{t\rightarrow+\infty}\frac{\partial^{2}f_T}{\partial t\partial s}(t \msf{y},0)=O(t^{-1-C}).
    \end{equation}
    Moreover, we have
    \begin{align}\label{diff-3-1}
        \ct\int_{1}^{+\infty}\frac{\textup{d}}{\textup{d}t}\left(\left(C_T^{3}\right)^\prime\left(t\msf{y},0\right)\right)\frac{\textup{d}t}{t^{\rho}}-\frac{\pi}{3}\cdot\ct\int_{1}^{+\infty}\frac{\textup{d}}{\textup{d}t}\left(\left(C_{T^\flat}^{1}\right)^\prime\left(t\msf{y}^{\flat},1\right)\right)\frac{\textup{d}t}{t^{\rho}}=-\frac{\partial f_T}{\partial s}(\msf{y},0).
    \end{align}
\end{lemma}
\begin{proof}
    By a similar strategy as used in Lemma \ref{second-term}, we obtain
    \begin{align}
        E_{T}^{3}(g_{i\msf{y}},s)&=(\det \msf{y}_1)^{s/2+1}(1+s)\Lambda(2+2s)E_{T^{\flat}}^{1}(g_{i\msf{y}^{\flat}},1+s)\\&-(\det \msf{y}_1)^{-s/2+1}(1-s)\Lambda(2-2s)E_{T^{\flat}}^{1}(g_{i\msf{y}^{\flat}},1-s)\notag\\&+\mathcal{F}_T(\msf{y},s)\notag,
    \end{align}
    where $\mathcal{F}_T(\msf{y},s)$ corresponds to the (normalized) ``middle'' terms in (\ref{singular-Eis}),
    i.e. $j = 2$.
    Equation \eqref{corank2-C_T} follows upon setting $f_T(\msf{y},s)=\det\msf{y}^{-1}\mc{F}_T(\msf{y},s)q^{-T}$,
    and the functional equation for $C^3_T(\msf{y}, s)$
    implies that $f_T(\msf{y}, s)$ is an odd function of $s$.
    Formula (\ref{estimate-middle}) follows from \cite[Proposition 5.8]{GS19}.
    
    Finally we prove (\ref{diff-3-1}). By (\ref{corank2-C_T}),
    \begin{align}\label{equation:Eisenstein:coeff_corank_1:deriv}
       \left(C_{T}^{3}\right)^\prime(\msf{y},0)=\left(\frac{\pi}{6}\log\det\msf{y}_1+\frac{\pi}{3}+4\Lambda^{\prime}(2)\right)C_{T^\flat}^{1}(\msf{y}^\flat,1)+\frac{\pi}{3}\left(C_{T^\flat}^1\right)^\prime(\msf{y}^\flat,1)+\frac{\partial f_T}{\partial s}(\msf{y},0).
    \end{align}
    Since $T^\flat\neq(0)$, the coefficient $C^1_{T^\flat}(\msf{y}^\flat,1)$ is independent of the matrix $\msf{y}^\flat$ by (\ref{genus1wt2}), hence for $t>0$,
    \begin{align*}
        &\left(C_{T}^{3}\right)^\prime(t\msf{y},0)-\frac{\pi}{3}\left(C_{T^\flat}^1\right)^\prime(t\msf{y}^\flat,1)\\=&\left(\frac{\pi}{3}\log t+\frac{\pi}{6}\log\det\msf{y}_1+\frac{\pi}{3}+4\Lambda^{\prime}(2)\right)C_{T^\flat}^{1}(\msf{y}^\flat,1)+\frac{\partial f_T}{\partial s}(t\msf{y},0)\\
        =&\left(\frac{\pi}{6}\log\det\msf{y}_1+\frac{\pi}{3}+4\Lambda^{\prime}(2)\right)C_{T^\flat}^{1}(\msf{y}^\flat,1)+\frac{\pi}{3}C_{T^\flat}^{1}(\msf{y}^\flat,1)\log t+\frac{\partial f_T}{\partial s}(t\msf{y},0).
    \end{align*}
    Then (\ref{diff-3-1}) follows from the estimate (\ref{estimate-middle}) and \cite[Lemma 5.11]{GS19}.
\end{proof}

\subsection{The case of corank 3}
Let $m=4$, $r=3$, and $t=0$, so $T=\boldsymbol{0}_3$.
\begin{lemma}\label{rank=0-CT}
    Let $\msf{y}\in\textup{Sym}_3(\mathbb{R})$ be a positive definite symmetric matrix, then
    \begin{align}\label{corank3-C_T}
        C_{\boldsymbol{0}_3}^{3}(\msf{y},s)&=(1+s)(2+s)\Lambda(2+s)\Lambda(2+2s)\det\msf{y}^{s/2}\\&-(1-s)(2-s)\Lambda(2-s)\Lambda(2-2s)\det\msf{y}^{-s/2}\notag\\&+f(\msf{y},s)\notag,
    \end{align}
    where $f(\msf{y},-s)=-f(\msf{y},s)$ and there exists a positive number $C$ such that
    \begin{equation}\label{estimate-middle-0}
        \lim\limits_{t\rightarrow+\infty}\frac{\partial f}{\partial s}(t \msf{y},0)=0,\,\,\,\,\lim\limits_{t\rightarrow+\infty}\frac{\partial^{2}f}{\partial t\partial s}(t \msf{y},0)=O(t^{-1-C}).
    \end{equation}
    Moreover, we have
    \begin{align}\label{diff-3-0}
        \ct\int_{1}^{+\infty}\frac{\textup{d}}{\textup{d}t}\left(\left(C_{\boldsymbol{0}_{3}}^{3}\right)^\prime\left(t\msf{y},0\right)\right)\frac{\textup{d}t}{t^{\rho}}=-\frac{\partial f}{\partial s}(\msf{y},0).
    \end{align}
\end{lemma}
\begin{proof}
    A calculation gives that $\Phi^2_3(g_{i\msf{y}},s)=\det\msf{y}^{s/2+1}$.
    Therefore, by (\ref{constant term}) and Lemma \ref{intertwining-r-2}, 
    \begin{align}
        E_{\boldsymbol{0}_3}^{3}(g_{i\msf{y}},s)&=(1+s)(2+s)\Lambda(2+s)\Lambda(2+2s)\det\msf{y}^{s/2+1}\\&-(1-s)(2-s)\Lambda(2-s)\Lambda(2-2s)\det\msf{y}^{-s/2+1}\notag\\&+\mathcal{F}(\msf{y},s)\notag,
    \end{align}
    where $\mathcal{F}(\msf{y},s)$ is a normalized sum of all the ``middle'' terms in (\ref{singular-Eis}),
    i.e. $j = 1$ and $j = 2$.
    Equation \eqref{corank3-C_T} now follows upon setting.
    $f(\msf{y},s)=\det\msf{y}^{-1}\mc{F}(\msf{y},s)$,
    and the functional equation for $f(\msf{y}, s)$ follows
    from the functional equation for $C^3_{\boldsymbol{0}_3}(\msf{y},s)$.
    Formula (\ref{estimate-middle-0}) follows from \cite[Proposition 5.8]{GS19}.
    
    Finally we prove (\ref{diff-3-0}). By (\ref{corank3-C_T}),
    \begin{align}\label{rk=0-ct}
       \left(C_{\boldsymbol{0}_3}^{3}\right)^\prime(\msf{y},0)=2\pi\Lambda^{\prime}(2)+\frac{\pi^2}{6}+\frac{\pi^2}{18}\log\det\msf{y}+\frac{\partial f}{\partial s}(\msf{y},0).
    \end{align}
    Hence for $t>0$,
    \begin{align*}
        \left(C_{\boldsymbol{0}_3}^{3}\right)^\prime(t\msf{y},0)=2\pi\Lambda^{\prime}(2)+\frac{\pi^2}{6}+\frac{\pi^2}{18}\log\det\msf{y}+\frac{\pi^2}{6}\log t+\frac{\partial f}{\partial s}(t\msf{y},0).
    \end{align*}
    Then (\ref{diff-3-0}) follows from the estimate (\ref{estimate-middle-0}) and \cite[Lemma 5.11]{GS19}.
\end{proof}

 \section{Regularized Siegel--Weil formula}
\subsection{Some Schwartz functions at $\infty$}\label{arc-sch}
Consider the quadratic space $V_{\infty} = \mrm{M_2}(\R)$
with quadratic form $x \mapsto \det(x)$
and symmetric bilinear pairing $(-,-)$
with $\frac{1}{2} (x,x) = \det(x)$.
We specialize some formulas from \cite{GS19}
to these explicit coordinates.

Let $z_0$ be the oriented negative $2$-plane in $M_2(\R)$ with oriented basis
    \begin{equation}
        \left \{ \begin{pmatrix}
        -1 & 0 \\ 0 & 1
        \end{pmatrix},
        \quad
        \begin{pmatrix}
        0 & 1 \\ 1 & 0
        \end{pmatrix}
        \right \}.
    \end{equation}
In the coordinates of Section \ref{ssec:Orthog_and_H^2},
we have $z_0 = (i,i)$.
The Siegel majorant (as in \cite[{(2.43)}]{GS19})
associated to $z_0$ is
    \begin{equation*}
        \|x\|^2=a^{2}+b^{2}+c^{2}+d^{2}
        \quad\quad
        \text{for}
        \quad
        x=\begin{pmatrix}
            a & b\\
            c & d
        \end{pmatrix}\in \textup{M}_2(\mathbb{R})=V_{\infty}.
    \end{equation*}
For $1\leq r\leq 3$ and $\boldsymbol{x}=(x_1,\cdots,x_r) \in V_{\infty}^r$, define the following real-valued functions:
    \begin{align}\label{some-functions}
        \nu(\boldsymbol{x})\coloneqq e^{-\pi\sum\limits_{i=1}^{r}\|x_i\|^2},
        \quad
        R(\boldsymbol{x})\coloneqq\sum\limits_{i=1}^{r}\left(\frac{\|x_i\|^2}{2}-\det(x_i)\right)
        \quad
        T(\bsm{x}) = \frac{1}{2} (\bsm{x}, \bsm{x})
    \end{align}
where $T(\bsm{x})$ is an $r \times r$ symmetric matrix.
When $r = 1$, the function $\nu$ is the Siegel Gaussian
from \cite[{(2.44)}]{GS19}, and we have $R(x) = R(x, z_0)$
where the latter is $R(x, z_0) = h_{z_0}(s_x)$ in the notation of \cite[{\S 2.3}]{GS19}.

We also set
    $\phi(\boldsymbol{x})\coloneqq R(\boldsymbol{x})+2\sum\limits_{i=1}^{r}\det(x_i)$, 
and for $1\leq i \leq r$, define $\xi^r_{\mrm{KM},i} \in \mathscr{S}(V_\infty^{r})$ by:

Let $\boldsymbol{x}=(x_1,\cdots,x_r)$ and $\hat{\boldsymbol{x}}_i=(\cdots,\hat{x}_i,\cdots)$,
\begin{itemize}
    \item For $r=1,2$, 
    \begin{equation}
        \xi_{\km,i}^{r}(\boldsymbol{x})=\nu(\boldsymbol{x})\left(1-2\pi\phi(\hat{\boldsymbol{x}}_i)\right)
    \end{equation}
    \item For $r=3$ and $j < k$, define $f(x_j,x_k)=(a_j+d_j)(c_k-b_k)-(a_k+b_k)(c_j-b_j)$,
    \begin{equation}\label{equ: xi33}
        \xi_{\km,i}^{3}(\boldsymbol{x})=\nu(\boldsymbol{x})\left(1-2\pi\phi(\hat{\boldsymbol{x}}_i) + 2 \pi^2 f(\hat{\boldsymbol{x}}_i)^2 \right)
    \end{equation}
\end{itemize}
We set $\xi_{\textup{KM}}^{r} \coloneqq \sum\limits_{i=1}^{r}\xi_{\textup{KM},i}^r(\boldsymbol{x})\in\mathscr{S}(V_{\infty}^{r})$.
These functions appear naturally in the Kudla--Millson forms on the Hermitian symmetric domain
associated to the quadratic space $V_\infty$ (the original
construction can be found in \cite{KudlaMillson1986, KudlaMillson1987},
see also \cite{GS19} for another construction). We refer the reader to $\S$\ref{green-currents-on-H2},
especially Lemma \ref{form-to-function}, for a more precise statement.

\begin{lemma}\label{KM-1}
    For $1\leq r\leq 3$, the Siegel--Weil section associated to the Schwartz function $\xi_{\textup{KM},r}^{r}$ is a highest weight vector of weight $\lambda_{r}$ where 
    \begin{equation}
        \lambda_{r}=\begin{cases}
            0, &\textup{if $r=1$;}\\
            (2,0), &\textup{if $r=2$;}\\
            (2,2,0), &\textup{if $r=3$.}
        \end{cases}
    \end{equation}
    More precisely, 
    \begin{equation}\label{equ: function to section}
        \Phi^{\lambda_r}(g,s)=\left(\omega_{r,\infty}(g)\xi_{\textup{KM},r}^{r}\right)(0)\cdot\vert\det(a(g))\vert_{\infty}^{s-(3-r)/2}.
    \end{equation}
\end{lemma}
\begin{proof}
    By Lemma \ref{form-to-function} and Remark \ref{rem: distinction}, the Schwartz function $\xi_{\textup{KM},r}^{r}$ is exactly $(-1)^{r-1}\tilde{\nu}_i$ of \cite[(3.87)]{GS19}. Formula (\ref{equ: function to section}) follows from \cite[Lemma 3.4]{GS19}.
\end{proof}

\subsection{An auxiliary Eisenstein series}
The group $H\coloneqq\textup{GSpin}(V)$ is isomorphic to $\textup{GL}_2\times_{\mathbb{G}_m}\textup{GL}_2$. We consider the action of $(h_1,h_2)\in\textup{GL}_2\times_{\mathbb{G}_m}\textup{GL}_2$ on $V$ by the following formula:
\begin{equation}\label{orthogonal-action}
    (h_1,h_2)\curvearrowright x\in V\mapsto \det(h_1)^{-1}\cdot h_1\,x\,{^{t}h_2}.
\end{equation}
This action induces the following exact sequence:
\begin{equation}\label{equation:orthog-action_exact_sequence}
    1\longrightarrow\mathbb{G}_m\longrightarrow H\longrightarrow\textup{SO}(V)\longrightarrow1.
\end{equation}
This action also induces an action of $H(\mathbb{Q}_v)$ (resp. $\textup{O}(V)(\mathbb{Q}_v)$) on $\mathscr{S}(V_v^{r})$ for all positive integers $r$ and all places $v$ of $\mathbb{Q}$ given by
\begin{equation*}
    h=(h_1,h_2)\curvearrowright \varphi\in \mathscr{S}(V_v^{r})\mapsto \left(x\in V_v^{r}\mapsto\varphi(h^{-1}x)\right).
\end{equation*}
We also use $\omega_{r,v}$ to denote this action.

Let $V = V_2 + V_2^\prime$ be a Witt decomposition of $V$ where $V_2$ is a totally isotropic subspace of dimension $2$, whose dual is $V_2^\prime$. Let $Q_2\subset O(V)$ be the parabolic subgroup which stabilizes the subspace $V_2^\prime$. Let $Q_2=M_2\cdot U_2$ where $U_2$ is the unipotent radical of $Q_2$ and $M_2$ is the Levi component which preserves the Witt decomposition \cite[$\S$5.4]{KudRallis94}. In the current situation, we have $M_2\simeq\textup{GL}_2$ (\cite[$\S$3.2]{Kudla1992}). Let $K\subset O(V)(\mathbb{A})$ be a maximal compact subgroup so that we have an Iwasawa decomposition
\begin{equation*}
    O(V)(\mathbb{A})=Q_2(\mathbb{A})\cdot K.
\end{equation*}
For $h\in O(V)(\mathbb{A})$, we write $h=u\cdot m(a)\cdot k$, let $\vert a(h)\vert=\vert\det(a)\vert_{\mathbb{A}}$. Define the Eisenstein series
\begin{equation}
    E(h,s)=\sum\limits_{\gamma\in Q_2(\Q) \backslash O(V)(\Q)} \vert a(\gamma h)\vert^{s+1/2}.
\end{equation}
It is proved in \cite[Proposition 5.4.1]{KudRallis94} that this series is absolutely convergent for $\textup{Re}(s)>\frac{1}{2}$ and has a meromorphic analytic continuation. The Eisenstein series has a simple pole at $s=\frac{1}{2}$ with a non-zero constant residue:
\begin{equation}\label{residue-Eis}
    \kappa=\mathop{\textup{Res}}\limits_{s=1/2}E(h,s).
\end{equation}

\subsection{A (regularized) Siegel--Weil formula}\label{reg-SW}
Kudla--Rallis \cite[Corollary 5.1.2]{KudRallis94} defines an operator $D$ in the center of the universal enveloping algebra for $\textup{Sp}(\mathsf{W}_{3,\infty})$ such that the regularized theta series $\theta(g,h,\omega_{3,\infty}(D)\varphi)$ decreases rapidly on $O(V)(\mathbb{Q})\backslash O(V)(\mathbb{A})$ for all $\varphi\in\mathscr{S}(V(\mathbb{A})^{3})$. Therefore the integral
\begin{equation*}
    \int\limits_{O(V)(\mathbb{Q})\backslash O(V)(\mathbb{A})}\theta(g,h,\omega_{3,\infty}(D)\varphi)E(h,s)\dd h
\end{equation*}
is absolutely convergent for all $s$ for which $E(h,s)$ is holomorphic,
where we give $O(V)(\A)$ the Haar measure with $O(V)(\Q) \backslash O(V)(\A)$ of total volume $1$.
\par
By Howe's result \cite{Howe1989Remarks}, there exists an element $D^{\prime}$ in the center of the universal enveloping algebra for $O(V)(\R)$ such that
\begin{equation*}
    \omega_{3,\infty}(D)=\omega_{3,\infty}(D^{\prime}),
\end{equation*}
where $\omega_{3,\infty}$ on the right means the action of the group $O(V)(\mathbb{R})$ on $\mathscr{S}(V_{\infty}^{3})$, and we have \cite[Lemma 5.5.3]{KudRallis94}
\begin{equation*}
    E(h,s)\ast D^{\pr}=P(s)E(h,s)
\end{equation*}
where $P(s)=\left(s-5/2\right)\left(s+5/2\right)\left(s-1/2\right)\left(s+1/2\right)$ by \cite[Lemma 5.5.3]{KudRallis94}.
Define
\begin{equation}
    \mathcal{E}(g,s,\varphi)\coloneqq\frac{1}{P(s)}\int\limits_{O(V)(\mathbb{Q})\backslash O(V)(\mathbb{A})}\theta(g,h,\omega_{3,\infty}(D)\varphi)E(h,s)\dd h.
\end{equation}
\begin{proposition}\label{regularized-SW}
    Let $\varphi\in\mathscr{S}(V(\mathbb{A})^{3})$ be a Schwartz function. Let $\kappa$ be the constant in (\ref{residue-Eis}). The associated Eisenstein series $E(g,s,\varphi)$ is holomorphic at $s=0$. The regularized theta integral $\mathcal{E}(g,s,\varphi)$ has a pole at $s=\frac{1}{2}$ of order at most $1$ with residue $B_{-1}(g,\varphi)$. Moreover,
     \begin{equation}\label{reg-SW-3}
         E(g,0,\varphi)=2\kappa^{-1}\cdot B_{-1}(g,\varphi)=\frac{1}{6}\int\limits_{O(V)(\mathbb{Q})\backslash O(V)(\mathbb{A})}\theta(g,h,\omega_{3,\infty}(D)\varphi)\dd h.
     \end{equation}
\end{proposition}
\begin{proof}
    The statement can be found in \cite[Theorem 4.1]{HarrisKudla1991}.
\end{proof}

\subsection{Absolutely convergent theta integrals}
\begin{lemma}\label{term-by-term SW}
    Let $\varphi\in\mathscr{S}(V(\mathbb{A})^{3})$ be a Schwartz function. Let $T\in\textup{Sym}_3(\mathbb{Q})$ be a matrix. If the function $\theta_T(g,h,\varphi)$ is absolutely integrable over $O(V)(\mathbb{Q})\backslash O(V)(\mathbb{A})$, then
    \begin{equation}\label{rk3-reg-SW}
        \int\limits_{O(V)(\mathbb{Q})\backslash O(V)(\mathbb{A})}\theta_T(g,h,\varphi)\dd h=\frac{1}{12}\int\limits_{O(V)(\mathbb{Q})\backslash O(V)(\mathbb{A})}\theta_T(g,h,\omega_{3,\infty}(D)\varphi)\dd h=\frac{1}{2}\cdot E_{T}(g,0,\varphi).
    \end{equation}
\end{lemma}
\begin{proof}
    By the absolute convergence of the theta integral, we have
    \begin{align*}
        \int\limits_{O(V)(\mathbb{Q})\backslash O(V)(\mathbb{A})}&\theta_T(g,h,\omega_{3,\infty}(D_3)\varphi)\dd h=\int\limits_{O(V)(\mathbb{Q})\backslash O(V)(\mathbb{A})}\theta_T(g,h,\varphi)\cdot (D_{3}')^{\ast} \{ 1 \}\dd h\\
        &=P\left(\frac{1}{2}\right)\cdot\int\limits_{O(V)(\mathbb{Q})\backslash O(V)(\mathbb{A})}\theta_T(g,h,\varphi)\dd h=12\int\limits_{O(V)(\mathbb{Q})\backslash O(V)(\mathbb{A})}\theta_T(g,h,\varphi)\dd h
    \end{align*}
    and the relation with $E_{T}(g,0,\varphi)$ follows upon applying Proposition \ref{regularized-SW}
    to the left-hand side.    
\end{proof}

\part{Geometric side}
\section{Kudla--Millson theory on the product of upper half-planes}
\label{sec:KM_upper_half}
In a series of papers \cite{KudlaMillson1986,KudlaMillson1987,KM}, Kudla and Millson developed a systematic construction of closed Schwartz forms representing the Poincaré duals of special cycles on symmetric domains of orthogonal and unitary type. García and Sankaran \cite{GS19} constructed Green currents for special cycles on these domains (see also Kudla's work \cite[Proposition 4.10]{Kudla2003} in the divisor case). The product $\mathbb{H}_1^{2}$ is a symmetric domain of orthogonal type. In this section, we make these constructions explicit in this case, following the appendix to Berndt and Kühn's work \cite{BerndtKuehn2012a}.

\subsection{Orthogonal description of \texorpdfstring{$\mathbb{H}_1^{2}$}{$\mathbb{H}_1^{2}$}}
\label{ssec:Orthog_and_H^2}
There is a symmetric bilinear pairing on $V_{\infty}=\textup{M}_2(\mathbb{R})$ given by the following formula:
\begin{equation}\label{equation:coordinates_pairing_M_2}
    \left(\begin{pmatrix}
        a_1 & b_1\\
        c_1 & d_1
    \end{pmatrix},\begin{pmatrix}
        a_2 & b_2\\
        c_2 & d_2
    \end{pmatrix}\right)=a_1d_2+a_2d_1-b_1c_2-b_2c_1.
\end{equation}
Recall that the group
$H(\mathbb{R})\simeq\textup{GL}_2(\mathbb{R})\times_{\mathbb{R}^{\times}}\textup{GL}_2(\mathbb{R})$
has an action on $V_{\infty}$ by the formula (\ref{orthogonal-action}):
\begin{equation*}
    (h_1,h_2)\cdot x\coloneqq (\det h_1)^{-1}h_1\, x\, {^{t}h_2}.
\end{equation*}
Let $\mathbb{D}^{\pm}$ be the
space parameterizing oriented negative $2$-planes $X\subset V_{\infty}=\textup{M}_2(\mathbb{R})$,
which carries an induced action by $H(\R)$.
The space $\mbb{D}^{\pm}$ has two connected components $\mathbb{D}^{+}$ and $\mathbb{D}^{-}$. We will fix an isomorphism $\iota:\mathbb{H}_1^{2}\overset{\sim}{\rightarrow}\mathbb{D}^{+}$ in the following way:
\begin{equation}\label{idH2andD}
    z=(z_1,z_2)\in\mathbb{H}_1^{2}\longmapsto\textup{Span}_{\mathbb{R}}\{\textup{Re}\,Z,\textup{Im}\,Z\},
\end{equation}
where $z_1=x_1+iy_1$, $z_2=x_2+iy_2$ and
\begin{equation}\label{matrixZ}
    Z=\begin{pmatrix}
        z_1z_2 & z_1\\
        z_2 & 1
    \end{pmatrix},\,\,\,\,\textup{Re}\,Z=\begin{pmatrix}
        x_1x_2-y_1y_2 & x_1\\
        x_2 & 1
    \end{pmatrix},\,\,\,\,\textup{Im}\,Z=\begin{pmatrix}
        x_1y_2+x_2y_1 & y_1\\
        y_2 & 0
    \end{pmatrix}.
\end{equation}
The right-hand side of \eqref{idH2andD} should be viewed as
the span of an oriented basis.
The isomorphism (\ref{idH2andD}) also respects
the actions of the group $H(\mathbb{R})^{\circ}\simeq\textup{GL}_2^{+}(\mathbb{R})\times_{\mathbb{R}_+^{\times}}\textup{GL}_2^{+}(\mathbb{R})$
on both sides \cite[(A.4.3)]{BerndtKuehn2012a}.
We remark that \cite[$\S$A.8]{BerndtKuehn2012a} also gives another isomorphism between $\mathbb{H}_1^{2}$ and $\mathbb{D}^{+}$ under a different action of $H$ on $V$. From now on, we denote $\mathbb{D}^{+}$ simply by $\mathbb{D}$.

\subsection{Alternative coordinates for \texorpdfstring{$\mathbb{D}^{\pm}$}{D+-}}
\label{ssec:alternative_coords_D}

A slightly different description
will also be useful.
This description is closer (but not identical)
to \cite[$\S$A.8]{BerndtKuehn2012a}.
This alternative is natural
for complex uniformization of special cycles.
Outside of Section \ref{ssec:alternative_coords_D},
we will default to the conventions of Section \ref{ssec:Orthog_and_H^2}
unless specified otherwise.

Let $W_1 \cong \Q^{\oplus 2}$ and $W_2 \cong \Q^{\oplus 2}$
be $2$-dimensional vector spaces over $\Q$ with chosen bases.
View $W_1$ and $W_2$ as consisting of column vectors
and consider $\opn{M}_2 = V \coloneqq \Hom(W_1,W_2)$.
Equip $W_1$ and $W_2$ with
symplectic pairings $\langle -,- \rangle_1$ and $\langle -,- \rangle_2$
respectively, given by the Gram matrix
    \begin{equation}\label{equation:symplectic_pairing_Gram_matrix}
    \begin{pmatrix}
    0 & 1
    \\
    -1 & 0
    \end{pmatrix}
    \end{equation}
for the given bases.
The symplectic pairings
induce an involution $\Hom(W_1, W_2) \ra \Hom(W_2, W_1)$
denoted $M \mapsto M^{\dagger}$,
i.e. $\langle x, M y \rangle_2 = \langle M^{\dagger} x, y \rangle_1$ for all $x \in W_2$ and $y \in W_1$.
There is an induced symmetric bilinear pairing $(M, M') \coloneqq \opn{tr}(M^{\dagger} M')$ on $\opn{M}_2$,
and this recovers
the formula in \eqref{equation:coordinates_pairing_M_2}.

Recall that we use $\mathbb{H}_1$ to denote the complex upper plane. Let $\mathbb{H}_1^-$
denote the complex lower
half-planes, and write $\mathbb{H}_1^{\pm}$
for disjoint union of $\mathbb{H}_1$ and $\mathbb{H}_1^{-}$.
For $W \coloneqq \Q^{\oplus 2}$,
we view $\mathbb{H}_1^{\pm}$ as parameterizing
Hodge structures $h$ on $W$
with Hodge types $(-1,0)$ and $(0,-1)$
in a standard way, for the following conventions.
We let $z \in \mathbb{H}_1^{\pm}$
correspond to the filtration $F^0 W_{\C} \subseteq W_{\C}$
where $F^0 W_{\C}$ is spanned by the column vector $(z, 1)$.
There is a corresponding complex structure $\C \acts W_{\R}$
where $a \in \C^{\times}$ acts by $h(a^{-1})$.
This is arranged so that if $E_{z}$ is the elliptic curve
$\C / (\Z + \Z z)$ and $W \ra H_1(E_{z}, \Q)$
sends the standard basis to the ordered pair $(1, -z)$,
then $W_{\R} \ra \Lie E_{z}$ is $\C$-linear.

We consider two copies of the above situation
associated to $W_1$ and $W_2$ from before, and write e.g.,
$z_j \in \mathbb{H}_{1,j}^{\pm}$ for $j = 1, 2$.
The spaces $\mathbb{H}_{1,j}^{\pm}$ are also equipped with
\emph{tautological bundles} $\omega_j$,
where $\omega_j$ is the complex line subbundle
of $\mathbb{H}_{1,j}^{\pm} \times W_{j,\C}$ whose fiber
over a point $F^0 W_{j,\C} \in \mathbb{H}_{1,j}^{\pm}$
is the complex line $F^0 W_{j,\C} \subseteq W_{j,\C}$.
Recall that $W_{j,\C}$ has a Hermitian pairing, which we normalize
by the formula $x,y \mapsto \langle \overline{x}, i y \rangle_j$.
We thus obtain a positive definite metric on $\omega_j$
by restriction from $W_{j,\C}$, after changing
the pairing by a sign if necessary (the sign will be locally constant
on $\mathbb{H}_{1,j}^{\pm}$). We write $\widehat{\omega}_{j,\mrm{nat}}$ for
the resulting metrized line bundles.\footnote{
    The subscript $\norm{-}_{\mrm{nat}}$ (``natural'') is to indicate a particular choice of normalization.
    We will eventually also consider $\norm{-}_{\mrm{Pet}}$ (``Petersson'') and $\norm{-}$, with the latter
    appearing in our main theorem.
}
In coordinates, the metric on $\widehat{\omega}_{j,\mrm{nat}}$ has square norm $\norm{-}^2_{\mrm{nat}}$ given by
    \begin{equation}\label{def-metric}
    \norm{\begin{pmatrix} z_j \\ 1 \end{pmatrix}}_{\mrm{nat}}^2
    =
    2 | \opn{Im}(z_j) |
    \quad\quad
    \text{for}
    \quad
    \begin{pmatrix} z_j \\ 1 \end{pmatrix}
    \in F^0 W_{j,\C}.    
    \end{equation}

Next, extend $(-,-)$ to a $\C$-bilinear pairing on $V_{\C}$.
This gives a Hermitian pairing $Y,Z \mapsto (\overline{Y}, Z)$
on $V_{\C}$.

The set $\mbb{D}^{\pm}$ of oriented negative $2$-planes 
$X \subseteq V_{\R}$ is in bijection
with the set of complex lines $F^1 V_{\C} \subseteq V_{\C}$
which are negative definite for the Hermitian pairing
and with $F^1 V_{\C} \perp \overline{F^1 V_{\C}}$
for the Hermitian pairing.
This bijection is canonical up to the following convention:
if $F^1 V_{\C} = \opn{span}_{\C} Z$
for some $Z \in V_{\C}$, we take the associated
oriented $2$-plane $X \subseteq V_{\R}$
with oriented basis $\{ \opn{Re}(Z), \opn{Im}(Z) \}$.
This convention is such that $X \ra X_{\C} / F^1 V_{\C}$
is orientation preserving.

The complex manifold structure
on $\mbb{D}^{\pm}$ is given by the parameterization of
these filtrations $F^1 V_{\C} \subseteq V_{\C}$.
We obtain a \emph{tautological bundle} $\mc{E}$
on $\mbb{D}^{\pm}$, which is the complex line subbundle
of $\mbb{D}^{\pm} \times V_{\C}$
whose fiber over a point $F^1 V_{\C} \in \mbb{D}^{\pm}$
is the complex line $F^1 V_{\C} \subseteq V_{\C}$.
We give $\mc{E}$ a Hermitian metric by restricting
the Hermitian pairing from $V_{\C}$ to $F^1 V_{\C}$
and taking the negative.
We write $\widehat{\mc{E}}_{\mrm{nat}} = (\mc{E}, \norm{-}_{\mrm{nat}})$
for the resulting metrized line bundle.

Giving $F^1 V_{\C} \subseteq V_{\C}$ as in the previous paragraph
is the same as giving a Hodge structure on $V_{\R}$
with Hodge types among $(1,-1), (0,0), (-1,1)$,
with the additional requirements
that the Hodge decomposition is orthogonal
for the Hermitian pairing, that the Hermitian pairing is
negative definite on $F^1 V_{\C}$,
and that we have $\dim_{\C} F^1 V_{\C} = 1$.
Using the definition $V = \Hom(W_1, W_2)$, 
we find that such Hodge structures are exactly those Hodge structures
induced by Hodge structures $h_j$ on $W_j$,
subject to the condition that the Hermitian pairings
    \begin{equation}
    \langle x, i y \rangle_j + i \langle x, y \rangle_j
\end{equation}
on $W_{j,\R}$ are either both positive definite
or both negative definite for $j = 1, 2$.
Here $i$ acts on $W_{j,\R}$ via the complex
structure associated to the Hodge structure $h_j$,
as described above.
For the reader's convenience, we recall that
the above Hermitian pairing on $W_{j,\R}$
agrees with the Hermitian pairing
on $\overline{F^0 W_{j,\C}}$ restricted from $W_{j,\C}$,
under the canonical $\C$-linear
isomorphisms $\overline{F^0 W_{j,\C}} \ra W_{j,\C} / F^0 W_{j,\C}$
and $W_{j,\R} \ra W_{j,\C} / F^0 W_{j,\C}$.

Unwinding the above identification gives
an isomorphism of complex manifolds
    \begin{equation}\label{equation:alternative_coords_D}
    \begin{tikzcd}[row sep = tiny]
    (\mathbb{H}_{1,1} \times \mathbb{H}_{1,2})
    \sqcup
    (\mathbb{H}_{1,1}^- \times \mathbb{H}_{1,2}^-)
    \arrow{r}{i_{\mbb{D}^{\pm}}} &
    \mbb{D}^{\pm}
    \\
    (z_1, z_2) \arrow[mapsto]{r} &
    \opn{span}_{\C}
    \begin{pmatrix}
    z_2 & -z_1 z_2
    \\
    1 & -z_1
    \end{pmatrix}
    \end{tikzcd}
    \quad\quad
    \begin{pmatrix}
    z_2 & -z_1 z_2
    \\
    1 & -z_1
    \end{pmatrix}
    =
    \begin{pmatrix}
    z_2 \\ 1
    \end{pmatrix}
    \begin{pmatrix}
    1 & -z_1
    \end{pmatrix}.
    \end{equation}
In other words, the pair $(F^0 W_{1, \C}, F^0 W_{2,\C})$
is sent to $F^1 V_{\C} \coloneqq \Hom_{\C}(W_{1,\C} / F^0 W_{1, \C}, F^0 W_{2,\C})$,
and we can recover the pair $(F^0 W_{1, \C}, F^0 W_{2,\C})$
from $F^1 V_{\C}$ as the kernel and image of any nonzero
element of $F^1 V_{\C}$.
With this convention,
the connected component $\mbb{D}^+ \subseteq \mbb{D}^{\pm}$
from Section \ref{ssec:Orthog_and_H^2}
corresponds to $\mathbb{H}_{1,1} \times \mathbb{H}_{1,2}$.

Recall that each $\mathbb{H}_{1,j}^{\pm}$ carries
a tautological bundle $\omega_j$, as described above.
There is a second line bundle over each $\mathbb{H}_{1,j}^{\pm}$
given by the quotients $W_{j,\C} / F^0 W_{j,\C}$
as $F^0 W_{j,\C}$ varies.
The fixed symplectic pairing on $W_{j,\C}$
identifies this second line bundle with the dual $\omega_j^{\vee}$,
and this identification is canonical up to a sign convention.\footnote{
\label{footnote:sign_convention_alternating_pairing}
Here, we simply mean the convention on whether the alternating pairing on $W_j$
induces $W_j \ra W_j^{\vee}$ given by $x \mapsto \langle x, -\rangle_j$
or $x \mapsto \langle -, x \rangle_j$.
}
Up to this sign,
we have a canonical identification
of tautological bundles
$i_{\mbb{D}^{\pm}}^* \mc{E} \cong p_1^* \omega_1 \otimes p_2^* \omega_2$
for pullback along $i_{\mbb{D}^{\pm}}$
from \eqref{equation:alternative_coords_D},
where $p_1 \colon \mathbb{H}_{1,1}^{\pm} \times \mathbb{H}_{1,2}^{\pm} \ra \mathbb{H}_{1,1}^{\pm}$ and $p_2 \colon \mathbb{H}_{1,1}^{\pm} \times \mathbb{H}_{1,2}^{\pm} \ra \mathbb{H}_{1,2}^{\pm}$ are the projections.
In the above coordinates, this isomorphism
(for one of the two choices of sign) is
    \begin{equation}\label{equation:alternative_coords_D:taut_bundles}
    \begin{tikzcd}[row sep = tiny]
    p_1^* \omega_1 \otimes p_2^* \omega_2
    \arrow{r}{\sim}
    & 
    i_{\mbb{D}^{\pm}}^* \mc{E}
    \\
    \begin{pmatrix}
    z_1 \\ 1
    \end{pmatrix}
    \otimes
    \begin{pmatrix}
    z_2 \\ 1
    \end{pmatrix}
    \arrow[mapsto]{r}
    &
    \begin{pmatrix}
    z_2 & -z_1 z_2
    \\
    1 & -z_1
    \end{pmatrix}\rlap{~.}
    \end{tikzcd}
    \end{equation}
The other choice of sign gives the negative of this isomorphism.
By construction, the above isomorphism must
preserve the Hermitian metrics.
This can also be verified by direct calculation,
i.e. both sides of \eqref{equation:alternative_coords_D:taut_bundles}
have square norm $4 |\opn{Im}(z_1) \opn{Im}(z_2)|$.
We thus have a canonical (up to sign)
isomorphism of metrized
tautological bundles
    \begin{equation}
    i_{\mbb{D}^{\pm}}^* \widehat{\mc{E}}_{\mrm{nat}}
    \cong p_1^* \widehat{\omega}_{1,\mrm{nat}} \otimes p_2^* \widehat{\omega}_{2,\mrm{nat}}.
    \end{equation}

We next describe actions of $H = \GL_2 \times_{\G_m} \GL_2$
in this setup. View the first and second factors of $\GL_2$
as acting on $W_1$ and $W_2$, respectively.
This induces actions $\GL_2(\R) \acts \mathbb{H}_{1,j}^{\pm}$
for $j = 1,2$ by pushforward of Hodge structures,
and this agrees with the usual action of $\GL_2(\R)$
by M\"obius transformations.
For each tautological bundle $\omega_j$,
we also define an action $\GL_2(\R) \acts \omega_j$
to be $|\det g|^{-1/2}$ times the naive one induced
by $\GL_2(\R) \acts W_{j,\C}$.
This is arranged so that the action preserves Hermitian metrics:
for $g \in \GL_2(\R)$ and $s$ any local section of $\omega_j$,
we have $\norm{g^*(s)}_{\mrm{nat}} = g^* \norm{s}_{\mrm{nat}}$.

There is also an action of $H$ on $V$ given by $(h_1,h_2)\circ x=h_2 x h_1^{-1}$.
Note that this is different
from the action in Section \ref{ssec:Orthog_and_H^2};
the two actions intertwine via
the involution $(h_1, h_2) \mapsto (h_1, J h_2 J^{-1})$
of $\GL_2 \times_{\G_m} \GL_2$
for $J$ the matrix from \eqref{equation:symplectic_pairing_Gram_matrix}.
This induces an action $H \acts \mbb{D}^{\pm}$,
and an action $H \acts \mc{E}$
on the tautological bundle which preserves
the Hermitian metrics:
for $h = (h_1,h_2) \in H(\R)$ and $s$ any local section of $\mc{E}$,
we have $\norm{h^*(s)}_{\mrm{nat}} = h^* \norm{s}_{\mrm{nat}}$.

With these conventions,
the description of $\mbb{D}^{\pm}$ in \eqref{equation:alternative_coords_D}
and the identification of tautological bundles
in \eqref{equation:alternative_coords_D:taut_bundles}
are both $H(\R)$-equivariant.
In these alternative coordinates,
the first Chern form $c_1(\widehat{\mc{E}}^{\vee})$
is also given by the formula \eqref{first-chern} (which was
calculated using the coordinates from Section \ref{ssec:Orthog_and_H^2}),
i.e.,
    \begin{equation}\label{equation:c_1_is_Omega}
    c_1(\widehat{\mc{E}}^{\vee}) = \Omega.
    \end{equation}
We refer to (\ref{first-chern}) for an explicit expression of $\Omega$.

\subsection{Special cycles on \texorpdfstring{$\mbb{D}^{\pm}$}{$\mbb{D}^{\pm}$}}
\label{ssec:local_cycles_sym_space}

We next recall local special cycles
on $\mbb{D}^{\pm}$,
which should trace back to the work
of Hirzebruch and Zagier \cite{HZ76}.
For $r \geq 0$ and $\bsm{x} = (x_1, \ldots, x_r) \in V_{\infty}^r$,
we have the Kudla local special cycle
\begin{equation}
    \mathbb{D}^{\pm}_{\bsm{x}} \coloneqq\{ Z \in \mbb{D}^{\pm} : (Z,x_i)=0\,\,\textup{for $1\leq i\leq r$} \}
        \subseteq \mbb{D}^{\pm}
\end{equation}
where $Z \in \mbb{D}^{\pm}$ denotes an
oriented negative $2$-plane in $V_{\infty}$.
The subset $\mathbb{D}^{\pm}_{\bsm{x}} \subseteq \mbb{D}^{\pm}$
is a complex submanifold.
We set $\mbb{D}_{\bsm{x}} = \mbb{D} \cap \mbb{D}_{\bsm{x}}^{\pm}$.

We have $\mathbb{D}^{\pm}_{\bsm{x}} \neq \emptyset$ if and only if $\opn{span}_{\R} \{x_1, \ldots, x_r \}$
is positive definite or the zero space.
We have $\mathbb{D}^{\pm}_{\bsm{x}}=\bigcap\limits_{i=1}^{r}\mathbb{D}^{\pm}_{x_i}$,
as follows immediately from the definition.
These assertions also hold with $\mbb{D}^{\pm}$ 
replaced by $\mbb{D}$.

Any element $x \in V_{\R}$ 
gives a $\C$-linear functional $(x, -)$ on
$V_{\C}$, and hence a section $s_x$ of
the dual tautological bundle $\mc{E}^{\vee}$ on $\mbb{D}^{\pm}$.
In the notation of Section \ref{ssec:alternative_coords_D},
recall that any point $Z \in \mbb{D}^{\pm}$
gives Hodge filtrations on $V_{\C}$, $W_{1,\C}$, and $W_{2,\C}$.
Given such $x$ and $Z$, the followings are equivalent.
    \begin{enumerate}[(1)]
        \item The map $x \colon W_{1,\R} \ra W_{2,\R}$ preserves Hodge structures.
        \item We have $x \in F^0 V_{\C}$, or equivalently, $x \in F^0 V_{\C}\cap\overline{F^0V_\mathbb{C}}$.
        \item We have $(x, Z) = 0$ (here $Z$ means the corresponding oriented negative 2-plane).
        \item The point $Z \in \mbb{D}^{\pm}$ lies in the zero locus of $s_x$.
    \end{enumerate}
The local special cycle $\mbb{D}_x^{\pm} \subseteq \mbb{D}^{\pm}$
is the locus where these equivalent conditions are satisfied.
For any $h \in H(\R)$,
we have $h \cdot \mbb{D}^{\pm}_{\bsm{x}} = \mbb{D}^{\pm}_{h \cdot \bsm{x}}$
for the actions $H(\R) \acts V_{\infty}$ and $H(\R) \acts \mbb{D}^{\pm}$
from Section \ref{ssec:Orthog_and_H^2}.
The same holds for the alternative action by $H(\R)$
used in Section \ref{ssec:alternative_coords_D}.

In the coordinates $\mathbb{D} \cong \mbb{H}_1^2$
from Section \ref{ssec:Orthog_and_H^2},
if $x=\left(\begin{smallmatrix}
    a & b\\
    c & d
\end{smallmatrix}\right)$,
then the set $\mathbb{D}_x \subseteq \mbb{H}_1^2$
consists of solutions $(z_1, z_2)$ to the equation
\begin{equation}\label{equationDx}
    dz_1z_2-cz_1-bz_2+a=0.
\end{equation}

\subsection{Kudla--Millson form}
\label{ssec:KM:KM_form}

For integers $r \geq 0$, a \emph{Schwartz form} $\alpha$ on $\mbb{D}$ of bidegree $(p,q)$
with coefficients in $\ms{S}(V_{\infty}^r)$
will mean an assignment $\bsm{x} \mapsto \alpha(\bsm{x})$ with
    \begin{equation}\label{equation:KM:KM_form:Schwartz_form}
    \alpha(\bsm{x}) \in A^{p,q}(\mbb{D})
        \quad \text{for all $\bsm{x} \in V_{\infty}^r$},
    \quad\quad
        \text{and}
    \quad\quad
    \alpha(-)_{z} \in \ms{S}(V_{\infty}^r) \otimes \bigwedge^{p,q} T^*_{z} \mbb{D}
        \quad \text{for all $z \in \mbb{D}$}
    \end{equation}
where $A^{p,q}(\mathbb{D})$ is the space of smooth $(p,q)$-forms on $\mathbb{D}$.
The Schwartz form $\alpha$ is said to be $H(\R)^{\circ}$-invariant
if $h^* \alpha (h \bsm{x}) = \alpha(\bsm{x})$ for all $h \in H(\R)^{\circ}$,
and is said to be closed if $\alpha(\bsm{x})$ is closed for all $\bsm{x}$.
For any fixed $z \in \mbb{D}$,
giving an $H(\R)^{\circ}$-invariant Schwartz form of degree $(p,q)$ is the same as
giving an element of $\bigwedge^{p,q} T^*_{z} \mbb{D}$
which is $K^{\circ}_z \coloneqq \opn{Stab}_{H(\R)^{\circ}}(z)$-invariant.

Suppose $\alpha$ is an $H(\R)^{\circ}$-invariant
Schwartz form on $\mbb{D}$ with coefficients in $\ms{S}(V^r_{\infty})$,
not necessarily of pure degree $(p,q)$ (allow sums of forms of different degree).
Then $\dd \alpha$ is also an $H(\R)^{\circ}$-invariant Schwartz form,
which may be described as follows.
Let $\mf{h} \coloneqq \Lie H(\R)^{\circ}$.
Given $z \in \mbb{D}$, set $\mf{k}_z \coloneqq \Lie K^{\circ}_z$,
and consider the associated Cartan 
decomposition $\mf{h} = \mf{p}_z \oplus \mf{k}_z$.
Pick any basis $\{ X_i \}_i$ of $\mf{p}_z$,
with associated dual basis $\{ X_i^* \}_i$ of $\mf{p}_z^*$.
After pulling back to $H(\R)^{\circ}$,
a calculation using the invariant formula for the exterior derivative
shows that $\alpha \mapsto \dd \alpha$ is given by
    \begin{equation}\label{equation:KM:KM_form:d_on_Schwartz_forms}
        \begin{tikzcd}[row sep = tiny]
            \left [ \ms{S}(V^r_{\infty}) \otimes \bigwedge^* T_z^* \mbb{D} \right ]^{K^{\circ}_z}
            \arrow{r} & 
            \left [ \ms{S}(V^r_{\infty}) \otimes \bigwedge^* T_z^* \mbb{D} \right ]^{K^{\circ}_z}
            \\
            \sum_j \phi_j \otimes \beta_j \arrow[mapsto]{r} & \sum_{i,j} (X_i \phi_j) \otimes (X_i^* \wedge \beta_j).
        \end{tikzcd}
    \end{equation}
For this, we use $H(\R)^{\circ}$-invariance of $\alpha$
and the fact that $[\mf{p}_z, \mf{p}_z] \subseteq \mf{k}_z$.

A special case of a main result
from \cite{KudlaMillson1986} and \cite{KudlaMillson1987} gives a certain closed
$H(\mathbb{R})^{\circ}$-invariant Schwartz form $\varphi_{\km}$ on $\mbb{D}$.
We refer to \cite[Theorem 3.1 (ii)]{KudlaMillson1986} for more details.
Following the Appendix of \cite{BerndtKuehn2012a}, we have an explicit expression for $\varphi_{\km}$:
For $z=(z_1,z_2)\in\mathbb{H}_1^{2}\simeq\mathbb{D}$,
and $x=\left(\begin{smallmatrix}
    a & b\\
    c & d
\end{smallmatrix}\right)$,
define
    \begin{equation}\label{maj-pro}
        R(x,z)\coloneqq\frac{\vert dz_1z_2-cz_1-bz_2+a\vert^{2}}{2y_1y_2}
            \quad\quad
        \norm{x}_z^{2} \coloneqq 2 R(x,z) + 2 \det(x),
    \end{equation}
where $R(x,z)$ corresponds to $h_z(s_x)$ from \cite[{\S 2.3}]{GS19},
and $\norm{-}_z$ is the Siegel majorant associated to $z$,
as in \cite[{(2.43)}]{GS19}.
Let $h_{z_i}=\left(\begin{smallmatrix}
    1 & x_i\\
    0 & 1
\end{smallmatrix}\right)\cdot\left(\begin{smallmatrix}
    y_i^{1/2} & 0 \\
    0 & y_{i}^{-1/2}
\end{smallmatrix}\right)\in\textup{SL}_2(\mathbb{R})$,
and define $h_z=(h_{z_1},h_{z_2}) \in H(\R)^{\circ}$.
If $h_{z}^{-1}\cdot x=\left(h_{z_1}\right)^{-1}x\,{^{t}h_{z_2}^{-1}}=\left(\begin{smallmatrix}
        a^{\prime} & b^{\prime}\\
        c^{\prime} & d^{\prime}
    \end{smallmatrix}\right)$, we have $R(x,z)=R(h_{z}^{-1}\cdot x)$ (recall our notation from $\S$\ref{arc-sch}).
    The formula \cite[(B.4.1)]{BerndtKuehn2012a} implies that
\begin{align}
        \label{equation:KM:KM_form:BK_phi_KM}
    \varphi_{\km}(x)_{z}=&\frac{e^{-\pi\left(a^{\prime2}+b^{\prime2}+c^{\prime2}+d^{\prime2}\right)}}{4}\times\\
    &\bigg(\left((a^{\prime}+d^{\pr})^{2}+(c^{\pr}-b^\pr)^{2}-\pi^{-1}\right)\cdot\left(\frac{\textup{d}x_1\wedge\textup{d}y_{1}}{y_1^{2}}+\frac{\textup{d}x_2\wedge\textup{d}y_{2}}{y_2^{2}}\right)\notag\\
    &+\left((a^{\prime}+d^{\pr})^{2} - (c^{\pr}-b^\pr)^{2}\right)\cdot\frac{\textup{d}x_1\wedge\textup{d}y_2-\textup{d}y_1\wedge\textup{d}x_2}{y_1y_2}\notag\\
    &-2\left(a^{\prime}+d^{\prime}\right)\left(c^{\prime}-b^{\prime}\right)\cdot\frac{\textup{d}x_1\wedge\textup{d}x_2+\textup{d}y_1\wedge\textup{d}y_2}{y_1y_2}\bigg).\notag
\end{align}
Our sign is the opposite of \cite[(B.4.1)]{BerndtKuehn2012a};
see the remark on \cite[{Page 70}]{BerndtKuehn2012a}.
From this explicit expression, we have 
\begin{equation}\label{first-chern}
    \varphi_{\km}(0)=\Omega\coloneqq-\frac{1}{4\pi}\cdot\left(\frac{\textup{d}x_1\wedge\textup{d}y_{1}}{y_1^{2}}+\frac{\textup{d}x_2\wedge\textup{d}y_{2}}{y_2^{2}}\right).
\end{equation}
Notice that $-\Omega$ is an invariant Kähler form on $\mathbb{D}$. 

\subsection{Polynomial Fock space}
For fixed $z \in \mbb{D}$, the \emph{polynomial Fock space}
$S(V^r_{\infty})_z \subseteq \ms{S}(V^r_{\infty})$
consists of those Schwartz functions of the form
    \begin{equation}\label{equation:KM:reg_Green:polynomial_Fock}
       \bsm{x} \mapsto ( \text{polynomial in the coordinates of $\bsm{x}$} )
        \cdot
        e^{- \pi \norm{\bsm{x}}_z^2}
    \end{equation}
for $\bsm{x} \in V_{\infty}^r$,
where $\norm{-}_z$ is the Siegel majorant associated to $z$,
as in \eqref{maj-pro}.
The action by the universal enveloping algebra for $\textup{Sp}(\mathsf{W}_{r,\infty})$
preserves $S(V^r_{\infty})_z$ (see e.g. \cite[{\S 3}]{BF10}).

\begin{lemma}\label{Fock-green-converge}
    Let $\alpha$ be an $H(\R)^{\circ}$-invariant
    Schwartz form on $\mbb{D}$ with coefficients
    in the polynomial Fock space.
    For any fixed $\bsm{x} = (x_1, \ldots, x_r)$,
    and for all $\rho \in \C$, the integral
        \begin{equation}
            \int_1^{+\infty}
                \alpha^{\circ}(\sqrt{t} \bsm{x}) \frac{\dd t}{t^{1 + \rho}}
            \quad\quad
            \alpha^{\circ}(\bsm{x})
                \coloneqq
                e^{2\pi \sum\limits_{i=1}^{r}\det(x_i)} \alpha(\bsm{x})
        \end{equation}
    converges uniformly absolutely on compact subsets of $\mbb{D} \setminus \mbb{D}_{\bsm{x}}$
    to a smooth form, and the convergence is also uniform for $\rho$ in compact subsets.
\end{lemma}
\begin{proof}
Without loss of generality, we may assume the polynomial-in-$\bsm{x}$ part of $\alpha(\bsm{x})_{z}$,
    as in \eqref{equation:KM:reg_Green:polynomial_Fock}, is homogeneous of some degree $k \geq 0$
    in the coordinates of $\bsm{x}$ for some $z$ (equivalently, any $z$).
We then have
        \begin{equation}\label{equation:Fock-green-converge}
\alpha^{\circ}(\sqrt{t}\boldsymbol{x})
        =        
        t^k e^{-2 \pi t R(\bsm{x},z)} \cdot (\text{smooth form on $\mbb{D}$ independent of $t$}).
        \end{equation}
    But $R(\bsm{x},z)$ is continuous and positive away from $\mbb{D}_{\bsm{x}}$;
    then note that the integral
        \begin{equation}
        \int_1^{+\infty} t^k e^{-2 \pi c t} \frac{\dd t}{t^{1 + \rho}}
        \end{equation}
    is absolutely convergent for any $c > 0$ and any $\rho \in \C$.
    The same argument shows that any order $z$-partial derivative of \eqref{equation:Fock-green-converge}
    gives an integrand which is uniformly absolutely convergent
    on compact subsets of $\mbb{D} \setminus \mbb{D}_{\bsm{x}}$,
    which implies smoothness.
\end{proof}

\subsection{Green currents of special cycles}\label{green-currents-on-H2}
For $\boldsymbol{x}=(x_1,\cdots,x_r)\in V_\infty^{r}$ where $1\leq r\leq 3$, define
\begin{equation}
    \varphi_\km(\boldsymbol{x})\coloneqq\varphi_\km(x_1)\wedge\cdots\wedge\varphi_\km(x_r).
\end{equation}
This is an $(r,r)$-form on $\mathbb{D}$, and is hence identically $0$ when $r=3$. For $1\leq i\leq r$, define
\begin{equation}\label{equ: def of nut}
    \nu_{r,i}(\boldsymbol{x})=e^{-2\pi\left(R(x_i,z)+\det(x_i)\right)}\cdot\varphi_\km(x_1,\cdots,\widehat{x_i},\cdots x_r)
\end{equation}
and $\nu_r(\boldsymbol{x})=\sum\limits_{i=1}^{r}\nu_{r,i}(\boldsymbol{x})$. Both $\nu_r$ and $\nu_{r,i}$ are $(r-1,r-1)$-smooth forms on $\mathbb{D}$,
and $\nu_r$ is an $H(\R)^{\circ}$-invariant Schwartz form \cite[{Proposition 2.6}]{GS19}. At the point $z=(i,i)$, we have
\begin{equation}\label{sch-func-green}
    \left(\nu_{r,i}(\boldsymbol{x})\wedge\Omega^{3-r}\right)_{(i,i)}=\tilde{\nu}_{r,i}(\boldsymbol{x})\cdot\Omega^{2}_{(i,i)},\,\,\,\left(\nu_r(\boldsymbol{x})\wedge\Omega^{3-r}\right)_{(i,i)}=\tilde{\nu}_r(\boldsymbol{x})\cdot\Omega^{2}_{(i,i)},
\end{equation}
for some Schwartz functions $\tilde{\nu}_{r,i},\tilde{\nu}_r\in\mathscr{S}(V_\infty^{r})$.
With notation as in $\S$\ref{arc-sch}, we have explicit expressions for these Schwartz functions.

\begin{lemma}\label{form-to-function}
    Consider $\boldsymbol{x}=(x_1,\cdots,x_r)\in V_\infty^{r}$ where $1\leq r\leq 3$. We have the following identities:
    \begin{equation}
        \tilde{\nu}_{r,i}(\boldsymbol{x})=\xi^{r}_{\km,i}(\boldsymbol{x}),\,\,\,\,1\leq i\leq r.
    \end{equation}
\end{lemma}
\begin{proof}
    The explicit formula of $\tilde{\nu}_{r,i}$ for $r=1$ can be found in \cite[(B.4)]{BerndtKuehn2012a}.
    For $r=2,3$, the explicit formula follows from its defining equation (\ref{equ: def of nut}).
    The final computation gives the expression exactly in the definition of $\xi^{r}_{\km,i}$
    from $\S$\ref{arc-sch}.
\end{proof}

\begin{remark}\label{rem: distinction}
    In \cite[{(3.87)}]{GS19},
    their $\Omega$ corresponds to our $-\Omega$,
    so that our $\tilde{\nu}_{r, i}(\bsm{x})$ corresponds to $(-1)^{r-1} \tilde{\nu}_i(\bsm{x})$
    of loc. cit..
\end{remark}

Following \cite[$\S$2.6]{GS19}, consider the following integral on $\mathbb{D} \setminus \mathbb{D}_{\boldsymbol{x}}$:
\begin{equation}\label{modified-green-integral}
    \mathfrak{g}_r^{\circ}(\boldsymbol{x},\rho)\coloneqq\int_1^{+\infty}
        \nu_r^{\circ}(\sqrt{t} \bsm{x})
    \frac{\textup{d}t}{t^{1+\rho}}.
    \quad\quad
    \nu_r^{\circ}(\bsm{x}) \coloneqq e^{2 \pi \sum\limits_{i = 1}^r \det(x_i)} \nu_r(\bsm{x}).
\end{equation}
Over $\mathbb{D}\setminus\mathbb{D}_{\boldsymbol{x}}$, this integral converges to a smooth form
for all $\rho\in\mathbb{C}$, e.g. by Lemma \ref{Fock-green-converge}. For the convergence over $\mathbb{D}$, we record the following lemma from \cite{GS19}.

\begin{lemma}[{\cite[Proposition 2.12]{GS19}}]\label{mod-green-converge}
    The integral (\ref{modified-green-integral}) converges to a locally integrable\footnote{
    The smoothness assertion of loc. cit. does not hold, as can be seen in the case of $r = 1$ and $\dim \mbb{D} = 1$.
    In general, at points on $\mbb{D}_{\mathbf{v}}$,
    the claim of boundedness of $\nu^{\circ}(\sqrt{t} \mathbf{v})$ as $t \ra \infty$
    also does not hold, as can be seen e.g. by inspection of
    \eqref{equation:KM:KM_form:BK_phi_KM}.
    The same unboundedness applies to $\varphi^{\circ}(\sqrt{t} \mathbf{v}_i')$ in the proof of loc. cit..
} form on $\mathbb{D}$ if $\textup{Re}(\rho)>0$, and extends meromorphically to the right half-plane $\textup{Re}(\rho)>-1/2$ as a current.  Moreover, if $\mathbb{D}_{\boldsymbol{x}}$ is either empty or of codimension $r$ in $\mathbb{D}$, the integral converges locally uniformly on $\mathbb{D}$ and in $\rho$ to a locally integrable form on $\mathbb{D}$ for $\textup{Re}(\rho)>-1/2$.
\end{lemma}
When $r=1$, the integral (\ref{modified-green-integral}) converges at $\rho=0$ if $x\neq0$,
and $e^{-2\pi\det(x)}\mathfrak{g}^{\circ}_1(x,0)$ recovers Kudla's Green function \cite[(4.28)]{Kudla2003},
as in \cite[{Example 2.10}]{GS19}.
\par
As in \cite[{Definition 2.13}]{GS19},
we use the meromorphic continuation of $\mf{g}^{\circ}_r(\bsm{x}, \rho)$ to define
\begin{equation}\label{green-currents}
    \mathfrak{g}_r^{\circ}(\boldsymbol{x})\coloneqq\underset{\rho=0}{\textup{CT}}\,\mathfrak{g}_r^{\circ}(\boldsymbol{x},\rho),
\end{equation}
where the symbol ``$\underset{\rho=0}{\textup{CT}}$'' means the constant term of the expansion at $\rho=0$,
i.e., for a compactly supported form $\eta\in A^{3-r,3-r}_c(\mathbb{D})$,
\begin{equation}\label{constant-term-1}
    \left\langle\mathfrak{g}_r^{\circ}(\boldsymbol{x}),\eta\right\rangle=\ct\left\langle\mathfrak{g}_r^{\circ}(\boldsymbol{x},\rho),\eta\right\rangle,
\end{equation}
where $\langle\cdot,\cdot\rangle$ denotes the pairing between currents and compactly supported forms.
We also use the notation
\begin{equation*}
    \int_\mathbb{D}\mf{g}_r^{\circ}(\bsm{x})\wedge\eta\coloneqq\left\langle\mf{g}_r^{\circ}(\bsm{x}),\eta\right\rangle,
\end{equation*}
but the readers should keep in mind that we do not claim $\mf{g}^{\circ}_r(\bsm{x})$ is a literal form.

For $r^{\prime}=\textup{dim}_\mathbb{R}\,\textup{span}_\mathbb{R}\{\boldsymbol{x}\}$, Garcia and Sankaran \cite[Proposition 2.12 (vi)]{GS19} proved that $\mathfrak{g}_r^\circ(\boldsymbol{x})$ satisfies
\begin{equation}\label{equ: GS original}
    \textup{dd}^{c}\mathfrak{g}_r^\circ(\boldsymbol{x})
    +\delta_{\mathbb{D}_{\boldsymbol{x}}}\wedge\Omega^{r-r^{\prime}}
    = \varphi_{\mrm{KM}}^{\circ}(\bsm{x})
    \quad\quad
    \varphi_{\mrm{KM}}^{\circ}(\bsm{x})
        \coloneqq e^{2\pi\sum\limits_{i=1}^{r}\det(x_i)}\cdot\varphi_\km(\boldsymbol{x})
\end{equation}
with the usual convention $\textup{dd}^c = -\frac{1}{2 \pi i} \partial \overline{\partial}$.

\section{Special cycles on the integral model}

\subsection{Moduli problems}
    \label{ssec:moduli_problems}
Let $\mathcal{E}ll$ be the moduli stack of elliptic curves. It is a smooth separated Deligne--Mumford stack of finite type over $\Spec\mathbb{Z}$. Let $\pi:\mathcal{A}\rightarrow\mathcal{E}ll$ be the universal elliptic curve over $\mathcal{E}ll$, and define the Hodge bundle on $\mathcal{E}ll$ to be $\omega_{\mathcal{A}}\coloneqq\pi_{\ast}\Omega_{\mathcal{A}/\mathcal{E}ll}^{1}$. This is a line bundle over $\mathcal{E}ll$,
and carries a Hermitian metric in the complex fiber.
We take the normalizations
    \begin{equation}\label{equation:moduli_problems:Archimedean_metric}
    (\alpha, \beta)_{\mrm{nat}} = \frac{i}{2} \int_{A(\C)} \beta \wedge \overline{\alpha}
    \quad\quad
    (\alpha, \beta)_{\mrm{Pet}} \coloneqq \frac{1}{\pi} (\alpha, \beta)_{\mrm{nat}}
    \quad\quad
    (\alpha, \beta) \coloneqq (4 \pi e^{\gamma})^{-1/2} (\alpha, \beta)_{\mrm{Pet}}
    \end{equation}
for $A \in \mc{E}ll(\C)$ and $\alpha, \beta \in H^0(A, \omega_A)$
(viewed as holomorphic $1$-forms on $A(\C)$),
where $\gamma$ is the Euler--Mascheroni constant.
The resulting metrized line bundles for $\norm{-}_{\mrm{nat}}$,
$\norm{-}_{\mrm{Pet}}$, and $\norm{-}$ will be denoted,
$\widehat{\omega}_{\mc{A}, \mrm{nat}}$, $\widehat{\omega}_{\mc{A}, \mrm{Pet}}$,
and $\widehat{\omega}_{\mc{A}}$, respectively,
with the only difference being the normalization of the metric.
\par
Define $\mathcal{M}\coloneqq\mathcal{E}ll\times_{\mathbb{Z}}\mathcal{E}ll$. Denote by $p_1$ (resp. $p_2$) the first (resp. second) projection morphism from $\M$ to $\mathcal{E}ll$.
We define the metrized tautological bundle on $\mc{M}$
to be $\widehat{\mc{E}} \coloneqq p_1^* \widehat{\omega}_{\mc{A}} \boxtimes p_2^* \widehat{\omega}_{\mc{A}}$.
We also set $\widehat{\omega}_j \coloneqq p_j^* \widehat{\omega}_{\mc{A}}$
for $j = 1, 2$.
We also write $\widehat{\mc{E}}_{\mrm{nat}}$ and $\widehat{\mc{E}}_{\mrm{Pet}}$
as helpful, to indicate the same line bundle but with metric defined
using $\norm{-}_{\mrm{nat}}$ or $\norm{-}_{\mrm{Pet}}$.

We write $\overline{\mc{E}ll}$ for the Deligne--Rapoport compactification
of $\mc{E}ll$ as in \cite[{\S III, Remarque 2.6}]{DR73},
i.e. $\overline{\mc{E}ll}$ parameterizes families of generalized elliptic curves
whose geometric fibers are irreducible.
Here $\overline{\mc{E}ll} \ra \Spec \Z$ is a smooth proper Deligne--Mumford stack.
Define $\overline{\mc{M}}=\overline{\mc{E}ll}\times_{\mathbb{Z}}\overline{\mc{E}ll}$,
which has dimension $3$, is also a smooth proper Deligne--Mumford over $\Spec \Z$,
and contains $\mc{M} \hookrightarrow \overline{\mc{M}}$ naturally as a dense open substack. 
Via the classical $j$-invariant,
the coarse moduli scheme associated to $\overline{\mc{M}}$ (resp. $\mc{M}$)
is $\overline{M}\simeq\mathbb{P}_{\mathbb{Z}}^{1}\times_{\mathbb{Z}}\mathbb{P}_{\mathbb{Z}}^{1}$
(resp. $M\simeq\mathbb{A}_\mathbb{Z}^{1}\times_\mathbb{Z}\mathbb{A}_\mathbb{Z}^{1}\simeq\mathbb{A}_{\mathbb{Z}}^{2}$),
as follows from the corresponding statements for $\mc{E}ll$ and $\overline{\mc{E}ll}$, given
in \cite[{\S VI}]{DR73}, and e.g. compatibility of coarse moduli maps with flat base change.

\subsection{Special cycles on $\M$}
    \label{ssec:global_special_cyles}
In the coming sections,
we discuss special cycles and associated Green currents
on the Deligne--Mumford stack $\M$.
\begin{definition}\label{special-cycle}
    For $1 \leq r\leq 3$, let $T=(T_{ij})\in\textup{Sym}_r(\mathbb{Q})$ be a symmetric matrix. Denote by $\mathcal{Z}(T)$ a stack whose fiber category over a scheme $S$ consists of the following objects,
\begin{equation*}
   \left((E,E^{\prime}), \boldsymbol{x}=(x_1,\cdots,x_r)\right),
\end{equation*}
where $\left(E, E^{\prime}\right)$ is an object in $\M(S)$ and $x_i\in\textup{Hom}(E,E^{\prime})$ satisfy
$T = \frac{1}{2} (\bsm{x}, \bsm{x})$, i.e.
\begin{equation*}
    T_{ij}=\frac{1}{2}\left(\textup{deg}(x_{i}+x_{j})-\textup{deg}(x_{i})-\textup{deg}(x_{j})\right)=\frac{1}{2}\left(x_{i}\circ x_{j}^{\vee}+x_{j}\circ x_{i}^{\vee}\right).
\end{equation*}
\end{definition}

There is a natural finite, unramified, and relatively representable (see the works of Howard and Madapusi \cite[Definition 2.2.11]{HM22}) morphism $i_{r}: \mathcal{Z}(T)\rightarrow \M$ given by forgetting the morphisms $\boldsymbol{x}$ of an object $\left((E,E^{\prime}), \boldsymbol{x}\right)$ of $\mathcal{Z}(T)$.
In particular, we know that $\mathcal{Z}(T)$ is a Deligne--Mumford stack.
We call $\mathcal{Z}(T) \ra \mc{M}$ a ``\emph{special cycle}''.
\'Etale locally
on the target,
the morphism $\mathcal{Z}(T) \ra \mc{M}$ becomes a disjoint union of
finitely many closed immersions
(equivalent to the finite unramified and relatively representable
property).

Since the quadratic space $\textup{Hom}(E,E^\prime)$ is positive-definite
(positivity of the Rosati involution), the stack $\mathcal{Z}(T)$ is non-empty only if $T$ is positive semi-definite.
We also see that $\mathcal{Z}(T)$ is empty unless $T$ is the normalized Gram matrix for an integral quadratic form.
By the above definition, the special cycle is $\textup{GL}_r(\mathbb{Z})$-invariant, i.e., for an element $\gamma\in\textup{GL}_r(\mathbb{Z})$, we have a natural isomorphism:
\begin{align}
    \mathcal{Z}(T)&\xrightarrow{\sim}\mathcal{Z}({^{t}\gamma}T\gamma)\label{linear-inv-cycle}\\
    \left((E,E^{\prime}), \boldsymbol{x}\right)&\mapsto\left((E,E^{\prime}), \boldsymbol{x}\cdot\gamma\right).\notag
\end{align}
\par
If $T$ is of the form $\textup{diag}\{\boldsymbol{0}_{r-\textup{rk}(T)}, T^\flat\}$ for a nonsingular matrix $T^\flat$, then we have $\mathcal{Z}(T)=\mathcal{Z}(T^\flat)$
(by positivity of $\textup{Hom}(E,E^\prime)$ again).
For general $T$, it is always possible to find a matrix $\gamma\in\textup{GL}_r(\mathbb{Z})$ such that ${^{t}\gamma}T\gamma$ is of the form $\textup{diag}\{\boldsymbol{0}_{r-\textup{rk}(T)}, T^\flat\}$ for a nonsingular matrix $T^\flat\in\textup{Sym}_{\textup{rk}(T)}(\mathbb{Q})$.
In this situation, we set $\det'(T) \coloneqq \det(T^\flat)$.
This is independent of the choice of $\gamma$. We caution the reader
that this notation $\det'$ differs from e.g. \cite[{Page 929}]{GS19}.

\begin{lemma}\label{correct-codim}
    For $1\leq r\leq 3$, let $T=(T_{ij})\in\textup{Sym}_r(\mathbb{Q})$ be a nonzero matrix. 
    \begin{itemize}
        \item [(a).]
        The morphism $\mathcal{Z}(T) \rightarrow \mathcal{M}$ is a local complete intersection morphism.
        The generic fiber $\mathcal{Z}(T)_{\Q} \rightarrow \Spec \Q$
        is smooth.
        If $\mathcal{Z}(T)$ is nonempty,
        then it is equidimensional with $\dim \mathcal{Z}(T) = \dim \mathcal{M} - \textup{rk}(T)$.
        \item [(b).] If $\textup{rk}(T) = 1$, then $\mathcal{Z}(T) \rightarrow \mathcal{M}$ is a generalized Cartier divisor in the sense of \cite[Definition 2.4.1]{HM22}.
        \item [(c).] If $\textup{rk}(T) = 2$, then $\mathcal{Z}(T) \rightarrow \Spec \mathbb{Z}$ is proper quasi-finite and flat with \'etale generic fiber.
        For any field $k$ and any object
        $\left((E,E^{\prime}), \boldsymbol{x}\right)$ in $\mathcal{Z}(T)(k)$,
        the elliptic curves $E$ and $E'$ both admit an action by $\Q(\sqrt{-\det' T})$. In particular, if $k$ has characteristic $p > 0$, then $E$ and $E'$ are both supersingular (resp. ordinary) if $p$ is nonsplit
        (resp. split) in
        $\Q(\sqrt{-\det' T})$.
        \item [(d).] If $\textup{rk}(T)=3$, if $\mathcal{Z}(T)$ is nonempty, it is supported on the supersingular locus over a unique prime $p$ determined by $T$ in the following way:
\begin{equation}\label{diff-set}
        \{p\} = \textup{Diff}(T)\coloneqq\{p:\,\,T\,\,\textup{cannot be represented in}\,\,\textup{M}_2(\mathbb{Q}_p)\}.
    \end{equation}
    \end{itemize}
\end{lemma}
\begin{proof}
    By the remarks preceding the lemma, it suffices to prove the case
    when $T$ is nonsingular by choosing some $\gamma\in\textup{GL}_r(\mathbb{Z})$
    such that ${^{t}\gamma}T\gamma$ is of the form $\textup{diag}\{\boldsymbol{0}_{r-\textup{rk}(T)}, S\}$ for a nonsingular matrix $S\in\textup{Sym}_{\textup{rk}(T)}(\mathbb{Q})$.
Therefore we assume $T$ is nonsingular for the rest of the proof. 
    \par
    (b) We first consider the case $r=1$. Let $T=(t)$ where $t\in\mathbb{Q}^{\times}$, the special cycle $\mathcal{Z}(T)$ is nonempty if and only if $t\in\mathbb{Z}_{>0}$, and it parametrizes degree $t$ isogenies between elliptic curves. This is exactly the notion of $[t-\textup{Isog}]$ moduli problem introduced by Katz and Mazur \cite[$\S$6.5]{KatzMazur1985}. That the special cycle $\mathcal{Z}(T)$ is a generalized Cartier divisor is essentially proved in \cite[Proposition 6.8.6]{KatzMazur1985}. The flatness over $\mathcal{Z}(T)$ over $\mathbb{Z}$ is proved in Theorem 6.8.1 of \cite{KatzMazur1985}. The projection map $\opn{pr}_1 \colon \mathcal{Z}(T)_{\Q} \ra \mc{E}ll_{\Q}$ is \'etale,
    so $\mathcal{Z}(T)_{\Q} \ra \Spec \Q$ is smooth.
    \par
    For general $r$, there is a finite unramified morphism from $\mathcal{Z}(T)$ to an
    open and closed
    substack of $\mathcal{Z}(T_{11}) \times_{\mathcal{M}} \cdots \times_{\mathcal{M}} \mathcal{Z}(T_{rr})$.
    If $\mathcal{Z}(T)$ is nonempty, then $T$ is positive definite,
    so each $T_{ii}$ is positive.
    Since $\mathcal{M}$ is regular
    and each $\mathcal{Z}(T_{ii}) \rightarrow \mathcal{M}$
    is a generalized Cartier divisor, the local complete intersection
    and dimension claims will follow as soon as we show
    $\dim \mathcal{Z}(T) \leq \dim \mathcal{M} - \textup{rk}(T)$
    (see also \cite[{Proposition 2.4.4}]{HM22}).
    \par
    (c) We next consider the case $r = 2$.
    Suppose $\gamma \in \GL_2(\Q)$ is such that
    ${}^t \gamma T \gamma = \diag(t_1, t_2)$.
    Given any field-valued point $(E, E', \bsm{x})$
    of $\mathcal{Z}(T)$, consider $(x_1', x_2') \coloneqq \bsm{x} \cdot \gamma$.
    If $x \coloneqq x_2^{\prime \vee} \circ x_1'$,
    then $x^{\vee} = - x$ and $x^{\vee} \circ x= \det T\cdot(\det\gamma)^2\equiv\det T$ mod $\mathbb{Q}^{\times2}$,
    so the $\Q$-subalgebra of $\End^{\circ}(E')$ generated by $x$
    is isomorphic to $\Q(\sqrt{-\det T})$.
    Since $E$ and $E'$ are isogenous (via e.g. $x_1'$),
    we conclude
    that $E$ must also have CM by the same field.
    Since every CM elliptic curve may be defined over
    either $\overline{\Q}$ or $\overline{\F}_p$,
    we conclude that every point of $\mathcal{Z}(T)_{\Q}$ is closed
    (in the generic fiber) and that every point of
    $\mathcal{Z}(T)$ in positive characteristic is also closed.
    This implies that the generic fiber
    $\mathcal{Z}(T)_{\Q}$ cannot have any positive
    dimensional components, and that $\mathcal{Z}(T)$ cannot
    have any positive dimensional components supported
    in positive characteristic, i.e. $\dim \mathcal{Z}(T) \leq 1$.
    We now know that $\mathcal{Z}(T) \ra \mc{M}$ is a local complete intersection, and since $\mc{M}$ is regular,
    we also know that $\mathcal{Z}(T)$
    is Cohen--Macaulay, so $\mathcal{Z}(T) \ra \Spec \Z$
    is flat by miracle flatness.

    Every $\overline{\Q}$-point of $\mathcal{Z}(T)$
    descends to $\Spec \mc{O}_L$ for some number field $L$,
    because CM elliptic curves have potentially everywhere good reduction,
    and $\bsm{x}$ extends by the N\'eron mapping property.
    Since $\mathcal{Z}(T) \ra \Spec \Z$ is flat and $0$-dimensional
    in the generic fiber, the generic points of the topological space
    $|\mathcal{Z}(T)|$ are all (images of) $\overline{\Q}$-points and there are only finitely many of them.
    This implies that $\mathcal{Z}(T)$ can be covered
    by finitely many schemes of the form $\Spec \mc{O}_L$,
    which is enough to imply that $\mathcal{Z}(T) \ra \Spec \Z$
    is proper and quasi-finite (we already know
    that $\mathcal{Z}(T) \ra \Spec \Z$ is separated and finite type).
    
    It remains to show that $\mathcal{Z}(T) \ra \Spec \Z$ has \'etale generic fiber. This can be shown
    in the style of e.g. \cite[{Theorem 2.2.1}]{Howard12}:
    for all but finitely many primes $p$,
    the morphism $\mathcal{Z}(T)_{\Spf \Z_p} \ra \Spf \Z_p$
    is formally \'etale, because (for all but finitely many primes $p$) the associated deformation
    problem can be identified (Serre--Tate) with the deformation
    problem for pairs $(X, \iota)$ where $X$ is a $p$-divisible group of height $2$ and dimension $1$, and $\iota$
    is an action on $X$ by the \emph{maximal} order
    in $\Q_p(\sqrt{-\det T})$.
    One can show that this deformation problem is
    formally \'etale using Grothendieck--Messing theory,
    after further discarding the primes which ramify in $\Q(\sqrt{-\det T})$.
    
\par
    (d) For $r=3$, let's assume that $\mathcal{Z}(T)\neq\varnothing$. Let $k$ be an algebraically closed field such that $\mathcal{Z}(T)(k)\neq\varnothing$. We will show that $\textup{char}(k)$ is positive. Let $\left((E,E^{\prime}), (x_1,x_2,x_3)\right)$ be an object in $\mathcal{Z}(T)(k)$. The nonsingularity of $T$ implies that $x_1$, $x_2$ and $x_3$ are three linearly independent vectors in the quadratic space $\textup{Hom}(E,E^{\prime})^{\circ}\coloneqq\textup{Hom}(E,E^{\prime})\otimes_{\mathbb{Z}}\mathbb{Q}$, therefore
    \begin{equation*}
        \dim_\mathbb{Q}\textup{Hom}(E,E^{\prime})^{\circ}\geq3.
    \end{equation*}
    Since the dimension of the space $\textup{Hom}(E,E^{\prime})^{\circ}$ can only be $1$, $2$ or $4$, and it is $4$ if and only if $E$ and $E^{\pr}$ are two isogenous supersingular elliptic curves over $k$, we have $\textup{char}(k)=p$ for some finite prime $p$, and $\left((E,E^{\prime}), (x_1,x_2,x_3)\right)$ lies on the supersingular locus of $\M$ over $p$. Moreover, the quadratic space $\textup{Hom}(E,E^{\prime})^{\circ}$ is isometric to the unique division quaternion algebra ramified at $p$ and $\infty$ over $\mathbb{Q}$. Therefore $\textup{Diff}(T)=\{p\}$ by \cite[Proposition 1.3]{Kud97Ann}.
\end{proof}

\subsection{Special cycles in Chow groups}
\label{ssec: special cycle in chow groups}
Let $Z^r(\mc{M})$ be the group of codimension $r$ cycles on $\mc{M}$,
i.e. formal $\Z$-linear combinations of integral closed substacks of codimension $r$.
As defined, $\mathcal{Z}(T)$ is not an actual cycle on $\mc{M}$,
and it only admits a finite unramified representable morphism $i_r$ to $\mc{M}$.
Define $\mathcal{Z}_{\textup{sub}}(T) \in Z^{\opn{rk} T}(\mc{M})$ be the cycle
given by the stack-theoretic image of $i_r$.

\begin{example}\label{sti-rank1}
    For $T=(t)$ where $t\in\mathbb{Z}_{>0}$, we set $\mc{T}_t \coloneqq \mathcal{Z}_{\mrm{sub}}(T)$.
    Generically, the morphism $i_1:\mathcal{Z}(T)\rightarrow\mc{T}_t$ has degree $2$ because over each object $(E,E^{\prime})$ of $\mc{T}_t$, there are two isogenies $x$ and $-x$ with degree $t$. For later use,
    we denote by $T_t \subseteq M$ (resp. $\overline{T_t} \subseteq \overline{M}$)
    the image of $\mc{T}_t$ (resp. $\overline{\mc{T}}_t$)
under the coarse moduli map $\mc{M} \ra M$ (resp. $\overline{\mc{M}} \ra \overline{M}$).
Alternatively, $T_t$ is
    defined by the product of all modular polynomials of degree $t/d^{2}$ where $d$ ranges over all positive integers such that $d^2\vert t$.
\end{example}

\subsection{Complex uniformization}
    \label{ssec:complex_uniformization}

Let $H \coloneqq \GL_2 \times_{\G_m} \GL_2$.
We write $h = (h_1, h_2)$ for elements of $H$.
By the \emph{standard open compact subgroup} of $H(\A_f)$,
we mean $K^{\circ} \coloneqq H(\hat{\Z})$.
We are mostly interested in the level group $K^{\circ}$,
but will make occasional use of general
open compact $K \subseteq K^{\circ}$ for the purpose of obtaining
a scheme rather than a stack.

Let $\mtt{S}$ be a (possibly infinite) set of primes,
and write e.g. $\Z[1/\mtt{S}]$ for the localization of $\Z$
given by inverting the multiplicative set generated by elements of $\mtt{S}$.
We say that an open compact subgroup $K \subseteq K^{\circ}$
is \emph{good away from $\mtt{S}$} if
$K = K_{\mtt{S}} \times K^{\mtt{S}}$
for open compact subgroups $K_{\mtt{S}} \subseteq H(\A_{f,\mtt{S}})$
and $K^{\mtt{S}} \subseteq H(\A_f^{\mtt{S}})$
with $K^{\mtt{S}} = K^{\circ \mtt{S}}$.
In particular, $K^{\circ}$ is good away from $\mtt{S} = \emptyset$.

Let $L \coloneqq \Z^{\oplus 2}$ be the $\Z$-lattice of rank $2$
with chosen basis, and equip $L$ with the symplectic form
with Gram matrix
    \begin{equation}
    \begin{pmatrix}
    0 & 1
    \\
    -1 & 0
    \end{pmatrix}
    \end{equation}
in this basis. Let $L'$ be another copy of $L$,
with the same alternating pairing.

Given an open compact subgroup $K \subseteq K^{\circ}$
good away from $\mtt{S}$,
we define a stack in groupoids $\mc{M}_K \ra \Spec \Z[1/\mtt{S}]$
whose fiber category $\mc{M}_K(S)$ 
over a connected locally Noetherian
scheme $S$ over $\Spec \Z[1/\mtt{S}]$ 
is the groupoid of tuples $(E, E', \overline{\eta}, \overline{\eta}')$ with
    \begin{align*}
    &
    E, E'
    &&
        \text{elliptic curves over $S$}
    \\
    &
    (\overline{\eta}, \overline{\eta}')
    &&
        \text{a $\pi_1(S,\overline{s})$-stable $K_{\mtt{S}}$-orbit of isomorphisms of $\hat{\Z}_{\mtt{S}}$-modules }
            \\
            & &&
            \quad
            \eta \colon L_{\hat{\Z}_{\mtt{S}}} \xra{\sim} \prod_{\ell \in \mtt{S}} T_{\ell}(E)
            \quad\quad
            \eta' \colon L'_{\hat{\Z}_{\mtt{S}}} \xra{\sim} \prod_{\ell \in \mtt{S}} T_{\ell}(E')
            \\
        & && \text{which preserve alternating pairings up to a common scalar}
    \end{align*}
where $E$ and $E'$ are equipped with their principal polarizations,
which induce alternating pairings on their Tate modules,
and where $\overline{s}$ is any choice of geometric point on $S$;
the choice of $\overline{s}$ does not matter and is suppressed from notation.
The isomorphisms in $\mc{M}_K(S)$
are pairs of isomorphisms of elliptic curves
which respect the orbits $(\overline{\eta}, \overline{\eta}')$.
As usual, the alternating pairings on Tate modules depend on trivializations
of roots of unity, but the above definition does not depend on this trivialization as we only require identification of alternating pairings
up to common scalar.

There is a smooth separated Deligne--Mumford stack of finite type
$\mc{M}_{K} \ra \Spec \Z[1/\mtt{S}]$ representing the above moduli problem
(after restriction to connected locally Noetherian test objects)
and the forgetful map $\mc{M}_{K} \ra \mc{M}[1/\mtt{S}]$
is finite \'etale of degree $[K^{\circ} : K]$.
Indeed, one can show that the given moduli description of $\mc{M}_K$
gives an \'etale local system on $\mc{M}$ (such a local system may
be described only on connected locally Noetherian test objects).
The morphism $\mc{M}_K \ra \mc{M}$ is determined by
its description on connected locally Noetherian test objects.
We also have a (somewhat indirect) description of
the groupoid $\mc{M}_K(S)$ for an arbitrary
scheme $S$ over $\Spec \Z[1/\mtt{S}]$ (not necessarily
connected or locally Noetherian): we can recover $\mc{M}_K \ra \mc{M}$
from its ``sheaf of sections'', i.e. we obtain a
``sheaf of $K$-level structures'' $\opn{Level}_K$ on $\mc{M}$
such that $\mc{M}_K(S)$ parameterizes pairs $(E,E') \in \mc{M}(S)$
along with a section of $\opn{Level}_K$ over $(E,E')$,
and after restricting to connected locally Noetherian $S$,
sections of $\opn{Level}_K$ may be functorially identified with
data $(\overline{\eta}, \overline{\eta}')$ as above.

The formation of $\mc{M}_K$ is compatible with enlargement
of $\mtt{S}$: in the situation above, if $\overline{s}$
is a point in characteristic $p$ (possibly $p = 0$),
then at any $\ell \not \in \mtt{S}$ with $\ell \neq p$, we have a
unique $\pi_1(S, \overline{s})$-stable $K_\ell$-orbit
of isomorphisms $\eta_{\ell} \colon L_{\Z_{\ell}} \ra T_{\ell}(E)$
and $\eta'_{\ell} \colon L'_{\Z_{\ell}} \ra T_{\ell}(E')$
which preserve alternating pairings up to the same scalar.
When $K = K^{\circ}$, we have $\mc{M} = \mc{M}_K$.
For sufficiently small $K$,
the stack $\mc{M}_K$ is (representable by)
a scheme. Whenever $\mc{M}_K$ is a scheme,
it is also quasi-projective
over $\Spec \Z[1/\mtt{S}]$.
It is enough for $K$ to act trivially on $L / N L$ and $L' / N' L'$
for some integers $N, N' \geq 3$.

For any $T \in \opn{Sym}_r(\Q)$ with $1 \leq r \leq 3$,
recall that we have a ``special cycle'' $\mathcal{Z}(T) \ra \mc{M}$
as defined in Section \ref{ssec:global_special_cyles}.
We write $\mathcal{Z}(T)_K \ra \mc{M}_K$ for the pullback over $\mc{M}_K$.

We have complex uniformization of $\mc{M}_K$ and $\mathcal{Z}(T)_K \ra \mc{M}_K$
in the standard way, as we now describe.
In Section \ref{ssec:complex_uniformization},
we take the conventions of Section \ref{ssec:alternative_coords_D},
except we slightly switch notation
and write $W, W'$ instead of $W_1, W_2$, etc.
Take the canonical embeddings $L \ra W$ and $L' \ra W'$,
using the given bases
(i.e. $\Z^{\oplus 2} \hookrightarrow \Q^{\oplus 2}$).
Fix $r$, and consider the characteristic function
$\varphi \coloneqq \pmb{1}_{\Hom(L, L') \otimes_{\Z} \hat{\Z}}^{\otimes r}
\in C^{\infty}_c(V(\A_f)^r)$.
Given any $r$-tuple $\bsm{x}^{\infty} \in V(\A_f)^r$,
we define the ``away-from-$\infty$ special cycle'' (nonstandard terminology)
to be the discrete set
    \begin{equation}
    \mc{Z}(\bsm{x}^{\infty})_K
    \coloneqq
    \{
    h \in H(\A_f) / K : \varphi(h^{-1} \bsm{x}^{\infty}) = 1
    \}.
    \end{equation}

Consider the setoid $\mathcal{Z}(T)_K(\C)_{\mrm{framed}}$,
consisting of data
    \begin{equation}
    \label{equation:complex_uniformization:framed_points}
    \begin{split}
    & (E, E', \bsm{x}, \overline{\eta}, \overline{\eta}')
    \in \mathcal{Z}(T)_K(\C)
    \\
    & \phi \colon H_1(E, \Q) \xra{\sim} W
    \quad\quad
    \phi' \colon H_1(E', \Q) \xra{\sim} W'
    \end{split}
    \end{equation}
where $\phi$ and $\phi'$ are isomorphisms of $\Q$ vector spaces
which preserve symplectic pairings up to a common scalar,
where the alternating pairings
on $H_1(E,\Q)$ come from the principal polarizations on $E$ and $E'$.
We call $(\phi, \phi')$ a ``framing''.
Morphisms in this category are morphisms
in $\mathcal{Z}(T)_K(\C)$ which are compatible with $(\phi, \phi')$.
The group $H(\Q)$ acts on the isomorphism classes
of $\mathcal{Z}(T)_K(\C)$ by acting on framings via
the formula $h \cdot (\phi, \phi') = (h_1 \phi, h_2 \phi')$
for $h = (h_1, h_2) \in H(\Q)$.

Hodge theory provides a canonical bijection
    \begin{equation}\label{equation:complex_uniformization:framed_adelic_description}
    \mathcal{Z}(T)_K(\C)_{\mrm{framed}} \xra{\sim} \coprod_{\substack{\bsm{x} \in V(\Q)^r \\ \frac{1}{2} (\bsm{x}, \bsm{x}) = T}} \mbb{D}^{\pm}_{\bsm{x}} \times \mc{Z}(\bsm{x}^{\infty})_K
    \end{equation}
which is $H(\Q)$-equivariant (after passing to isomorphism classes),
described as follows.
Given an object as in \eqref{equation:complex_uniformization:framed_points},
we obtain $\bsm{x} \in V(\Q)^r$ and $h = (h_1, h_2) \in \mc{Z}(\bsm{x}^{\infty})_K$ defined by the commutative diagram
    \begin{equation}
    \begin{tikzcd}
    W_{\A_f} & \arrow{l} W \arrow{d}{\bsm{x}} & \arrow{l}[swap]{\phi} H_1(E, \Q) & \arrow{l} H_1(E, \Z) \arrow{d}{\bsm{x}} \arrow{r} & \prod_{\ell} T_{\ell}(E) &
    \arrow{l}[swap]{\eta} L_{\hat{\Z}} \arrow{r} & W_{\A_f}
    \arrow[bend right = 10]{llllll}[swap]{h_1}
    \\
    W'_{\A_f} & \arrow{l} W' & \arrow{l}{\phi'} H_1(E', \Q) & \arrow{l} H_1(E', \Z) \arrow{r} & \prod_{\ell} T_{\ell}(E') \arrow{r} &
    \arrow{l}{\eta'} L'_{\hat{\Z}} \arrow{r} & W'_{\A_f}.
    \arrow[bend left = 10]{llllll}{h_2}
    \end{tikzcd}
    \end{equation}
A datum as in \eqref{equation:complex_uniformization:framed_points}
also gives a point of $\mbb{D}^{\pm}_{\bsm{x}}$,
as the Hodge structures on $H_1(E, \Q)$ and $H_1(E', \Q)$
induce Hodge structures on $W$ and $W'$ respectively (using $\phi$ and $\phi'$).
Here we are using the identification $(\mathbb{H}_1^+ \times \mathbb{H}_1^{\prime +}) \sqcup (\mathbb{H}_1^- \times \mathbb{H}_1^{\prime -}) \cong \mbb{D}^{\pm}$
from Section \ref{ssec:alternative_coords_D}
and the description of local special cycles 
from Section \ref{ssec:local_cycles_sym_space}.

The canonical forgetful map $\mathcal{Z}(T)_K(\C)_{\mrm{framed}} \ra \mathcal{Z}(T)_K(\C)$
now induces a canonical isomorphism (equivalence) of orbifolds\footnote{
    For us, a (complex) \emph{orbifold} is a stack over complex analytic spaces
    with the \'etale topology, which admits a representable \'etale cover 
    by complex manifolds.
}
    \begin{equation}\label{equation:complex_uniformization:special_cycle}
    \bigg [ H(\Q) \backslash \bigg ( \coprod_{\substack{\bsm{x} \in V(\Q)^r \\ \frac{1}{2} (\bsm{x}, \bsm{x}) = T}} \mbb{D}^{\pm}_{\bsm{x}} \times \mc{Z}(\bsm{x}^{\infty})_K \bigg ) \bigg ]
    \xra{\sim}
    \mathcal{Z}(T)_{K,\C}^{\mrm{an}}
    \end{equation}
where $\mathcal{Z}(T)_{K,\C}^{\mrm{an}}$ means analytification
of $\mathcal{Z}(T)_{K,\C}$.
The case $T = \boldsymbol{0}_r$
gives complex uniformization
    \begin{equation}
    \bigg [ H(\Q) \backslash \bigg ( \mbb{D}^{\pm} \times H(\A_f) / K \bigg ) \bigg ]
    \xra{\sim}
    \mc{M}_{K,\C}^{\mrm{an}} \rlap{~.}
    \end{equation}
Outside of Section \ref{ssec:complex_uniformization},
we will typically be looser and just write $\mathcal{Z}(T)_K(\C)$
and $\mc{M}_K(\C)$ instead of $\mathcal{Z}(T)^{\mrm{an}}_{K,\C}$
and $\mc{M}^{\mrm{an}}_{K,\C}$, etc..

The complex uniformization also gives an analytic
description of the metrized tautological
bundle $\widehat{\mc{E}}$ on $\mc{M}$
in the complex fiber.
We spell this out, as some care is needed for the metrics to align.
To simplify slightly, we assume that $K$ is chosen
such that $\mc{M}_K$ is a scheme.

Given a point $(E, E', \overline{\eta}, \overline{\eta}') \in \mc{M}_K(\C)$,
the analytification of $\widehat{\omega}_{1, \mrm{nat}}$
has fiber over this point which may be identified
with the line $F^0 H_1(E,\C) \subseteq H_1(E,\C)$.
This identification is canonical up to
a sign convention (Footnote \ref{footnote:sign_convention_alternating_pairing});
pick either of the two conventions.
The same holds for $\widehat{\omega}_{2, \mrm{nat}}$ and $F^0 H_1(E', \C)$.
Given a lift $(E, E', \overline{\eta}, \overline{\eta}', \phi, \phi') \in \mc{M}_K(\C)_{\mrm{framed}}$, the $\C$-lines $F^0 W_{\C} \subseteq W_{\C}$ and $F^0 W'_{\C} \subseteq W'_{\C}$.
are (by definition) the images $\phi( F^0 H_1(E,\C) )$ and $\phi' (F^0 H_1(E',\C))$.
Recall that the maps $\phi$ and $\phi'$
preserve alternating pairing only up to scalar.
If ``$\det \phi \in \Q^{\times}$'' and ``$\det \phi' \in \Q^{\times}$''
are the similitude factors (e.g. the alternating pairing on $W$ pulls
back to $\det \phi$ times the alternating pairing on $H_1(E,\Q)$,
then rescaling gives isomorphisms
    \begin{equation}
    F^0 W_{\C} \xleftarrow{|\det \phi|^{-1/2} \phi} F^0(E,\C)
    \quad\quad
    F^0 W'_{\C} \xleftarrow{|\det \phi'|^{-1/2} \phi'} F^0(E',\C)
    \end{equation}
which preserve Hermitian pairings exactly.
This pair of isomorphisms is $H(\Q)$-equivariant
(recall the rescaled action from Section \ref{ssec:alternative_coords_D}).
In other words, we have written down an $H(\Q)$-equivariant metric-preserving
isomorphism between (1) the analytification of $\widehat{\mc{E}}_{\mrm{nat}}$ on $\mc{M}_K$,
pulled back to $\mbb{D}^{\pm} \times H(\A_f) / K$ and (2) the metrized tautological
bundle on $\mbb{D}^{\pm}$
(from Section \ref{ssec:alternative_coords_D})
pulled back to $\mbb{D}^{\pm} \times H(\A_f) / K$.

We can form the first Chern form $c_1(\widehat{\mc{E}}^{\vee}_{\mrm{nat}})
= c_1(\widehat{\mc{E}}^{\vee}_{\mrm{Pet}}) = c_1(\widehat{\mc{E}}^{\vee})$
on $\mc{M}_{K,\C}^{\mrm{an}}$. Since the $(1,1)$-form $c_1(\widehat{\mc{E}}^{\vee})$
on $\mbb{D}^{\pm}$ is $H(\R)$-invariant, it descends (uniquely)
to a $(1,1)$-form $\Omega \coloneqq c_1(\widehat{\mc{E}}^{\vee})$
on the orbifold $\mc{M}_{K,\C}^{\mrm{an}}$ even for possibly stacky level $K$.

\subsection{Degrees of special cycles}
    \label{ssec:complex_fiber_degree}
    For a matrix $T \in \opn{Sym}_r(\Q)$, the generic fiber $\mathcal{Z}(T)_{\mathbb{Q}}$ is nonempty only if $\textup{rk}(T)\leq 2$. Define the degree of $\mathcal{Z}(T)$ (resp. $\mc{Z}(T)$) when $\mrm{rk}(T)\leq 2$ with respect to the line bundle $\mc{E}$ in the following way:
    \begin{equation}\label{degree-special-cycle}
        \deg_{\mc{E}^{\vee}_{\mathbb{Q}}}(\mathcal{Z}(T)_{\mathbb{Q}})\coloneqq\int\limits_{\mathcal{Z}(T)(\mathbb{C})}c_1(i_r^{\ast}\widehat{\mc{E}}^\vee)^{2-\textup{rk}(T)}.
    \end{equation}
\subsubsection{The case of rank $0$}
Let $T=\boldsymbol{0}_r$. The special cycle is $\mathcal{Z}(T)=\mc{M}$.
We have the usual complex uniformizations $\mc{M}(\C) \cong [\opn{SL}_2(\Z) \backslash \mathbb{H}_1]^2$
and $M(\C) \cong (\opn{SL}_2(\Z) \backslash \mathbb{H}_1)^2$,
where $[-]$ means stacky quotient and $(-)$ means non-stacky quotient.
Using the quotient integral formula and the trivial action of the size $4$
group $\{ \pm 1 \} \times \{ \pm 1 \} \subseteq \opn{SL}_2(\Z) \times \opn{SL}_2(\Z)$
on $\mathbb{H}_1 \times \mathbb{H}_1$,
we compute
\begin{equation}
    \deg_{\mc{E}^{\vee}_{\mathbb{Q}}}(\mathcal{Z}(T)_{\mathbb{Q}})=\int\limits_{\mc{M}(\mathbb{C})}c_1(\widehat{\mc{E}}^\vee)^2=\frac{1}{4}\int\limits_{M(\C)}\Omega^2=\frac{1}{4}\int\limits_{\left(\textup{SL}_2(\mathbb{Z})\backslash\mathbb{H}_1\right)^{2}}\frac{\textup{d}x_1\mrm{d}x_2\textup{d}y_1\mrm{d}y_2}{8\pi^{2}y_1^{2}y_2^2}=\frac{1}{288}.
\end{equation}

\subsubsection{The case of rank $1$}
Consider $T \in \opn{Sym}_r(\Q)$ with $\opn{rk}(T) = 1$. Let $\alpha\in\textup{GL}_r(\mathbb{Z})$ be a matrix such that (\ref{reduce-to-diagonal}) holds, i.e., ${^{t}\alpha}T\alpha=\diag\{\boldsymbol{0}_{r-1},T^\flat=(t)\}$. Clearly we have $\mathcal{Z}(T)=\mathcal{Z}(T^{\flat})$. Generically, the morphism $i_1:\mathcal{Z}(T^{\flat})\rightarrow\mc{Z}_{\textup{sub}}(T^{\flat})=\mc{T}_t$ has degree $2$, while the generic automorphism group of $\mc{T}_t$ has size $4$,
\begin{equation}\label{reduce-rk-1}
    \deg_{\mc{E}^{\vee}_{\mathbb{Q}}}(\mathcal{Z}(T)_{\mathbb{Q}})=\deg_{\mc{E}^{\vee}_{\mathbb{Q}}}(\mathcal{Z}(T^\flat)_{\mathbb{Q}})=2\deg_{\mc{E}^{\vee}_{\mathbb{Q}}}\left(\mathcal{T}_{t,{\mathbb{Q}}}\right)=\frac{1}{2}\int\limits_{\vert\mc{T}_t(\mathbb{C})\vert}\Omega\vert_{\vert\mc{T}_t(\mathbb{C})\vert}.
\end{equation}
\begin{lemma}\label{lemma:rk1-deg}
    For $T \in \opn{Sym}_r(\Q)$ with $\opn{rk}(T) = 1$ and all $\msf{y}^\flat\in\textup{Sym}_1(\mathbb{R})_{>0}$,
    \begin{equation}\label{rk1-deg}
        \deg_{\mc{E}^{\vee}_{\mathbb{Q}}}(\mathcal{Z}(T)_{\mathbb{Q}})=(96\pi)^{-1}C_{T^\flat}^{1}(\msf{y}^\flat,1).
    \end{equation}
\end{lemma}
\begin{proof}
    If $T^\flat$ is not positive definite and integral, the special cycle $\mathcal{Z}(T)=\mathcal{Z}(T^\flat)=\varnothing$. The Fourier coefficient $C_{T^\flat}^{1}(\msf{y}^\flat,1)$ is $0$ by Example \ref{genus1wt2example}. Hence both sides of (\ref{rk1-deg}) are $0$. Now we consider the case that $T^\flat$ is positive definite and integral. By \cite[Theorem 4.23]{Kudla2003}, there exists a number $C$ such that
    \begin{equation}\label{rank1-degree-linear-relation}
        \deg_{\mc{E}^{\vee}_{\mathbb{Q}}}(\mathcal{Z}(T^\flat)_{\mathbb{Q}})=C\cdot C_{T^{\flat}}^1(\msf{y}^\flat,1).
    \end{equation}
    For $T^\flat=(1)$, the left hand side of (\ref{rank1-degree-linear-relation}) is given by (\ref{reduce-rk-1}) as
    \begin{equation*}
        \frac{1}{2}\int\limits_{\vert\mc{T}_1(\mathbb{C})\vert}\Omega\vert_{\vert\mc{T}_1(\mathbb{C})\vert}=\frac{1}{2}\int\limits_{\textup{SL}_2(\mathbb{Z})\backslash\mathbb{H}_1}-\frac{\textup{d}x\wedge\textup{d}y}{2\pi y^{2}}=-\frac{1}{12}.
    \end{equation*}
    However, the right hand side of (\ref{rank1-degree-linear-relation}) is $C\times-8\pi$. Hence $C=(96\pi)^{-1}$.
\end{proof}
    
    \subsubsection{The case of rank $2$}
Consider $T \in \opn{Sym}_r(\Q)$ with $\opn{rk}(T) = 2$.
By Lemma \ref{correct-codim},
we know that $\mathcal{Z}(T) \ra \Spec \Z$ is flat proper quasi-finite
with \'etale generic fiber. We will need to compute the (stacky)
degree $\deg_{\mc{E}^{\vee}_{\mathbb{Q}}}\left(\mathcal{Z}(T)_{\Q}\right) \in \Q$ of the generic fiber $\mathcal{Z}(T)_{\Q} \ra \Spec \Q$
(in the sense of e.g. \cite[{(A.1.11)}]{Chen2024}, which should largely
go back to Vistoli).

\begin{lemma}
    \label{lemma:complex_degree:corank_1}
    Let $T\in\textup{Sym}_2(\mathbb{Q})$ be a positive definite matrix. For all $\msf{y}\in\textup{Sym}_2(\mathbb{R})_{>0}$,
    \begin{equation}\label{rk2-geo-SW}
        \deg_{\mc{E}^{\vee}_{\mathbb{Q}}}\left(\mathcal{Z}(T)_{\Q}\right)=\deg_{\mc{E}^{\vee}_{\mathbb{Q}}}\left(\mc{Z}(T)_{\Q}\right)=(16\pi^{2})^{-1}C_T^2\left(\msf{y},\frac{1}{2}\right).
    \end{equation}
\end{lemma}
\begin{proof}
    Denote by $H(\Delta)$ the Hurwitz class number,
    i.e. the weighted number of equivalence classes of positive definite binary quadratic forms with discriminant $-\Delta$. For half-integral positive definite matrix $T\in\textup{Sym}_2(\mathbb{Q})$, the following formula follows from \cite[(2.19)]{GK93} and \cite[Proposition 1.1]{Gortz2007Sum} (see also \cite[Theorem 7.3]{CY}),
    \begin{align*}
        \deg_{\mc{E}^{\vee}_{\mathbb{Q}}}\left(\mathcal{Z}(T)_{\Q}\right)&=\sum\limits_{\alpha\in\vert\mathcal{Z}(T)(\mathbb{C})\vert}\frac{1}{\#\textup{Aut}(\alpha)}\\
        &=\sum\limits_{(E,E^\prime)\in\vert\mathcal{Z}_{\textup{sub}}(T)(\mathbb{C})\vert}\frac{R_{\textup{Hom}(E,E^\prime)}(T)}{\#\textup{Aut}(E)\cdot\#\textup{Aut}(E)}\\
        &=\sum\limits_{d\vert\textup{cont}(T)}d\cdot H\left(\frac{4\det(T)}{d^{2}}\right),
    \end{align*}
    where $R_{\textup{Hom}(E,E^\prime)}(T)$ is the representation number of the quadratic form corresponding to $T$ in the quadratic space $\textup{Hom}(E,E^\prime)$, and for $T=\left(\begin{smallmatrix}
        t_1 & t/2\\
        t/2 & t_2
    \end{smallmatrix}\right)$, the number $\textup{cont}(T)=\gcd(t_1,t_2,t)$.

    By \cite[Theorem 2.1.1]{Nagaoka1992}, 
    \begin{equation*}
        C_T\left(\msf{y},\frac{1}{2},\Phi_2^2\right)=288\sum\limits_{d\vert\textup{cont}(T)}d\cdot H\left(\frac{4\det(T)}{d^{2}}\right).
    \end{equation*}
    Hence
    \begin{align*}
        \deg_{\mc{E}^{\vee}_{\mathbb{Q}}}(\mathcal{Z}(T)_{\Q})=288^{-1}C_T\left(\msf{y},\frac{1}{2},\Phi_2^2\right)=(16\pi^{2})^{-1}C_T^2\left(\msf{y},\frac{1}{2}\right).
    \end{align*}
    If $T$ is positive definite but not half-integral, both sides of (\ref{rk2-geo-SW}) are $0$.
    Therefore (\ref{rk2-geo-SW}) holds for all positive definite matrices $T\in\textup{Sym}_2(\mathbb{Q})$.
\end{proof}

\section{Green currents and arithmetic cycles}
\label{sec:Green_and_arith}
\subsection{Hermitian vector bundles}\label{her-vec-bundle}
We recall some basics on stacky degrees of Hermitian line bundles,
which is summarized and specialized from \cite[$\S$4.1]{Chen2024}.

Given a smooth algebraic stack $\mathcal{X}$ over $\Spec\mathbb{C}$, a Hermitian vector bundle
$\widehat{\mathcal{L}}$ on $\mathcal{X}$ is the following functorial assignment: for every morphism $S\rightarrow\mathcal{X}$ with $S$ a scheme smooth over $\Spec\mathbb{C}$, the assignment gives a vector bundle on $S$ with a smooth Hermitian metric on the analytification. We sometimes write $\widehat{\mathcal{L}} = (\mathcal{L},\|-\|)$ where $\mathcal{L}$ is the underlying line bundle on $\mathcal{X}$ and $\|-\|$ is the norm associated with the smooth Hermitian metric (on the analytification), defined functorially.

Given an algebraic stack $\mc{X}$ which is flat and finite type over $\Spec \Z$
with smooth generic fiber, a \emph{Hermitian vector bundle} on $\mc{X}$ is a vector bundle
$\mc{L}$ on $\mc{X}$ equipped with a smooth Hermitian metric on $\mc{E}|_{\mc{X}_{\C}}$
which is complex conjugation invariant.
\par
Next, we discuss stacky degrees of Hermitian line bundles along the lines of \cite[$\S$4.1]{Chen2024}.
Assume that $\mc{Z}$ is a reduced 1-dimensional Deligne--Mumford stack which is proper and flat over $\Spec\mathbb{Z}$. Let $\widehat{\mathcal{L}} = (\mathcal{L},\|-\|)$ be a Hermitian line bundle on $\mc{Z}$. The Arakelov arithmetic degree $\widehat{\deg}(\widehat{\mathcal{L}})$ of $\widehat{\mathcal{L}}$ on $\mc{Z}$ is valued in $\mathbb{R}$, and may be described as follows.  If $\mc{Z} = \Spec \mathcal{O}_E$ for a number field $E$, we have the standard definition
\begin{equation}\label{ari-degree}
    \widehat{\deg}(\widehat{\mathcal{L}})\coloneqq-\sum\limits_{v}\log\|s\|_v\,\,\,\textup{with}\,\,\,\|s\|\coloneqq q_v^{-\textup{ord}_v(s)}\,\,\textup{if}\,\,v<\infty,
\end{equation}
where the sum runs over all places $v$ of $E$, the quantity $q_v$ is the cardinality of the residue field at $v$, and $s$ is any rational section of the line bundle $\mathcal{L}$.
\par
If $\mc{Z}$ is integral (equivalently, reduced and irreducible by quasi-separatedness), there exists a finite surjection $\operatorname{Spec} \mathcal{O}_E \to \mc{Z}$ and set
\begin{equation}
    \widehat{\deg}(\widehat{\mathcal{L}}):= \frac{1}{\deg(E/\mc{Z}_\mathbb{Q})} \widehat{\deg}\left(\mathcal{L}|_{\operatorname{Spec} \mathcal{O}_E}\right)
\end{equation}
where $\deg(E/\mc{Z}_\mathbb{Q})$ denotes the degree of the finite étale morphism $\operatorname{Spec} E \to \mc{Z}_K$. This definition does not depend on the choice of $E$ or the morphism $\operatorname{Spec} \mathcal{O}_E \to \mc{Z}$,
since any two such choices may be covered by finite surjections
from a third such choice $\operatorname{Spec} \mathcal{O}_{E'}$
(and these finite surjections can be made compatible with the maps to $\mc{Z}$).
If $\mc{Z}$ is not necessarily irreducible, we define $\widehat{\deg}(\widehat{\mc{L}})$
by summing over the irreducible components of $\mc{Z}$.
If $\mc{X}$ a Deligne--Mumford stack which is flat, finite type,
and quasi-separated over $\Spec \Z$, equipped with a Hermitian line bundle $\widehat{\mc{L}}$,
and if $Z$ is a cycle on $\mc{X}$
whose components satisfy the conditions imposed on $\mc{Z}$ from above,
then we define $\widehat{\deg}(\widehat{\mc{L}}|_Z)$ by summing
over the components of $Z$.

\subsection{Green currents on $\M$}
    \label{ssec:Green_current_def}
Now we are going to construct Green currents for $\mathcal{Z}(T)$ following \cite{GS19}. For a positive definite matrix $y=a\cdot{^ta}\in\textup{Sym}_r(\mathbb{R})$ where $a\in\textup{GL}_r(\mathbb{R})$ and a complex number $\rho\in\mathbb{C}$ such that $\textup{Re}(\rho)>-1/2$, consider the following summation of currents on $\mathbb{D}$:
    \begin{equation}\label{green-currents-T}
        \mathfrak{g}^{\circ}(T,\msf{y},\rho)=\sum\limits_{\boldsymbol{x}\in\Omega_T(V)\cap\textup{M}_2(\mathbb{Z})^{r}}\mf{g}_{r}^{\circ}(\boldsymbol{x}\cdot a,\rho).
    \end{equation}
The goal of this section is to prove the convergence and the meromorphic continuation property of the above summation. The summation \eqref{green-currents-T} is closely related to another summation:
\begin{align}\label{nuT}
         \nu^{\circ}(T,\msf{y})\coloneqq\sum\limits_{\boldsymbol{x}\in\Omega_T(V)\cap\textup{M}_2(\mathbb{Z})^{r}}e^{2\pi \textup{tr}(T\msf{y})}
             \nu_{r}(\boldsymbol{x}\cdot a),
     \end{align}
the above summation \eqref{nuT} converges to a $\textup{SL}_2(\mathbb{Z})\times\textup{SL}_2(\mathbb{Z})$-invariant smooth $(r-1,r-1)$-form on $\mathbb{D}$ since $\nu_r$ is a $(r-1,r-1)$ smooth form on $\mathbb{D}$ with coefficients in the polynomial Fock space.

We have the following convergence lemma, 
as in \cite[Proposition 4.3]{GS19}.
For completeness, we give a self-contained proof, with an explicit bound
on $\opn{Re}(\rho)$, and also give the interchangeability of the $\bsm{x}$
sum and $t$ integral.

\begin{lemma}\label{sum-con-current}
\begin{itemize}
    \item [(a).]The summation (\ref{green-currents-T}) converges for $\textup{Re}(\rho)>0$ to a $\textup{SL}_2(\mathbb{Z})\times\textup{SL}_2(\mathbb{Z})$-invariant locally integrable form on $\mathbb{D}$, and 
\begin{equation}\label{int-first-sum}
         \mathfrak{g}^{\circ}(T,\msf{y},\rho)=\int_{1}^{+\infty}\nu^{\circ}(T,t\msf{y})\frac{\textup{d}t}{t^{1+\rho}}.
     \end{equation}
    \item[(b).] The summation \eqref{green-currents-T} extends meromorphically\footnote{
        For the purpose of this paper, the notion of ``meromorphic family of currents''
        is as in \cite[{Footnote 2}]{GS19}.
    }
    to the right half-plane $\textup{Re}(\rho)>-1/2$ as a current on $\mathbb{D}$,
    which is holomorphic except for a possible simple pole at $\rho = 0$.
\end{itemize}
\end{lemma}
\begin{proof}
    Let $U\subset\mathbb{D}$ be an open subset whose closure $\overline{U}$ is compact. The following set is finite by Lemma \ref{finite-bound}:
    \begin{equation*}
        S=\{\boldsymbol{x}\in \textup{M}_2(\mathbb{Z})^{r}: T(\boldsymbol{x})=T\,\,\textup{and}\,\,\mathbb{D}_{\boldsymbol{x}}\cap \overline{U}\neq\varnothing\}.
    \end{equation*}
    Let $S^\pr=\Omega_T(V)\cap\textup{M}_2(\mathbb{Z})^{r}\setminus S$, so that $S \sqcup S^\pr$ indexes the non-zero terms appearing on the right hand side of (\ref{green-currents-T}). By Lemma \ref{finite-bound}, there exists a number $C > 0$ such that for all $\boldsymbol{x}=(x_1,\cdots,x_r)\in S^\pr$,
    \begin{equation}\label{R>C}
        R(\boldsymbol{x},z)=\sum\limits_{i=1}^rR(x_i,z)>C,\,\,z\in\overline{U}.
    \end{equation}
    Similar inequality also holds with $\boldsymbol{x}$ replaced by $\boldsymbol{x}\cdot a$ and $C$ by another positive number $C^{\pr}$.
    
    Consider the following sums:
    \begin{align}\label{sum-partial}
        A=\sum\limits_{\boldsymbol{x}\in S}\mf{g}^{\circ}_r(\boldsymbol{x}\cdot a,\rho),\quad\quad B=\sum\limits_{\boldsymbol{x}\in S^{\pr}}\mf{g}^{\circ}_r(\boldsymbol{x}\cdot a,\rho)
    \end{align}
   
    For the summation $A$, it converges to a locally integrable form when $\textup{Re}(\rho)>0$ and extends meromorphically to the right half-plane $\textup{Re}(\rho)>-1/2$ as a current by the finiteness of the set $S$ and Lemma \ref{mod-green-converge}. Lemma \ref{mod-green-converge} also implies that for $\mrm{Re}(\rho)>0$,
    \begin{equation*}
        A=\sum\limits_{\boldsymbol{x}\in S}\int_{1}^{+\infty}\nu^\circ(\sqrt{t}\boldsymbol{x}\cdot a)\frac{\mrm{d}t}{t^{1+\rho}}=\int_{1}^{+\infty}\sum\limits_{\boldsymbol{x}\in S}\nu^\circ(\sqrt{t}\boldsymbol{x}\cdot a)\frac{\mrm{d}t}{t^{1+\rho}}
    \end{equation*}
    
    For the summation $B$, each term $\mf{g}^{\circ}_r(\boldsymbol{x}\cdot a,\rho)$
    is a smooth form over $U$ (because $\boldsymbol{x}\notin S$) for all $\rho\in\mathbb{C}$
    by Lemma \ref{Fock-green-converge}. Let $z\in\overline{U}$ be a point, and fix a basis of
    the vector space $\bigwedge^{r-1,r-1}T_z^{\ast}\mathbb{D}$.
    The coordinates of the smooth forms $\mf{g}^{\circ}_r(\boldsymbol{x}\cdot a,\rho)$
    under this basis are bounded by finite linear combinations of
integrals of the following type:
        \begin{equation}\label{coordinate-bound-z}
            \int_1^{+\infty}\left(R(\sqrt{t}\boldsymbol{x}\cdot a,z)+t\cdot\mrm{tr}(T\msf{y})\right)^{k}\times e^{-2\pi t R(\boldsymbol{x}\cdot a,z)}\frac{\textup{d}t}{t^{1+\rho}},\quad\quad k\in\mathbb{Z}_{>0}.
        \end{equation} 
    Now we try to estimate these coordinates by
    estimating (\ref{coordinate-bound-z}). Firstly, we have $R(\boldsymbol{x},z)>0$ since $\boldsymbol{x}\in S'$, then each integral \eqref{coordinate-bound-z} converges for all $\rho\in\mathbb{C}$.
    
    Moreover, by the compactness of $\overline{U}$, there exist positive numbers $a_U$ and $b_U$ such that
    \begin{equation*}
        a_U\cdot\|\boldsymbol{x}\|^2\leq R(\boldsymbol{x}\cdot a,z)+\mrm{tr}(T\msf{y})\leq b_U\cdot\|\boldsymbol{x}\|^2.
    \end{equation*}
    Then except finitely many $\boldsymbol{x}$, we have $R(\boldsymbol{x},z)\geq a_U/2\cdot\|\boldsymbol{x}\|^2$ for all the other $\boldsymbol{x}\in S'$. For those $\boldsymbol{x}$, the integral (\ref{coordinate-bound-z}) is bounded in the following way:
    \begin{align*}
        &\int_1^{+\infty}t^{k}\left(R(\boldsymbol{x}\cdot a,z)+\mrm{tr}(T\msf{y})\right)^{k}\times e^{-2\pi t R(\boldsymbol{x}\cdot a,z)}\frac{\textup{d}t}{t^{1+\rho}}\leq\int_{1}^{+\infty}b_U^k\|\boldsymbol{x}\|^{2k}\times e^{-a_U\pi t \|\boldsymbol{x}\|^2}\frac{\textup{d}t}{t^{1+\rho-k}}\\
        &\leq C_{\rho,k}b_U^k\|\boldsymbol{x}\|^{2k}\int_{1}^{+\infty} e^{-a_U\pi t \|\boldsymbol{x}\|^2/2}\textup{d}t=2C_{\rho,k}b_U^k(a_U\pi)^{-1}\|\boldsymbol{x}\|^{2k-2}\cdot e^{-a_U\pi\|\boldsymbol{x}\|^2/2},
    \end{align*}
    where $C_{\rho,k}$ is a positive number such that
    \begin{equation*}
        \frac{1}{t^{1+\rho-k}}\leq C_{\rho,k}\cdot e^{a_U\pi t\|\boldsymbol{x}\|^2/2},\,\,\textup{for all}\,\,t\geq1.
    \end{equation*}
    Then we can find a positive number $C_{\rho,U}$ depending only on $\rho$ and $U$, such that for all $\boldsymbol{x}\in S^\pr$, the coordinates of $\mf{g}^{\circ}_r(\boldsymbol{x}\cdot a,\rho)$ are bounded by the following function
    \begin{equation*}
        C_{\rho, U}\cdot e^{-\frac{a_U\pi}{4}\cdot\|\boldsymbol{x}\|^2},\,\,\,\textup{for all}\,\,z\in\overline{U}.
    \end{equation*}
    Therefore the sum $B$ converges absolutely for all $\rho\in\mathbb{C}$ to a smooth form on $U$ by standard arguments for the convergence of theta series. Hence the following identity holds for all $\rho\in\mathbb{C}$,
    \begin{equation*}
        B=\sum\limits_{\boldsymbol{x}\in S'}\int_{1}^{+\infty}\nu^\circ(\sqrt{t}\boldsymbol{x}\cdot a)\frac{\mrm{d}t}{t^{1+\rho}}=\int_{1}^{+\infty}\sum\limits_{\boldsymbol{x}\in S'}\nu^\circ(\sqrt{t}\boldsymbol{x}\cdot a)\frac{\mrm{d}t}{t^{1+\rho}}.
    \end{equation*}
\end{proof}

The summation \eqref{green-currents-T} is independent of the choice of $a\in\mrm{GL}_r(\mathbb{R})$ by the $\textup{O}_r(\mathbb{R})$-invariance of the forms $\nu_r(\boldsymbol{x})$ (\cite[Proposition 2.12(ii)]{GS19}). It is also easy to verify that for $\gamma\in\textup{GL}_r(\mathbb{Z})$,
\begin{equation}
    \mathfrak{g}^{\circ}({^{t}\gamma}T\gamma,\msf{y},\rho)=\mathfrak{g}^{\circ}({T,\gamma \msf{y}{^{t}\gamma}},\rho).
\end{equation}

The current $\mathfrak{g}^{\circ}(T,\msf{y},\rho)$
descends to a current on $\mc{M}(\mathbb{C})$
by its invariance under $\textup{SL}_2(\mathbb{Z})\times\textup{SL}_2(\mathbb{Z})$.
This descended current will also be denoted $\mathfrak{g}^{\circ}(T,\msf{y},\rho)$,
and can be calculated using the following ``quotient intgral'' procedure:
any compactly supported smooth form $\eta \in A^{3 - r, 3 - r}_c(\mc{M}(\C))$
is specified by an $(\opn{SL}_2(\Z) \times \opn{SL}_2(\Z))$-invariant
form $\tilde{\eta} \in A^{3 - r, 3 - r}(\mbb{D})$,
which can be written as $\tilde{\eta} = \sum_{\gamma \in \opn{SL}_2(\Z) \times \opn{SL}_2(\Z)} \gamma^* \alpha$
for some $\alpha \in A^{3 - r, 3 - r}_c(\mbb{D})$ via a partition of unity argument.
We then have
    \begin{equation}
    \int_{\mc{M}(\C)} \mathfrak{g}^{\circ}(T,\msf{y},\rho) \wedge \eta
    =
    \int_{\mbb{D}} \mathfrak{g}^{\circ}(T,\msf{y},\rho) \wedge \alpha.
    \end{equation}
Equivalently, if $\mc{F} \subseteq \mbb{D}$ is a fundamental domain
(say, the self-product of the usual fundamental domain for $\mbb{H}_1$),
we have
    \begin{equation}\label{equation:KM_forms_currents_part2_new:fundamental_domain}
    \int_{\mc{M}(\C)} \mathfrak{g}^{\circ}(T,\msf{y},\rho) \wedge \eta
    =
    \frac{1}{4} \int_{\mc{F}} \mathfrak{g}^{\circ}(T,\msf{y},\rho) \wedge \eta
    \end{equation}
for $\opn{Re}(\rho) > 0$, where $\mathfrak{g}^{\circ}(T,\msf{y},\rho)$ is represented
by a locally $L^1$ form on $\mbb{D}$.
The factor $\frac{1}{4}$ accounts for the trivial action of $\{ \pm 1 \} \times \{ \pm 1 \}
\subseteq \opn{SL}_2(\Z) \times \opn{SL}_2(\Z)$ on $\mbb{D}$.
For $\opn{Re}(\rho) > -1/2$, the current $\mathfrak{g}^{\circ}(T,\msf{y},\rho)$
can then be calculated from \eqref{equation:KM_forms_currents_part2_new:fundamental_domain}
via unique meromorphic continuation in $\rho$.

By Lemma \ref{sum-con-current} and its proof, the constant term in the Laurent expansion
\begin{equation}\label{g-T-ct}
    \mf{g}^{\circ}(T,\msf{y})\coloneqq\ct\,\mf{g}^{\circ}(T,\msf{y},\rho)
\end{equation}
descends to a current on $\mc{M}(\mathbb{C})$.
The meaning of (\ref{g-T-ct}) is: for an element $\eta\in A^{3-r,3-r}_c(\mc{M}(\mathbb{C}))$, the pairing $\int_{\mc{M}(\mathbb{C})}\mf{g}^{\circ}(T,\msf{y},\rho)\wedge\eta$ defines a meromorphic function in $\rho$ for $\textup{Re}(\rho)>-1/2$,
and then
\begin{equation}
    \int_{\mc{M}(\mathbb{C})}\mf{g}^{\circ}(T,\msf{y})\wedge\eta\coloneqq\ct\,\int_{\mc{M}(\mathbb{C})}\mf{g}^{\circ}(T,\msf{y},\rho)\wedge\eta.
\end{equation}

\subsection{Extension to the compactification}
\label{ssec: extending to compactification}
We seek to extend the current $\mf{g}^{\circ}(T,\msf{y},\rho)$ on $\mc{M}(\mathbb{C})$
to a current on the compactification $\overline{\mc{M}}(\mathbb{C})$.
At least for $\opn{Re}(\rho) > 0$, where $\mf{g}^{\circ}(T, \msf{y}, \rho)$
is represented by a locally $L^1$ form on $\mc{M}(\C)$, 
one can attempt to define the value of $\mf{g}^{\circ}(T,\msf{y},\rho)$
on $\eta\in A_c^{3-r,3-r}(\overline{\mc{M}}(\mathbb{C}))$
by the integral
\begin{equation}\label{eq:extended current def}
    \int_{\overline{\mc{M}}(\mathbb{C})}
    \mf{g}^{\circ}(T,\msf{y},\rho)\wedge\eta
    =
    \int_{\mc{M}(\mathbb{C})}
    \mf{g}^{\circ}(T,\msf{y},\rho)\wedge\eta|_{\mathcal{M}(\mathbb{C})}.
\end{equation}
The convergence of this integral, however,
is not automatic since the restriction $\eta|_{\mathcal{M}(\mathbb{C})}$ need not be compactly supported.

For $r=3$, Theorem \ref{thm:extension}
shows that $\mf{g}^{\circ}(T, \msf{y}, \rho)$
is in fact $L^1$ on $\mc{M}(\C)$ at least if $\opn{Re}(\rho) > 2$,
so that $\mf{g}^{\circ}(T, \msf{y}, \rho)$ defines a current on $\overline{\mc{M}}(\C)$
for such $\rho$.
That theorem also explains the unique meromorphic continuation of $\mf{g}^{\circ}(T, \msf{y}, \rho)$
to a family of currents on $\overline{\mc{M}}(\C)$ for $\opn{Re}(\rho) > -1/2$,
with at most a simple pole at $\rho = 0$.
The constant term at $\rho = 0$ is then a current on $\overline{\mc{M}}(\C)$
which extends the previously defined current $\mf{g}^{\circ}(T, \msf{y})$ on $\mc{M}(\C)$.
We simplify notation and also write $\mf{g}^{\circ}(T, \msf{y})$ for this extension to $\overline{\mc{M}}(\C)$.

We expect an analogue of Theorem \ref{thm:extension} to hold for $r=1$ and $r = 2$ as well.
In particular, we expect a meromorphic extension of $\mf{g}(T, \msf{y}, \rho)$
on $\overline{\mc{M}}(\C)$ to $\mrm{Re}(\rho)>-1/2$, with a possible pole at $\rho = 0$.
Since these cases are not needed for the present paper, we do not pursue them here.

However, in the special case where $r=1$ and $T=(t)\neq\boldsymbol{0}_1$,
the proof of \cite[{Proposition 4.3}]{GS19}
and the calculation \cite[{Example 2.10}]{GS19}
shows that $\mf{g}^{\circ}(T, \msf{y}) = \mf{g}^{\circ}(T, \msf{y}, 0)$
is represented by a locally $L^1$ function on $\mc{M}(\C)$.
In this case, the pairing
\eqref{eq:extended current def} already converges absolutely at
$\rho=0$, as follows from
\cite[Theorem 4.2]{BerndtKuehn2012b}.
Indeed, the function $\mathfrak{g}^{\circ}(T,\msf{y})$ is two times the Kudla's Green function
studied by Berndt and K\"uhn. This function is nonnegative,
since each summand in its defining series is an integral
of a nonnegative Gaussian. Hence
\cite[Theorem 4.2]{BerndtKuehn2012b} implies
\begin{equation}\label{eq:genus 1 converge}
    \int_{\mc{M}(\mathbb{C})}
    \abs{\mf{g}^{\circ}(T,\msf{y})}
    \,\Omega^2<\infty.
\end{equation}
The absolute convergence of $\int_{\mc{M}(\C)} \mf{g}^{\circ}(T, \msf{y}) \wedge \eta$
for any continuous $(2,2)$-form $\eta$ on $\overline{\mc{M}}(\C)$ then follows from the following lemma.

\begin{lemma}
    Let $\eta$ be a continuous $(2,2)$-form on $\overline{\mc{M}}(\C)$.
    Write $\eta|_{\mc{M}(\C)} = f \cdot \Omega^2$ for some continuous function
    $f$ on $\overline{\mc{M}}(\C)$. Then $f$ is bounded.
\end{lemma}
\begin{proof}
    Points on the boundary $\overline{\mc{M}}(\C) \setminus \mc{M}(\C)$
    are of the form $(\infty, z_2)$ or $(z_1, \infty)$ or $(\infty, \infty)$
    for $z_1, z_2$ in $\mc{E}ll(\C)$.
    By compactness of $\overline{\mathcal{M}}(\C)$, it suffices to show that
    the coefficients of $\Omega^2$ are bounded away from $0$
    on any orbifold chart near any given boundary point.
    For points on $\mc{E}ll(\C)$, we use the usual orbifold chart
    $\mbb{H}_1 \ra \mc{E}ll(\C)$ with the upper half-plane $\mbb{H}_1$.
    Near $\infty$, we use the usual coordinate $q = e^{2 \pi i z}$,
    for $|q| \ll 1$, which gives an orbifold chart near $\infty$ for $q$ in a complex disc
    of radius $\ll 1$ equippped with the trivial action by $\Z / 2 \Z$.
    See for instance \cite[{\S 4}]{Hain14} for this orbifold description near the cusp $\infty$.
    
    On the above orbifold chart near a boundary point $(\infty, z_2)$ for $z_2 \neq \infty$, the form $\Omega^2$
    is a constant multiple of $(|q_1| \log |q_1|)^{-2} y_2^{-2} \dd q_1 \wedge \dd \overline{q_1} \wedge \dd x_2 \wedge \dd y_2$,
    where $q_1 = e^{2 \pi i z_1}$ is the usual local coordinate near $\infty$,
    and where $x_2 + i y_2$ are coordinates in the upper-half plane.
    This expression for $\Omega^2$ is clearly bounded away from $0$ near $q_1 = 0$
    and near (a lift in $\mbb{H}_1$ of) a fixed $z_2$.
    A similar argument applies near a boundary point $(z_1,\infty)$.
    For the boundary point $(\infty, \infty)$,
    the form $\Omega^2$ is a constant multiple of
    $(|q_1| \log |q_1|)^{-2} (|q_1| \log |q_1|)^{-2} \dd q_1 \wedge \dd \overline{q}_1
    \wedge \dd q_2 \wedge \dd \overline{q}_2$,
    which is also bounded away from $0$ for small $q_1, q_2$.
\end{proof}

\subsection{Theta series associated to currents}
Let $1\leq r \leq3$ be an integer.
Consider the theta series associated to a Schwartz function $\boldsymbol{\varphi}\in\mathscr{S}(\mathbb{V}^{r})$,
with $\omega$ the Weil representation from 
\begin{equation*}
    \theta\left(g,h,\boldsymbol{\varphi}\right)=\sum\limits_{\boldsymbol{x}\in V^{r}}\left(\omega(g)\boldsymbol{\varphi}\right)\left(h^{-1}\boldsymbol{x}\right).
\end{equation*}
For $T\in\textup{Sym}_r(\mathbb{Q})$, define $\theta_T\left(g,h,\boldsymbol{\varphi}\right)=\sum\limits_{\boldsymbol{x}\in \Omega_T(V)}\left(\omega(g)\boldsymbol{\varphi}\right)\left(h^{-1}\boldsymbol{x}\right)$.

Now assume that $\bsm{\varphi}$ is $O_r(\R)$-invariant,
where $O_r(\R) \subseteq K_{r,\infty}$ is the diagonally embedded
orthogonal group for the Gram matrix $1_r$, with
$K_{r, \infty} \subseteq \Sp_{2r}(\R)$
the standard maximal compact as in \eqref{maximalcompact}.
For $g_{\msf{z}}\in\textup{Sp}(\msf{W}_r)(\mathbb{A})$ as in (\ref{def_gy})
where $\msf{z}=\msf{x}+i\msf{y}\in\mathbb{H}_r$ and $\msf{y}=a\cdot{^{t}a}$,
we define $\theta\left(\msf{z},h,\boldsymbol{\varphi}\right)\coloneqq\det(\msf{y})^{-1}\cdot\theta\left(g_{\msf{z}},h,\boldsymbol{\varphi}\right)$ and for all $T\in\textup{Sym}_r(\mathbb{Q})$,
set
\begin{align*}
    C_{\theta,T}(\msf{y},h,\boldsymbol{\varphi})\coloneqq\det(\msf{y})^{-1}\cdot\theta_T\left(g_{\msf{z}},h,\boldsymbol{\varphi}\right)\cdot q^{-T},\,\,q^{T}=e^{2\pi i\textup{tr}(T\msf{z})}.
\end{align*}
The $O_r(\R)$-invariance ensures that this does not depend on the choice of $g_{\msf{z}}$.
In particular, if $\bsm{\varphi}$ is moreover factorizable
as $\boldsymbol{\varphi}=\varphi_{\infty}\otimes\boldsymbol{1}_{\textup{M}_2(\hat{\mathbb{Z}})^{r}}\in\mathscr{S}(\mathbb{V}^{r})$,
then for $h=(h_{\infty},e_f)$, we have
\begin{equation}\label{T-coefficient}
    C_{\theta,T}\left(\msf{y},h,\boldsymbol{\varphi}\right)=\sum\limits_{\boldsymbol{x}\in\Omega_T(V)\cap\textup{M}_2(\mathbb{Z})^{r}}e^{2\pi\textup{tr}(T\msf{y})}\varphi_{\infty}\left(h_\infty^{-1}\boldsymbol{x}\cdot a\right).
\end{equation}
\par
For a matrix $T\in\textup{Sym}_r(\mathbb{Q})$,
a matrix $\msf{y}\in\textup{Sym}_r(\mathbb{R})_{>0}$,
a point $z\in\mathbb{D}$,
and $\rho \in \C$ with $\mrm{Re}(\rho)>0$, using the formula (\ref{int-first-sum}),
definitions \eqref{nuT} and \eqref{sch-func-green},
and the $O(V)(\R)$-invariance of $\nu_r$ from \cite[{Proposition 2.6(d)}]{GS19}
gives
\begin{align}\label{wedge-omega}
    \left(\mf{g}^{\circ}(T,\msf{y},\rho)\wedge\Omega^{3-r}\right)_z=&\int_{1}^{+\infty}\nu^{\circ}(T,t\msf{y})_z\wedge\Omega^{3-r}_z\frac{\textup{d}t}{t^{1+\rho}}\\=
    &\int_{1}^{+\infty}C_{\theta,T}\left(t\msf{y},h_z,\tilde{\nu}_r\otimes\boldsymbol{1}_{\textup{M}_2(\hat{\mathbb{Z}})^{r}}\right)\frac{\textup{d}t}{t^{1+\rho}}\cdot\Omega^{2}_z\notag
\end{align}
for any $h_z \in O(V)(\R)$ with $h_z \cdot (i,i) = z$.
Note that the Schwartz function $\tilde{\nu}_r \in \ms{S}(V_{\infty}^r)$
is $O_r(\R)$-invariant by \cite[{Proposition 2.6(f)}]{GS19},
and is $\opn{Stab}_{O(V)(\R)}(i,i)$-invariant by \cite[{Proposition 2.6(d)}]{GS19}.

\subsection{Green current equation}\label{ssec: green cureents equation}
Define
\begin{equation}
    \varphi_\km^\circ(T,\mathsf{y})=e^{2\pi\textup{tr}(T\mathsf{y})}\sum\limits_{\boldsymbol{x}\in\Omega_T(V)\cap\textup{M}_2(\mathbb{Z})^{r}} 
        \varphi_\km(\boldsymbol{x}\cdot a).
\end{equation}

If $T=(t)$ is a $1\times1$ matrix where $t\in\mathbb{Q}^{\times}$,
it is proved that we have the following identity of currents on $M(\mathbb{C})$
\cite[Proposition 2.4]{BerndtKuehn2012a}:
\begin{equation}\label{green-currents-equation-rk1-coarse}
    \frac{1}{2}\textup{dd}^c\mathfrak{g}^{\circ}(T,\mathsf{y})_{\mrm{coarse}}
    +
    \delta_{T_t}=
    \frac{1}{2}\varphi_\km^\circ(T,\mathsf{y})_{\mrm{coarse}},
\end{equation}
Where $\varphi_\km^\circ(T,\mathsf{y})_{\mrm{coarse}}$ is the current
on the $M(\C)$ defined by \eqref{equation:KM_forms_currents_part2_new:fundamental_domain}
without the $1/4$ factor. Similarly, $\textup{dd}^c\mathfrak{g}^{\circ}(T,\mathsf{y})_{\mrm{coarse}}$
is defined on invariant forms on $\mbb{D}$ with compact support modulo $\opn{SL}_2(\Z) \times \opn{SL}_2(\Z)$
as $4$ times the stacky count defined in loc. cit.
Recall that $T_t \subseteq M$ denotes the (reduced, scheme-theoretic)
image of $\mc{T}_t \subseteq \mc{M}$
under the coarse moduli map $\mc{M} \ra M$.
This $T_t$ is nonempty if and only if $t\in\mathbb{Z}_{>0}$.
Acting on invariant forms on $\mbb{D}$, we have $\delta_{T_t} = 4 \delta_{\mc{T}_t} = 2 \delta_{\mc{Z}(t)}$.

The formula (\ref{green-currents-equation-rk1-coarse}) thus implies
the following current identity on $\mc{M}(\mathbb{C})$:
\begin{equation}\label{green-currents-equation-rk1}
    \textup{dd}^c\mathfrak{g}^{\circ}(T,\mathsf{y})+\delta_{\mc{Z}(T)}=\varphi_\km^\circ(T,\mathsf{y}).
\end{equation}

Garc\'ia and Sankaran give a much more general form of (\ref{green-currents-equation-rk1}).
We record the case relevant for this paper.
\begin{lemma}
    For $1 \leq r \leq 3$, let $T\in\textup{Sym}_{r}(\mathbb{Q})$ be any matrix.
    The following equality of currents holds on $\mc{M}(\mathbb{C})$:
    \begin{equation}\label{g-c-e-full-version}
        \textup{dd}^c\mathfrak{g}^{\circ}(T,\mathsf{y})+\delta_{\mc{Z}(T)}\wedge\Omega^{r-\textup{rk}T}=\varphi_\km^\circ(T,\mathsf{y}).
    \end{equation}
\end{lemma}
\begin{proof}
    The claimed
    equality of currents on $\mc{M}(\C)$ can be checked locally on the orbifold chart $\mbb{D} \ra \mc{M}(\C)$,
    via partitions of unity. 
    Using the complex uniformization of $\mc{Z}(T)$ from \eqref{equation:complex_uniformization:special_cycle},
    the claimed current identity now
    follows from the corresponding Green current identity on $\mbb{D}$,
    which appears in \eqref{equ: GS original}.
    See also the proof of \cite[{Proposition 4.4}]{GS19}.

\end{proof}

\subsection{Arithmetic special cycle classes}
    \label{ssec:arithmetic_special_cycles}
For $1\leq r\leq 3$, let $T\in\textup{Sym}_r(\mathbb{Q})$ be a matrix and $\mathsf{y}\in\textup{Sym}_r(\mathbb{R})$ be a positive definite matrix.
Given such a pair $(T, \mathsf{y})$,
we define a number $\det(\mathsf{y}_1) \in \R_{>0}$ as follows.
Let $\alpha\in\textup{GL}_r(\mathbb{Z})$ be a matrix such that
\begin{equation}\label{reduce-to-diagonal}
    {^{t}\alpha}T\alpha=\begin{pmatrix}
        \boldsymbol{0}_{r-\textup{rk}(T)} & \\
         & T^\flat
    \end{pmatrix}
\end{equation}
is block diagonal, where $T^\flat\in\textup{Sym}_{\textup{rk}(T)}(\mathbb{Q})$ is a nonsingular matrix. Such $\alpha$ are unique up to right multiplication
by elements of the form
    \begin{equation}
    \begin{pmatrix}
    * & *
    \\
    0 & *
    \end{pmatrix}
    \end{equation}
with block diagonal entries in $\GL_{r - \textup{rk}(T)}(\Z)$
and $\GL_{\textup{rk}(T)}(\Z)$, respectively. Therefore the following number is well-defined:
\begin{equation}\label{det'}
    \textup{det}'\,T=\det T^\flat.
\end{equation}
As elements in $Z^{\textup{rk}(T)}(\mathcal{M})$,
we also have
\begin{equation}\label{special-cycle-true-allT}
\mathcal{Z}(T) = \mathcal{Z}(T^{\flat})
\end{equation}

By Gram--Schmidt, we may decompose
\begin{equation}\label{equation:y_decomp_for_current}
    \alpha^{-1}\mathsf{y}{^{t}\alpha}^{-1} =
    \begin{pmatrix}
    1 & c
    \\
    0 & 1
    \end{pmatrix}
    \begin{pmatrix}
        \mathsf{y}_1 & 0 \\
         0 & \mathsf{y}^\flat
    \end{pmatrix}
    \begin{pmatrix}
    1 & 0
    \\
    {}^t c & 1
    \end{pmatrix}
\end{equation}
for $c \in \mrm{M}_{r-\textup{rk}(T), \textup{rk}(T)}(\mathbb{R})$, and
where $\mathsf{y}_1\in\textup{Sym}_{r-\textup{rk}(T)}(\mathbb{R})$ and $\mathsf{y}^\flat\in\textup{Sym}_{\textup{rk}(T)}(\mathbb{R})$.
The matrix $c$ is uniquely determined by
$\alpha^{-1}\mathsf{y}{^{t}\alpha}^{-1}$,
and both $\mathsf{y}_1$ and $\mathsf{y}^\flat$ must be
positive definite.
The above discussion on uniqueness of $\alpha$
up to right translation by block upper triangular matrices
shows that $\det(\mathsf{y}_1)$ and $\det(\mathsf{y}^\flat)$
are both independent of the choice of $\alpha$.
Given $(T, \mathsf{y})$, the $\GL_{\opn{rk}(T)}(\Z)$-orbit
of the pair $(T^\flat, \mathsf{y}^\flat)$ is also well-defined,
where $\gamma \in \GL_{\opn{rk}(T)}(\Z)$
acts as $\gamma \cdot (T^\flat, \mathsf{y}^\flat) = \left({}^t \gamma^{-1} T^\flat \gamma^{-1}, \gamma \mathsf{y}^\flat \left({^t \gamma}\right)\right)$.

By construction, we find that
the quantities $\det(\mathsf{y}_1)$ and $\det(\mathsf{y}^\flat)$
do not change if the pair $(T, \mathsf{y})$ is replaced by $({}^t \alpha T \alpha, \alpha^{-1} \mathsf{y} {}^t \alpha^{-1})$
for any $\alpha \in \GL_r(\Z)$.
The $\GL_{\opn{rk}(T)}(\Z)$-orbit of $(T^\flat, \mathsf{y}^\flat)$ also does not change. Moreover, the Fourier coefficient is also unchanged under the above operations.

\begin{lemma}\label{reduce-both-diagonal}
    Let $T\in\textup{Sym}_r(\mathbb{Q})$ be a matrix and $\mathsf{y}\in\textup{Sym}_r(\mathbb{R})$ be a positive definite matrix. Let $\tilde{T}$ be the right hand side of (\ref{reduce-to-diagonal}). Let $\tilde{\msf{y}}=\diag\{\msf{y}_1,\msf{y}^\flat\}$ be the middle matrix on the right hand side of (\ref{equation:y_decomp_for_current}),
    \begin{equation}\label{simplify-C_T}
        C_T^{r}(\msf{y},s)=C_{\tilde{T}}^{r}(\tilde{\msf{y}},s)
    \end{equation}
\end{lemma}
\begin{proof}
    It's easy to show that $C_T^{r}(\msf{y},s)=C_{\tilde{T}}^r(\alpha^{-1}\mathsf{y}{^{t}\alpha}^{-1},s)$, so it suffices to prove
    \begin{equation*}
        C_{\tilde{T}}^r(\alpha^{-1}\mathsf{y}{^{t}\alpha}^{-1},s)=C_{\tilde{T}}^{r}(\tilde{\msf{y}},s).
    \end{equation*}
    By \eqref{equation:y_decomp_for_current},
    an easy calculation implies that for the matrix $c\in\mrm{M}_{r-\mrm{rk}(T),\mrm{rk}(T)}(\mathbb{R})$,
    \begin{equation*}
        g_{i\alpha^{-1}\msf{y}{^t\alpha^{-1}}}=m(\theta) g_{i\tilde{\msf{y}}},\quad\quad\theta=\begin{pmatrix}
            1 & c\\
            0 & 1
        \end{pmatrix}\in\mrm{GL}_r(\mathbb{A}).
    \end{equation*}
    By \cite[(5.56)]{GS19}, for all $\Phi\in I_r(s)$,
    \begin{equation*}
        E_{\tilde{T}}(m(\theta)g,s,\Phi)=E_{\tilde{T}}(g,s,\Phi).
    \end{equation*}
    By the definition of the coefficient $C_T$, we have
    \begin{equation*}
        C_{\tilde{T}}^r(\alpha^{-1}\mathsf{y}{^{t}\alpha}^{-1},s)e^{-2\pi\textup{tr}(\tilde{T}\alpha^{-1}\mathsf{y}{^{t}\alpha}^{-1})}=C_{\tilde{T}}^{r}(\tilde{\msf{y}},s)e^{-2\pi\textup{tr}(\tilde{T}\tilde{\mathsf{y}})}.
    \end{equation*}
    It's easy to verify that $\textup{tr}(\tilde{T}\alpha^{-1}\msf{y}{^t\alpha^{-1}})=\textup{tr}(\tilde{T}\tilde{\msf{y}})$. Then (\ref{simplify-C_T}) follows.
\end{proof}

\begin{definition}\label{green-currents-T2}
    For $1\leq r\leq 3$, let $T\in\textup{Sym}_r(\mathbb{Q})$ be a matrix and $\mathsf{y}\in\textup{Sym}_r(\mathbb{R})$ be a positive definite matrix. Define the following current on $\mc{M}(\mathbb{C})$
    \begin{equation}\label{eq: corrected currents}
        \mathfrak{g}(T,\mathsf{y})\coloneqq\begin{cases}
            \mathfrak{g}^{\circ}(T,\mathsf{y}), &\textup{if $\det(T)\neq0$;}\\
            \mathfrak{g}^{\circ}(T,\mathsf{y})-\log(\det \mathsf{y}_1)\delta_{\mathcal{Z}(T)}\wedge\Omega^{r-\textup{rk}(T)-1}, &\textup{if $\det(T)=0$.}
        \end{cases}        
    \end{equation}
    where $\det(\mathsf{y}_1) \in \R_{>0}$ is associated to the pair
    $(T,\mathsf{y})$ as in \eqref{equation:y_decomp_for_current}.
\end{definition}

\begin{remark}\label{remark:Green_currents:GS_modification}
There is a similar modification of $\mf{g}^{\circ}(T,\msf{y})$
which appears in \cite[{Definition 4.5}]{GS19}, but our definition is not identical to theirs.
Our modification of $\mf{g}^{\circ}(T, \msf{y})$ is closer to the construction of \cite[{\S 12.4}]{Chen2024},
but without the complications of loc. cit. involving class number issues.
\end{remark}

\begin{remark}
    The current $\delta_{\mc{Z}(T)}\wedge\Omega^{r-\mrm{rk}(T)-1}$ extends naturally to the compactification as follows. For $\eta\in A_c^{3-r,3-r}(\overline{\mc{M}}(\mathbb{C}))$, define
\begin{equation*}
\int_{\overline{\mc{M}}(\mathbb{C})}
\left(\delta_{\mc{Z}(T)}\wedge\Omega^{r-\mrm{rk}(T)-1}\right)\wedge\eta
\coloneqq
\int_{\mc{Z}(T)(\mathbb{C})}
\left.(\eta\wedge\Omega^{r-\mrm{rk}(T)-1})\right|_{\mc{Z}(T)(\mathbb{C})}.
\end{equation*}
The integral on the right is finite. Indeed, if $T$ is positive semidefinite of rank $1$, it is an integral over a finite union of modular curves, and convergence follows from their finite volume. If $T$ is positive semidefinite of rank $2$, it is a finite sum of values of the function $\eta$ at CM point.

Therefore for $r=3$, the current $\mf{g}(T,\msf{y})$ extends to a current on $\overline{\mc{M}}(\mathbb{C})$ by our discussion in $\S$\ref{ssec: extending to compactification}. 
\end{remark}

In the situation of  Definition \ref{green-currents-T2}, the above discussion shows that we also have $\mathfrak{g}({^{t}\gamma}T\gamma,\mathsf{y})=\mathfrak{g}({T,\gamma\mathsf{y}{^{t}\gamma}})$ for any $\gamma\in\textup{GL}_r(\mathbb{Z})$.
Note that $\mf{g}(T,\mathsf{y}) = \mathfrak{g}^{\circ}(T,\mathsf{y})$ if $\det T \neq 0$. Moreover, the following green current equation holds on $\mc{M}(\mathbb{C})$,
\begin{equation}
    \textup{dd}^c\mathfrak{g}(T,\mathsf{y})+\delta_{\mathcal{Z}(T)}\wedge\Omega^{r-\textup{rk}(T)}=\varphi_\km^\circ(T,\mathsf{y}).
\end{equation}

With notation as in Section \ref{ssec:Green_current_def},
we define an arithmetic special cycle class
$[\hz(T,\mathsf{y})] \in \widehat{\opn{Ch}}^3(\mc{M})$
by
    \begin{equation}\label{full-definition-ari-cycle}
        [\hz(T,\mathsf{y})] \coloneqq\begin{cases}
            \left(\mathcal{Z}(T^\flat),\mathfrak{g}(T^\flat,\mathsf{y}^\flat)\right)\cdot\hc(\widehat{\mc{E}}^{\vee})^{r-\textup{rk}(T)}+\left(0,\mathfrak{g}(T,\mathsf{y})-\mathfrak{g}(T^\flat,\mathsf{y}^\flat)\wedge\Omega^{r-\textup{rk}(T)}\right), &\textup{if $T\neq\boldsymbol{0}_r$}\\
            \hc(\widehat{\mc{E}}^{\vee})^{r}+\left(0,\mathfrak{g}(\boldsymbol{0}_r,\mathsf{y})\right), &\textup{if $T=\boldsymbol{0}_r$},
        \end{cases}
    \end{equation}
for the pair $(T,\msf{y})$.
This construction \eqref{full-definition-ari-cycle}
is essentially \cite[{(5.158)}]{GS19},
except we have the additional complication
that $\mc{M}$ is not proper,
and also an additional complication from stackiness of $\mc{M}$.
There is also the discrepancy from
Remark \ref{remark:Green_currents:GS_modification}.
Due to the local complete intersection property
for the special cycles $\mc{Z}(T)$ as explained in
Lemma \ref{correct-codim} and its proof,
we do not make the additional derived modifications
discussed in \cite[{\S 4.4}]{Chen2024}.

To remedy the non-properness, it should be preferable
to work with $\overline{\mc{M}}$ instead of $\mc{M}$,
and to replace the Gillet--Soul\'e arithmetic intersection theory
with a version of the theory from \cite{Burgos2007}.
On the other hand, one can formulate the arithmetic intersection
numbers more directly, as we explain next.
We take this simpler approach,
as the technical issues in defining the arithmetic cycle classes
on the compactified stack $\overline{\mc{M}}$
are largely separate from the main ideas of this paper.

\subsection{Arithmetic degrees of special cycles}
\label{ssec: height of arithmetic_special_cycles}
For regular proper Deligne–Mumford stacks,
there is a well-defined notion of arithmetic degree
on the top-codimension arithmetic Chow group \cite{gillet2009arithmeticintersectiontheorydelignemumford}.
However, the stack $\mc{M}$ is not proper over $\mathbb{Z}$.
Although the line bundle underlying the Hermitian line bundle $\widehat{\mc{E}}^{\vee}$
extends to $\overline{\mc{M}}$, its metric has logarithmic singularities along the
boundary $\overline{\mc{M}}\setminus \mc{M}$,
while its first Chern form has log-log growth \cite[(7.60)]{Burgos2007}.
Thus it seems natural to expect that the arithmetic special cycles should define
classes in arithmetic Chow groups allowing log-log singularities along the
boundary $\overline{\mc{M}}\setminus \mc{M}$,
though there is a technical complication due to stackiness.
Moreover, by the analysis in \cite[Proposition 2.16]{BerndtKuehn2012a}, Kudla’s Green function $\mathfrak{g}^{\circ}(T,\msf{y})$ also fails to be a Green function with log-log singularities in the sense of Burgos–Kramer–Kühn \cite{Burgos2007}.

Nonetheless, one can make a precise definition
for the expected arithmetic degrees of the
cycles $\widehat{\mc{Z}}(T,\msf{y})$,
for example as in \cite[{\S 4.7}]{Chen2024}.

If $\opn{rk}(T) = 3$, the special cycle $\mathcal{Z}(T)$ is an Artinian Deligne--Mumford stack supported over at most one prime. Let's say $\mc{Z}(T)$ is supported over $p$, define
\begin{equation*}
    \widehat{\deg}\left(\left(\widehat{\mc{E}}^{\vee}\right)^{0}\big|_{\mc{Z}(T)}\right)=\log p\cdot\left(\textup{the stacky length of the Artinian Deligne--Mumford stack $\mathcal{Z}(T)$}\right)
\end{equation*}
More precisely, choose a finite \'etale cover $Z\to\mc{Z}(T)$ of constant degree $d$, where $Z$ is a scheme. The stacky length is the $\Z$-length of $Z$ divided by $d$, and the arithmetic degree is obtained by multiplying this quantity by $\log p$.

Next consider the case $\opn{rk}(T) = 2$.
By Lemma \ref{correct-codim}, we know that $\mc{Z}(T) \ra \Spec \Z$
is proper quasi-finite flat with \'etale generic fiber.
The number $\widehat{\deg}((\widehat{\mc{E}}^{\vee})|_{\mc{Z}(T)})$ was defined in Section \ref{her-vec-bundle}.

If $\opn{rk}(T)=1$, let $T^{\flat}=(t)$ be the matrix obtained in (\ref{reduce-to-diagonal}).
In \cite[$\S$7.8]{Burgos2007}, the authors showed that the
metrized tautological bundle on the coarse space $M$ (in the sense of loc. cit.)
extends to a Hermitian $\mathbb{Q}$-line bundle with \textit{good metric} (we refer to \cite[Definition 7.4]{Burgos2007} for the notion of \textit{good metric}) along $\overline{M}\backslash M$.
Denote by $\widehat{\mc{E}}^{\vee}_{\overline{M}}$ this metrized $\mathbb{Q}$-line bundle on $\overline{M}$. Define
\begin{equation}\label{height-rank1-to-hecke}
    \widehat{\deg}\left(\left(\widehat{\mc{E}}^{\vee}\right)^{2}|_{\mc{Z}(T)}\right)=2^{-1}\cdot\widehat{\deg}\left(\left(\widehat{\mc{E}}_{\overline{M}}^{\vee}\right)^{2}|_{\overline{T_t}}\right).
\end{equation}
The factor $2^{-1}$ is inserted because we divide by $4$
for stacky considerations,
and then multiply by $2$ since $\mathcal{Z}(T)\rightarrow\mathcal{T}_t$
has degree $2$ as in Example \ref{sti-rank1}.

Finally if $\opn{rk}(T)=0$, we have $\mc{Z}(T)=\mc{M}$. Define
\begin{equation}\label{height-rank0-to-hecke}
    \widehat{\deg}\left(\left(\widehat{\mc{E}}^{\vee}\right)^{3}|_{\mc{Z}(\boldsymbol{0}_r)}\right)=4^{-1}\cdot\widehat{\deg}\left(\left(\widehat{\mc{E}}_{\overline{M}}^{\vee}\right)^{3}|_{\overline{M}}\right).
\end{equation}
The factor $4^{-1}$ is inserted for
stacky considerations.

Motivated by the \eqref{full-definition-ari-cycle}
on the expected arithmetic special cycle class $\widehat{\mc{Z}}(T,\msf{y})$,
define
\begin{equation}\label{def: arithmetic degree}
    \widehat{\deg}[\widehat{\mc{Z}}(T,\msf{y})]\coloneq\widehat{\deg}\left(\left(\widehat{\mc{E}}^{\vee}\right)^{3-\textup{rk}(T)}\big|_{\mc{Z}(T)}\right)+\frac{1}{2}\int_{\mc{M}(\mathbb{C})}\mathfrak{g}(T,\msf{y})\wedge\Omega^{3-r}.
\end{equation}
Since $\mc{M}(\mathbb{C})$ is noncompact, convergence of the integral is not automatic.
For $r=3$, however, convergence follows
from the estimates in Appendix \ref{sec append: genus 3 convergence},   
as discussed in $\S$\ref{ssec: extending to compactification}.
Thus the arithmetic degree in \eqref{def: arithmetic degree}
is well defined in this case.
Our main theorem explains that these arithmetic degrees
appear in the Fourier coefficients of the classical
genus 3 Siegel Eisenstein series.
\begin{theorem}\label{thm: first time}
    For every $T\in\textup{Sym}_3(\mathbb{Q})$ and $\msf{y}\in\textup{Sym}_3(\mathbb{R})_{>0}$, we have
    \begin{equation}\label{equ: ari-SW-rk3}
        \widehat{\deg}[\widehat{\mc{Z}}(T,\msf{y})]=-(32\pi^{2})^{-1}\left(C_T^3\right)^\prime(\msf{y},0).
    \end{equation}
\end{theorem}
The proof will be completed in Section \ref{sec:proof}.

\section{Faltings heights of corank 1 cycles}
\label{sec:Faltings_heights_rk2}
For a matrix $T$ of rank $1$ or $3$,
the height of the corresponding cycle $\mc{Z}(T)$
is more or less known by \cite[Theorem 7.61]{Burgos2007} (rank 1) and \cite{GK93} (rank 3).
In this paper, we handle the rank $2$ case, i.e. when $\mc{Z}(T)$ is a $1$-cycle.

\subsection{Faltings heights of CM elliptic curves}
\label{ssec:Faltings_height_CM}

For a rank $2$ matrix $T$,
recall (Lemma \ref{correct-codim}) that
every object in $\mathcal{Z}(T)(\overline{\Q})$ consists of a pair
of CM elliptic curves (and additional data).
The right-hand side of \eqref{def: arithmetic degree}
involves the Faltings heights
of these two elliptic curves, which have CM by possibly non-maximal orders.
This section contains some preparatory material for describing these heights.
The procedure from \cite[{Part 3}]{Chen2024} for studying ``height differences''
can be simplified for the case of Faltings heights of elliptic curves,
as we explain below.

If $E$ is any elliptic curve over a number field $K$,
and if $E$ has everywhere semi-abelian reduction,
recall that the stable Faltings height
    \begin{equation}
    \opn{h}_{\mrm{Fal}}(E) \coloneqq \frac{1}{[K : \Q]}
    \widehat{\deg}(\widehat{\omega}_{E,\mrm{Pet}})
    \in \R
    \end{equation}
is a normalized arithmetic degree of the Hodge bundle
$\widehat{\omega}_{E,\mrm{Pet}} \coloneqq e^* \Omega^1_{\tilde{E} / \mc{O}_K}$
for the N\'eron model $\tilde{E}$ of $E$ over $\Spec \mc{O}_K$,
with $e$ being the identity section. Here we normalize our Archimedean
metrics using the $\norm{-}_{\mrm{Pet}}$ normalization
from \eqref{equation:moduli_problems:Archimedean_metric}
The quantity $\opn{h}_{\mrm{Fal}}(E)$ does not change
if we extend the base field $K$. We always consider stable Faltings height.

\begin{remark}
\label{remark:Faltings_height_CM:height_over_more_general_base_fields}
It will be convenient notation to define the (stable) Faltings height $\opn{h}_{\mrm{Fal}}(E)$
for elliptic curves $E$ over arbitrary fields $\kappa$
of characteristic zero, subject to the condition
that $E_{\overline{\kappa}}$ descends to $\overline{\Q}$.
Indeed, we can pick any descent
of $E_{\overline{\kappa}}$ to a number field (sufficiently large to ensure
everywhere semi-abelian reduction) and take its Faltings height.
This will not depend on the choice of descent.
In particular, this applies to any CM elliptic curve in characteristic zero,
which always descends to $\overline{\Q}$ in the sense above.
\end{remark}

Next, let $F / \Q$ be an arbitrary imaginary quadratic field.
If $E$ is an elliptic curve with CM by the maximal order $\mc{O}_F$,
then we have
    \begin{equation}\label{equation:Faltings_height_CM:Chowla--Selberg}
    \opn{h}_{\mrm{Fal}}(E)
    =
    \opn{h}_{\mrm{Fal,F}}^{\mrm{CM}}
    \coloneqq
        \frac{1}{2} \frac{L'(1,\eta_F)}{L(1,\eta_F)}
        +
        \frac{1}{2} \frac{\Gamma'(1)}{\Gamma(1)}
        +
        \frac{1}{4} \log |D|
        -
        \frac{1}{2} \log (2 \pi)
    \end{equation}
from the Chowla--Selberg formula (compare \cite[{Proposition 10.10}]{KRY04}),
where $\eta_F$ is the quadratic character associated to $F / \Q$,
and $D$ is the discriminant of $\mc{O}_F$.

If $E$ instead has CM by a possibly non-maximal order,
the Faltings height was largely computed in the work of Nakkajima and Taguchi
\cite{NT91}. We give a reformulation of their result
in a more convenient form for us.
The height will be given in terms of
\eqref{equation:Faltings_height_CM:Chowla--Selberg}
and some explicit local quantities as in Lemma \ref{lemma:Faltings_height_CM:arbitrary_order}.

We first recall the following ``height difference'' formula,
which is a (not difficult) reformulation of \cite[{Lemma 5}]{Faltings1983}.
Given any isogeny $\phi \colon E_1 \ra E_2$
of arbitrary elliptic curves with semi-abelian reduction
over a number field $K$,
we have
    \begin{equation}\label{equation:Faltings_height_CM:height_change_formula}
    \opn{h}_{\mrm{Fal}}(E_2) - \opn{h}_{\mrm{Fal}}(E_1)
    =
    \frac{1}{2} \log \deg \phi
    -
    \frac{1}{[K : \Q]} \log |\opn{coker}(\phi_* \colon \Lie \tilde{E}_1 \ra \Lie \tilde{E}_2)|
    \end{equation}
where $\tilde{E}_i$ are the respective N\'eron models over $\Spec \mc{O}_K$,
and where $|-|$ means cardinality of the finite $\mc{O}_K$-module
(and this is also valid for abelian varieties).

In particular, if $E$ is any elliptic curve with CM by
a (possibly non-maximal) order in $F$, we have
    \begin{equation}\label{equation:Faltings_height_CM:height_diff_supported_primes}
    \opn{h}_{\mrm{Fal}}(E) - \opn{h}_{\mrm{Fal,F}}^{\mrm{CM}}
    \in
    \sum_{\text{primes } p} \Q \cdot \log p
    \end{equation}  
because any two elliptic curves with CM by $F$ are isogenous
over an algebraic closure. Note also that $\End(E)$ does not change
with extension of the base field, if $E$ has CM.

Although the quantity $\deg\phi$
appears as a measure of ``Archimedean ramification'',
we will eventually view the quantities $\log \deg \phi$
as contributing to ``non-Archimedean terms'' which can be calculated
by $p$-divisible groups as $p$ varies.

In light of \eqref{equation:Faltings_height_CM:height_diff_supported_primes},
we work locally, one prime at a time.
We set $\mc{O}_p \coloneqq \mc{O} \otimes_{\Z} \Z_p$
for orders $\mc{O} \subseteq \mc{O}_F$,
and also set $F_p \coloneqq F \otimes_{\Q} \Q_p$.

\begin{lemma}\label{lemma:CM_EC_height_determined_by_endos}
Let $E$ and $E'$ be elliptic curves
in characteristic zero, with CM by orders $\mc{O}$ and $\mc{O}'$
in the same imaginary quadratic field $F$.
If $p$ is a prime such that $\mc{O}_p \cong \mc{O}'_p$, then we have
    \begin{equation}
    \opn{h}_{\mrm{Fal}}(E) - \opn{h}_{\mrm{Fal}}(E')
    \in
    \sum_{\ell \neq p} \Q \cdot \log \ell.
    \end{equation}
In particular, if $\mc{O} \cong \mc{O}'$,
we have $\opn{h}_{\mrm{Fal}}(E) = \opn{h}_{\mrm{Fal}}(E')$.
\end{lemma}
\begin{proof}
By \eqref{equation:Faltings_height_CM:height_change_formula},
it is enough to show the existence of a prime-to-$p$ isogeny
$E \ra E'$.
Without loss of generality, we may assume that $E$ and $E'$
are given over $\C$, as any isogeny over $\C$ will descend
to $\overline{\Q}$ (for any descents of $E$ and $E'$).\footnote{
This type of argument also shows that any characteristic zero
abelian varieties $A$ and $A'$
with CM by the same maximal order in a finite \'etale $\Q$-algebra
$\tilde{F}$ (not necessarily a product of CM fields,
but necessarily totally complex) and the same CM type
(for fixed embeddings $\tilde{F} \ra \End^0(A)$ and $\tilde{F} \ra \End^0(A')$)
must have the same (stable) Faltings height, i.e. there exist
prime-to-$p$ isogenies (which can be chosen to be $\tilde{F}$-linear)
between $A$ and $A'$ for every prime $p$.
}

After picking appropriate embeddings $F \ra \End^{\circ}(E)$ and $F \ra \End^{\circ}(E')$,
we may assume that $E$ and $E'$ have the same CM type.
We can assume $E \cong \C / \mf{a}$ and $E' \cong \C / \mf{a}'$
as complex tori, where $\mf{a}, \mf{a}' \subseteq F$ are
$\Z$-lattices of rank $2$ which satisfy $\End_F(\mf{a}) = \mc{O}$
and $\End_F(\mf{a}') = \mc{O}'$.
Here we set e.g. $\End_F(\mf{a}) \coloneqq \End_{\Z}(\mf{a}) \cap \End_F(\mf{a} \otimes_{\Z} \Q)$.

Set $\mf{a}_p \coloneqq \mf{a} \otimes_{\Z} \Z_p$
and $\mf{a}'_p \coloneqq \mf{a}' \otimes_{\Z} \Z_p$.
One can show that both $\mf{a}_p$ and $\mf{a}'_p$
are free of rank one as $\mc{O}_p$-modules.
Briefly, one can find an element $u \in F_p^{\times}$
such that $u \cdot \mf{a}_p \subseteq \mc{O}_{F,p}$ and $1 \in u \cdot \mf{a}_p$,
and then if $s \in \Z_{\geq 0}$ is the integer such
that $[\mc{O}_{F,p} : \mc{O}_p] = p^s$,
we must have $u \cdot \mf{a}_p = \Z_p + p^s \mc{O}_{F,p}$
being equal to the unique order of index $p^s$
(and the same applies to $\mf{a}'_p$).

In other words, we can find some $b \in F_p^{\times}$
such that $b \cdot \mf{a}_p = \mf{a}'_p$,
as submodules of $F_p$. Since $F^{\times}$ is dense in $F_p^{\times}$,
we can moreover assume $b \in F^{\times}$. Then multiplication by $b$
gives a prime-to-$p$ quasi-isogeny $E \ra E'$
(which can be assumed an isogeny if one prefers, by clearing denominators
away from $p$).
\end{proof}

\begin{remark}
For the reader's convenience,
we remark that the condition
$\mc{O} \otimes_{\Z} \Z_p \cong \mc{O}' \otimes_{\Z} \Z_p$
from Lemma \ref{lemma:CM_EC_height_determined_by_endos}
is equivalent to the condition $v_p([\mc{O}_F : \mc{O}]) = v_p([ \mc{O}_F : \mc{O}'])$, because there is a unique order of index $p^s$ inside
$\mc{O}_{F,p}$ for each integer $s \geq 0$,
namely $\Z_p + p^s \mc{O}_{F,p}$, and these are all of the orders.
\end{remark}

Next, we use the notation
    \begin{equation}
    \eta_F^*(p)
        \coloneqq
        \begin{cases}
        1 & \text{if $p$ is split in $\mc{O}_F$}
        \\
        0 & \text{if $p$ is ramified in $\mc{O}_F$}
        \\
        -1 & \text{if $p$ is inert in $\mc{O}_F$}
        \end{cases}    
    \end{equation}
and set
    \begin{equation}\label{equation:delta_Fal_def}
    \delta_{\mrm{Fal},F_p}(s) \coloneqq
    \frac{1}{2}
    \left ( s - \frac{(1-p^{-s})(1-\eta^*_F(p))}{(1-p^{-1})(p-\eta^*_F(p))}
    \right )
    \end{equation}
for integers $s \geq 0$.
Note that this only depends on the isomorphism class of
the $\Q_p$-algebra $F_p \coloneqq F \otimes_{\Q} \Q_p$.

\begin{lemma}\label{lemma:Faltings_height_CM:arbitrary_order}
Let $E$ and $E_0$ be elliptic curves in characteristic zero,
with CM by the same imaginary quadratic field $F$.
Write $\mc{O}$ for the endomorphism ring of $E$,
and assume $E_0$ has CM by the maximal order $\mc{O}_F$.

Fix any prime $p$, and let $s \in \Z_{\geq 0}$ be the integer
such that $\mc{O}_p$ is the unique order of index $p^s$ inside $\mc{O}_{F,p}$.
We then have
    \begin{equation}
    \opn{h}_{\mrm{Fal}}(E) - \opn{h}_{\mrm{Fal}}(E_0)
    =
    \delta_{\mrm{Fal},F_p}(s) \cdot \log p
    \pmod{\sum_{\ell \neq p} \Q \cdot \log \ell}.
    \end{equation}
\end{lemma}
\begin{proof}
By Lemma \ref{lemma:CM_EC_height_determined_by_endos},
we may assume that $E$ and $E_0$ are the complex elliptic curves
$\C / \mc{O}$ and $\C / \mc{O}_F$ respectively,
for some embedding $F \ra \C$.
The inclusion $\mc{O} \ra \mc{O}_F$ gives an isogeny
$E \ra E_0$ whose degree has $p$-part being $p^s$. Its dual $E_0 \ra E$ is an
isogeny degree whose $p$-part is $p^s$. Descend the data of $E_0 \ra E$
and the actions $\mc{O}_F \acts E_0$ and $\mc{O} \acts E$
to some number field $K$, and we may assume that $E_0$ and $E$
have everywhere good reduction over $K$.

The computation of
Nakkajima and Taguchi \cite{NT91}
(see also \cite[Proposition 10.3]{KRY04} and its proof)
can be formulated as follows:
for any isogeny $\phi \colon E_{0,\breve{K}} \ra E_{\breve{K}}$
over a finite extension $\breve{K}$ of $\breve{\Q}_p$,
if the degree of $\phi$ has $p$-part $p^s$, then
    \begin{equation}\label{equation:Faltings_height_CM:Nakkajima--Taguchi_length}
    \frac{1}{[\breve{K} : \breve{\Q}_p]} \opn{length}_{\mc{O}_{\breve{K}}}(\opn{coker}(\phi_* \colon \Lie \tilde{E}_{0,\Spec \mc{O}_{\breve{K}}} \ra \Lie \tilde{E}_{\Spec \mc{O}_{\breve{K}}} ))
    =
    \frac{1}{2} \frac{(1-p^{-s})(1-\eta^*_F(p))}{(1-p^{-1})(p-\eta^*_F(p))}
    \end{equation}
for the N\'eron models $\tilde{E}_0$ and $\tilde{E}$ over $\mc{O}_{\breve{K}}$.
See also \cite[{(7.2.6)}]{Chen2024}, stated in terms of $p$-divisible groups.
Upon summing over the $p$-adic places of $K$,
it follows that any isogeny $\phi \colon E_0 \ra E$
(over $K$) with $\deg \phi$ having $p$-part $p^s$
has $\frac{1}{[K : \Q]} v_p(|\opn{coker}(\phi_* \colon \Lie \tilde{E}_0 \ra \Lie \tilde{E})|)$
equal to the right-hand
side of \eqref{equation:Faltings_height_CM:Nakkajima--Taguchi_length};
then apply \eqref{equation:Faltings_height_CM:height_change_formula}.
\end{proof}

The following lemma will help detect
CM orders using $p$-divisible groups
(after using the equivalence between \'etale $p$-divisible groups
and Tate modules).
The point is that we will apply the lemma
in situations where the base field is not a number field.

\begin{lemma}\label{lemma:Faltings_height_CM:Tate_modules}
Let $E$ be an elliptic curve over a field of characteristic $0$.
Assume $E$ has CM by an imaginary quadratic field $F$,
and set $\mc{O} \coloneqq \End(E)$.

Fix a prime $p$.
Suppose we have
$\dim_{\Z_p} \End(T_p(E)) = 2$,
where $\End(T_p(E))$ denotes Galois equivariant
endomorphisms of the $p$-adic Tate module.

Then we have $\mc{O}_p \cong \End(T_p(E))$.
\end{lemma}
\begin{proof}
Recall that the canonical map
$\End(E) \otimes_{\Z} \Z_p \ra \End(T_p(E))$
is injective \cite[{Theorem 3, page 176}]{Mumford85}.
Dimension considerations then show that the map
$\End(E) \otimes_{\Z} \Q_p \ra \End(T_p(E)) \otimes_{\Z_p} \Q_p$
is an isomorphism.
We find that an element $\phi \in \End(E) \otimes_{\Z} \Q$
maps to $\End(T_p(E))$ (equivalently, preserves the $p$-adic Tate module)
if and only if $\phi \in \End(E) \otimes_{\Z} \Z_{(p)}$,
as can be shown by clearing denominators away from $p$.
Since $\End(E) \otimes_{\Z} \Q$ has dense image in $\End(T_p(E)) \otimes_{\Z_p} \Q_p$, we conclude that $\End(E) \otimes_{\Z} \Z_{(p)} \ra \End(T_p(E))$
also has dense image, so $\End(E) \otimes_{\Z} \Z_p \subseteq \End(T_p(E))$
is closed (by compactness) and dense and hence the inclusion is an equality.
\end{proof}

\subsection{Decomposition of heights}
    \label{ssec:height_decomp}

In Section \ref{ssec:height_decomp}, we consider $T \in \opn{Sym}_r(\Q)$
with $\opn{rk}(T) = 2$ (so either $r = 2$ or $r = 3$).
We decompose the term
$\widehat{\deg}(\widehat{\mc{E}}^{\vee}|_{\mathcal{Z}(T)})$
over the prime numbers.
For the reader's convenience,
we recall that the metric $\|\cdot\|$ on $\mc{E}^{\vee}$ defined in
Section \ref{ssec:moduli_problems}
is different from Petersson metric $\|\cdot\|_{\textup{Pet}}$ also defined in loc. cit.:
the two metrics on $\mc{E}^{\vee}$ are related by
\begin{equation}\label{relation of metrics}
    \|\cdot\|^{2}=4\pi e^{\gamma}\|\cdot\|_{\textup{Pet}}^{2}.
\end{equation}

Recall the definition
$\widehat{\mc{E}}_{\mrm{Pet}} = p_1^* \widehat{\omega}_{\mc{A}, \mrm{Pet}}
\otimes p_2^* \widehat{\omega}_{\mc{A}, \mrm{Pet}}$
from Section \ref{ssec:moduli_problems}.
The (stacky) arithmetic degree $\widehat{\deg}(\widehat{\mc{E}}_{\textup{Pet}}^{\vee}|_{\mathcal{Z}(T)})$
of the Hermitian line bundle $\widehat{\mc{E}}^{\vee}$ on $\mathcal{Z}(T)$
(as defined in \cite[{Section 4.1}]{Chen2024})
may be calculated by the formula
    \begin{equation}\label{equation:height_decomp:arith_degree_as_height}
    \widehat{\deg}(\widehat{\mc{E}}_{\textup{Pet}}^{\vee}|_{\mathcal{Z}(T)})
    =
    -\sum_{\substack{\alpha \in \vert\mathcal{Z}(T)(\kappa)\vert \\ \alpha = (E, E')}}
    \frac{\opn{h}_{\mrm{Fal}}(E) + \opn{h}_{\mrm{Fal}}(E')}{\# \opn{Aut}(\alpha)}
    \end{equation}
for any algebraically closed field $\kappa$ of characteristic zero,
where the sum runs over the set of isomorphism classes
in the groupoid $\mathcal{Z}(T)(\kappa)$.
The formula \eqref{equation:height_decomp:arith_degree_as_height}
may be verified by covering $\mathcal{Z}(T)$
by a disjoint union of finite morphisms
from schemes of the form $\Spec \mc{O}_K$ for various
number fields $K$.
The well-definedness of the heights $\opn{h}_{\mrm{Fal}}(E)$
and $\opn{h}_{\mrm{Fal}}(E')$ even if $\kappa \neq \overline{\Q}$
follows because Lemma \ref{correct-codim} states that the condition
$(E, E', \bsm{x}) \in \mathcal{Z}(T)(\kappa)$ implies that $E$ and $E'$
both have CM by the imaginary quadratic field $\Q(\sqrt{-\det' T})$;
then recall that any CM elliptic curve
descends to $\overline{\Q}$ and apply the discussion
of Remark \ref{remark:Faltings_height_CM:height_over_more_general_base_fields}.
The use of $\kappa \neq \overline{\Q}$ will largely be a matter
of convenient bookkeeping later.
For the rest of Section \ref{ssec:height_decomp},
we set $F \coloneqq \Q(\sqrt{-\det' T})$.

 Recall that we consider the stacky degree $\deg_{\mc{E}^{\vee}_{\mathbb{Q}}}(\mathcal{Z}(T)_{\Q})$
 of the generic fiber $\mathcal{Z}(T)_{\Q} \ra \Spec \Q$,
then we must have
    \begin{equation}\label{equation:height_decomp:difference}
    \widehat{\deg}(\widehat{\mc{E}}_{\textup{Pet}}^{\vee}|_{\mathcal{Z}(T)})
    +
    2 (\deg_{\mc{E}^{\vee}_{\mathbb{Q}}}(\mathcal{Z}(T)_{\Q})) ( \opn{h}^{\mrm{CM}}_{\mrm{Fal},F} )
    \in
    \sum_{\text{primes } p} \Q \cdot \log p
    \end{equation}
via \eqref{equation:Faltings_height_CM:height_diff_supported_primes}.
We define
    \begin{equation}\label{eq: local height change}
    \widehat{\deg}_{\widehat{\mc{E}}^{\vee}}(\mathcal{Z}(T))_p
    \in \Q \cdot \log p
    \end{equation}
to be the ``$p$-part'' of \eqref{equation:height_decomp:difference},
defined by linear independence of $\log p$ for varying primes $p$.
Given $\alpha \in \mathcal{Z}(T)(\kappa)$ as
in \eqref{equation:height_decomp:arith_degree_as_height}, we
define $\delta_{\mrm{Fal},p}(\alpha) \in \Q \cdot \log p$ by the formulas
    \begin{equation}
    \opn{h}_{\mrm{Fal}}(E) + \opn{h}_{\mrm{Fal}}(E')
    -
    2 \opn{h}_{\mrm{Fal},F}^{\mrm{CM}}
    =
    \sum_p \delta_{\mrm{Fal},p}(\alpha)
    \quad\quad
    \text{with}
    \quad
    \delta_{\mrm{Fal},p}(\alpha) \in \Q \cdot \log p
    \end{equation}
and linear independence of $\log p$ for varying $p$. We then have
    \begin{equation}\label{equation:height_decomp:p-part_groupoid_sum}
    \widehat{\deg}_{\widehat{\mc{E}}^{\vee}}(\mathcal{Z}(T))_p
    =
        - \sum_{\substack{\alpha \in \vert\mathcal{Z}(T)(\kappa)\vert \\ \alpha = (E, E')}} \frac{\delta_{\mrm{Fal},p}(\alpha)}{\# \opn{Aut}(\alpha)}.
    \end{equation}
The point is that $\widehat{\deg}_{\widehat{\mc{E}}^{\vee}}(\mathcal{Z}(T))_p$
will be calculated via a local analogue
(Definition \ref{definition:height_change:special_cycles})
in combination with Rapoport--Zink uniformization (Section \ref{sec:RZ_uniformization}).

\section{Local special cycles}
\label{sec:local_special_cycles}
    We fix an arbitrary prime $p$ throughout this section,
allowing $p = 2$.
Let $\mbb{X}^{\textup{s}}$ (resp. $\mbb{X}^{\textup{o}}$) be a supersingular (resp. ordinary) $p$-divisible group of dimension $1$ and height $2$ over $\F$.
Given a choice of $* \in \{\mrm{s}, \mrm{o}\}$,
the object $\mbb{X}^{*}$ is unique up to isomorphism.

We set
    \begin{equation}
    \mbb{B}^* \coloneqq \Hom^{\circ}(\mbb{X}^*, \mbb{X}^{*}).
    \end{equation}
Fix any principal polarization $\lambda$ on $\mbb{X}^{*}$ (such $\lambda$ is unique up
to $\Z_p^{\times}$-scalar).
There is an induced Rosati involution $\dagger$ on $\mbb{B}^*$
given by $x^{\dagger} \coloneqq \lambda^{-1} x^{\vee} \lambda$
(which does not depend on the choice of $\lambda$).
This induces a quadratic form $q \colon \mbb{B}^* \ra \Q_p$
given by $q(x) \coloneqq x^{\dagger} x$,
with associated $\Q_p$-bilinear pairing $(-,-)$ on $\mbb{B}^*$ given by the formulas
    \begin{equation}
    (x, y) \coloneqq \opn{tr}^{\circ}(x^{\dagger} \circ y)
    \quad\quad
    \opn{tr}^{\circ}(a) \coloneqq a^{\dagger} + a
    \end{equation}
for $x, y, a \in \mbb{B}^*$.
More explicitly, the pair $(\mbb{B}^*, \dagger)$
is isomorphic (as a $\Q_p$-algebra with involution)
to the unique ($4$-dimensional) quaternion division algebra over $\Q_p$
with its unique standard involution if $* = \mrm{s}$,
and is isomorphic to $\Q_p \times \Q_p$ with its unique nontrivial involution
if $* = \mrm{o}$.

For each $* \in \{ \mrm{s}, \mrm{o} \}$, 
we have a canonical full rank $\Z_p$-lattice
    \begin{equation}
    \Hom(\mbb{X}^*, \mbb{X}^*) \subseteq \Hom^{\circ}(\mbb{X}^*, \mbb{X}^*)
    \end{equation}
which we denote $\mbb{L}^* \subseteq \mbb{B}^*$.

We find it useful to consider ``schemification''
of locally Noetherian formal schemes which are finite relative schemes
over e.g. $\Spf \breve{\Z}_p$ (compare \cite[{Appendix B}]{Chen2024}),
where $\breve{\Z}_p$ is the discrete valuation ring given by the ring of integers
of the completion of the maximal unramified extension of $\Q_p$.
That is, the association $\Spec B \mapsto \Spf B$
gives an equivalence of categories of finite relative schemes
over $\Spec \breve{\Z}_p$ and $\Spf \breve{\Z}_p$;
we write $(\Spf B)^{\mrm{sch}} \coloneqq \Spec B$ and call
this the ``schemification'' of $\Spf B$.

    \subsection{Rapoport--Zink spaces}
    \label{ssec:local_special_cycles:RZ}
        For $* \in \{ \mrm{s}, \mrm{o} \}$,
we first consider
the functor $\mc{N}_0^{*} \colon (\Sch / \Spf \breve{\Z}_p)^{\op} \ra \Set$
defined as follows: given a scheme $S$ over $\Spf \breve{\Z}_p$,
the set $\mc{N}_0^{*}(S)$ parameterizes data $(X, \rho)$
where $X / S$ is a $p$-divisible group and $\rho \colon X \times_S \overline{S} \ra \mbb{X}^* \times_{\F} \overline{S}$ is a quasi-isogeny which is an isomorphism in geometric fibers,
with $\overline{S} \coloneqq S \times_{\Spf \breve{\Z}_p} \F$.
It is a standard fact that $\mc{N}_0^*$ is (non-canonically) represented by $\Spf \breve{\Z}_p[[t]]$
(take any lift to $\Spf \breve{\Z}_p$, e.g. a suitable canonical lifting
in the sense of Gross,
and then apply Grothendieck--Messing theory).

We next consider some Rapoport--Zink spaces
    \begin{equation}\label{equation:local_special_cycles:cycles:RZ}
    \mc{N}^{*,0}
        \subseteq
    \mc{N}^*
        \subseteq
    \tilde{\mc{N}}^*
    \end{equation}
which are slight variants of one another,
and whose uses will be explained shortly.
Each of the spaces in \eqref{equation:local_special_cycles:cycles:RZ}
will be (non-canonically) a countable
disjoint union of copies of $\Spf \breve{\Z}_p[[t, t']]$.

We first define $\tilde{\mc{N}}^* \colon (\Sch / \Spf \breve{\Z}_p)^{\op} \ra \Set$
to parameterize tuples $(X, \rho, X', \rho')$
with $(X, \rho)$ and $(X', \rho')$ data as above,
but where $\rho$ and $\rho'$ are no longer required to be isomorphisms
in geometric fibers, and we instead require $\deg(\rho) = \deg(\rho')$.

We let $\mc{N}^* \subseteq \tilde{\mc{N}}^*$
be the open and closed subfunctor consisting of tuples $(X, \rho, X', \rho')$
with $\rho$ being an isomorphism in geometric fibers.

We let $\mc{N}^{*,0} \subseteq \tilde{\mc{N}}^*$ be the open and closed subfunctor
consisting of tuples $(X, \rho, X', \rho')$
with both $\rho$ and $\rho'$ being isomorphisms in geometric fibers.
There is a canonical isomorphism $\mc{N}^{*,0} \ra \mc{N}_0^* \times_{\Spf \breve{\Z}_p} \mc{N}_0^*$.

We find that the (disconnected) space $\tilde{\mc{N}}^*$ arises
naturally in Rapoport--Zink uniformization.
On the other hand, the natural local special cycles
on $\tilde{\mc{N}}^*$
(Definition \ref{definition:local_special_cycles:cycles})
are not quasi-compact in general, so we find it convenient to
restrict to $\mc{N}^*$ for counting purposes.
We have $\mc{N}^{*,0} = \mc{N}^*$ in the supersingular case,
but not in the ordinary case.
The connected space $\mc{N}^{*,0}$
seems to be ``too small'' for our counting purposes in the ordinary case,
though we still find it useful in some arguments.

Consider groups
    \begin{align}
    H'(\Q_p) & \coloneqq \{ (h, h') \in \End^{\circ}(\mbb{X}^*)^{\times} \times \End^{\circ}(\mbb{X}^*)^{\times} : \deg(h) = \deg(h') \}
    \\
    H'_0(\Q_p) & \coloneqq \{ (h, h') \in H'(\Q_p) : h \in \End(\mbb{X}^*)^{\times} \}
    \\
    H'(\Z_p) & \coloneqq \{ (h, h') \in H'(\Q_p) : h \in \End(\mbb{X}^*)^{\times} \quad \text{and} \quad h' \in \End(\mbb{X}^*)^{\times} \}
    \end{align}
depending on the choice of $* \in \{ \mrm{s}, \mrm{o} \}$,
which we suppress from notation.

There is an action $H'(\Q_p) \acts \tilde{\mc{N}}^*$,
where $(h,h') \in H'(\Q_p)$ acts as $(X, \rho, X', \rho') \mapsto (X, h \rho, X', h' \rho')$.
The subgroups $H'(\Z_p) \subseteq H'(\Q_p)$
and $H'_0(\Q_p) \subseteq H'(\Q_p)$ are normal,
and the action $H'(\Q_p) \acts \tilde{\mc{N}}^*$ induces
a simply transitive action $H'(\Q_p) / H'(\Z_p) \acts \tilde{\mc{N}}^*(\F)$.
We pick the identification $H'(\Q_p) / H'(\Z_p) \cong \tilde{\mc{N}}^*(\F)$
sending $1 \in H'(\Q_p) / H'(\Z_p)$ to $\mc{N}^{*,0}(\F)$.

We have an open and closed disjoint union decomposition
    \begin{equation}\label{equation:local_special_cycles:RZ:decomposition}
        \tilde{\mc{N}}^* = \coprod_{i \in H'(\Q_p) / H'_0(\Q_p)} \tilde{\mc{N}}^*_i
    \end{equation}
where $\tilde{\mc{N}}^*_i \subseteq \tilde{\mc{N}}^*$
is the open and closed subfunctor with
$\tilde{\mc{N}}^*_i(\F) \subseteq \tilde{\mc{N}}^*(\F) \cong H'(\Q_p) / H'(\Z_p)$
given by the $H'_0(\Q_p)$-orbit corresponding to $i$.
When $i = 1$, we have $\tilde{\mc{N}}_i^* = \mc{N}^*$.

We have
    \begin{equation}
    \begin{aligned}
    & H'(\Q_p) / H'(\Z_p)
    \cong
    \begin{cases}
    \Z & \text{in the supersingular case}
    \\
    \Z^3 & \text{in the ordinary case}
    \end{cases}
    \quad\quad
    && H'(\Q_p) / H'_0(\Q_p) \cong \Z^2
    \\
    & H'_0(\Q_p) / H'(\Z_p)
    \cong
    \begin{cases}
    \{ 0 \} & \text{in the supersingular case}
    \\
    \Z & \text{in the ordinary case}.
    \end{cases} 
    \end{aligned}
    \end{equation}

    \subsection{Cycles}
    \label{ssec:local_special_cycles:cycles}

\begin{definition}\label{definition:local_special_cycles:cycles}
Given any subset $H\subset \mbb{B}^*$,
we define the \emph{special cycle} $\tilde{\mc{Z}}^*(H) \subset \tilde{\mc{N}}^*$
to be the subfunctor cut out by the condition
\begin{equation*}
    (\rho^{\prime})^{-1} \circ x \circ \rho \in \textup{Hom}(X^{}, X^{\prime})
\end{equation*}
for all $x\in H$, for tuples $(X, \rho, X', \rho') \in \tilde{\mc{N}}^*(S)$
for schemes $S$ over $\Spf \breve{\Z}_p$.
We then define $\mc{Z}^*(H) \coloneqq \tilde{\mc{Z}}^*(H) \cap \mc{N}^*$
and $\mc{Z}^{*, 0}(H) \coloneqq \tilde{\mc{Z}}^*(H) \cap \mc{N}^{*, 0}$ .
\end{definition}

Several standard remarks apply
to Definition \ref{definition:local_special_cycles:cycles},
i.e. $(\rho^{\prime})^{-1} \circ x \circ \rho \colon X \ra X'$
is viewed as a quasi-homomorphism by Drinfeld rigidity,
and the subfunctor $\tilde{\mc{Z}}^*(H) \subset \tilde{\mc{N}}^*$
is represented by a closed formal subscheme by (the proof of)
\cite[{Proposition 2.9}]{RZ96}.
The involution of $\tilde{\mc{N}}^*$ given by swapping
$(X, \rho)$ and $(X', \rho')$ takes $\tilde{\mc{Z}}^*(H)$
isomorphically to $\tilde{\mc{Z}}^*(H^{\dagger})$.
To see this, we use the fact that given any $(X, \rho) \in \mc{N}^*_0(S)$,
the fixed principal polarization
on $\mbb{X}^*$ lifts (uniquely) to a principal polarization
on $X$ (e.g. \cite[{Lemma 5.6.1}]{Chen2024}).

Given a tuple $\bsm{x} = (x_1, \ldots, x_r)$ of elements 
in $\mbb{B}^*$,
we can forget information to view $\bsm{x}$
as a subset of $\mbb{B}^*$,
so we abuse notation to write e.g. $\tilde{\mc{Z}}^*(\bsm{x})$
for the associated special cycle.
Consider the temporary notation
$\tilde{\varphi}_p \coloneqq \pmb{1}_{\mbb{L}^*}^{\otimes r}
\in C^{\infty}_c((\mbb{B}^*)^r)$ for an understood
$r \geq 1$. If $\bsm{x} \in (\mbb{B}^*)^r$ is an $r$-tuple,
we have
    \begin{equation}\label{equation:RZ_similitude:connected_components_and_phi}
    \tilde{\mc{Z}}^*(\bsm{x})(\F) = \{ h \in H'(\Q_p) / H'(\Z_p) : h \in \opn{supp}(h \mapsto \tilde{\varphi}_p(h^{-1} \bsm{x})) \}
    \end{equation}
as a subset $\tilde{\mc{N}}^*(\F) = H'(\Q_p) / H'(\Z_p)$.
In particular, we have
    \begin{equation}
    \tilde{\mc{Z}}^*(\bsm{x})(\F)
    =
    \begin{cases}
    \tilde{\mc{N}}^*(\F) & \text{if $\bsm{x} \subseteq \mbb{L}^*$}
    \\
    \emptyset & \text{if $\bsm{x} \not \subseteq \mbb{L}^*$}
    \end{cases}
    \end{equation}
in the supersingular case (because $H'(\Q_p)$ preserves $\mbb{L}^*$
in the supersingular case).

There is an action $H'(\Q_p) \acts \mbb{B}^{*}$
where $(h,h') \in H'(\Q_p)$ acts as $x \mapsto h' x h^{-1}$.
This preserves quadratic forms up to $\Z_p^{\times}$-scalar.
For $\tilde{h} \in H'(\Q_p)$ and any subset $H \subseteq \mbb{B}^*$,
we then have
    \begin{equation}
    \tilde{h} \cdot \tilde{\mc{Z}}^*(H)
    =
    \tilde{\mc{Z}}^*(\tilde{h} \cdot H).
    \end{equation}

Considering the decomposition from \eqref{equation:local_special_cycles:RZ:decomposition}
and defining $\tilde{\mc{Z}}^*(\bsm{x})_i
\coloneqq \tilde{\mc{Z}}^*(\bsm{x}) \cap \tilde{\mc{N}}_i^*$
gives an open and closed decomposition
    \begin{equation}\label{equation:local_special_cycles:cycles:decomposition}
        \tilde{\mc{Z}}^*(\bsm{x}) = \coprod_{i \in H'(\Q_p) / H'_0(\Q_p)} \tilde{\mc{Z}}^*(\bsm{x})_i .
    \end{equation}

    \subsection{Cycles and lattices of rank \texorpdfstring{$2$}{2}}
        The rest of Section \ref{sec:local_special_cycles}
is largely devoted to studying the cycles $\mc{Z}^*(M)$
for $\Z_p$-lattices $M \subseteq \mbb{B}^*$ of rank $2$.
Note that any such $M$ is automatically non-degenerate:
if $* = \mrm{s}$, then $\mbb{B}^{*}$ is anisotropic,
and if $* = \mrm{o}$, then $M$ spans all of $\mbb{B}^{*}$.

Given any $\Z_p$-lattice $M$ of rank $2$
equipped with a non-degenerate quadratic form,
we consider its Gross--Keating
invariants $\opn{GK}(M) = (a, b)$
for integers $a \leq b$
(compare \cite[{\S 4}]{GK93}),
as recalled in Section \ref{sec:local_den_and_p=2}.
We set
    \begin{equation}
        \nu(M) \coloneqq a + b
        \quad\quad
        d(M) \coloneqq b - a.
\end{equation}
The lattice $M$ is integral if and only if $a \geq 0$.

Given a rank $2$ lattice $M \subseteq \mbb{B}^{*}$ which is not integral,
we have $\tilde{\mc{Z}}^*(M) = \emptyset$.
Indeed, by the transitive action $H'(\Q_p) \acts \tilde{\mc{N}}^*(\F)$, 
it is enough to check $\mc{Z}^{*,0}(M) = \emptyset$.
Then note that for any $(X, \rho, X', \rho') \in \mc{Z}^{*,0}(M)(\F)$,
and every $x \in M$, we must have $x^{\dagger} x \in \End(X)$, i.e. $x^{\dagger} x \in \Z_p$.

We know that $\Q_p[X]/(X^2 - \opn{disc}(M))$ is a ramified extension of
$\Q_p$ if and only if $\nu(M)$ is odd (e.g. follows from Lemma \ref{gk-dep-V}).
If $M$ is integral, we define the type $t(M)$ of $M$ to be
\begin{equation}
    t(M)=\begin{cases}
        0, &\textup{if $a=b=0$;}\\
        1, &\textup{if $a=0$ and $b>0$;}\\
        2, &\textup{if $a>0$.}
    \end{cases}
\end{equation}
We then have $\nu(M)=d(M)$ if and only if $t(M)\leq1$.

Given a finite \'etale $\Q_p$-algebra $F$ of degree $2$,
and given an integer $s \geq 0$, we
write $\mc{O}_{F,s} \coloneqq \Z_p + p^s \mc{O}_{F}$
for the unique order of index $p^s$ inside $\mc{O}_{F}$.

\begin{lemma}\label{lemma:local_special_cycles:cycles_rank_2_lattices:lattice_orbit}
Two $\Z_p$-lattices $M, M' \subseteq \mbb{B}^{*}$ of rank $2$
lie in the same $H'_0(\Q_p)$-orbit if and only if $\opn{GK}(M) = \opn{GK}(M')$
and $\opn{disc}(M_{\Q_p}) = \opn{disc}(M'_{\Q_p})$.
This is also equivalent to lying in the same $H'(\Q_p)$-orbit.
\end{lemma}
\begin{proof}
Since the $H'(\Q_p)$-action preserves quadratic forms up to $\Z_p^{\times}$-scalar,
it certainly preserves the invariants in the lemma statement.

Suppose instead that $M$ and $M'$ have the same invariants as in the lemma statement.
Set $D \coloneqq \opn{disc}(M_{\Q_p}) = \opn{disc}(M'_{\Q_p})$
and $F \coloneqq \Q_p[X]/(X^2 - D)$ (we only use this $\Q_p$-algebra up to isomorphism).

We first claim that the condition on discriminants implies that $M_{\Q_p}$
and $M'_{\Q_p}$ lie in the same $H'_0(\Q_p)$-orbit.
In the ordinary case, there is nothing to show as $M_{\Q_p} = M'_{\Q_p} = \mbb{B}^{\mrm{o}}$.
In the supersingular case, Lemma \ref{gk-inv-rk2}
implies that we can find nonzero $x \in M$ and $x' \in M'$ with $v_p(q(x)) = v_p(q(x'))$.
It is enough to show that $x^{-1} M_{\Q_p}$ and $x^{\prime -1} M'_{\Q_p}$
lie in the same $H'(\Q_p)$-orbit.
But $\Q_p \subseteq x^{-1} M_{\Q_p}$ and $\Q_p \subseteq x^{\prime -1} M'_{\Q_p}$.
This implies that $x^{-1} M_{\Q_p} \subseteq \mbb{B}^{\mrm{s}}$
and $x^{\prime -1} M'_{\Q_p} \subseteq \mbb{B}^{\mrm{s}}$ are both
$\Q_p$-subalgebras, not just $\Q_p$-submodules.
These $\Q_p$-subalgebras are both isomorphic to $F$.
Skolem--Noether then implies that $x^{-1} M_{\Q_p}$
and $x^{\prime -1} M'_{\Q_p}$ are both conjugate inside $\mbb{B}^{\mrm{s}}$,
but we would like to show that they are conjugate by an element of $\mc{O}_{\mbb{B}^{s}}^{\times}$,
i.e. a unit in the maximal order of $\mbb{B}^{s}$.
It suffices to note the existence of $u \in \mbb{B}^{\mrm{s}}$ with $v_p(q(u)) = 1$
which normalizes the $\Q_p$-subalgebra $x^{-1} M_{\Q_p}$.
If $F / \Q_p$ is ramified, we can take $u \in x^{-1} M_{\Q_p}$.
If $F / \Q_p$ is inert, there exists $u$ for which conjugation induces
the nontrivial Galois involution on $x^{-1} M_{\Q_p}$.
This finishes the proof that $M_{\Q_p}$ and $M'_{\Q_p}$
lie in the same $H'_0(\Q_p)$-orbit.

To prove the lemma, it thus suffices to verify the following claim.
View $F$ as the $2$-dimensional $\Q_p$ quadratic space
with quadratic form $N_{F / \Q_p}(-)$, and
let $M, M' \subseteq F$ be $\Z_p$-lattices of rank $2$
with the same Gross--Keating invariants.
Then we claim that there exists $b \in F^{\times}$
with $v_p(N_{F / \Q_p}(b)) = 0$ such that $b M = M'$.

First, select any $a, a' \in F^{\times}$
such that $M = a \cdot \mc{O}_{F,s}$ and $M' = a' \cdot \mc{O}_{F,s'}$
for some integers $s, s' \geq 0$
(compare with the proof of Lemma \ref{lemma:CM_EC_height_determined_by_endos},
where we argued that $\mf{a}_p$ and $\mf{a}'_p$ were both free of rank $1$
over an appropriate order).

If $e \coloneqq v_p(N_{F/\Q_p}(a))$ and $e' \coloneqq v_p(N_{F/\Q_p}(a'))$,
then we have
    \begin{align}
    \opn{GK}(M)
    & =
    \begin{cases}
    (e, e + 2s) & \text{if $F / \Q_p$ is unramified}
    \\
    (e, e + 2s + 1) & \text{if $F / \Q_p$ is ramified}
    \end{cases}
    \\
    \opn{GK}(M')
    & =
    \begin{cases}
    (e', e' + 2s') & \text{if $F / \Q_p$ is unramified}
    \\
    (e', e' + 2s' + 1) & \text{if $F / \Q_p$ is ramified}
    \end{cases}
    \end{align}  
(e.g. as in Lemma \ref{lemma:local_special_cycles:qcan:hom_space}).
In particular, $\opn{GK}(M) = \opn{GK}(M')$ if and only if $e' = e$
and $s = s'$, in which case we can take $b \coloneqq a' / a$.
\end{proof}

\begin{lemma}\label{lemma:local_special_cycles:cycles:endo_order}
Let $M \subseteq \mbb{B}^*$ be an integral $\Z_p$-lattice of rank $2$.
Set $F \coloneqq \Q_p[X]/(X^2 - \opn{disc}(M))$.

The two $\Z_p$-subalgebras of $\opn{End}^{\circ}(\mbb{X}^{*})$
            generated by elements of the form $y^{\dagger} x$ (resp. $x y^{\dagger}$)
            for $x,y \in M$ are both (non-canonically) isomorphic
            to $\mc{O}_{F, \lfloor \nu(M) / 2 \rfloor}$.
            Given any $\Z_p$-basis $\{x, y \}$ for $M$,
            these $\Z_p$-subalgebras are equal to $\Z_p[y^{\dagger} x]$ and
            $\Z_p[x y^{\dagger}]$, respectively.
            In particular, $F$ is a field (resp. split) if and only if $\mbb{X}^*$ is supersingular
            (resp. ordinary).
\end{lemma}
\begin{proof}
The supersingular case
should be contained in the work of Gross--Keating (see \cite[{Lemma 3.1}]{R}).
We give a uniform explanation for the supersingular and ordinary cases.

The claim $y^{\dagger} x \in \End(\mbb{X}^*)$
follows because $(y^{\dagger} x)^{\dagger} (y^{\dagger} x) \in \Z_p$,
and $y^{\dagger} x + x^{\dagger} y \in \Z_p$,
using integrality of $M$.
Note that we also conclude $x^{\dagger} y \in \End(\mbb{X}^*)$.
It suffices to prove the lemma for $y^{\dagger} x$:
in the supersingular case, we have $x y^{\dagger} = x (y^{\dagger} x) x^{-1}$
(and recall that $\End^{\circ}(\mbb{X}^{*})^{\times}$ normalizes $\End(\mbb{X}^*)$),
and in the ordinary case, we can use commutativity of $\End^{\circ}(\mbb{X}^*)$.

We compute $\opn{disc}(M) = \frac{1}{4} (y^{\dagger} x - x^{\dagger} y)^2$.
In particular, we find that the $\Q_p$-subalgebra of $\End^{\circ}(\mbb{X}^*)$
generated by $y^{\dagger} x$ is isomorphic to $F$.
We must have $y^{\dagger} x \in \mc{O}_{F}$ (because $y^{\dagger} x \in \End(\mbb{X}^*)$),
and also $y^{\dagger} x \not \in \Q_p$ because $\opn{disc}(M) \neq 0$.

Recall (Lemma \ref{gk-dep-V}) that we have
    \begin{equation}\label{equation:local_special_cycles:cycles:endo_order:index}
    \lfloor \nu(M) / 2 \rfloor
    =
    \begin{cases}
    \lfloor \frac{v_p(4 \opn{disc}(M))}{2} \rfloor - 1 & \text{if $p = 2$ and $F / \Q_p$ is ramified}
    \\
    \lfloor \frac{v_p(4 \opn{disc}(M))}{2} \rfloor & \text{else.}
    \end{cases}
    \end{equation}
On the other hand, for any $a \in \mc{O}_{F}$ with $a \not \in \Z_p$,
some casework shows that the order $\Z_p[a] \subseteq \mc{O}_{F}$
has index $p^{\lfloor (v_p((a - a^{\s})^2)) / 2 \rfloor - 1}$
if $p = 2$ and $F / \Q_p$ is ramified, and has index
$p^{\lfloor (v_p((a - a^{\s})^2)) / 2 \rfloor}$ otherwise,
where $\s \colon F \ra F$ is the unique nontrivial involution.
Since $\dagger$ induces $\s$ on $\Q_p[y^{\dagger} x] \cong F$,
the lemma now follows from the formulas
in \eqref{equation:local_special_cycles:cycles:endo_order:index}
and the expression $\opn{disc}(M) = \frac{1}{4} (y^{\dagger} x - x^{\dagger} y)^2$.
\end{proof}

\begin{corollary}\label{corollary:local_special_cycles:cycles:endo_order}
Let $M$ and $F$ be as in the statement of
Lemma \ref{lemma:local_special_cycles:cycles:endo_order}.
Pick any $\Z_p$-basis $\{x, y \}$ for $M$,
and let $F_M$ (resp. $F'_M$) be the $\Q_p$-subalgebras
of $\End^{\circ}(\mbb{X}^*)$ given by $\Q_p[y^{\dagger} x]$
and $\Q_p[x y^{\dagger}]$, respectively. Pick any isomorphisms $\iota_F \colon F \ra F_M$
and $\iota'_F \colon F \ra F'_M$.

Let $S$ be any scheme over $\Spf \breve{\Z}_p$,
and consider any $(X, \rho, X', \rho') \in \tilde{\mc{Z}}^*(M)(S)$.
Then the map $\iota_F \colon \mc{O}_{F,\lfloor \nu(M) / 2 \rfloor} \ra \End(\mbb{X}^*)$
lifts via $\rho$ to an embedding $\mc{O}_{F,\lfloor \nu(M) / 2 \rfloor} \hookrightarrow
\End(X)$, and similarly for $\iota'_F$, $\rho'$, and $X'$.
\end{corollary}
\begin{proof}
This follows from Lemma \ref{lemma:local_special_cycles:cycles:endo_order},
i.e. $y^{\dagger} x$ lifts via $\rho$ to $\End(X)$,
and $x y^{\dagger}$ lifts via $\rho'$ to $\End(X')$.
We are using the standard lifting of principal polarizations,
as explained after Definition \ref{definition:local_special_cycles:cycles}.
\end{proof}

\begin{lemma}\label{lemma:local_special_cycles:cycles_rank_2_lattices:adic_finite_flat_components}
Let $M \subseteq \mbb{B}^*$ be any $\Z_p$-lattice of rank $2$.
Then $\tilde{\mc{Z}}^*(M)$ is a countable disjoint union of connected components,
each of which is adic finite flat over $\Spf \breve{\Z}_p$,
and with schemification having \'etale generic fiber.
\end{lemma}
\begin{proof}
By the transitive action of $H'(\Q_p)$
on the connected components of $\tilde{\mc{N}}^*$,
it is enough to check that $\mc{Z}^{*,0}(M) \ra \Spf \breve{\Z}_p$
is adic finite flat over $\Spf \breve{\Z}_p$,
with schemification having \'etale generic fiber.

The claim should then be standard in the supersingular case
after \cite[{\S 5}]{GK93} and \cite[{\S 3}]{R}.
We explain this in our language (treating the ordinary case
at the same time), for the reader's convenience.

Given $\gamma \in \End^{\circ}(\mbb{X}^*)$,
we obtain a local special cycle $\mc{Z}^*_0(\gamma) \subseteq \mc{N}^*_0$
defined similarly as in Definition \ref{definition:local_special_cycles:cycles},
but with $(X', \rho')$ replaced with $(X, \rho)$ and $H$ replaced by
the singleton set $\{ \gamma \}$.
For $\gamma \not \in \Q_p$, it is known that $\mc{Z}^*_0(\gamma) \subseteq \mc{N}^*_0$
is an effective Cartier divisor which is adic finite flat over $\Spf \breve{\Z}_p$
with schemification having \'etale generic fiber;
this is explained e.g. in \cite{CLZ} (both ordinary and supersingular cases
are needed).

Pick any $\Z_p$-basis $\{x_1, x_2\}$ for $M$.
Set $\gamma_1 = x_2^{\dagger} x_1$
and $\gamma_2 = x_1 x_2^{\dagger}$.
Then the closed immersion $\mc{Z}^{*,0}(M) \ra \mc{N}^{*,0}$
factors through $\mc{Z}^*_0(\gamma_1) \times_{\Spf \breve{\Z}_p} \mc{Z}^*_0(\gamma_2)$,
by Corollary \ref{corollary:local_special_cycles:cycles:endo_order}.
From this, it immediately follows that $\mc{Z}^{*,0}(M) \ra \Spf \breve{\Z}_p$
is adic finite and with schemification having \'etale generic fiber.

It remains to show that $\mc{Z}^{*,0}(M) \ra \Spf \breve{\Z}_p$ is flat.
For this, first note
$\mc{Z}^{*,0}(M) = \mc{Z}^{*,0}(x_1) \cap \mc{Z}^{*,0}(x_2)$.
For each $i = 1, 2$, we also know that
$\mc{Z}^{*,0}(x_i) \subseteq \mc{N}^{*,0}$
is either an effective Cartier divisor or equal
to all of $\mc{N}^{*,0}$
by a standard method\footnote{This
method at least goes back to Katz--Mazur in \cite[{Proposition 6.8.6}]{KatzMazur1985};
see also e.g. \cite[{Proposition 5.1}]{W}.}
which can be explained as follows:
a deformation theory argument (e.g. Grothendieck--Messing deformation theory)
shows that on the square-zero thickening of $\mc{Z}^{*,0}(x_i)$,
the closed formal subscheme $\mc{Z}^{*,0}(x_i)$ is the zero locus of a
canonical section (depending on $x_i$) of
a certain line bundle coming
from Hodge filtrations; then apply Nakayama's lemma.
We already know that $\mc{Z}^{*,0}(M) \ra \Spf \breve{\Z}_p$ is finite adic
with $\dim \mc{Z}^{*,0}(M)^{\mrm{sch}} \leq 1$,
so it follows that $\mc{Z}^{*,0}(M)^{\mrm{sch}}$ is Cohen--Macaulay,
i.e. because after fixing
an identification $\mc{N}^{*,0} \cong \Spf \breve{\Z}_p[[t, t']]$,
the equations defining $\mc{Z}^{*,0}(x_i)$ and $\mc{Z}^{*,0}(x_i)$
must give a regular sequence in $\Spf \breve{\Z}_p[[t, t']]$ at all points
in the image of $\mc{Z}^{*,0}(M)^{\mrm{sch}} \ra \Spec \breve{\Z}_p[[t, t']]$.
This reasoning also shows that $\mc{Z}^{*,0}(M)^{\mrm{sch}}$
is of pure dimension $1$, so we conclude that $\mc{Z}^{*,0}(M)^{\mrm{sch}} \ra \Spec \breve{\Z}_p$ is flat by miracle flatness.
\end{proof}

    \subsection{Quasi-canonical liftings}
    \label{ssec:local_special_cycles:qcan}
        Let $F$ be a finite \'etale $\Q_p$-algebra of degree $2$.
Then either $F$ is a quadratic field extension of $\Q_p$,
or we have $F \simeq \Q_p \times \Q_p$.
We say that $F$ is non-split if $F/\mathbb{Q}_p$ is a quadratic field extension, and that $F$ is split if $F \simeq \mathbb{Q}_p \times \mathbb{Q}_p$.
We let $\mc{O}_F \subseteq F$ be the integral closure of $\Z_p$,
i.e. $\mathcal{O}_F$ is the ring of integers of $F$ if $F$ is a field,
while we have $\mathcal{O}_F = \mathbb{Z}_p \times \mathbb{Z}_p$
in the split case.
Given any integer $s \geq 0$, we set
    \begin{equation}\label{quasi-can-endos}
        \mathcal{O}_{F,s}=\Z_p+p^{s}\cdot\mathcal{O}_F.
    \end{equation}
for the unique order of index $p^s$ inside $\mc{O}_F$.

Denote by $\Breve{F}$ the field $\Breve{\mathbb{Q}}_p$ when $F$ is split
(equipped with a morphism $F \ra \Breve{\Q}_p$), and otherwise
write $\breve{F}$ for the completion of the maximal unramified extension of $F$,
equipped with a canonical map $F \ra \breve{F}$.

Consider an embedding $\iota_F \colon F \ra \End^{\circ}(\mbb{X}^{*})$.
This necessarily means $* = \mrm{s}$ if $F$ is nonsplit
and $* = \mrm{o}$ if $F$ is split.

We recall the notion of quasi-canonical lifting from \cite{G} (when $F$ is non-split) and \cite{Vol} (when $F$ is split), slightly modified to suit our current purposes.
For $s \geq 0$ and $\ast\in\{\textup{s},\textup{o}\}$, a \emph{quasi-canonical lifting} (and canonical lifting when $s = 0$) for the pair $(\mbb{X}^{*}, \iota_F)$
is an object $(\mf{X}_{F,s}^*, \rho) \in \mc{N}_0^*(\Spf \mc{O}_{\breve{K}})$
for some finite extension $\breve{K} / \breve{\Q}_p$,
with the property that the action $\iota_{F} \colon \mc{O}_{F,s} \ra \End(\mbb{X}^*)$
lifts to $\mf{X}^*_{F,s}$, and that moreover $\End(\mf{X}^*_{F,s}) \cong \mc{O}_{F,s}$.
Quasi-canonical liftings are unique up to isomorphism after forgetting the framing $\rho$.

In the non-split case (resp. the split case),
the subgroup $\mathcal{O}_{F,s}^{\times}\subset\mathcal{O}_{F}^{\times}$
(resp. $(1 + p^s \Z_p)^{\times} \subset \Z_p^{\times}$) has an associated finite totally ramified abelian extension $\Breve{F}_{s}$ of $\Breve{F}$
with ring of integers $\mathcal{O}_{\Breve{F},s}$, by local class field theory. The extension degree is
\begin{equation}\label{local-deg-horizontal}
    [\Breve{F}_s:\Breve{F}]=\begin{cases}
        p^{s}(1-\eta^*(p)p^{-1}), &\textup{if $s\geq1$;}\\
        1, &\textup{if $s=0$.}
    \end{cases}
\end{equation}
where $\eta^*(p)\coloneqq-1,0,1$ in the inert, ramified, and split cases respectively.
In the split case, we have $\Breve{F}_{s}=\Breve{F}[\zeta_{p^{s}}]$ where $\zeta_{p^{s}}$ is a primitive $p^{s}$-th root of unity. By \cite[Proposition 5.3]{G} (the non-split case) and \cite[Proposition 3.5]{Vol} (the split case), the ring of definition of $\mathfrak{X}_{F,s}^{\ast}$ is $\mathcal{O}_{\Breve{F}_s}$.

If $(\mf{X}^*_{F,s}, \rho)$
is a level $s$ quasi-canonical lifting
for $(\mbb{X}^*, \iota_F)$ over $\Spf \mc{O}_{\breve{F}_s}$,
then the induced morphism $\Spf \mc{O}_{\breve{F}_s} \ra \mc{N}_0^*$
is a closed immersion.
Any quasi-canonical lifting of level $s$ for $(\mbb{X}^*, \iota_F)$ over some $\Spf \mc{O}_{\breve{K}}$ is the pullback of $(\mf{X}^*_{F,s}, \rho)$
from the previous sentence along a unique morphism $\Spf \mc{O}_{\breve{K}} \ra \Spf \mc{O}_{\breve{F}_s}$.

\begin{lemma}\label{lemma:local_special_cycles:qcan:hom_space}
Let $r,s$ be nonnegative integers.
Fix embeddings $\iota_F \colon F \ra \End^{\circ}(\mbb{X}^*)$
and $\iota_F' \colon F \ra \End^{\circ}(\mbb{X}^*)$.
Let $\breve{K} / \breve{F}$ be any finite degree field extension,
and $(\mf{X}_{F,r}^{*}, \rho)$ and $(\mf{X}_{F,s}^{*}, \rho')$
be quasi-canonical liftings over $\Spf \mc{O}_{\breve{K}}$ of level $r$ and $s$,
for the pairs $(\mbb{X}^*, \iota_F)$ and $(\mbb{X}^*, \iota_F')$ respectively.
Form the $\Z_p$-submodule
    \begin{equation}\label{equation:local_special_cycles:qcan:hom_space}
    M \coloneqq \Hom(\mf{X}_{F,r}^{*}, \mf{X}_{F,s}^{*})
        \subseteq
        \Hom^{\circ}(\mbb{X}^*, \mbb{X}^*)
        = \mbb{B}^{*}.
    \end{equation}
If $r \leq s$, then the following conclusions hold.
\begin{enumerate}[(1)]
    \item The space $M$
        is a free $\mc{O}_{F,r}$-module of rank one by pre-composition,
        and any generator $\psi$ satisfies $q(\psi) = \deg \psi = p^{|s - r|}$.
        The quadratic form on $M$ is given by $q(\psi \cdot a) = N_{F / \Q_p}(a) q(\psi)$.
        In the ordinary case, any such $\psi \in \Hom(\mbb{X}^*, \mbb{X}^*)$
        is an isomorphism on the connected part.
    \item We have $F \cong \Q_p[X]/(X^2 - \opn{disc}(M))$ (non-canonically).
    \item The lattice $M$ has Gross--Keating invariants $(a,b)$
        which are characterized by the expression
            \begin{equation}\label{equation:local_special_cycles:qcan:hom_space:GK-invariants}
            (r,s) = (\lfloor d(M) / 2 \rfloor, \lfloor \nu(M)/2 \rfloor).
            \quad\quad
            d(M) \coloneqq b - a
            \quad\quad
            \nu(M) \coloneqq a + b
            \end{equation}
        and the additional data of whether $F / \Q_p$ is ramified, i.e.
            \begin{align}
            & (a,b) = ( |s - r|, r + s ) && \text{if $F / \Q_p$ is unramified}
            \\
            & (a,b) = ( |s - r|, r + s + 1) && \text{if $F / \Q_p$ is ramified.}
            \end{align}
    \item Given $f \in \End^{\circ}(\mbb{X}^*)$, we have $M \cdot f \subseteq M$
    if and only if $f \in \iota_F(\mc{O}_{F,r})$.
\end{enumerate}
If $r \geq s$, the above claims hold with the following modifications. Conclusion (1) holds but with $\Hom(\mf{X}_{F,r}^{*}, \mf{X}_{F,s}^{*})$
instead being free of rank one over $\mc{O}_{F,s}$ by post-composition,
and any generator $\psi$ is instead an isomorphism on the \'etale part in the ordinary case.
Conclusion (2) holds verbatim. Conclusion (3) also holds after swapping $r$ and $s$ in 
\eqref{equation:local_special_cycles:qcan:hom_space:GK-invariants}.
Conclusion (4) holds upon considering $f \cdot M$ instead of $M \cdot f$,
and replacing $\iota_F(\mc{O}_{F,r})$ with $\iota'_F(\mc{O}_{F,s})$.
\end{lemma}
\begin{proof}
Note that the $r \leq s$ and $r \geq s$ cases are equivalent
upon applying the Rosati involution
(with respect to the principal polarizations
lifting the fixed principal polarization on $\mbb{X}^*$).
Recall also that $\nu(M)$ is odd if and only if $\Q_p[X]/(X^2 - \opn{disc}(M))$
is a ramified extension of $\Q_p$,
as follows from Lemma \ref{gk-dep-V}.
Also note that claim (2) follows from claim (1),
e.g. as in the proof of Lemma \ref{lemma:local_special_cycles:cycles:endo_order}.
Note that claim (1) implies that the quadratic form on $M$ is integral.

We now explain how claim (4) follows from claims (1) and (2).
It is enough to check the $r \leq s$ case.
The embedding $\iota'_F$ induces an embedding $F \ra \End_{\iota_F(F)}(M_{\Q_p})$,
where $\iota_F(F) \acts M_{\Q_p}$ is given by pre-composition
and $\iota'_F(F) \acts M_{\Q_p}$ is given by post-composition.
But since $M_{\Q_p}$ is free of rank one as a $\iota_F(F)$-module,
we conclude that it is also free of rank one as a $\iota'_F(F)$-module,
and that $\End_{\iota_F(F)}(M_{\Q_p}) = \End_{\iota'_F(F)}(M_{\Q_p})$.
In particular, if $f \in \End^{\circ}(\mbb{X}^*)$ satisfies $M \cdot f \subseteq M$,
then we have $f \in \End_{\iota'_F(F)}(M_{\Q_p}) = \iota_F(F)$.
But since $M$ is free of rank one as an $\iota_F(\mc{O}_{F,r})$-module,
we must moreover have $f \in \iota_F(\mc{O}_{F,r})$.

Given the above remarks, the supersingular case
of the lemma is proved in \cite[{Proposition 1.1}]{W2}.
We thus restrict to the ordinary case.
We fix an identification
$\mbb{X}^{\mrm{o}} \cong \pmb{\mu}_{p^{\infty}} \times \underline{\Q_p / \Z_p}$,
which fixes an identification $\mbb{B}^{\mrm{o}} \cong \Q_p \times \Q_p$.
Given this fixed data,
the lifts $(\mf{X}_{F,r}^{\mrm{o}}, \rho)$ and $(\mf{X}_{F,s}^{\mrm{o}}, \rho')$
correspond to choices of primitive $p^r$-th and $p^s$-th roots of unity
$\zeta_{p^r}, \zeta_{p^s} \in \mc{O}_{\breve{K}}$, respectively
(see the discussion on lifting theory for ordinary $p$-divisible groups
in \cite[{\S 6.1}]{Chen2024}).
By e.g. \cite[{(6.1.5)}]{Chen2024},
an element $(f_1, f_2) \in \mbb{B}^{\mrm{o}}$
lies in $M$ if and only if
    \begin{equation}
    (\zeta_{p^r})^{f_1} = (\zeta_{p^s})^{f_2}.
    \end{equation}
Without loss of generality, we may assume $r \leq s$.
Select $a \in \Z_p$ satisfying $(\zeta_{p^s})^a = \zeta_{p^r}$,
and note $v_p(a) = |s - r|$.
Then $(f_1, f_2)$ lies in $M$ if and only if $f_2 - a f_1 \equiv 0 \pmod{p^s}$.
In other words, every element of $M$ can be written in the form
    \begin{equation}
    \alpha \cdot (1, a) + \beta p^{r} (0,1) \cdot (1, a)
    \end{equation}
for unique values $\alpha, \beta \in \Z_p$, i.e. $M$ is free of
rank one over $\mc{O}_{F,r}$ with generator $(1,a)$.
This proves claims (1) and (2).
It remains to prove (3) in the ordinary case,
i.e. to show that $M$ has Gross--Keating invariants
$(|s - r|, s + r)$.
It is enough to show that the $\mc{O}_{F,r}$
viewed as a $\Z_p$-lattice (of rank $2$) with quadratic form 
$N_{F / \Q_p}(-)$ has Gross--Keating invariants
$(0, 2r)$. For the $\Z_p$-basis $\{ (1,1), (0,p^r) \}$
of $\mc{O}_{F,r}$, the associated
quadratic form is given by $x^2 + p^r xy$.
Inspecting the definition of Gross--Keating invariants
from \eqref{GK-condition} shows that the invariants are indeed
$(0,2r)$ in this case.
\end{proof}

\begin{remark}
The isomorphism class of the quadratic lattice $M$
from Lemma \ref{lemma:local_special_cycles:qcan:hom_space}
can in general depend on the other data, and not just $r,s$.
In both the supersingular and ordinary cases,
we can rescale the quadratic form on $M$ by
an arbitrary element of $\Z_p^{\times}$ by
changing the pair $(\iota'_F, \rho')$,
and this can even change the isomorphism class
of the rational quadratic space $M_{\Q_p}$
(consider the Hasse invariant).
\end{remark}

    \subsection{Primitive cycles}
        \label{ssec:local_special_cycles:primitive_cycles}
        Consider any $\Z_p$-lattice $M \subseteq \mbb{B}^{*}$
of rank $2$.

\begin{definition}\label{def:primitive_cycle}
The \emph{primitive cycle} $\mc{Z}^*(M)^{\circ} \subseteq \mc{Z}^{\ast}(M)$
is the closed formal subscheme given by the (scheme-theoretic) union of
the components which are not contained in $\mc{Z}^*(M')$ for any
rank $2$ lattice $M' \supsetneq M$.
\end{definition}

Lemma \ref{lemma:local_special_cycles:cycles_rank_2_lattices:adic_finite_flat_components}
implies that $\mc{Z}^*(M) \ra \Spf \breve{\Z}_p$
is a countable disjoint union of formal schemes which are adic finite flat
over $\Spf \breve{\Z}_p$, and where the schemification of $\mc{Z}^*(M)$ has \'etale generic fiber.
The ``components'' in Definition \ref{def:primitive_cycle}
mean the closed formal subschemes associated to the irreducible components of the schemification.
Note that $\mc{Z}^*(M)^{\circ} = \emptyset$ if $M$ is not integral
(because $\mc{Z}^*(M) = \emptyset$ if $M$ is not integral).

\begin{lemma}\label{hor-decomposition}
For any $\Z_p$-lattice $L \subseteq \mbb{B}^*$
of rank $2$, we have a scheme-theoretic union\footnote{
We mean to intersect ideal sheaves.
} decomposition
    \begin{equation}
    \mc{Z}^*(L) = \bigcup_{\substack{L \subseteq M \\ M \text{ integral}}}
    \mc{Z}^*(M)^{\circ}
    \end{equation}
of closed formal subschemes of $\mc{N}^*$,
where the union runs over $\Z_p$-lattices $M \subseteq L_{\Q_p}$ of rank $2$.
\end{lemma}
\begin{proof}
This follows from the definition of $\mc{Z}^*(M)^{\circ}$.
For fixed $L$, there are only finitely many integral $\Z_p$-lattices
$M \subseteq L_{\Q_p}$ satisfying $M \supseteq L$.
\end{proof}

\begin{lemma}\label{lemma:local_special_cycles:primitive_cycles:hom_lattice}
Let $\breve{K}$ be any finite degree field extension of $\breve{\Q}_p$.
Given any object $(X, \rho, X', \rho') \in \mc{N}^*(\Spf \mc{O}_{\breve{K}})$,
the corresponding morphism $\Spf \mc{O}_{\breve{K}} \ra \mc{N}^*$
factors through $\mc{Z}^*(M)^{\circ}$
if and only if
    \begin{equation}
    \Hom(X,X') = M
    \end{equation}
as $\Z_p$-submodules of $\mbb{B}^*$.
\end{lemma}
\begin{proof}
We are using the usual Drinfeld rigidity to view $\Hom(X,X')$
as a submodule of $\mbb{B}^*$ (using $\rho$ and $\rho')$.
If $\Hom(X,X') = M$, then the morphism $\Spf \mc{O}_{\breve{K}} \ra \mc{N}^*$
certainly factors through $\mc{Z}^*(M)$, 
and it also cannot factor through $\mc{Z}^*(M')$ for any rank $2$ lattice $M' \supsetneq M$,
i.e. $\Spf \mc{O}_{\breve{K}} \ra \mc{N}^*$ factors through $\mc{Z}^*(M)^{\circ}$.

Conversely, consider $(X, \rho, X', \rho') \in \mc{Z}^*(M)^{\circ}(\Spf \mc{O}_{\breve{K}})$.
Set $F \coloneqq \Q_p[X]/(X^2 - \opn{disc}(M))$.
Corollary \ref{corollary:local_special_cycles:cycles:endo_order}
implies that both $X$ and $X'$ admit faithful actions by $\mc{O}_{F,\lfloor \nu(M) / 2 \rfloor}$.
We claim that $\End^{\circ}(X) \cong \End^{\circ}(X') \cong F$ (non-canonically).
It is enough to check that $\End^{\circ}(X)$ and $\End^{\circ}(X')$ are both commutative,
since any embedding $F \hookrightarrow \End^{\circ}(\mbb{X}^*)$ has image which is self-centralizing.
In the ordinary case, this follows because $\End^{\circ}(\mbb{X}^*) \cong \Q_p \times \Q_p$
is commutative.
In the supersingular case, this follows because there exists a morphism $\End^{\circ}(X') \ra \breve{K}$
given by the action $\End(X') \ra \End(\Lie X') \cong \mc{O}_{\breve{K}}$
(then use the fact that any $\Q_p$-subalgebra of $\End^{\circ}(\mbb{X}^*)$ is simple).
It now follows that there exists $(h, h') \in H'_0(\Q_p)$ such that
$(X, h \rho)$ and $(X', h' \rho')$
are both quasi-canonical liftings of levels $s, s' \leq \lfloor \nu(M) / 2 \rfloor$,
with respect to certain embeddings $\iota_F, \iota'_F \colon F \ra \End^{\circ}(\mbb{X}^*)$,
which we will not need to describe explicitly.
It suffices to take any $(h,h') \in H'_0(\Q_p)$ such that $h \rho$ and $h' \rho'$ are
both isomorphisms in geometric fibers; this is automatic in the supersingular case.

Then Lemma \ref{lemma:local_special_cycles:qcan:hom_space}
implies that $M' \coloneqq \Hom(X,X') \subseteq \mbb{B}^*$ is a $\Z_p$-sublattice of rank $2$,
where we make the identification $\Hom(X,X') \subseteq \mbb{B}^*$ using $(\rho, \rho')$.
We must have $M' \supseteq M$, by construction.
The morphism $\Spf \mc{O}_{\breve{K}} \ra \mc{N}^*$ associated
to $(X, \rho, X', \rho')$ factors through $\mc{Z}^*(M')$.
The construction of $\mc{Z}^*(M)^{\circ}$ now implies $M' = M$.
\end{proof}

\begin{corollary}\label{corollary:local_special_cycles:primitive_cycles:distinct_components}
Given distinct $\Z_p$-lattices $M, M' \subseteq \mbb{B}^*$ of rank $2$,
the closed formal subschemes $\mc{Z}^*(M)^{\circ} \subseteq \mc{N}^*$
and $\mc{Z}^*(M')^{\circ} \subseteq \mc{N}^*$ do not share any components.
\end{corollary}
\begin{proof}
This follows immediately from Lemma \ref{lemma:local_special_cycles:primitive_cycles:hom_lattice},
i.e. the existence of such a common component
implies the existence of $(X, \rho, X', \rho') \in \mc{N}^*(\Spf \mc{O}_{\breve{K}})$
for some finite degree field extension $\breve{K} / \breve{\Q}_p$
such that $\Hom(X, X') = M$ and $\Hom(X, X') = M'$.
We are again implicitly using the fact that $\mc{Z}^*(M)^{\circ}$ and $\mc{Z}^*(M')^{\circ}$
are disjoint unions of formal schemes which are adic finite flat over $\Spf \breve{\Z}_p$
with \'etale generic fiber,
in order to work with components.
\end{proof}

In the next lemma, we again set $F \coloneqq \Q_p[X]/(X^2 - \opn{disc}(M))$.

\begin{lemma}\label{primitive-part}
Let $M \subseteq \mbb{B}^*$ be any integral $\Z_p$-lattice of rank $2$.
There exists a scheme-theoretic union decomposition
    \begin{equation}
    \mc{Z}^*(M)^{\circ} = \mc{Z}^*(M)^{\circ}_+ \cup \mc{Z}^*(M)^{\circ}_{-}
    \end{equation}
of closed formal subschemes, characterized by the following requirements.
    \begin{enumerate}
        \item We have $\mc{Z}^*(M)^{\circ}_+ \cong \Spf \mc{O}_{\breve{F}_{\lfloor \nu(M) / 2\rfloor}}$
        and $\mc{Z}^*(M)^{\circ}_- \cong \Spf \mc{O}_{\breve{F}_{\lfloor \nu(M) / 2\rfloor}}$ as formal schemes over $\Spf \breve{\Z}_p$.
        \item The universal pairs of $p$-divisible groups over $\mc{Z}^*(M)^{\circ}_{+}$ and $\mc{Z}^*(M)^{\circ}_{-}$
        (after forgetting framings $\rho$ and $\rho'$)
        are of the form 
            \begin{equation}
            (\mf{X}^*_{F, \lfloor d(M) / 2 \rfloor}, \mf{X}^*_{F, \lfloor \nu(M) / 2 \rfloor})
            \quad\quad
            \text{and}
            \quad\quad
            (\mf{X}^*_{F, \lfloor \nu(M) / 2 \rfloor}, \mf{X}^*_{F, \lfloor d(M) / 2 \rfloor})
            \end{equation}
        respectively, where each $\mf{X}^*_{F, s}$ denotes any $p$-divisible group\footnote{Note
        that the condition $\End(\mf{X}^*_{F, s}) \cong \mc{O}_{F,s}$ forces
        $\mf{X}^*_{F, s}$ to be supersingular (resp. ordinary) if $F$ is nonsplit (resp. split).}
        of height $2$ dimension $1$ such that $\End(\mf{X}^*_{F, s}) \cong \mc{O}_{F,s}$.
    \end{enumerate}
If $t(M) \leq 1$, then we have $\mc{Z}^*(M)^{\circ}_+ = \mc{Z}^*(M)^{\circ}_{-}$.
If $t(M) = 2$, then we have $\mc{Z}^*(M)^{\circ}_+ \neq \mc{Z}^*(M)^{\circ}_{-}$.
The formal scheme $\mc{Z}^*(M)^{\circ}$ is disconnected
in the ordinary case if $t(M) = 2$, and is connected in all other cases.
\end{lemma}
\begin{proof}
Let $\mf{X}^*_{F, \lfloor d(M) / 2 \rfloor}$ and $\mf{X}^*_{F, \lfloor \nu(M) / 2 \rfloor}$
be $p$-divisible groups over $\Spf \mc{O}_{\breve{F}_{\lfloor \nu(M) / 2 \rfloor}}$
as in the lemma statement.
Recall that these $p$-divisible groups are unique up to isomorphism
(from the theory of quasi-canonical liftings,
as recalled in Section \ref{ssec:local_special_cycles:qcan}).

Select framings $\pmb{\rho}$ and $\pmb{\rho}'$ 
to obtain
$(\mf{X}^*_{F, \lfloor d(M) / 2 \rfloor}, \pmb{\rho}, \mf{X}^*_{F, \lfloor \nu(M) / 2 \rfloor}, \pmb{\rho}') \in \mc{N}^{*,0}(\Spf \mc{O}_{\breve{F}_{\lfloor \nu(M) / 2 \rfloor}})$.
Any $\pmb{\rho}$ and $\pmb{\rho}'$ which are isomorphisms in geometric fibers will suffice.
We claim that the associated morphism $\Spf \mc{O}_{\breve{F}_{\lfloor \nu(M) / 2 \rfloor}}
\ra \mc{N}^{*,0}$ is a closed immersion. Indeed, the theory of quasi-canonical liftings
(recalled in Section \ref{ssec:local_special_cycles:qcan})
implies that it is a closed immersion after projection to the second factor
in $\mc{N}^{*,0} \cong \mc{N}_0^* \times \mc{N}_0^*$,
so $\Spf \mc{O}_{\breve{F}_{\lfloor \nu(M) / 2 \rfloor}}
\ra \mc{N}^{*,0}$ being a closed immersion follows from separatedness
of that projection morphism (and the usual graph argument).

By Lemma \ref{lemma:local_special_cycles:qcan:hom_space},
the data $(\pmb{\rho}, \pmb{\rho}')$ identifies
$\Hom(\mf{X}^*_{F, \lfloor d(M) / 2 \rfloor}, \mf{X}^*_{F, \lfloor \nu(M) / 2 \rfloor})$
with a $\Z_p$-lattice $M' \subseteq \mbb{B}^*$ of rank $2$,
with $\opn{disc}(M') = \opn{disc}(M)$ modulo $(\Q_p^{\times})^2$
and with $\opn{GK}(M') = \opn{GK}(M)$.
By Lemma \ref{lemma:local_special_cycles:cycles_rank_2_lattices:lattice_orbit},
we conclude that $M$ and $M'$ lie in the same $H'_0(\Q_p)$-orbit.
That is, we can find $(h, h') \in H'_0(\Q_p)$ such that
    \begin{equation}\label{equation:local_special_cycles:primitive_cycles:+_component_def}
    (\mf{X}^*_{F, \lfloor d(M) / 2 \rfloor}, h \pmb{\rho}, \mf{X}^*_{F, \lfloor \nu(M) / 2 \rfloor}, h' \pmb{\rho}') \in \mc{N}^{*}(\Spf \mc{O}_{\breve{F}_{\lfloor \nu(M) / 2 \rfloor}})
    \end{equation}
is such that the pair $(h \pmb{\rho}, h' \pmb{\rho}')$
induces $\Hom(\mf{X}^*_{F, \lfloor d(M) / 2 \rfloor}, \mf{X}^*_{F, \lfloor \nu(M) / 2 \rfloor}) = M$ as sublattices of $\mbb{B}^*$. Fix one such choice of $(h, h')$.
The morphism $\Spf \mc{O}_{\breve{F}_{\lfloor \nu(M) / 2 \rfloor}} \ra \mc{N}^*$
associated to \eqref{equation:local_special_cycles:primitive_cycles:+_component_def}
is also a closed immersion, because it is a $H'_0(\Q_p)$-translate
of the previous morphism.
We define $\mc{Z}^*(M)^{\circ}_+ \subseteq \mc{N}^*$ to be the closed subfunctor
defined by the closed immersion corresponding
to \eqref{equation:local_special_cycles:primitive_cycles:+_component_def}.
We define $\mc{Z}^*(M)^{\circ}_{-} \subseteq \mc{N}^*$ by a similar procedure,
where we swap $\mf{X}^*_{F, \lfloor d(M) / 2 \rfloor}$
and $\mf{X}^*_{F, \lfloor \nu(M) / 2 \rfloor}$.
If $t(M) = 2$, we must have $\mc{Z}^*(M)^{\circ}_{+} \neq \mc{Z}^*(M)^{\circ}_{-}$
as the two universal pairs of $p$-divisible groups over
$\mc{Z}^*(M)^{\circ}_{+}$ and $\mc{Z}^*(M)^{\circ}_{-}$
are non-isomorphic even after forgetting the framings.

Next, consider any finite degree field extension $\breve{K} / \breve{\Q}_p$,
and consider any object
    \begin{equation}
    (X, \rho, X', \rho') \in \mc{Z}^*(M)^{\circ}(\Spf \mc{O}_{\breve{K}})
    \end{equation}
corresponding to a morphism $\Spf \mc{O}_{\breve{K}} \ra \mc{Z}^*(M)^{\circ}$.
As in the proof of Lemma \ref{lemma:local_special_cycles:primitive_cycles:hom_lattice},
we have $\End(X) \cong \mc{O}_{F,s}$ and $\End(X') \cong \mc{O}_{F,s'}$
for some integers $0 \leq s,s' \leq \lfloor \nu(M) / 2 \rfloor$.
The pair $(\rho, \rho')$ identifies $\Hom(X,X') = M$
as sublattices of $\mbb{B}^*$. Inspecting the formulas for Gross--Keating invariants
in Lemma \ref{lemma:local_special_cycles:qcan:hom_space}(3)
(and also the $H'_0(\Q_p)$-invariance of Gross--Keating invariants
from Lemma \ref{lemma:local_special_cycles:cycles_rank_2_lattices:lattice_orbit})
implies that the only possibilities are
$(s,s') = (\lfloor d(M) / 2 \rfloor, \lfloor \nu(M) / 2 \rfloor)$ 
or $(s, s') = (\lfloor \nu(M) / 2 \rfloor, \lfloor d(M) / 2 \rfloor)$.

It suffices to show that $\Spf \mc{O}_{\breve{K}} \ra \mc{Z}^*(M)^{\circ}$
factors through $\mc{Z}^*(M)^{\circ}_+$ in the first case,
and through $\mc{Z}^*(M)^{\circ}_{-}$ in the second case.
Note that this will also show that $\mc{Z}^*(M)^{\circ}_+ = \mc{Z}^*(M)^{\circ}_{-}$
if $t(M) \leq 1$, which is equivalent to
the condition $\nu(M) = d(M)$ (also recall that $d(M) \equiv \nu(M) \pmod{2}$).

Suppose $(s,s') = (\lfloor d(M) / 2 \rfloor, \lfloor \nu(M) / 2 \rfloor)$.
As $\mf{X}^*_{F, \lfloor d(M) / 2 \rfloor}$ and $\mf{X}^*_{F, \lfloor \nu(M) / 2 \rfloor}$
are characterized up to isomorphism by their endomorphism rings (and this holds
also over $\Spf \mc{O}_{\breve{K}}$; see Section \ref{ssec:local_special_cycles:qcan} again),
we may write
    \begin{equation}
    (X, \rho, X', \rho') = (\mf{X}^*_{F, \lfloor d(M) / 2 \rfloor}, g \pmb{\rho}, \mf{X}^*_{F, \lfloor \nu(M) / 2 \rfloor}, g' \pmb{\rho}')
    \end{equation}
for an element $(g, g') \in H'_0(\Q_p)$,
and where $\mf{X}^*_{F, \lfloor d(M) / 2 \rfloor}$ and $\mf{X}^*_{F, \lfloor \nu(M) / 2 \rfloor}$ are defined by pullback
of the $p$-divisible groups fixed above (abusing notation and recycling the same symbols)
along some morphism
$\phi_0 \colon \Spf \mc{O}_{\breve{K}} \ra \Spf \mc{O}_{\breve{F}_{\lfloor \nu(M) / 2 \rfloor}}$.

We first claim that $(g, g') = (h, h')$ modulo $H'(\Z_p)$
with $(h,h')$ as in the above
construction of $\mc{Z}^*(M)^{\circ}_{+}$.
This is automatic in the supersingular case, since
$H'_0(\Q_p) / H'(\Z_p) = \{ 0 \}$ in that case.
Assume instead that $X$ and $X'$ are ordinary.
By Lemma \ref{lemma:local_special_cycles:qcan:hom_space},
(and the above construction of $\mc{Z}^*(M)^{\circ}_+$)
we know that the action $\mc{O}_{F,s} \acts \mf{X}^*_{F, \lfloor d(M) / 2 \rfloor}$
gives $M$ the structure of a free $\mc{O}_{F,s}$-module of rank one.
Pick any generator $\psi \in M$.
Fix an identification $\mbb{X}^{\mrm{o}} \cong \pmb{\mu}_{p^{\infty}} \times \underline{\Q_p / \Z_p}$ as in the proof of Lemma \ref{lemma:local_special_cycles:qcan:hom_space}.
By part (1) of that lemma, we must have
$g^{\prime -1} \psi g = h^{\prime -1} \psi h = (1, p^{s' - s})$ modulo $\Z_p^{\times} \times \Z_p^{\times}$. This is enough to show $(g, g') = (h, h')$ modulo $H'(\Z_p)$,
since $g' = h' = \psi \cdot (1, p^{s - s'})$ modulo $\Z_p^{\times} \times \Z_p^{\times}$.

Pick any $\Z_p$-basis $\{x, y\}$ for $M$. Let $F_M, F'_M \subseteq \End^{\circ}(\mbb{X}^*)$
be the $\Q_p$-subalgebras generated by $y^{\dagger} x$ and $x y^{\dagger}$, respectively.
Fix any isomorphisms $\iota_F \colon F \ra F_M$ and $\iota'_F \colon F \ra F'_M$
(which exist by Lemma \ref{lemma:local_special_cycles:cycles:endo_order}).

We claim that the pair $(\mf{X}^*_{F, \lfloor \nu(M) / 2 \rfloor}, g' \pmb{\rho}')$
over $\Spf \mc{O}_{\breve{K}}$ is (up to isomorphism compatible with framings) the pullback
of $(\mf{X}^*_{F, \lfloor \nu(M) / 2 \rfloor}, h' \pmb{\rho}')$
along a unique morphism $\phi \colon \Spf \mc{O}_{\breve{K}} \ra \Spf \mc{O}_{\breve{F}_{\lfloor \nu(M) / 2 \rfloor}}$
(which need not coincide with $\phi_0$).
Indeed, first note that
$(\mf{X}^*_{F, \lfloor \nu(M) / 2 \rfloor}, h^{\prime -1} g' \pmb{\rho}')$
over $\Spf \mc{O}_{\breve{K}}$
is a quasi-canonical lifting (of level $\lfloor \nu(M) / 2 \rfloor$)
for the pair $(\mbb{X}^*, a \mapsto h^{\prime -1} \iota'_F(a) h^{\prime})$
over $\Spf \mc{O}_{\breve{K}}$, i.e.
Corollary \ref{corollary:local_special_cycles:cycles:endo_order}
implies that $\iota'_F(\mc{O}_{F, \lfloor \nu(M) / 2 \rfloor})$ 
lifts to $\End(\mf{X}^*_{F, \lfloor \nu(M) / 2 \rfloor})$ with respect
to the framing $(g' \pmb{\rho}')$.
We are also using $h^{\prime -1} g' \in \End(\mbb{X}^*)$ here.
We similarly know that $(\mf{X}^*_{F, \lfloor \nu(M) / 2 \rfloor}, \pmb{\rho}')$
over $\Spf \mc{O}_{\breve{F}_{\lfloor \nu(M) / 2 \rfloor}}$
is a quasi-canonical lifting for the pair $(\mbb{X}^*, a \mapsto h^{\prime -1} \iota'_F(a) h^{\prime})$.
The discussion preceding 
Lemma \ref{lemma:local_special_cycles:qcan:hom_space}
now shows that $(\mf{X}^*_{F, \lfloor \nu(M) / 2 \rfloor}, h^{\prime -1} g' \pmb{\rho}')$
arises as pullback of $(\mf{X}^*_{F, \lfloor \nu(M) / 2 \rfloor}, \pmb{\rho}')$
along a unique morphism
$\phi \colon \Spf \mc{O}_{\breve{K}} \ra \Spf \mc{O}_{\breve{F}_{\lfloor \nu(M) / 2 \rfloor}}$.

Pulling back the entire tuple \eqref{equation:local_special_cycles:primitive_cycles:+_component_def}
along $\phi$ yields a tuple
    \begin{equation}
    (\mf{X}^*_{F, \lfloor d(M) / 2 \rfloor}, f g \pmb{\rho}, \mf{X}^*_{F, \lfloor \nu(M) / 2 \rfloor}, g' \pmb{\rho}') \in \mc{Z}^*(M)^{\circ}_{+}(\Spf \mc{O}_{\breve{\breve{K}}})
    \end{equation}
for some $f \in \End(\mbb{X}^{*})^{\times}$.
Our task is to show that the pairs $(\mf{X}^*_{F, \lfloor d(M) / 2 \rfloor}, f g \pmb{\rho})$
and $(\mf{X}^*_{F, \lfloor d(M) / 2 \rfloor}, g \pmb{\rho})$ are isomorphic.

By construction, we have $\Hom(\mf{X}^*_{F, \lfloor d(M) / 2 \rfloor}, \mf{X}^*_{F, \lfloor \nu(M) / 2 \rfloor}) = M$ as sublattices of $\mbb{B}^*$
with respect to both framings $(g \pmb{\rho}, g' \pmb{\rho}')$
and $(f g \pmb{\rho}, g' \pmb{\rho}')$.
This implies $M \cdot f = M$.
Applying Lemma \ref{lemma:local_special_cycles:qcan:hom_space}(4),
we conclude $f \in \iota_F(\mc{O}_{F, \lfloor d(M) / 2 \rfloor}^{\times})$.
But by construction, the embedding
$\iota_F \colon \mc{O}_{F,\lfloor d(M) / 2 \rfloor} \ra \End(\mbb{X}^*)$
lifts via $g \pmb{\rho}$
to an action $\mc{O}_{F, \lfloor d(M) / 2 \rfloor} \acts \mf{X}^*_{F, \lfloor d(M) / 2 \rfloor}$, which shows that
the pairs $(\mf{X}^*_{F, \lfloor d(M) / 2 \rfloor}, f g \pmb{\rho})$
and $(\mf{X}^*_{F, \lfloor d(M) / 2 \rfloor}, g \pmb{\rho})$
are isomorphic.\footnote{In more detail: Corollary \ref{corollary:local_special_cycles:cycles:endo_order}
shows that $\iota_F$ lifts via $g \pmb{\rho}$ to an embedding $\mc{O}_{F, \lfloor \nu(M) / 2 \rfloor} \hookrightarrow \End(\mf{X}^*_{F, \lfloor d(M) / 2 \rfloor})$,
but since $\End(\mf{X}^*_{F, \lfloor d(M) / 2 \rfloor}) \cong \mc{O}_{F, \lfloor d(M) / 2 \rfloor}$, this lift of $\iota_F$ must extend to the possibly larger ring
$\mc{O}_{F, \lfloor d(M) / 2 \rfloor}$.}

It remains to consider the case
$(s,s') = (\lfloor \nu(M) / 2 \rfloor, \lfloor d(M) / 2 \rfloor)$.
This is largely similar to the case considered above,
but we briefly indicate the argument.
First write 
    \begin{equation}
    (X, \rho, X', \rho') = (\mf{X}^*_{F, \lfloor \nu(M) / 2 \rfloor}, g \pmb{\rho}, \mf{X}^*_{F, \lfloor d(M) / 2 \rfloor}, g' \pmb{\rho}')
    \end{equation}
for some $(g, g') \in H'_0(\Q_p)$.
Let $(h,h') \in H'_0(\Q_p)$ be as in the construction for $\mc{Z}^*(M)^{\circ}_{-}$
(analogously to the construction of $\mc{Z}^*(M)^{\circ}_{+}$).
We can skip the step of showing $(g, g') = (h, h')$ modulo $H'(\Z_p)$,
though it will also be true, upon replacing $(1, p^{s' - s})$ with $(p^{s - s'}, 1)$.
Take the same $\iota_F$ and $\iota'_F$ as above.
Arguing as above shows that the pair $(\mf{X}^*_{F, \lfloor \nu(M) / 2 \rfloor}, g \pmb{\rho})$
is the pullback of $(\mf{X}^*_{F, \lfloor \nu(M) / 2 \rfloor}, h \pmb{\rho})$
along a unique morphism $\phi \colon \Spf \mc{O}_{\breve{K}} \ra \Spf 
\mc{O}_{\breve{F}_{\lfloor \nu(M) / 2 \rfloor}}$.
We then obtain a point
    \begin{equation}
    (\mf{X}^*_{F,\lfloor \nu(M) / 2 \rfloor}, g \pmb{\rho}, \mf{X}^*_{F, \lfloor d(M) / 2 \rfloor}, f g' \pmb{\rho}') \in \mc{Z}^*(M)^{\circ}_{-}(\Spf \mc{O}_{\breve{K}})
    \end{equation}
for some $f \in \End^{\circ}(\mbb{X}^*)$, and would like to show
$f \in \iota'_F(\mc{O}_{F, \lfloor d(M) / 2 \rfloor}^{\times})$.
Again, we find $f \cdot M = M$, which implies the claim
via Lemma \ref{lemma:local_special_cycles:qcan:hom_space}(4).
This then shows that the pairs $(\mf{X}^*_{F, \lfloor d(M) / 2 \rfloor}, f g' \pmb{\rho}')$
and $(\mf{X}^*_{F, \lfloor d(M) / 2 \rfloor}, g' \pmb{\rho}')$ are isomorphic.

The last remaining claim of the lemma statement is to show that
$\mc{Z}^*(M)^{\circ}$ is disconnected if $t(M) = 2$ in the ordinary case
(as $\mc{Z}^*(M)^{\circ}$ is certainly connected in the other cases;
recall that $\mc{N}^*(\F)$ is a singleton in the supersingular case).
Recall that we have a canonical identification $\mc{N}^*(\F) = H'_0(\Q_p) / H'(\Z_p)$
(Section \ref{ssec:local_special_cycles:RZ}).
In the above notation, let $\psi \in M$ be generator for $M$
as a free $\iota_F(\mc{O}_{F,\lfloor d(M) / 2 \rfloor})$-module
(the element $\psi$ will also be a generator
for $M$ as a free $\iota'_F(\mc{O}_{F, \lfloor d(M) / 2 \rfloor})$-module).
Then we have
    \begin{align}
    \mc{Z}^*(M)^{\circ}_+(\F) & = \{ (h, h') \in H'_0(\Q_p) / H'(\Z_p) : h' = \psi \cdot (1, p^{\lfloor d(M) / 2 \rfloor - \lfloor \nu(M) / 2 \rfloor}) \text{ modulo } \Z_p^{\times} \times \Z_p^{\times} \}
    \\
    \mc{Z}^*(M)^{\circ}_{-}(\F) & = \{ (h, h') \in H'_0(\Q_p) / H'(\Z_p) : h' = \psi \cdot (p^{\lfloor d(M) / 2 \rfloor - \lfloor \nu(M) / 2 \rfloor}, 1) \text{ modulo } \Z_p^{\times} \times \Z_p^{\times} \}
    \end{align}
(using the identification $\End^{\circ}(\mbb{X}^*) \cong \Q_p \times \Q_p$ induced
by the identification $\mbb{X}^* \cong \pmb{\mu}_{p^{\infty}} \times \underline{\Q_p / \Z_p}$ from above).
In particular, we have $\lfloor d(M) / 2 \rfloor \neq \lfloor \nu(M) / 2 \rfloor$
if $t(M) = 2$, and thus $\mc{Z}^*(M)^{\circ}_+(\F) \neq \mc{Z}^*(M)^{\circ}_{-}(\F)$
in this case.
\end{proof}

\begin{corollary}\label{corollary:local_special_cycles:primitive_cycles:deg_GK_disc}
For $\Z_p$-lattices $M \subseteq \mbb{B}^*$
of rank $2$, the quantities
    \begin{equation}
    \deg( \mc{Z}^*(M) \ra \Spf \breve{\Z}_p )
    \quad\quad
    \deg( \mc{Z}^*(M)^{\circ} \ra \Spf \breve{\Z}_p )
    \end{equation}
depend only on the data of $\opn{GK}(M)$
and $\opn{disc} M_{\Q_p}$.
\end{corollary}
\begin{proof}
For $\mc{Z}^*(M)^{\circ}$, this follows
immediately from Lemma \ref{primitive-part}.
For $\mc{Z}^*(M)$, this follows
from the fact that the $H'(\Q_p)$-orbit
of $M$ determines and is determined by
$\opn{GK}(M)$ and $\opn{disc} M_{\Q_p}$ as an element of $\Q_p^{\times} / \Q_p^{\times 2}$
(Lemma \ref{lemma:local_special_cycles:cycles_rank_2_lattices:lattice_orbit});
one can then apply the claim for $\mc{Z}^*(M)^{\circ}$ in combination
with Lemmas \ref{hor-decomposition}
and Corollary \ref{corollary:local_special_cycles:primitive_cycles:distinct_components},
as $H'(\Q_p) \acts \mbb{B}^*$ preserves the quadratic form up to $\Z_p^{\times}$-scalar.
\end{proof}

    \subsection{Local ``height defect'' of special cycles}
        \label{ssec:local_special_cycles:height_defect}
        We conclude this section with a key definition needed to formulate our
main non-Archimedean local theorem (Theorem \ref{corank-1-local}).
Let $L$ be a non-degenerate quadratic $\Z_p$-lattice of rank $2$,
and set $F = \Q_p[X]/(X^2 - \opn{disc} L)$.
Pick an embedding of quadratic spaces $L \ra \mbb{B}^*$: this always exists,
and we necessarily have $* = \mrm{s}$ if $F$ is nonsplit and $* = \mrm{o}$
if $F$ is split.

Pick any embedding $\breve{\Q}_p \ra \C_p$.
Given $z \in \mc{Z}^*(L)^{\mrm{sch}}(\C_p)$,
pick any descent of $z$ to a morphism $\Spf \mc{O}_{\breve{L}} \ra \mc{Z}(L)$
for some finite degree field extension $\breve{L}$ of $\breve{\Q}_p$.
By Lemmas \ref{primitive-part} and \ref{hor-decomposition},
pulling back the universal pair of $p$-divisible groups from $\mc{Z}^*(L)$
to $\Spf \mc{O}_{\breve{L}}$ (forgetting the extra data such as framings, etc.)
gives a pair of $p$-divisible groups $\mf{X}$ and $\mf{X}'$
whose endomorphism rings are (up to isomorphism)
the unique orders in $\mc{O}_F$ of index $p^{s_{z}}$ and $p^{s'_{z}}$,
for some $s_z, s'_z \in \Z_{\geq 0}$.

\begin{definition}
    \label{definition:height_change:special_cycles}
With notation as above, we set
    \begin{equation}
        \label{equation:height_change:special_cycles:groupoid_sum}
    \bdf(L)
    \coloneqq
    \sum_{z \in \mc{Z}^*(L)^{\mrm{sch}}(\C_p)} \bdf(z)
        \quad\quad
    \bdf(z)
    \coloneqq
        \delta_{\mrm{Fal},F}(s_z) + \delta_{\mrm{Fal},F}(s'_z)
    \end{equation}
with $\delta_{\mrm{Fal},F}(s) \in \Q$
as defined in \eqref{equation:delta_Fal_def}.
\end{definition}
By Lemmas \ref{primitive-part} and \ref{hor-decomposition}
and Corollary \ref{corollary:local_special_cycles:primitive_cycles:distinct_components},
we equivalently have
    \begin{equation}\label{local-change-delta}
    \bdf(L) \sum\limits_{\substack{ L\subseteq M \\ \text{$M$ integral} }} \deg\mathcal{Z}^{\ast}(M)^{\circ}\cdot\left(\bdf\left(\left \lfloor \frac{d(M)}{2} \right \rfloor\right)+\bdf\left(\left \lfloor \frac{\nu(M)}{2} \right \rfloor\right)\right)
    \end{equation}
where the sum runs over lattices $M \subseteq L_{\Q_p}$.
Using Lemma \ref{lemma:local_special_cycles:cycles_rank_2_lattices:lattice_orbit},
we see that $\bdf(L)$ depends only on $\opn{GK}(L)$ and $\opn{disc}(L_{\Q_p})$,
and is independent of the choice of embedding $L \ra \mbb{B}^*$.
We also see that the right-hand side of \eqref{equation:height_change:special_cycles:groupoid_sum}
does not change if $\mc{Z}^*(L)^{\mrm{sch}}$ is replaced by $\tilde{\mc{Z}}^*(L)^{\mrm{sch}}_i$,
upon using the action $H'(\Q_p) \acts \tilde{\mc{N}}^*$,
which is transitive on the decomposition in \eqref{equation:local_special_cycles:RZ:decomposition}
since the action $H'(\Q_p) \acts \mbb{B}^*$
which preserves the quadratic form up to scalar,
and hence preserves discriminants and Gross--Keating invariants
of rank $2$ lattices (Lemma \ref{lemma:local_special_cycles:cycles_rank_2_lattices:lattice_orbit}).

Next, consider $T \in \opn{Sym}_r(\Q_p)$ of rank $2$.
Let $L$ be the unique non-degenerate quadratic $\Z_p$-lattice of rank $2$
which admits a tuple $\bsm{x} \in L^r$
with $\opn{span}(\bsm{x}) = L$ and with
normalized Gram matrix $\frac{1}{2} (\bsm{x}, \bsm{x}) = T$.
For any choice of embedding $L \ra \mbb{B}^*$,
with ``$*$'' uniquely determined by $L$, we set
    \begin{equation}\label{equation:delta_T_deg_T}
    \delta_{\mrm{Fal},p}(T) \coloneqq \delta_{\mrm{Fal}}(L)
    \quad\quad
    \deg_p(T) \coloneqq \deg \mc{Z}^*(L).
    \end{equation}
Again the choice of embedding $L \ra \mbb{B}^*$
does not matter, and the right-hand sides
moreover depend only $\opn{GK}(L)$ and $\opn{disc}(L_{\Q_p})$.

\section{Heights and Rapoport--Zink uniformization}
    \label{sec:RZ_uniformization}
Fix a prime $p$ and write $\F \coloneqq \overline{\F}_p$.
In this section, we study ordinary and supersingular
Rapoport--Zink uniformization.
We treat the ordinary and supersingular cases in parallel.
We use the notation $* \in \{ \mrm{s}, \mrm{o} \}$
where ``$\mrm{s}$'' and ``$\mrm{o}$'' mean ``supersingular''
and ``ordinary'', respectively.

\subsection{Base points}
    \label{ssec:base_points}

The Rapoport--Zink uniformization morphisms depend on a choice of
base point, which plays a larger role in the ordinary case.

In the notation of Section \ref{ssec:complex_uniformization},
we take $\mtt{S}$ to be all primes but $p$,
and let $K \subseteq H'(\A_f)$ be any open compact subgroup
contained in $K^{\circ}$ and good away from $\mtt{S}$, i.e. ``good at $p$''.
We are mostly interested in $K = K^{\circ}$, but this does not play
much of a role in Section \ref{sec:RZ_uniformization}.
We will abuse notation and suppress $K$ from notation
in Section \ref{sec:RZ_uniformization},
e.g. writing $\mc{M}$ instead of $\mc{M}_K$.
Recall that if $S$ is a connected locally Noetherian scheme
over $\Spec \Z_{(p)}$,
we can express objects in $\mc{M}(S)$
as tuples $(E, E', \overline{\eta}, \overline{\eta}')$
where the pair $(\overline{\eta}, \overline{\eta}')$
is a $K^{p}$-orbit of trivializations of away-from-$p$
Tate modules which preserve alternating pairings up to a common scalar
(more precisely, see loc. cit.).

By a supersingular (resp. ordinary) \emph{base point}, we mean an object
    \begin{equation}
    (\mbf{E}, \mbf{E}', \overline{\bsm{\eta}}, \overline{\bsm{\eta}}')
    \in \mc{M}(\overline{\F})
    \end{equation}
with $\mbf{E}$ and $\mbf{E}'$ both supersingular (resp. ordinary),
equipped with two additional pieces of data.
    \begin{enumerate}[(1)]
        \item A pair of isomorphisms
            \begin{equation}
            \bsm{\rho} \colon \mbb{X}^* \ra \mbf{E}[p^{\infty}]
            \quad\quad
            \bsm{\rho}' \colon \mbb{X}^* \ra \mbf{E}'[p^{\infty}]
            \end{equation}
        for $\mbb{X}^*$ the framing $p$-divisible group
        used to define Rapoport--Zink spaces in Section \ref{sec:local_special_cycles}
        (with $* = \mrm{s}$ and $* = \mrm{o}$ in the supersingular
        and ordinary cases, respectively).
        \item A fixed representative $(\bsm{\eta}, \bsm{\eta}')$
        for $(\overline{\bsm{\eta}}, \overline{\bsm{\eta}}')$,
        (not just a $K^{p}$-orbit).
    \end{enumerate}

Given two base points
    \begin{equation}
    (\mbf{E}_1, \mbf{E}'_1)
    \quad\quad
    (\mbf{E}_2, \mbf{E}'_2)
    \end{equation}
(additional data suppressed), a quasi-isogeny from the first to the second
will mean a pair of quasi-isogenies $\mbf{E}_1 \ra \mbf{E}_2$
and $\mbf{E}'_1 \ra \mbf{E}'_2$ of the same degree. In this case,
we say that the two base points are \emph{isogenous}
(and this is an equivalence relation).
Note that the notion of isogenous base points
does not depend on the
additional data $(\bsm{\rho}, \bsm{\rho}', \bsm{\eta}, \bsm{\eta}')$.

In the supersingular case, there is a unique base point up to isogeny.
As in the proof of \cite[{Theorem 6.30}]{RZ96}, this can be viewed
as largely a consequence of the group $H = \GL_2 \times_{\G_m} \GL_2$ 
satisfying the Hasse principle.
For our purposes, we can also explain this phenomenon as follows.
If $\mbf{E}$ is any supersingular elliptic curve over $\F$,
then $\End^{\circ}(\mbf{E})$ is the quaternion algebra over $\Q$
which is nonsplit only at $p$ and $\infty$,
and so $\End^{\circ}(\mbf{E})^{\times}$ (viewed as an algebraic group over $\Q$)
is a form of $\GL_2$ and hence satisfies the Hasse principle.
From this, it follows that every positive rational number
is a norm of an element in this quaternion algebra,
i.e. there is a quasi-endomorphism of $\mbf{E}$ of degree $n$
for any positive integer $n$.
An explanation of this Hasse principle for quaternion
algebras is given in \cite[{Theorem 14.7.4}]{Voight21}
and attributed to Hasse--Schilling.

In the ordinary case, base points exist but are no longer unique
up to isogeny. This failure of uniqueness persists even
if we consider only elliptic curves with CM by the same maximal
order $\mc{O}_F$ in some imaginary quadratic field $F$.
For example, let $F = \Q(\sqrt{-6})$
and choose $p$ to be split in $\mc{O}_F$.
Let $\mf{a} \subseteq \mc{O}_F$ be the unique prime ideal
lying over the prime $3$.
Consider the complex elliptic curves $E = \C / \mc{O}_F$ and $E' = \C / \mf{a}$
(pick an embedding $F \ra \C$).
Descend $E$ and $E'$ to any number field,
assume $E$ and $E'$ have good (ordinary) reduction
at all primes over $p$, and
form $\mbf{E}$ and $\mbf{E}'$ by base-change to $\F$
(pick any $p$-adic place of the number field and pick
an embedding of the residue field into $\F$).
We claim that the base points $(\mbf{E}, \mbf{E})$ and $(\mbf{E}, \mbf{E}')$
are not isogenous.
Indeed, first note that $\Hom^{\circ}(\mbf{E}, \mbf{E})$
and $\Hom^{\circ}(\mbf{E}, \mbf{E}')$ are both free $F$-modules
of rank one, but generated by isogenies of degree $1$ and $3$ respectively.
Then note that the number $3$ is not the norm of
any element of $F$ (since $6$ is a norm but $2$
is not a norm $3$-adically, e.g. by local class field theory).

\subsection{Uniformization}
    \label{ssec:RZ_uniformization:uniformization}
We use our usual notation $H = \GL_2 \times_{\G_m} \GL_2$
from Section \ref{ssec:complex_uniformization}.
Consider the special cycle $\mathcal{Z}(T) \ra \mc{M}$
from Section \ref{ssec:global_special_cyles},
and assume $\opn{rk}(T) = 2$ or $\opn{rk}(T) = 3$.
This means $r = 2$ or $r = 3$
(we also allow the case of singular $T$ with $\opn{rk}(T) = 2$ and $r = 3$).
In these cases, recall (Lemma \ref{correct-codim})
that $\mathcal{Z}(T) \ra \Spec \Z$ is proper quasi-finite.
If $\opn{rk}(T) = 2$, we set $F \coloneqq \Q(\sqrt{-\det'T})$,
with notation as in loc. cit..

Consider the base change $\mathcal{Z}(T)_{\Spec \breve{\Z}_p}$.
For $* \in \{ \mrm{s}, \mrm{o} \}$,
we write $\breve{\mc{Z}}^*(T)_p$ for the formal completion
of $\mathcal{Z}(T)_{\Spec \breve{\Z}_p}$ along its supersingular or ordinary
locus,\footnote{
By this, we mean the set of points on
(the underlying topological space of) $\mc{Z}(T)_{\Spec \breve{\Z}_p}$
in positive characteristic which correspond to $(E, E', \bsm{x})$
with $E$ and $E'$ both supersingular or both ordinary,
depending on whether $* = \mrm{s}$ or $* = \mrm{o}$.
} depending on whether $* = \mrm{s}$ or $* = \mrm{o}$.
Note that Lemma \ref{correct-codim} implies that the supersingular and ordinary
loci are always closed on $\mathcal{Z}(T)_{\Spec \breve{\Z}_p}$,
and either empty or the entire special fiber.

If $\opn{rk}(T) = 2$, Lemma \ref{correct-codim} implies that we simply have
    \begin{equation}
    \breve{\mc{Z}}^*(T)_p
    =
    \begin{cases}
    \mathcal{Z}(T)_{\Spf \breve{\Z}_p}
        & \text{if $p$ is nonsplit in $F$ and $* = \mrm{s}$}
    \\
    \mathcal{Z}(T)_{\Spf \breve{\Z}_p}
        & \text{if $p$ is split in $F$ and $* = \mrm{o}$}
    \\
    \emptyset
        & \text{else.}
    \end{cases}
    \end{equation}
If $\opn{rk}(T) = 3$, Lemma \ref{correct-codim} implies that we have
$\breve{\mc{Z}}^*(T)_p = \emptyset$ unless $* = \mrm{s}$
and $p$ is the unique prime specified
in Lemma \ref{correct-codim}(d).

Fix a base point $(\mbf{E}, \mbf{E}', \bsm{\rho}, \bsm{\rho}', \bsm{\eta}, \bsm{\eta}')$
as in Section \ref{ssec:base_points}.
If $* = \mrm{o}$ and $\opn{rk}(T) = 2$,
we also assume that $\mbf{E}$ and $\mbf{E}'$
both have CM by $F$ (so that $p$ splits in $F$).
One can construct uniformizations
for other base points, but this is the case we will need.

Let $\mbf{V} \coloneqq \Hom^{\circ}(\mbf{E}, \mbf{E}')$.
This is a $\Q$-vector space and is equipped with a positive definite
bilinear pairing $(-,-)$ given by
    \begin{equation}
    (x, y) \coloneqq \opn{tr}^{\circ}(x^{\vee} \circ y)
    \quad\quad
    \opn{tr}^{\circ}(a) \coloneqq a^{\vee} + a
    \end{equation}
for $x,y \in \Hom^{\circ}(\mbf{E}, \mbf{E}')$
and $a \in \End^{\circ}(\mbf{E})$.

Consider the group scheme $\mbf{H}' \ra \Spec \Z$ given by
    \begin{equation}
    \mbf{H}'(R) \coloneqq \{ (\gamma, \gamma') \in (\End(\mbf{E}) \otimes_{\Z} R)^{\times} \times (\End(\mbf{E}') \otimes_{\Z} R)^{\times} : \gamma^{\vee} \circ \gamma = \gamma^{\prime \vee} \circ \gamma' \in R^{\times} \}
    \end{equation}
for $\Z$-algebras $R$ (where we make our usual abuse of notation
to use the principal polarization to identify $\mbf{E}$ with its dual,
and similarly for $\mbf{E}'$).

The datum $(\bsm{\rho}, \bsm{\rho}')$ induces maps
    \begin{equation}
        \mbf{H}'(\Q_p) \ra H'(\Q_p)
            \quad\quad
        \mbf{H}'(\Z_p) \ra H'(\Z_p)
    \end{equation}
for $H'(\Q_p)$ and $H'(\Z_p)$
from Section \ref{ssec:local_special_cycles:RZ}.
There are similar induced identifications
of quadratic spaces $\mbf{V}(\Q_p) \xra{\sim} \mbb{B}^*$,
where $\mbb{B}^*$ is as in Section \ref{sec:local_special_cycles}.

The datum $(\bsm{\eta}, \bsm{\eta}')$ induces a morphism
of group schemes
    \begin{equation}
    \mbf{H}'_{\A_f^p} \ra H_{\A_f^p}.
    \end{equation}
In the supersingular case, the $\ell \neq p$ case of
\cite[{Corollary 6.29}]{RZ96}\footnote{
The statement of loc. cit. should be modified
to additionally require that the isocrystals $N$ and $N'$ (associated
to $A$ and $A'$ respectively) have $N_{\bsm{P}}$ and $N_{\bsm{P}}'$
with the same slope for all $\bsm{P} \mid p$,
in the notation of loc. cit.. This condition
is automatic for us in the supersingular case, but not
automatic in the ordinary case.
} shows that this map is an isomorphism.
The map need not be an isomorphism in the ordinary case.
There is a similar induced map $\mbf{V}(\A_f^p) \ra V(\A_f^p)$
which preserves quadratic forms. Again, this is an isomorphism
in the supersingular case, but need not be an isomorphism
in the ordinary case.

Using the techniques from \cite[{\S 6}]{RZ96},
one can uniformize $\breve{\mc{Z}}^*(T)_p$ by local special cycles
on Rapoport--Zink spaces, as follows.
Our notation and presentation of ideas are similar to that of
\cite[{\S 11}]{Chen2024}.

We consider a stack $\breve{\mc{Z}}^*(T)_{p, \mrm{framed}}$
given by the setoids
    \begin{equation}
    \breve{\mc{Z}}^*(T)_{p,\mrm{framed}}(S)
    =
    \left \{
    \begin{array}{c}
    (E, E', \overline{\eta}, \overline{\eta}', \bsm{x})
        \in \breve{\mc{Z}}^*(T)_p(S)
    \\
    \phi \colon E_{\overline{S}} \ra \mbf{E}_{\overline{S}}
    \text{ and }
    \phi' \colon E'_{\overline{S}} \ra \mbf{E}'_{\overline{S}}
    \text{ quasi-isogenies of the same degree}
    \end{array}
    \right \}
    \end{equation}
for schemes $S$ over $\Spf \breve{\Z}_p$,
with $\overline{S} \coloneqq S \times_{\Spf \breve{\Z}_p} \F$
as usual, and where the degrees of $\phi$ and $\phi'$ 
are locally constant functions on $S$.
We call $(\phi, \phi')$ a ``framing''.
After viewing $\breve{\mc{Z}}^*(T)_{p,\mrm{framed}}$
as a sheaf of sets (passing to isomorphism classes),
there is a left action $\mbf{H}'(\Q) \acts \breve{\mc{Z}}^*(T)_{p, \mrm{framed}}$
given by acting on the framing,
i.e. $(h, h') \cdot (\phi, \phi') = (h \phi, h' \phi')$.

Recall the spaces $V$, $W$, and $W'$ from Section \ref{ssec:complex_uniformization}.
Set
    \begin{equation}\label{equation:RZ_uniformization:varphi^p}
        \varphi^p \coloneqq \pmb{1}_{\Hom(L,L') \otimes_{\Z} \hat{\Z}^p}^{\otimes r} \in C^{\infty}_c(V(\A_f^p)^r).
    \end{equation}
Given $\bsm{x}^p \in V(\A_f^p)^r$, we consider
an ``away-from-$p$ local special cycle'' defined as the discrete set
    \begin{equation}
    \mc{Z}(\bsm{x}^p) = \{ h \in H(\A_f^p) / K^p : \varphi^p(h^{-1} \bsm{x}^p) = 1 \}.
    \end{equation}
There is a canonical isomorphism (equivalence)
    \begin{equation}\label{equation:RZ_uniformization:framed}
    \breve{\mc{Z}}^*(T)_{p, \mrm{framed}}
        \xra{\sim}
    \coprod_{\substack{\bsm{x} \in \mbf{V}(\Q)^r \\ \frac{1}{2}(\bsm{x}, \bsm{x}) = T}}
    \tilde{\mc{Z}}^*(\bsm{x}_p) \times \mc{Z}(\bsm{x}^p)
    \end{equation}
where $\bsm{x}_p$ and $\bsm{x}^p$
denote the images of $\bsm{x}$ in $\mbf{V}(\Q_p)^r \cong (\mbb{B}^*)^r$
and $V(\A_f^p)^r$ respectively, using the canonical map $\mbf{V}(\A_f^p) \ra V(\A_f^p)$, and where
$\tilde{\mc{Z}}^*(\bsm{x}_p) \subseteq \tilde{\mc{N}}^*$
is the local special cycle associated to the tuple $\bsm{x}_p$,
as in Section \ref{ssec:local_special_cycles:cycles}.
The equivalence in \eqref{equation:RZ_uniformization:framed}
is constructed by sending an object
of $\breve{\mc{Z}}^*(T)_p$ to the following data.
    \begin{enumerate}[(1)]
        \item We define $\bsm{x} \in V(\Q)^r$ by the commutative
        diagram
            \begin{equation}
            \begin{tikzcd}
            E_{\overline{S}} \arrow{r} \arrow{d}{\bsm{x}}
                & \mbf{E}_{\overline{S}} \arrow{d}{\bsm{x}}
            \\
            E'_{\overline{S}} \arrow{r}
                & \mbf{E}'_{\overline{S}} \rlap{~.}
            \end{tikzcd}
            \end{equation}
        \item (At $p$) We define a point of $\tilde{\mc{Z}}^*(\bsm{x}_p)$
        by passing to underlying $p$-divisible groups, with framings induced by $(\phi, \phi')$ and $(\bsm{\rho}, \bsm{\rho}')$.
        \item (Away from $p$) We define a point $h = (h_1, h_2) \in \mc{Z}(\bsm{x}^p)$
        by the commutative diagrams
            \begin{equation}
            \begin{tikzcd}
            W_{\A_f^p} \arrow{r}{\pmb{\eta}} & T_{\A_f^p}(\mbf{E})^{\circ}
            \\
            W_{\A_f^p} \arrow{r}{\overline{\eta}} \arrow{u}{h_1} & T_{\A_f^p}(E)^{\circ} \arrow{u}{\phi_*}
            \end{tikzcd}
            \quad\quad
            \begin{tikzcd}
            W'_{\A_f^p} \arrow{r}{\pmb{\eta}'} & T_{\A_f^p}(\mbf{E}')^{\circ}
            \\
            W'_{\A_f^p} \arrow{r}{\overline{\eta}'} \arrow{u}{h_2} & T_{\A_f^p}(E')^{\circ} \arrow{u}{\phi'_*} \rlap{~.}
            \end{tikzcd}
            \end{equation}
    \end{enumerate}
The resulting map \eqref{equation:RZ_uniformization:framed}
is $\mbf{H}'(\Q)$-equivariant, where the action of $\mbf{H}'(\Q)$
on the right-hand side is given by $\bsm{x} \mapsto h \bsm{x}$
and the maps
    \begin{equation}
    \tilde{\mc{Z}}^*(\bsm{x}_p)
        \ra 
        \tilde{\mc{Z}}^*(h \cdot \bsm{x}_p)
    \quad\quad
    \mc{Z}(\bsm{x}^p)
        \ra
        \mc{Z}(h \cdot \bsm{x}^p)
    \end{equation}
induced by the Rapoport--Zink action
$\mbf{H}'(\Q) \ra H'(\Q_p) \acts \tilde{\mc{N}}^*$ (Section \ref{ssec:local_special_cycles:RZ}) and
the left-translation action $\mbf{H}'(\Q) \ra H'(\A_f^p) \acts H(\A_f^p)$.

Thus, after choosing a (quasi-)inverse
to \eqref{equation:RZ_uniformization:framed}
(unique up to unique $2$-isomorphism), we compose
with the forgetful map
$\breve{\mc{Z}}^*(T)_{p, \mrm{framed}} \ra \breve{\mc{Z}}^*(T)_p$
and obtain a \emph{uniformization map}
    \begin{equation}\label{equation:RZ_uniformization:cover}
    \bigg (
    \coprod_{\substack{\bsm{x} \in \mbf{V}(\Q)^r \\ \frac{1}{2} (\bsm{x}, \bsm{x}) = T}} \tilde{\mc{Z}}^*(\bsm{x}_p) \times \mc{Z}(\bsm{x}^p)
    \bigg )
        \ra
    \breve{\mc{Z}}^*(T)_p
    \end{equation}
which is canonical given the fixed choice of base point (up
to the choice of inverse for \eqref{equation:RZ_uniformization:framed}).
The uniformization map induces a map from the stack quotient
    \begin{equation}\label{equation:RZ_uniformization:quotient}
    \tilde{\Theta}
    \colon
    \bigg [
    \mbf{H}'(\Q) \backslash \bigg (
    \coprod_{\substack{\bsm{x} \in \mbf{V}(\Q)^r \\ \frac{1}{2} (\bsm{x}, \bsm{x}) = T}} \tilde{\mc{Z}}^*(\bsm{x}_p) \times \mc{Z}(\bsm{x}^p)
    \bigg )
    \bigg ]
        \ra
    \breve{\mc{Z}}^*(T)_{p}.
    \end{equation}
This is a formally locally of finite type and formally \'etale
monomorphism of locally Noetherian formal algebraic stacks
(compare \cite[{Lemma 11.6.1}]{Chen2024}).

\begin{lemma}\label{lemma:RZ_uniformization:supersingular}
In the supersingular case, the map $\tilde{\Theta}$ is an isomorphism.
\end{lemma}
\begin{proof}
This can be verified using the techniques from \cite[{\S 6}]{RZ96},
i.e. one can show that \eqref{equation:RZ_uniformization:cover}
exhibits $\breve{\mc{Z}}^*(T)_p$ as a stack quotient
for the action of $H'(\Q)$ on the left-hand side of \eqref{equation:RZ_uniformization:cover}.
For more details on a similar uniformization for unitary groups,
see \cite[{\S 11.6}]{Chen2024}.
A key point is to check surjectivity on $\F$-points,
which follows from the uniqueness of base points up to quasi-isogeny
(which came from a Hasse principle)
as explained in Section \ref{ssec:base_points}.
\end{proof}

\begin{lemma}\label{lemma:RZ_uniformization:ordinary}
In the ordinary case, the map $\tilde{\Theta}$ is an open and closed immersion.
\end{lemma}
\begin{proof}
First note that $\tilde{\Theta}$ is an open and closed immersion on underlying
topological spaces, since we already know $\tilde{\Theta}$ is a monomorphism
(hence injective on underlying topological spaces), and since the source
and target of $\tilde{\Theta}$ have underlying topological spaces
which are both finite and discrete
(where each point of the underlying topological spaces is an image
of an $\F$-point).
For $\breve{\mc{Z}}^*(T)_p$, here we are using
the proper quasi-finiteness proved in Lemma \ref{correct-codim},
and for the source of $\tilde{\Theta}$, we are using
the fact that each $\tilde{\mc{Z}}^*(\bsm{x}_p)$
has discrete underlying topological space, where every point
is the image of an $\F$-point
(follows because $\tilde{\mc{Z}}^*(\bsm{x}_p) \subseteq \mc{N}^*$
is a closed subfunctor, and $\mc{N}^*$ is discrete by the discussion
after \eqref{equation:local_special_cycles:cycles:RZ}).

Via the same proof strategy as in the supersingular case (again
compare \cite[{\S 11.6}]{Chen2024} for more details),
we find that $\tilde{\Theta}$ is formally \'etale, surjective onto its
image, a monomorphism, and proper on reduced substacks,
which is enough to imply that $\tilde{\Theta}$ is an isomorphism onto
its (open and closed) image.
\end{proof}

The uniformization map $\tilde{\Theta}$ depends on a choice of
base point. Given a base point $b$, we sometimes write $\tilde{\Theta}_b$
instead of $\tilde{\Theta}$ to emphasize the dependence on $b$.

\begin{lemma}\label{lemma:RZ_uniformization:isog_base_points_unif_image}
Given two base points $b$ and $b'$, the images 
of $\tilde{\Theta}_b$ and $\tilde{\Theta}_{b'}$
coincide (resp. are disjoint) if $b$ and $b'$ are isogenous
(resp. not isogenous).
\end{lemma}
\begin{proof}
This follows from the definition of isogenous base points
and the definition of $\breve{\mc{Z}}^*(T)_{p,\mrm{framed}}$,
i.e. an object of $\breve{\mc{Z}}^*(T)_p(S)$ admits framings
to two different base points if and only if those two base points
are isogenous.
\end{proof}

\begin{lemma}\label{lemma:RZ_uniformization:multiple_base_points_cover}
We can cover $\breve{\mc{Z}}^*(T)_p$
by the images of finitely many uniformization maps
$\tilde{\Theta}_{b_1}, \ldots, \tilde{\Theta}_{b_n}$
(allowing $n = 0$) which can be chosen to have pairwise
disjoint images. In the supersingular case, at most one uniformization
map is needed.
\end{lemma}
\begin{proof}
The supersingular case was already proved in
Lemma \ref{lemma:RZ_uniformization:supersingular}.
For the ordinary case, we can associate to any object
$\beta = (E, E', \overline{\bsm{\eta}}, \overline{\bsm{\eta}}', \bsm{x}) \in \breve{\mc{Z}}^*(T)_p(\F)$
an ordinary base point $b$ by forgetting $\bsm{x}$
and choosing additional data
$(\bsm{\rho}, \bsm{\rho}', \bsm{\eta}, \bsm{\eta}')$
as in the definition of a base point (always possible).

Then $\beta$ is contained in the image of the open and closed
uniformization map $\tilde{\Theta}_b$, so we can cover
$\breve{\mc{Z}}^*(T)_p$ by the images of finitely many uniformization
maps by quasi-compactness of $\breve{\mc{Z}}^*(T)_p$.
We can moreover assume that these uniformization maps
have pairwise disjoint images,
by Lemma \ref{lemma:RZ_uniformization:isog_base_points_unif_image}.
\end{proof}

\subsection{Heights}
    \label{ssec:RZ_uniformization:heights}

In Section \ref{ssec:RZ_uniformization:heights},
we assume $\opn{rk}(T) = 2$.
We explain how to use Rapoport--Zink uniformization to compute
the quantity $\widehat{\deg}_{\widehat{\mc{E}}^{\vee}}(\mathcal{Z}(T))_p$
from Section \ref{ssec:height_decomp}.

We consider uniformization along the supersingular (resp. ordinary) locus
of $\mathcal{Z}(T)_{\Spec \breve{\Z}_p}$ if $p$
is nonsplit (resp. split) in $F = \Q(\sqrt{-\det'T})$.
Then Lemma \ref{correct-codim} implies
that we have $\breve{\mc{Z}}^*(T)_{p} = \mathcal{Z}(T)_{\Spf \breve{\Z}_p}$,
for $* \in \{ \mrm{s}, \mrm{o} \}$ chosen as in 
the previous sentence.

Form the uniformization map $\tilde{\Theta}_b$
as in \eqref{equation:RZ_uniformization:quotient},
associated to a fixed base point $b$.
Fix an embedding $\breve{\Q}_p \ra \C_p$.
Recall from \eqref{equation:height_decomp:p-part_groupoid_sum}
that $\widehat{\deg}_{\widehat{\mc{E}}^{\vee}}(\mc{Z}(T))_p$
may be expressed as a weighted sum over the groupoid
$\mathcal{Z}(T)_{\Spec \breve{\Z}_p}(\C_p)$.
We parameterize this groupoid by a suitable ``generic fiber''
of the uniformization \eqref{equation:RZ_uniformization:quotient},
expressed by schemification as follows.

Recall the discussion on schemification at the beginning
of Section \ref{sec:local_special_cycles}.
More generally, we obtain an equivalence e.g. between the categories
    \begin{enumerate}
        \item Stacks in groupoids $\mc{X}$ over $(\mathit{Sch} / \Spec \breve{\Z}_p)_{\mathit{fppf}}$
        admitting a collection of morphisms
        $\{ U_i \ra \mc{X} \}_i$ where each $U_i \ra \mc{X}$
        is representable by schemes and finite locally free,
        where the collection is jointly covering,
        and where each $U_i \ra \Spec \breve{\Z}_p$ is representable by schemes and finite.
        Consider $1$-morphisms only up to $2$-isomorphism.
        \item Stacks in groupoids $\mc{X}$ over $(\mathit{Sch} / \Spf \breve{\Z}_p)_{\mathit{fppf}}$
        admitting a collection of morphisms
        $\{ U_i \ra \mc{X} \}_i$ where each $U_i \ra \mc{X}$
        is representable by schemes and finite locally free,
        where the collection is jointly covering,
        and where each $U_i \ra \Spf \breve{\Z}_p$ is representable by schemes and finite.
        Consider $1$-morphisms only up to $2$-isomorphism.
    \end{enumerate}
The equivalence is given by $\mc{X} \mapsto \mc{X}_{\Spf \breve{\Z}_p}$,
and admits a quasi-inverse by ``schemification of presentations''.
This stacky schemification is convenient bookkeeping for us, but it can be avoided if one prefers
(Remark \ref{remark:RZ_uniformization:heights:no_stacky_schemification}).
A similar ``schemification'' appears in \cite{CLZ}.

The stacks in the uniformization \eqref{equation:RZ_uniformization:quotient}
satisfy the property in (2) above. For example, $\mathcal{Z}(T)_{\Spec \breve{\Z}_p}$ satisfies property (2) because it admits a finite \'etale cover
by a scheme (add level structure away from $p$) and the structure map
$\mathcal{Z}(T)_{\Spec \breve{\Z}_p} \ra \Spec \breve{\Z}_p$
is proper and quasi-finite (Lemma \ref{correct-codim}).
Schemifying, we obtain a morphism (unique up to $2$-isomorphism)
    \begin{equation}\label{equation:RZ_uniformization:heights:quotient_schemified}
    \bigg [
    \mbf{H}'(\Q) \backslash \bigg (
    \coprod_{\substack{\bsm{x} \in \mbf{V}(\Q)^r \\ \frac{1}{2} (\bsm{x}, \bsm{x}) = T}} \tilde{\mc{Z}}^*(\bsm{x}_p)^{\mrm{sch}} \times \mc{Z}(\bsm{x}^p)
    \bigg )
    \bigg ]
        \ra
    \mathcal{Z}(T)_{\Spec \breve{\Z}_p}
    \end{equation}
The morphism in \eqref{equation:RZ_uniformization:heights:quotient_schemified}
is an isomorphism in the supersingular case,
and an open and closed immersion in the
ordinary case (follows from the formal algebraic stack version).

We next recall the structure of the local special cycle 
$\tilde{\mc{Z}}^*(\bsm{x}_p) \subseteq \tilde{\mc{N}}^*$
from Sections \ref{ssec:local_special_cycles:RZ} and \ref{ssec:local_special_cycles:cycles}.
Recall that there is an open and closed disjoint union decomposition
$\tilde{\mc{Z}}^*(\bsm{x}_p) = \coprod_i \tilde{\mc{Z}}^*(\bsm{x}_p)_i$
indexed over $i \in H'(\Q_p) / H'_0(\Q_p)$, in the notation \eqref{equation:local_special_cycles:cycles:decomposition}.
The terms $\tilde{\mc{Z}}^*(\bsm{x}_p)_i$ are connected in the supersingular case,
but not necessarily connected in the ordinary case.
This decomposition is compatible with the action $H'(\Q_p) \acts \tilde{\mc{N}}^*$
and the action $H'(\Q_p) \acts H'(\Q_p) / H'_0(\Q_p)$ by left-translation.

Using Corollary \ref{corollary:local_special_cycles:primitive_cycles:deg_GK_disc}
and Lemma \ref{lemma:local_special_cycles:cycles_rank_2_lattices:lattice_orbit},
we find that the structure maps
$\tilde{\mc{Z}}^*(\bsm{x}_p)_i \ra \Spf \breve{\Z}_p$
have the same degree as $i$ varies,
since the action $H'(\Q_p) \acts \tilde{\mc{Z}}^*$
is transitive on the terms in the decomposition
\eqref{equation:local_special_cycles:RZ:decomposition}
indexed over $i \in H'(\Q_p) / H'_0(\Q_p)$.
For any $i$, we have
    \begin{equation}
    \deg(\tilde{\mc{Z}}^*(\bsm{x}_p)_i \ra \Spf \breve{\Z}_p)
    =
    \deg_p(T)
    =
    \deg(\mc{Z}^*(\bsm{x}_p) \ra \Spf \breve{\Z}_p)
    =
    \sum_{x \in \mc{Z}^*(\bsm{x}_p)^{\mrm{sch}}(\C_p)} 1
    \end{equation}
with $\deg_p(T)$ defined in \eqref{equation:delta_T_deg_T}.

From the schemified
uniformization \eqref{equation:RZ_uniformization:heights:quotient_schemified},
if $\deg_{\Q} (\mathcal{Z}(T)_{\Q})$ denotes the stacky degree of $\mathcal{Z}(T)_{\Q} \ra \Spec \Q$
(denoted $\deg_{\mc{E}^{\vee}_{\Q}} ( \mathcal{Z}(T)_{\Q} )$ in Section \ref{ssec:complex_fiber_degree}),
it now follows formally that we have
    \begin{align}
    \deg_{\Q} \mathcal{Z}(T)_{\Q}
    & =
    \sum_{\alpha \in \mc{Z}(T)^{\mc{M}}_{\Spec \breve{\Z}_p}(\C_p)} \frac{1}{\# \opn{Aut}(\alpha)}
    \\
    \label{equation:RZ_uniformization:heights:degree:local-to-global}
    & =
    \deg_p(T)
    \cdot
    \# \bigg [
    \mbf{H}'(\Q) \backslash \bigg (
    \coprod_{\substack{\bsm{x} \in \mbf{V}(\Q)^r \\ \frac{1}{2} (\bsm{x}, \bsm{x}) = T}} H'(\Q_p) / H'_0(\Q_p) \times \mc{Z}(\bsm{x}^p)
    \bigg )
    \bigg ]
    \end{align}
in the supersingular case, where the first equality is largely by definition
(Section \ref{ssec:complex_fiber_degree}),
where the sum runs over isomorphism classes of the groupoid
$\mc{Z}(T)^{\mc{M}}_{\Spec \breve{\Z}_p}(\C_p)$,
and where $\# [ \cdots ]$
means groupoid cardinality.
In the ordinary case, the second equality \eqref{equation:RZ_uniformization:heights:degree:local-to-global}
still holds if we replace the sum over
$\mc{Z}(T)^{\mc{M}}_{\Spec \breve{\Z}_p}(\C_p)$
with the sum over only the open and closed substack uniformized
by the image of $\tilde{\Theta}_b$.

We next make a similar-looking decomposition for heights.
As explained before
Definition \ref{definition:height_change:special_cycles},
every object $z \in \tilde{\mc{Z}}^*(\bsm{x}_p)^{\mrm{sch}}_i(\C_p)$
has an associated ``local height defect'' $\delta_{\mrm{Fal}}(z) \in \Q$,
defined in terms of the endomorphism rings of the pair of universal $p$-divisible groups
over the descent of $z$ to $\Spf \mc{O}_{\breve{L}}$ for some finite degree
field extension $\breve{L}$ of $\breve{\Q}_p$.
Given $\tilde{\mc{Z}}^*(\bsm{x}_p)^{\mrm{sch}}_i$
mapping to the source of the uniformization map $\tilde{\Theta}_b$
from \eqref{equation:RZ_uniformization:quotient},
and if $\alpha \in \mc{Z}(T)_{\Spec \breve{\Z}_p}(\C_p)$
is the image of some $z \in \tilde{\mc{Z}}^*(\bsm{x}_p)^{\mrm{sch}}_i(\C_p)$,
then Lemmas \ref{lemma:Faltings_height_CM:arbitrary_order}
and \ref{lemma:Faltings_height_CM:Tate_modules} imply that we have
    \begin{equation}
    \delta_{\mrm{Fal}}(z) \cdot \log p
    =
    \delta_{\mrm{Fal},p}(\alpha)
    \end{equation}
where $\delta_{\mrm{Fal},p}(\alpha)$ was defined in Section \ref{ssec:height_decomp}.
This is because $\delta_{\mrm{Fal},p}(\alpha)$ is defined in terms
of the $p$-index of the endomorphism rings of the pair of universal elliptic curves,
while $\delta_{\mrm{Fal}}(z)$ is defined similarly but using
universal $p$-divisible groups, but Lemma \ref{lemma:Faltings_height_CM:Tate_modules}
(and the equivalence between \'etale $p$-divisible groups and Tate modules,
applied over $\Spec \breve{L}$, as well as fully-faithfulness of the generic fiber functor
of $p$-divisible groups as in \cite[{Theorem 4}]{tate1967})
implies that these endomorphism rings may be identified, since Rapoport--Zink uniformization
(by construction) is compatible with universal $p$-divisible groups and universal abelian varieties.

Upon considering the formula \eqref{equation:height_decomp:p-part_groupoid_sum}
for $\widehat{\deg}_{\widehat{\mc{E}}^{\vee}}(\mathcal{Z}(T))_p$ and
the formula \eqref{equation:height_change:special_cycles:groupoid_sum}
for $\delta_{\mrm{Fal}}(\opn{span}(\bsm{x}_p))$,
it then follows formally from the above discussion and the schemified
uniformization \eqref{equation:RZ_uniformization:heights:quotient_schemified}
that we have
    \begin{align}
    -\widehat{\deg}_{\widehat{\mc{E}}^{\vee}}(\mathcal{Z}(T))_p
    & =
    \sum_{\alpha \in \mathcal{Z}(T)_{\Spec \breve{\Z}_p}(\C_p)}
        \frac{\delta_{\mrm{Fal},p}(\alpha)}{\# \opn{Aut}(\alpha)}
    \\
    \label{equation:RZ_uniformization:heights:height:local-to-global}
    &=
    \delta_{\mrm{Fal},p}(T)
        \cdot
    \# \bigg [
    \mbf{H}'(\Q) \backslash \bigg (
    \coprod_{\substack{\bsm{x} \in \mbf{V}(\Q)^r \\ \frac{1}{2} (\bsm{x}, \bsm{x}) = T}} H'(\Q_p) / H'_0(\Q_p) \times \mc{Z}(\bsm{x}^p)
    \bigg )
    \bigg ] \cdot \log p
    \end{align}
in the supersingular case,
with $\delta_{\mrm{Fal},p}(T)$ defined in \eqref{equation:delta_T_deg_T}.
In the ordinary case, the second equality \eqref{equation:RZ_uniformization:heights:height:local-to-global}
continues to hold if we replace the sum over
$\mathcal{Z}(T)_{\Spec \breve{\Z}_p}(\C_p)$
with the sum over only the open and closed substack uniformized by
by the image of $\tilde{\Theta}_b$.

In all cases, supersingular or ordinary, we find that the ``normalized height difference''
is given by
    \begin{equation}\label{equation:RZ_uniformization:heights:normalized_height}
    \frac{\widehat{\deg}_{\widehat{\mc{E}}^{\vee}}(\mathcal{Z}(T))_p}{\deg_{\mathcal{E}^\vee_\mathbb{Q}} (\mathcal{Z}(T)_{\Q})}
    =
    -
    \frac{\delta_{\mrm{Fal},p}(T)}{\deg_p(T)} \cdot \log p
    \end{equation}
whenever $\mathcal{Z}(T) \neq \emptyset$,
which is equivalent to $\deg_{\Q} (\mathcal{Z}(T)_{\Q} \neq 0)$.
Indeed, the supersingular case of \eqref{equation:RZ_uniformization:heights:normalized_height}
follows immediately from dividing \eqref{equation:RZ_uniformization:heights:height:local-to-global}
by \eqref{equation:RZ_uniformization:heights:degree:local-to-global}.
The ordinary case follows as well, after we range over the finitely many uniformization
maps $\tilde{\Theta}_b$ for different choices of base point $b$,
as in Lemma \ref{lemma:RZ_uniformization:multiple_base_points_cover}.

\begin{remark}
    \label{remark:RZ_uniformization:heights:no_stacky_schemification}
    It is possible to prove \eqref{equation:RZ_uniformization:heights:degree:local-to-global},
    \eqref{equation:RZ_uniformization:heights:height:local-to-global},
    and \eqref{equation:RZ_uniformization:heights:normalized_height}
    while avoiding schemification for stacks, although we find that language to be convenient
    bookkeeping.
    In brief, one can add level structure away from $p$, take constant finite \'etale covers
    by formal schemes, and divide by the degrees of these covers; compare \cite{CLZ}.
\end{remark}

\part{Arithmetic degrees of special cycles}

\section{Proof of non-Archimedean local identities}
\label{sec:proof_non-Arch}
The following is the main non-Archimedean local theorem of this article,
which is an identity between the ``change of Faltings height along special cycles''
and near-central derivatives of normalized local densities.

In the theorem below, the left-hand sides were defined in \eqref{equation:delta_T_deg_T},
and the right-hand sides were defined in \eqref{nor-der-local-density}.

\begin{theorem}\label{corank-1-local}
    Let $L\subset\mathbb{B}^{\ast}$ be a $\zp$-lattice of rank $2$ with $\ast\in\{\textup{s},\textup{o}\}$. Then
    \begin{itemize}
        \item [(a)] For the special cycle $\mathcal{Z}^{\ast}(L)$, we have
        \begin{equation}\label{deg-den}
            \deg\mathcal{Z}^{\ast}(L)=\den^{\star}(L).
        \end{equation}
        \item [(b)] For the change of Faltings height along the special cycle $\mathcal{Z}^{\ast}(L)$, we have
        \begin{equation}\label{bdf-dden}
        \bdf(L)=\partial\den^{\star}(L).
    \end{equation}
    \end{itemize}
\end{theorem}

The goal of this section is to prove Theorem \ref{corank-1-local}.
Fix a rank $2$ lattice $L\subset\mathbb{B}^{\ast}$ where $\ast\in\{\textup{s},\textup{o}\}$
with $\textup{GK}(L)=(a,b)$. Recall that $\nu(L)=a+b$ and $d(L)=b-a$.
\subsection{Base cases}\label{base-cases}
We first consider the case that $\nu(L)\leq2$.
\subsubsection{$\nu(L)=0$}
In this case $a=b=0$, the special cycle $\mathcal{Z}^{\ast}(L)=\mathcal{Z}^{\ast}(L)^{\circ}\simeq\textup{Spf}\,\mc{O}_{\breve{F}_0}$ by Lemma \ref{primitive-part}, hence $\deg\mathcal{Z}^{\ast}(L)=1$ by (\ref{local-deg-horizontal}), and $\bdf(L)=0$ by (\ref{local-change-delta}).
By (\ref{aequivb-an-poly}) and (\ref{aequivb-iso-poly}),
\begin{equation*}
    \den^{\star}(X,L)=1,
\end{equation*}
hence $\den^{\star}(L)=1=\deg\mathcal{Z}^{\ast}(L)$ and $\partial\den^{\star}(L)=0=\bdf(L)$. Theorem \ref{corank-1-local} is true.

\subsubsection{$\nu(L)=1$}
In this case we must have $a=0$ and $b=1$. Such a lattice $L$ exists only if $\ast=\textup{s}$, and the only integral lattice containing $L$ is the lattice $L$ itself. Hence the special cycle $\mathcal{Z}^{\ast}(L)=\mathcal{Z}^{\ast}(L)^{\circ}\simeq\textup{Spf}\,\mc{O}_{\breve{F}_0}$ by Lemma \ref{primitive-part}, and
\begin{equation*}
    \bdf(L)=\deg\mathcal{Z}^{\textup{s}}(L)^{\circ}\cdot\left(\bdf\left(0\right)+\bdf\left(0\right)\right)=0.
\end{equation*}
By (\ref{anequivb-poly}),
\begin{equation*}
    \den^{\star}(X,L)=1,
\end{equation*}
hence $\den^{\star}(L)=1=\deg\mathcal{Z}^{\ast}(L)$ and $\partial\den^{\star}(L)=0=\bdf(L)$. Theorem \ref{corank-1-local} is true.

\subsubsection{$\nu(L)=2$}
In this case $(a,b)=(0,2)$ or $(1,1)$. We first consider the case that $(a,b)=(0,2)$.
Let $M\subset\mathbb{B}^{\ast}$ be an integral rank $2$ lattice
such that $L \subsetneq M$.
The Gross--Keating invariant of $M$ must be $(0,0)$ (this is obvious for the case $p>2$, the case $p=2$ can be proved using Lemma \ref{gk-inv-rk2} (c) and (d)) and there is only one such $M$, hence the contribution of $M$ to $\deg\mathcal{Z}^{\ast}(L)$ is 1, while $0$ to $\bdf(L)$. By Lemma \ref{primitive-part} and (\ref{local-deg-horizontal}),
\begin{align*}
    &\deg\mathcal{Z}^{\ast}(L)=\deg\mathcal{Z}^{\ast}(L)^{\circ}+\deg\mathcal{Z}^{\ast}(M)^{\circ}=\begin{cases}
        p+1+1=p+2, &\textup{if $\ast =$ s;}\\
        p-1+1=p, &\textup{if $\ast =$ o.}
    \end{cases}\\
    &\bdf(L)=\deg\mathcal{Z}^{\ast}(L)^{\circ}\cdot\left(\bdf\left(1\right)+\bdf\left(1\right)\right)=\begin{cases}
        (p+1)\cdot\left(1-\frac{2}{p+1}\right)=p-1, &\textup{if $\ast=\textup{s}$}\\
        (p-1)\cdot\left(1-0\right), &\textup{if $\ast=\textup{o}$}
    \end{cases}\Bigg\}=p-1.
\end{align*}
On the other hand, by (\ref{aequivb-an-poly}) and (\ref{aequivb-iso-poly}),
\begin{equation*}
    \den^{\star}(X,L)=\begin{cases}
        X^{-1}+1+pX, &\textup{if $\ast=\textup{s}$;}\\
        X^{-1}-1+pX, &\textup{if $\ast=\textup{o}$.}
    \end{cases}
\end{equation*}
Therefore $\den^{\star}(L)=\deg\mathcal{Z}^{\ast}(L)$ and $\partial\den^{\star}(L)=p-1=\bdf(L)$. Theorem \ref{corank-1-local} is true.
\par
Next we consider the case that $(a,b)=(1,1)$. Let $M\subset\mathbb{B}^{\ast}$ be a rank $2$ lattice such that $L\subsetneq M\subset M^{\vee}\subsetneq L^{\vee}$. Such a lattice $M$ corresponds to an isotropic line in the 2-dimensional nondegenerate quadratic space $L^{\vee}/L\simeq\mathbb{F}_p^{2}$, where the quadratic form is given by $\overline{x}\in L^{\vee}/L\rightarrow\overline{p\cdot q(x)}$. 
\begin{itemize}
    \item If $\ast=\textup{s}$, such $M$ doesn't exist because $L^{\vee}/L$ is anisotropic (this fact is obvious for $p>2$, the case $p=2$ can be proved using Lemma \ref{gk-inv-rk2} (b)).
    \item If $\ast=\textup{o}$, there are in total $2$ such lattices $M$ because $L^{\vee}/L$ is nondegenerate and isotropic. However, the Gross--Keating invariant of such $M$ must be $(0,0)$, hence the contribution of $M$ to $\deg\mathcal{Z}^{\ast}(L)$ is 1, while $0$ to $\bdf(L)$.
\end{itemize}
By Lemma \ref{primitive-part} and (\ref{local-deg-horizontal}),
\begin{align*}
   \deg\mathcal{Z}^{\ast}(L)&=\begin{cases}
        2(p+1), &\textup{if $\ast=\textup{s}$;}\\
        2(p-1)+2=2p, &\textup{if $\ast=\textup{o}$,}
    \end{cases}
\end{align*}
\begin{align*}
   \bdf(L)=\deg\mathcal{Z}^{\ast}(L)^{\circ}\cdot\left(\bdf\left(0\right)+\bdf\left(1\right)\right)=\begin{cases}
        2(p+1)\cdot\left(\frac{1}{2}-\frac{1}{p+1}\right), &\textup{if $\ast=\textup{s}$;}\\
        2(p-1)\cdot\left(\frac{1}{2}-0\right), &\textup{if $\ast=\textup{o}$,}
    \end{cases}\Bigg\}=p-1.
\end{align*}
On the other hand, by (\ref{aequivb-an-poly}) and (\ref{aequivb-iso-poly}),
\begin{equation*}
    \den^{\star}(X,L)=\begin{cases}
        X^{-1}+p+1+pX, &\textup{if $\ast=\textup{s}$;}\\
        X^{-1}+p-1+pX, &\textup{if $\ast=\textup{o}$.}
    \end{cases}
\end{equation*}
Therefore $\den^{\star}(L)=\deg\mathcal{Z}^{\ast}(L)$ and $\partial\den^{\star}(L)=p-1=\bdf(L)$. Theorem \ref{corank-1-local} is true.

\subsection{Counting Lattices}
\begin{lemma}\label{lattice-M}
    Let $L$ be a rank $2$ $\zp$-lattice spanned by elements $x$ and $y$, there is a one-to-one correspondence between the following two sets
    \begin{align}
        \iota_{L}^{x}:\{(k,c):k\in\mathbb{Z}_{\geq1},\,\,c\in\zp/p^{k}\zp\}&\longrightarrow\left\{M\supsetneq L:x\notin pM\right\}\label{the-counting}\\
        (k,c)&\longmapsto M=\zp\textup{-span of}\,\,x\,\,\textup{and}\,\,p^{-k}\left(y+cx\right).\notag
    \end{align}
\end{lemma}
\begin{proof}
    Since $x\notin pM$, the lattice $M$ is spanned over $\zp$ by $x$ and $y^{\pr}=ay+c^{\pr}x$. If $\nu_p(a)\geq0$, the lattice $M$ is spanned by $x$, $y$ and $c^{\pr}x$ since $L\subset M$. Since $x\notin pM$, we must have $c^{\pr}\in\zp$, then $L=M$, which is a contradiction. Hence $\nu_p(a)<0$. Scaling by an element in $\zp^{\times}$, we can assume that $a=p^{-k}$ for some $k\geq1$ and $y^{\pr}=p^{-k}\left(y+cx\right)$ where $c'=p^{-k}c$. Therefore we have
    \begin{equation*}
        y=p^{k}y^{\pr}-cx,
    \end{equation*}
    hence we must have $c\in\zp$ since $L\subset M$ and $x,y^\pr$ is a $\zp$-basis of $M$.
    \par
    However, the element $c$ is only determined up to $\zp/p^{k}\zp$ because we can change the element $y^{\pr}$ by adding $\zp$-multiples of $x$. Therefore the map $\iota_L^x$ (\ref{the-counting}) is surjective.
    \par
    Now let $(k_1,c_1)$ and $(k_2,c_2)$ be two elements of the set on the left hand side of (\ref{the-counting}). Let $M_i$ be the lattice spanned by $x$ and $y_i=p^{-k_i}(y+c_i^\pr x)$ where $i=1,2$, here we take $c_i^\pr\in\zp$ as a lift of the element $c_i\in\zp/p^{k_i}\zp$. It's easy to see that
    \begin{equation*}
        [x,y_2]=[x,y_1]\cdot\begin{pmatrix}
            1 & p^{-k_2}(c_2^\pr-c_1^\pr)\\
            0 & p^{k_1-k_2}
        \end{pmatrix}.
    \end{equation*}
    Therefore if $M_1=M_2$, we have $k_1=k_2$ and $p^{-k_2}(c_2^\pr-c_1^\pr)\in\zp$, which is equivalent to saying that $(k_1,c_1)=(k_2,c_2)$. Hence the map $\iota_L^x$ (\ref{the-counting}) is injective.
\end{proof}

\subsection{Gross--Keating invariants of intermediate lattices}
Let $L$ be a nondegenerate quadratic $\zp$-lattice of rank $2$ with $\textup{GK}(L)=(a,b)$. In $\S$\ref{larger-lattice}, we construct a lattice $L^\pr$ such that 
\begin{itemize}
    \item [(a).]$L\subsetneq L^{\prime}$ and $L^\pr/L\simeq\mathbb{F}_p$.
    \item [(b).]$\textup{GK}(L^{\pr})$ is either $(a,b-2)$, or $(b-2,a)$ if $b-2< a$.
\end{itemize}
By (b), we can find a basis $\boldsymbol{x}=(x_1,x_2)$ of $L$ such that $x_1,p^{-1}x_2$ is a basis of $L^\pr$.
In the following paragraphs, denote by $L_{k,c}$ the lattice $\iota_L^{x_2}\left((k,c)\right)$ for $k\geq1$ and $c\in\zp/p^k\zp$. We also define $L_{0,0}$ to be the lattice $L$ itself.
\begin{lemma}\label{gk-lkc-other}
    Let $L$ be a nondegenerate quadratic $\zp$-lattice of rank $2$ with $\textup{GK}(L)=(a,b)$. For $k\geq0$ and $c\in\zp/p^k\zp$, if $L$ is not in the Case 1 of $\S$\ref{larger-lattice}, the Gross--Keating invariant of the lattice $L_{k,c}$ is $\textup{GK}(L_{k,c})=(a-2k,b)$.
\end{lemma}
\begin{proof}
    Let $\textup{GK}(L_{k,c})=(\mu_1,\mu_2)$. By Lemma \ref{gk-dep-V}, $\mu_1+\mu_2=a+b-2k$, hence we only need to show that $\mu_1=a-2k$.
    We first prove the case that $p>2$ (Case 2 in $\S$\ref{larger-lattice}), notice that under the basis $\boldsymbol{x}=(x_1,x_2)$ of $L$, the quadratic form is diagonalized in the way that $q(xx_1+yx_2)=u_1p^{a}x^{2}+u_2p^{b}y^{2}$. Then
    \begin{align*}
        q\left(x\cdot p^{-k}(x_1+cx_2)+yx_2\right)=\left(u_1p^{a-2k}+u_2 c^{2}p^{b-2k}\right)x^{2}+2u_2cp^{b-k}xy+u_2p^{b}y^{2}.
    \end{align*}
    Under the assumption that we are not in the Case $1$ of $\S$\ref{larger-lattice}, the lattice $L$ is either isotropic with $a<b$, or $L$ is anisotropic. Therefore we have $\nu_p\left(u_1p^{a-2k}+u_2 c^{2}p^{b-2k}\right)=a-2k<\min\{\nu_p\left(2u_2cp^{b-k}\right), \nu_p(u_2p^{b})=b\}$. Therefore $\mu_1=a-2k$ by Lemma \ref{gk-inv-rk2} $(a)$.
    \par
    Now we focus on the case that $p=2$. If the lattice $L$ is not diagonalizable, we have $a=b$ and $L$ is anisotropic by Lemma \ref{gk-inv-rk2} $(c)$, so we are in Case 3 of $\S$\ref{larger-lattice}. Under the basis $\boldsymbol{x}=(x_1,x_2)$ of $L$, the quadratic form is given by $q(xx_1+yx_2)=2^{a}\left(x^2+xy+y^{2}\right)$. Then
    \begin{align*}
        q\left(x\cdot 2^{-k}(x_1+cx_2)+yx_2\right)=2^{a-2k}(1+c+c^{2})x^{2}+2^{a-k}(1+2c)xy+2^{a}y^{2}.
    \end{align*}
    Notice that $\nu_2\left(2^{a-2k}(1+c+c^{2})\right)=a-2k<\min\{\nu_2(2^{a-k}(1+2c))=a-k,\nu_2(2^{a})=a\}$, therefore $\mu_1=a-2k$ by Lemma \ref{gk-inv-rk2} $(a)$. 
    \par
    If the lattice is diagonalizable, following the notations in Case 4 of $\S$\ref{larger-lattice}, let $\boldsymbol{e}$ be a basis of $L$ as in Lemma \ref{gk-inv-rk2} $(d)$. Then 
    \begin{align*}
        q\left(xe_1+y\cdot e_2\right)=u_12^{a}x^{2}+u_22^{b-\delta}y^{2},
    \end{align*}
    where $\delta=0,1,2$ according to (\ref{gk-2-diag}).
    \begin{itemize}
        \item [Case 4-0] For $\delta=0$, we have $a\leq b$ and $a\nequiv b$ mod $2$. Then $x_1=e_1$, $x_2=e_2$ and 
        \begin{align*}
        q\left(x\cdot 2^{-k}(x_1+cx_2)+yx_2\right)=\left(u_12^{a-2k}+c^{2}u_22^{b-2k}\right)x^{2}+cu_22^{b+1-k}xy+u_22^{b}y^{2}.
    \end{align*}
    Notice that $\nu_2\left(u_12^{a-2k}+c^{2}u_22^{b-2k}\right)=a-2k$, therefore $\mu_1=a-2k$ by Lemma \ref{gk-inv-rk2} $(a)$.
        \item [Case 4-1] For $\delta=1$, we have $a\leq b-1$ and $a\equiv b-1$ mod $2$. If $a\leq b-3$, we have $x_1=e_1$, $x_2=e_2$ and
        \begin{align*}
        q\left(x\cdot 2^{-k}(x_1+cx_2)+yx_2\right)=\left(u_12^{a-2k}+c^{2}u_22^{b-1-2k}\right)x^{2}+cu_22^{b-k}xy+u_22^{b-1}y^{2}.
    \end{align*}
    Notice that $\nu_2\left(u_12^{a-2k}+c^{2}u_22^{b-1-2k}\right)=a-2k$, therefore $\mu_1=a-2k$ by Lemma \ref{gk-inv-rk2} $(a)$.
        If $a= b-1$, we have $x_1=e_1$, $x_2=e_1+e_2$ and
        \begin{align*}
        &q\left(x\cdot 2^{-k}(x_1+cx_2)+yx_2\right)\\=&\left((1+c)^{2}u_12^{a-2k}+c^{2}u_22^{a-2k}\right)x^{2}+\left(u_1+c(u_1+u_2)\right)2^{a+1-k}xy+\left(u_1+u_2\right)2^{a}y^{2}.
    \end{align*}
    Notice that $\nu_2\left((1+c)^{2}u_12^{a-2k}+c^{2}u_22^{a-2k}\right)=a-2k$, therefore $\mu_1=a-2k$ by Lemma \ref{gk-inv-rk2} $(a)$.
        \item [Case 4-2] For $\delta=2$, we have $a\leq b-2$ and $a\equiv b$ mod $2$. If $a\leq b-4$, we have $x_1=e_1$, $x_2=e_2$ and
        \begin{align*}
        q\left(x\cdot 2^{-k}(x_1+cx_2)+yx_2\right)=\left(u_12^{a-2k}+c^{2}u_22^{b-2-2k}\right)x^{2}+cu_22^{b-1-k}xy+u_22^{b-2}y^{2}.
    \end{align*}
    Notice that $\nu_2\left(u_12^{a-2k}+c^{2}u_22^{b-2-2k}\right)=a-2k$, therefore $\mu_1=a-2k$ by Lemma \ref{gk-inv-rk2} $(a)$. If $a= b-2$, we have $x_1=e_1$, $x_2=e_1+e_2$ and
        \begin{align*}
        &q\left(x\cdot 2^{-k}(x_1+cx_2)+yx_2\right)\\=&\left((1+c)^{2}u_12^{a-2k}+c^{2}u_22^{a-2k}\right)x^{2}+\left(u_1+c(u_1+u_2)\right)2^{a+1-2k}xy+\left(u_1+u_2\right)2^{a}y^{2}.
    \end{align*}
    Notice that $\nu_2\left((1+c)^{2}u_12^{a-2k}+c^{2}u_22^{a-2k}\right)=a-2k$, therefore $\mu_1=a-2k$ by Lemma \ref{gk-inv-rk2} $(a)$.
    \end{itemize}
\end{proof}
\begin{lemma}\label{gk-Lkc-case1}
    Let $L$ be an isotropic quadratic $\zp$-lattice of rank $2$ with $\textup{GK}(L)=(a,a)$ for some $a\in\mathbb{Z}$, i.e., Case $1$ of $\S$\ref{larger-lattice}. For $k\geq1$ and $c\in\zp/p^{k}\zp$,
    \begin{equation*}
        \textup{GK}(L_{k,c})=\begin{cases}
            (a-2k+t,a-t), &\textup{if $c$ or $1+c\in p^{t}\zp^{\ast}/p^{k}\zp$ for some $1\leq t\leq k$;}\\
            (a-2k,a), &\textup{otherwise.}
        \end{cases}
    \end{equation*}
\end{lemma}
\begin{proof}
    Let $\textup{GK}(L_{k,c})=(\mu_1,\mu_2)$. By Lemma \ref{gk-dep-V}, $\mu_1+\mu_2=2a-2k$, hence we only need to find $\mu_1$ using Lemma \ref{gk-inv-rk2} $(a)$. Using the notations in Case 1 of $\S$\ref{larger-lattice},
    \begin{align}\label{iso-rk-2-change}
        q\left(x\cdot p^{-k}(x_1+cx_2)+y\cdot x_2\right)&=p^{a-2k}\left((1+c)x+p^{k}y\right)\left(cx+p^{k}y\right)\\
        &=c(1+c)p^{a-2k}x^{2}+(1+2c)p^{a-k}xy+p^{a}y^{2}.\notag
    \end{align}
    \par
    If $c$ or $1+c\in p^{t}\zp^{\ast}/p^{k}\zp$ for some $1\leq t\leq k-1$, notice that $\nu_p\left(q(\boldsymbol{x})\right)\geq a-2k+t$. Moreover, $q\left(p^{-k}(x_1+cx_2)\right)=p^{a-2k}c\left(1+c\right)$ has $p$-adic valuation $a-2k+t$, hence $\mu_1=a-2k+t$ and $\mu_2=a-t$.
    \par
    If $c\equiv 0$ or $-1$ mod $p^{k}$, we can just let $c=0$ or $-1$ since the lattice $L_{k,c}$ is determined up to $c\in\zp/p^{k}\zp$, then the quadratic form of $L_{k,c}$ is equivalent to $p^{a-k}xy$, hence $\textup{GK}(L_{k,c})=(a-k,a-k)$ by Lemma \ref{gk-inv-rk2}.
    \par
    Otherwise, notice that $\nu_p\left(q(\boldsymbol{x})\right)\geq a-2k$ and $q\left(p^{-k}(x_1+cx_2)\right)=p^{a-2k}c\left(1+c\right)$ has $p$-adic valuation $a-2k$, hence $\mu_1=a-2k$ and $\mu_2=a$.
\end{proof}
\begin{corollary}\label{lkc-integral}
    Let $L$ be a nondegenerate integral quadratic $\zp$-lattice of rank $2$ with $\textup{GK}(L)=(a,b)$. 
    \begin{itemize}
        \item If $L$ is not in the Case 1 of $\S$\ref{larger-lattice}, the lattice $L_{k,c}$ is integral if and only if $1\leq k\leq [a/2]$.
        \item If $L$ is in the Case 1 of $\S$\ref{larger-lattice}, the lattice $L_{k,c}$ is integral if and only if $1\leq k\leq [a/2]$ or
        \begin{equation*}
            \left[\frac{a}{2}\right]+1\leq k\leq a\,\,\textup{and}\,\,c\,\,\textup{or}\,\,1+c\in p^{t}\zp^{\ast}/p^{k}\zp\,\,\textup{for some}\,\,t\geq 2k-a.
        \end{equation*}
    \end{itemize}
\end{corollary}
\begin{proof}
    A nondegenerate quadratic $\zp$-lattice of rank $2$ with $\textup{GK}(L)=(a,b)$ is integral if and only if $a\geq0$. Then this corollary follows from Lemma \ref{gk-lkc-other} and Lemma \ref{gk-Lkc-case1}.
\end{proof}

\subsection{Geometric difference formula}
\begin{proposition}\label{diff-geo}
    Let $L$ be a nondegenerate integral quadratic $\zp$-lattice of rank $2$ with $\textup{GK}(L)=(a,b)$ such that $b\geq2$. Let $L^\pr$ be the lattice constructed in $\S$\ref{larger-lattice}.
    \begin{itemize}
        \item If $a\nequiv b$ mod $2$, 
        \begin{align}\label{a-neq-b-mod-2-geo-deg}
             &\deg\mathcal{Z}^{\ast}(L)-\deg\mathcal{Z}^{\ast}(L^{\prime})=(a+1)p^{(a+b-1)/2},
        \end{align}
        \begin{align}\label{a-neq-b-mod-2-geo}
             &\bdf(L)-\bdf(L^{\prime})\\&=\frac{(a+1)(b-1)p^{(a+b+3)/2}-2(a+1)bp^{(a+b+1)/2}+(a+1)(b+1)p^{(a+b-1)/2}+2p^{a+1}-2}{2(p-1)^{2}}.\notag
        \end{align}
        \item If $a\equiv b$ mod $2$ and $L$ is anisotropic, 
        \begin{align}\label{a-eq-b-mod-2-geo-aniso-deg}
             \deg\mathcal{Z}^{\ast}(L)-\deg\mathcal{Z}^{\ast}(L^{\prime})=(a+1)p^{(a+b)/2-1}(p+1),
        \end{align}
        \begin{align}\label{a-eq-b-mod-2-geo-aniso}
             \bdf(L)-\bdf(L^{\prime})&\\=\frac{1}{2(p-1)^{2}}\cdot\bigg(&(a+1)bp^{(a+b+4)/2}-(a+1)(b+4)p^{(a+b+2)/2}\notag\\&-(a+1)(b-4)p^{(a+b)/2}+(a+1)bp^{(a+b-2)/2}+4p^{a+1}-4\bigg).\notag
        \end{align}
        \item If $a\equiv b$ mod $2$ and $L$ is isotropic, 
        \begin{align}\label{a-eq-b-mod-2-geo-iso-deg}
             \deg\mathcal{Z}^{\ast}(L)-\deg\mathcal{Z}^{\ast}(L^{\prime})=\begin{cases}
                 (a+1)p^{(a+b)/2-1}(p-1), &\textup{if $a<b$;}\\
                 p^{a-1}\left((a+1)p-a+1\right), &\textup{if $a=b$.}
             \end{cases}
        \end{align}
        \begin{align}\label{a-eq-b-mod-2-geo-iso}
             \bdf(L)-\bdf(L^{\prime})=\frac{(a+1)b}{2}\cdot p^{(a+b-2)/2}(p-1).
        \end{align}
    \end{itemize}
\end{proposition}
\begin{proof}
    By the decomposition of $\mathcal{Z}^{\ast}(L)$ in Lemma \ref{hor-decomposition}, and the definition of $\bdf(L)$ in (\ref{local-change-delta}), the differences $\deg\mathcal{Z}^{\ast}(L)-\deg\mathcal{Z}^{\ast}(L^\pr)$ and $\bdf(L)-\bdf(L^{\prime})$ are both summations ranging over integral lattices $M$ such that $L\subset M$ and $x_2\notin pM$, such lattices $M$ and their Gross--Keating invariant have been described explicitly in Lemma \ref{lattice-M}, \ref{gk-lkc-other} and \ref{gk-Lkc-case1}. 
    \par
    Notice that for all the prime numbers $p$ (including $p=2$), the field $\mathbb{Q}_{L,p}=\mathbb{Q}_p\left(\sqrt{\disc\,L}\right)$ is a ramified extension of $\mathbb{Q}_p$ if and only if $a\nequiv b$ mod $2$, an unramified extension of $\mathbb{Q}_p$ if and only if $L$ is anisotropic, and $\mathbb{Q}_p\times\mathbb{Q}_p$ if and only if $L$ is isotropic (we refer to the proof of Lemma \ref{gk-dep-V} and Remark \ref{Q_L}).
    \par
    Let's first consider the case that $L$ is not in the Case 1 of $\S$\ref{larger-lattice}. For $1\leq k\leq [a/2]$ and $c\in\zp/p^{k}\zp$, the Gross--Keating invariant of the lattice $L_{k,c}$ is $(a-2k,b)$, therefore
    \begin{align*}
        \deg\mathcal{Z}^{\ast}(L)-\deg\mathcal{Z}^{\ast}(L^{\prime})=\sum\limits_{k=0}^{[a/2]}\sum\limits_{c\in\zp/p^{k}\zp}\deg\mathcal{Z}^{\ast}(L_{k,c})^{\circ}=\sum\limits_{k=0}^{[a/2]}p^{k}\cdot\deg\mathcal{Z}^{\ast}(L_{k,0})^{\circ}.
    \end{align*}
    and
    \begin{align*}
        \bdf(L)-\bdf(L^{\prime})&=\sum\limits_{k=0}^{[a/2]}\sum\limits_{c\in\zp/p^{k}\zp}\deg\mathcal{Z}^{\ast}(L_{k,c})^{\circ}\cdot\left(\bdf\left(\left[\frac{b-a}{2}\right]+k\right)+\bdf\left(\left[\frac{a+b}{2}\right]-k\right)\right)\\&=\sum\limits_{k=0}^{[a/2]}p^{k}\cdot\deg\mathcal{Z}^{\ast}(L_{k,0})^{\circ}\cdot\left(\bdf\left(\left[\frac{b-a}{2}\right]+k\right)+\bdf\left(\left[\frac{a+b}{2}\right]-k\right)\right).
    \end{align*}
    Then formulae (\ref{a-neq-b-mod-2-geo-deg})--(\ref{a-eq-b-mod-2-geo-iso}) (excluding the Case 1) follow from explicit computations based on Lemma \ref{primitive-part}, (\ref{local-deg-horizontal}) and \eqref{equation:delta_Fal_def}.
    \par
    If $L$ is in the Case 1 of $\S$\ref{larger-lattice}, by Corollary \ref{lkc-integral},
    \begin{align*}
        \deg\mathcal{Z}^{\textup{o}}(L)-\deg\mathcal{Z}^{\textup{o}}(L^{\prime})=&\deg\mathcal{Z}^{\textup{o}}(L)^{\circ}+\sum\limits_{k=1}^{[a/2]}\left(\sum\limits_{t=0}^{k}N(k,t)\cdot\deg\mathcal{Z}^{\textup{o}}(L_{k,p^{t}})^{\circ}\right)\\
        &+\sum\limits_{k=[a/2]+1}^{a}\left(\sum\limits_{t=2k-a}^{k}N(k,t)\cdot\deg\mathcal{Z}^{\textup{o}}(L_{k,p^{t}})^{\circ}\right),
    \end{align*}
    and
    \begin{align*}
        \bdf(L)-\bdf(L^{\prime})=&\deg\mathcal{Z}^{\textup{o}}(L)^{\circ}\cdot\left(\bdf(0)+\bdf(a)\right)\\&+\sum\limits_{k=1}^{[a/2]}\left(\sum\limits_{t=0}^{k}N(k,t)\cdot\deg\mathcal{Z}^{\textup{o}}(L_{k,p^{t}})^{\circ}\cdot\left(\bdf(k-t)+\bdf(a-k)\right)\right)\\
        &+\sum\limits_{k=[a/2]+1}^{a}\left(\sum\limits_{t=2k-a}^{k}N(k,t)\cdot\deg\mathcal{Z}^{\textup{o}}(L_{k,p^{t}})^{\circ}\cdot\left(\bdf(k-t)+\bdf(a-k)\right)\right),
    \end{align*}
    where the number $N(k,t)$ counts how many $c\in\zp/p^{k}\zp$ satisfy certain conditions in Lemma \ref{gk-Lkc-case1}, explicitly,
    \begin{equation*}
        N(k,t)=\begin{cases}
            2\varphi(p^{k-t}), &\textup{if $t>0$;}\\
            (p^{k}-2p^{k-1}), &\textup{if $t=0$.}
        \end{cases}
    \end{equation*}
    Then formulae (\ref{a-neq-b-mod-2-geo-deg})--(\ref{a-eq-b-mod-2-geo-iso}) (in the Case 1) follow from explicit computations based on Lemma \ref{primitive-part}, (\ref{local-deg-horizontal}) and \eqref{equation:delta_Fal_def}.
\end{proof}

\subsection{Analytic difference formula}
\begin{lemma}\label{diff-ana}
    Let $L$ be a nondegenerate integral quadratic $\zp$-lattice of rank $2$ with $\textup{GK}(L)=(a,b)$ such that $b\geq2$. Let $L^\pr$ be the lattice constructed in $\S$\ref{larger-lattice}.
    \begin{itemize}
        \item If $a\nequiv b$ mod $2$,
        \begin{align}\label{a-neq-b-mod-2-ana-value}
             &\den^{\star}(L)-\den^{\star}(L^{\pr})=(a+1)p^{(a+b-1)/2},
        \end{align}
        \begin{align}\label{a-neq-b-mod-2-ana}
             &\partial\den^{\star}(L)-\partial\den^{\star}(L^{\pr})\\&=\frac{(a+1)(b-1)p^{(a+b+3)/2}-2(a+1)bp^{(a+b+1)/2}+(a+1)(b+1)p^{(a+b-1)/2}+2p^{a+1}-2}{2(p-1)^{2}}.\notag
        \end{align}
        \item If $a\equiv b$ mod $2$ and $L$ is anisotropic, 
        \begin{align}\label{a-eq-b-mod-2-ana-aniso-value}
             \den^{\star}(L)-\den^{\star}(L^{\pr})=(a+1)p^{(a+b)/2-1}(p+1),
        \end{align}
        \begin{align}\label{a-eq-b-mod-2-ana-aniso}
             \partial\den^{\star}(L)-\partial\den^{\star}(L^{\pr})&\\=\frac{1}{2(p-1)^{2}}\cdot\bigg(&(a+1)bp^{(a+b+4)/2}-(a+1)(b+4)p^{(a+b+2)/2}\notag\\&-(a+1)(b-4)p^{(a+b)/2}+(a+1)bp^{(a+b-2)/2}+4p^{a+1}-4\bigg).\notag
        \end{align}
        \item If $a\equiv b$ mod $2$ and $L$ is isotropic, 
        \begin{align}\label{a-eq-b-mod-2-ana-iso-value}
             \den^{\star}(L)-\den^{\star}(L^{\pr})=\begin{cases}
                 (a+1)p^{(a+b)/2-1}(p-1), &\textup{if $a<b$;}\\
                 p^{a-1}\left((a+1)p-a+1\right), &\textup{if $a=b$.}
             \end{cases}
        \end{align}
        \begin{align}\label{a-eq-b-mod-2-ana-iso}
             \partial\den^{\star}(L)-\partial\den^{\star}(L^{\pr})=\frac{(a+1)b}{2}\cdot p^{(a+b-2)/2}(p-1).
        \end{align}
    \end{itemize}
\end{lemma}
\begin{proof}
    The lemma follows from the straightforward computation based on the explicit computation of the polynomial $\den^{\star}(X,L)$ in Lemma \ref{explicit-local-den-poly} (since we know the Gross--Keating invariants of $L$ and $L^{\pr}$).
\end{proof}

\noindent\textit{Proof of Theorem \ref{corank-1-local}}. If the lattice $L$ is not integral, the numbers $\deg\mathcal{Z}^\ast(L)$, $\den^{\star}(L)$, $\delta_{\mrm{Fal}}(L)$ and $\partial\den^\star(L)$ are all $0$, so there is nothing to prove. Therefore we assume that the lattice $L$ is integral.

Comparing Proposition \ref{diff-geo} and Lemma \ref{diff-ana}, we observe that for all the rank $2$ integral $\zp$-lattices $L\subset\mathbb{B}^{\ast}$ with $\textup{GK}(L)=(a,b)$ where $b\geq2$, 
\begin{equation*}
    \deg\mathcal{Z}^{\ast}(L)-\deg\mathcal{Z}^{\ast}(L^{\prime})=\den^{\star}(L)-\den^{\star}(L^{\pr}),\,\,\,\,\bdf(L)-\bdf(L^{\prime})=\partial\den^{\star}(L)-\partial\den^{\star}(L^{\pr}),
\end{equation*}
where $L^{\pr}$ is another sublattice of $\mathbb{B}^{\ast}$ constructed in $\S$\ref{larger-lattice} whose Gross--Keating invariant is either $(a,b-2)$, or $(b-2,a)$ if $b-2< a$. By mathematical induction, we can reduce Theorem \ref{corank-1-local} to the cases that $\nu(L)=0,1$ or $2$, which we have already verified in $\S$\ref{base-cases}. This concludes the proof of Theorem \ref{corank-1-local}.$\hfill\qed$

\section{Proof of the arithmetic Siegel--Weil formula}
    \label{sec:proof}
\subsection{Integrals of Green currents}
In this subsection, we evaluate the integrals of the Green currents $\mathfrak{g}^\circ(T,\msf{y})$ for $3\times 3$ matrices $T$. These integrals converge by the discussion in $\S$\ref{ssec: extending to compactification}. We also recall Berndt and Kühn's explicit formulas for the integrals of Kudla's Green function \cite[Theorem 4.2]{BerndtKuehn2012b}, which will simplify our proof.
\begin{lemma}\label{integral-3}
    Let $T\in\textup{Sym}_3(\mathbb{Q})$ and $\msf{y}\in\textup{Sym}_3(\mathbb{R})_{>0}$ be symmetric matrices.
    We then have
    \begin{align}\label{integral of size 3}
        \int\limits_{\overline{\mathcal{M}}(\mathbb{C})}\mathfrak{g}^\circ(T,\msf{y})=(16\pi^2)^{-1}\ct\int_{1}^{+\infty}\frac{\textup{d}}{\textup{d}t}\left(\left(C_T^{3}\right)^\prime\left(t\msf{y},0\right)\right)\frac{\textup{d}t}{t^{\rho}}.
    \end{align}
\end{lemma}
\begin{proof}
    By Theorem \ref{thm:extension} and the discussion in $\S$\ref{ssec: extending to compactification}, 
    we have
    \begin{align*}
        \int\limits_{\overline{\mathcal{M}}(\mathbb{C})}\mathfrak{g}^\circ(T,\msf{y})
        =\ct\int\limits_{\overline{\mathcal{M}}(\mathbb{C})}\mathfrak{g}^\circ(T,\msf{y},\rho)\overset{\eqref{wedge-omega}}{=}\ct\int\limits_{\mathcal{M}(\mathbb{C})}\int_{1}^{+\infty}C_{\theta,T}\left(t \msf{y},h_z,\tilde{\nu}_3\otimes\boldsymbol{1}_{\textup{M}_2(\hat{\mathbb{Z}})^{3}}\right)\frac{\textup{d}t}{t^{1+\rho}}\cdot\Omega^{2}_z.
    \end{align*}
    By Theorem \ref{thm:main}, the order of the two integrals on the right hand side can be interchanged for $\mathrm{Re}(\rho)\gg0$:
    \begin{align*}
        \int\limits_{\mathcal{M}(\mathbb{C})}\int_{1}^{+\infty}C_{\theta,T}\left(t\msf{y},h_z,\tilde{\nu}_3\otimes\boldsymbol{1}_{\textup{M}_2(\hat{\mathbb{Z}})^{3}}\right)\frac{\textup{d}t}{t^{1+\rho}}\cdot\Omega^{2}_z=\int_{1}^{+\infty}\int\limits_{\mathcal{M}(\mathbb{C})}C_{\theta,T}\left(t\msf{y},h_z,\tilde{\nu}_3\otimes\boldsymbol{1}_{\textup{M}_2(\hat{\mathbb{Z}})^{3}}\right)\Omega^{2}_z\cdot\frac{\textup{d}t}{t^{1+\rho}}.
    \end{align*}

    In Appendix $\S$\ref{sec append: genus 3 convergence}, we denote the function $C_{\theta,T}(t\msf{y},h_z,\tilde{\nu}_3\otimes\boldsymbol{1}_{\textup{M}_2(\hat{\mathbb{Z}})^{3}})$ by $K_T(t,\msf{y},z)$. Remark \ref{rem:boundedness of K} implies that the integral of $C_{\theta,T}(t\msf{y},h_z,\tilde{\nu}_3\otimes\boldsymbol{1}_{\textup{M}_2(\hat{\mathbb{Z}})^{3}})$ converges absolutely.
    Therefore the Siegel--Weil formula in Lemma \ref{term-by-term SW} implies
    \begin{align}
        \int_{O(V)(\mathbb{Q})\backslash O(V)(\mathbb{A})}\theta_T\left(g,h,\tilde{\nu}_3\otimes\boldsymbol{1}_{\textup{M}_2(\hat{\mathbb{Z}})^{3}}\right) \dd h =\frac{1}{2}\cdot E_{T}\left(g,0,\tilde{\nu}_3\otimes\boldsymbol{1}_{\textup{M}_2(\hat{\mathbb{Z}})^{3}}\right),
    \end{align}
    where the measure on $O(V)(\A)$ is chosen such that the total volume of $O(V)(\mathbb{Q})\backslash O(V)(\mathbb{A})$ is 1.
    The integration on $O(V)(\A)$ can be reduced to an integration on $SO(V)(\A)$ with a factor of $1/2$
    by \cite[(4.14)]{Kudla2003}, and then reduced to an integration on $\mathcal{M}(\mathbb{C})$
    by the invariance of $\tilde{\nu}_3 \otimes 1_{M_2(\hat{\Z})^3}$
    under the standard compact subgroup of $SO(V)(\A)$.
    By complex uniformization of $\mathcal{M}(\C)$,
    there exists a positive constant $C$ such that for all functions $F$
    on $O(V)(\mathbb{Q})\backslash O(V)(\mathbb{A})$ which are
    $\opn{Stab}_{O(V)(\R)}(i,i) \times O(V)(\hat{\Z})$-invariant,\footnote{
        Here $O(V)(\hat{\Z})$ is abuse of notation for the stabilizer of $L \otimes_{\Z} \hat{\Z}$
        where $L = M_2(\Z) \subseteq V$.
        The notation $\opn{Stab}_{O(V)(\R)}(i,i)$ denotes
        the stabilizer in $O(V)(\R)$ of $(i,i) \in \mbb{D}^{\pm}$ for the
        canonical action $O(V)(\R) \acts \mbb{D}^{\pm}$.
    }
    we have
    \begin{equation*}
        \int_{O(V)(\mathbb{Q})\backslash O(V)(\mathbb{A})} F(h) \dd h=C\cdot\int\limits_{\mathcal{M}(\mathbb{C})}F(h_z)\Omega^{2}_z
    \end{equation*}
    whenever the integrals are absolutely convergent.
    Taking $F = 1$ shows $C=\textup{vol}(\mathcal{M}(\mathbb{C}),\Omega^2)^{-1}=288$.
    Therefore
    \begin{align}\label{theta-to-eis-rk3}
        \int\limits_{\mathcal{M}(\mathbb{C})}C_{\theta,T}\left(t\msf{y},h_z,\tilde{\nu}_3\otimes\boldsymbol{1}_{\textup{M}_2(\hat{\mathbb{Z}})^{3}}\right)\Omega_z^2=576^{-1}C_{T}\left(t\msf{y},0,\tilde{\nu}_3\otimes\boldsymbol{1}_{\textup{M}_2(\hat{\mathbb{Z}})^{3}}\right).
    \end{align}
    Therefore we have
    \begin{align*}
        \int\limits_{\overline{\mathcal{M}}(\mathbb{C})}\mathfrak{g}^\circ
        (T,\msf{y})&=576^{-1}\ct\int_{1}^{+\infty}C_{T}\left(t\msf{y},0,\tilde{\nu}_3\otimes\boldsymbol{1}_{\textup{M}_2(\hat{\mathbb{Z}})^{3}}\right)\frac{\textup{d}t}{t^{1+\rho}}\\&\stackrel{\textup{\cite[(5.139)]{GS19}}}{=}288^{-1}\ct\int_{1}^{+\infty}\frac{\textup{d}}{\textup{d}t}\left(C_{T}^\prime\left(t\msf{y},0,\Phi_3^{2}\right)\right)\frac{\textup{d}t}{t^{\rho}}\\&=(16\pi^2)^{-1}\ct\int_{1}^{+\infty}\frac{\textup{d}}{\textup{d}t}\left(\left(C_T^{3}\right)^\prime\left(t\msf{y},0\right)\right)\frac{\textup{d}t}{t^{\rho}}.
    \end{align*}
\end{proof}

In $\S$\ref{ssec: extending to compactification}, we saw that when $T$ is a $1\times1$ nonsingular matrix,
the function $\mf{g}(T, \msf{y}) = \mathfrak{g}^{\circ}(T,\msf{y})$
is absolutely integrable against the volume form $\Omega^2$
(see \eqref{eq:genus 1 converge}).
The integral is closely related to the derivative of the
classical genus 1 Eisenstein series. 

\begin{lemma}\label{integral-1-1}
For all $T\in\textup{Sym}_1(\mathbb{Q})\backslash\{(0)\}$ and $\msf{y}\in\textup{Sym}_1(\mathbb{R})_{>0}$,
we have
    \begin{align}\label{integral-rank-1}
        \int\limits_{\mathcal{M}(\mathbb{C})}\mathfrak{g}(T,\msf{y})\wedge\Omega^2=(96\pi)^{-1}\lim\limits_{t\rightarrow+\infty}\left(\left(C_{T}^{1}\right)^\prime\left(t\msf{y},1\right)-\left(C_{T}^{1}\right)^\prime\left(\msf{y},1\right)\right).
    \end{align}
\end{lemma}
\begin{proof}
    Both sides are nonzero only if $T$ is integral. Suppose that $T=(m)$ where $m\in\mathbb{Z}\backslash\{0\}$.
    For the left hand side, it has been calculated explicitly in \cite[Theorem 4.2]{BerndtKuehn2012b} that
    \begin{align}\label{integral-rank-1-g}
        \int\limits_{\mathcal{M}(\mathbb{C})}\mathfrak{g}(T,\msf{y})\wedge\Omega^2&=\frac{1}{4}\int\limits_{M(\mathbb{C})}\mathfrak{g}^{\circ}(T,\msf{y},0)\wedge\Omega^2\\
        &=\frac{\sigma_1(\vert m\vert)}{96\pi\vert m\vert\msf{y}}\begin{cases}
            1, &\textup{$m\in\mathbb{Z}_{>0}$;}\\
            e^{4\pi m\msf{y}}-4\pi m\msf{y}\opn{Ei}(4\pi m\msf{y}), &\textup{$m\in\mathbb{Z}_{<0}$.}
        \end{cases}\notag
    \end{align}
    Notice that the Green function $\Xi(\msf{y},z,m)$ on $M(\C)$ (\cite[Definition 2.3]{BerndtKuehn2012a})
    is $1/2$ times our green function $\mathfrak{g}(T,\msf{y})$ descended to $M(\C)$.

    Berndt and Kühn \cite[Definition 2.1]{BerndtKuehn2012b} introduced an Eisenstein series $E_2(\msf{z},s)$
    of weight $2$ with the $T$-th Fourier coefficient $a_{T}(\msf{y},s)$.
    The relation of $E_2(\msf{z},s)$ series to $E^{1}(\msf{z},s)$ is given by
    \begin{equation}\label{relation-Eis}
        E^{1}(\msf{z},2s-1)=-8\pi E_2(\msf{z},s).
    \end{equation}
    The derivative of its Fourier coefficient $a_{T}(\msf{y},s)$ at $s=1$ is computed explicitly in \cite[(2.2.3)]{BerndtKuehn2012b}.
    By (\ref{relation-Eis}),
    \begin{equation}\label{C_T-rank1-1}
        \left(C_{T}^1\right)^\prime(\msf{y},1)=-4\pi a_{T}^\prime(\msf{y},1)=-\frac{\sigma_1(\vert m\vert)}{\vert m\vert\msf{y}}\begin{cases}
            1+4\pi m\msf{y}\frac{\sigma_1^\prime(m)}{\sigma_1(m)}, &\textup{$m\in\mathbb{Z}_{>0}$;}\\
            e^{4\pi m\msf{y}}-4\pi m\msf{y}\opn{Ei}(4\pi m\msf{y}), &\textup{$m\in\mathbb{Z}_{<0}$.}
        \end{cases}
    \end{equation}
    Here for a positive integer $m$, $\sigma_1(m)=\sum\limits_{d>0,d|m}d$ and $\sigma_1'(m) \coloneqq \sigma_1(m)\log m-2\sum\limits_{d>0,d|m}d\log d$. Then
    \begin{equation}\label{C_T-rank1-2}
        \lim\limits_{t\rightarrow+\infty}\left(C_{T}^1\right)^\prime(t\msf{y},1)=\begin{cases}
            -4\pi \sigma_1^\prime(m), &\textup{$m\in\mathbb{Z}_{>0}$;}\\
            0, &\textup{$m\in\mathbb{Z}_{<0}$.}
        \end{cases}
    \end{equation}
    Then (\ref{integral-rank-1}) follows from (\ref{integral-rank-1-g}), (\ref{C_T-rank1-1}), and (\ref{C_T-rank1-2}).
\end{proof}
\begin{remark}
The explicit expression \eqref{C_T-rank1-1} implies that the following integral
\begin{equation*}
    \int_{1}^{+\infty}\frac{\textup{d}}{\textup{d}t}\left(\left(C_{T}^{1}\right)^\prime\left(t\msf{y},1\right)\right)\frac{\textup{d}t}{t^{\rho}}
\end{equation*}
converges for $\mathrm{Re}(\rho)>-1$ and defines a holomorphic function in $\rho$ in that range. In particular, it is regular at $\rho=0$. In other words, (\ref{integral-rank-1}) is equivalent to
    \begin{align}\label{integral-rank-ct}
        \int\limits_{\mathcal{M}(\mathbb{C})}\mathfrak{g}(T,\msf{y})\wedge\Omega^2=(96\pi)^{-1}\ct\int_{1}^{+\infty}\frac{\textup{d}}{\textup{d}t}\left(\left(C_{T}^{1}\right)^\prime\left(t\msf{y},1\right)\right)\frac{\textup{d}t}{t^{\rho}}.
    \end{align}
\end{remark}

\subsection{The case of corank 0}
Let $T\in\textup{Sym}_3(\mathbb{Q})$ be a nonsingular matrix. Lemma \ref{correct-codim} ($d$)
implies that the special cycle $\mathcal{Z}(T)$ is nonempty only if $T$ is half-integral
off the diagonal, integral on the diagonal, positive definite, and the set $\textup{Diff}(T)$ in (\ref{diff-set}) has cardinality $1$.
In that case, the special cycle $\mathcal{Z}(T)$ is supported over the supersingular locus of $\mc{M}$ at the only prime $p\in\textup{Diff}(T)$. Moreover, the special cycle $\mathcal{Z}(T)$ is an Artinian Deligne--Mumford stack. Let $L$ be a quadratic lattice of rank $3$ over $\mathbb{Z}$ with moment matrix $T$. The stacky length of $\mathcal{Z}(T)$ is given by the following formula by \cite[Proposition 3.22]{GK93} (see also \cite[$\S$4]{Gortz2007Arithmetic}),
\begin{align*}
    \widehat{\deg}\left(\left(\widehat{\mc{E}}^{\vee}\right)^{0}\big|_{\mc{Z}(T)}\right)
=&\sum\limits_{(E,E^{\prime})\in\mc{M}^{\textup{ss}}(\overline{\mathbb{F}}_p)}\frac{R_{\textup{Hom}(E,E^\prime)}(T)}{\#\textup{Aut}(E)\cdot\#\textup{Aut}(E^\prime)}\cdot\left(-\partial\den^{\star}\left(1,L\otimes\mathbb{Z}_p\right)\log p\right)
    \\
    =&\left(\prod\limits_{l\neq p}\den^{\star}\left(1,L\otimes\mathbb{Z}_l\right)\right)\cdot\left(-\partial\den^{\star}\left(1,L\otimes\mathbb{Z}_p\right)\log p\right)
    \\
    \stackrel{(\ref{whitt-local-density})}{=}&2^{3/2}\left(\prod\limits_{l\neq p}\frac{W_{T,l}(e,0,1_{\textup{M}_2(\mathbb{Z}_l)^3})}{(1-l^{-2})}\right)\cdot\frac{W_{T,p}^\prime(e,0,1_{\textup{M}_2(\mathbb{Z}_p)^3})}{(1-p^{-2})}.\notag
\end{align*}

By \cite[Proposition 3.2 (iii)]{GS19} (see also \cite[Proposition 4.3 (ii)]{BY20}),
\begin{equation*}
    \det\msf{y}^{-1}W_T(g_{i\msf{y}})e^{2\pi\textup{tr}(T\msf{y})}=-2^{11/2}\pi^{4}.
\end{equation*}
Therefore the stacky length of $\mc{Z}(T)$ is given by
\begin{equation}\label{step-1-rk3-cal}
    \widehat{\deg}\left(\left(\widehat{\mc{E}}^{\vee}\right)^{0}\big|_{\mc{Z}(T)}\right)=-(32\pi^2)^{-1}\left(C_T^3\right)^\prime(\msf{y},0).
\end{equation}
\begin{theorem}
    Let $T\in\textup{Sym}_3(\mathbb{Q})$ be a nonsingular matrix. For all $\msf{y}\in\textup{Sym}_3(\mathbb{R})_{>0}$,
    \begin{equation}\label{ari-SW-rk3}
        \widehat{\deg}[\widehat{\mc{Z}}(T,\msf{y})]=-(32\pi^{2})^{-1}\left(C_T^3\right)^\prime(\msf{y},0).
    \end{equation}
\end{theorem}
\begin{proof}
    We first consider the case that $T$ is positive definite.
    In that case, the sum in \eqref{green-currents-T}
    is empty so $\mf{g}(T, \msf{y}) = 0$,
and we hence have
    \begin{equation*}
        \int\limits_{\overline{\mathcal{M}}(\mathbb{C})}\mathfrak{g}(T,\msf{y})=0.
    \end{equation*}
    If $\mc{Z}(T)=\varnothing$, then
    both sides of (\ref{ari-SW-rk3}) vanish.
    If $\mc{Z}(T)\neq\varnothing$, then \eqref{ari-SW-rk3}
    holds by (\ref{step-1-rk3-cal}). Therefore (\ref{ari-SW-rk3})
    is true for positive definite matrices $T$.

    If $T$ is not positive definite, there exists a positive number $C$ by \cite[Proposition 3.3]{GS19} such that
    \begin{equation*}
        \left(C_T^3\right)^\prime(t\msf{y},0)=O(e^{-Ct})\,\,\,\,\textup{as $t\rightarrow\infty$}.
    \end{equation*}
    Therefore by Lemma \ref{integral-3}, we have
    \begin{equation*}
        \widehat{\deg}[\widehat{\mc{Z}}(T,\msf{y})]=\frac{1}{2}\int\limits_{\overline{\mathcal{M}}(\mathbb{C})}\mathfrak{g}(T,\msf{y})=-(32\pi^2)^{-1}\left(C_T^3\right)^\prime(\msf{y},0).
    \end{equation*}
\end{proof}

\subsection{The case of corank 1}
After Lemma \ref{integral-3}, 
the term which remains to be evaluated in the corank 1 case is the Faltings height of the special cycle.
We relate this height with the Eisenstein series Fourier coefficient
by combining Theorem \ref{corank-1-local}
with the Rapoport--Zink uniformization from $\S$\ref{sec:RZ_uniformization}.
\begin{proposition}\label{ari-deg-rk-2-full}
    Let $T^{\flat}\in\textup{Sym}_2(\mathbb{Q})$ be a nondegenerate matrix such that $\disc(T^\flat)\notin\left(\mathbb{Q}^{\times}\right)^{2}$. Then
    \begin{equation}\label{ht-rk2-}
        \widehat{\deg}\left(\widehat{\mc{E}}^{\vee} |_{\mc{Z}(T^{\flat})}\right)=W^{\star}_{T^{\flat},\infty}(\msf{y}^{\flat}, 1/2)
    \cdot\frac{\textup{d}}{\textup{d}s} \bigg|_{s = -1/2}
    \left ( \Lambda^*(s + 1/2, \eta_{T^{\flat}})
    \prod_p W^{\star}_{T^{\flat},p}(s)
    \right )
    \end{equation}
\end{proposition}
\begin{proof}
    If $\disc(T^\flat)\notin\left(\mathbb{Q}^{\times}\right)^{2}$ and $T^{\flat}$ is not positive definite, the left hand side of (\ref{ht-rk2-}) is $0$ because $\mathcal{Z}(T^\flat)=\emptyset$, while the right hand side of (\ref{ht-rk2-}) is also $0$ by (\ref{norW-archi-value}) (notice that the function $\Lambda^*(s + 1/2, \eta_{T^{\flat}})
    \prod_p W^{\star}_{T^{\flat},p}(s)$ has no pole at $s=-1/2$ because $\Lambda^{\ast}(s,\eta_{T^\flat})$ is not $\xi(s)$ by our assumption).
    
    Next we assume that $T^{\flat}$ is positive definite. If $T^{\flat}$ does not define an integral quadratic form,
then we have $\mathcal{Z}(T^{\flat}) = \emptyset$
and we also have $W^{\star}_{T^{\flat},p}(s) = 0$
whenever $T^{\flat}$ does not define an integral quadratic $\Z_p$-lattice,
as explained at the end of Section \ref{ssec:normalized_non-Arch_Whitt}.
In this case, both sides of (\ref{ht-rk2-}) are $0$ again. Conversely, if $T^{\flat}$ does define an integral quadratic form, Lemma \ref{lemma:complex_degree:corank_1} and \eqref{equation:normalized_global_Whittaker:m=4} imply that
    \begin{equation}\label{geo-SW-rk2-nor}
    \Lambda^*(0, \eta_{T^{\flat}})
    \prod_p W^{\star}_{T^{\flat},p}(-1/2)
    =
     \deg_{\mc{E}_\mathbb{Q}^{\vee}}(\mathcal{Z}(T^{\flat})_{\Q}).
    \end{equation}
In this case, we have $W^{\star}_{T^{\flat},p}(1/2) \neq 0$ for all $p$,
as follows e.g. from inspection of the explicit formulas
in Lemma \ref{explicit-local-den-poly}, as the polynomials of loc. cit.
are non-vanishing at $X = 1$.
Alternatively, the non-vanishing $W^{\star}_{T^{\flat},p}(1/2) \neq 0$
also follows from \eqref{corank-1-local}
and inspection of \eqref{hor-decomposition}.
The value $\Lambda^*(0, \eta_{T^{\flat}})$
is always finite and nonzero, since $\eta_{T^{\flat}}$
is nontrivial at the Archimedean place.

For simplicity, denote the right hand side of (\ref{ht-rk2-}) by $F(T^\flat)$. 
We thus have
    \begin{equation}
    \frac{F(T^\flat)}
    {\deg_{\mc{E}^{\vee}_{\Q}}(\mathcal{Z}(T^{\flat})_{\Q})}
    =    
    \frac{\Lambda^{* \prime}(0, \eta_{T^{\flat}})}
    {\Lambda^*(0, \eta_{T^{\flat}})}
    +
    \sum_p \frac{W^{\star \prime}_{T^{\flat},p}(-1/2)}{W^{\star}_{T^{\flat},p}(-1/2)}
    \end{equation}
where the sum has only finitely many nonzero terms.
The formula \eqref{equation:Faltings_height_CM:Chowla--Selberg}
and the functional equation $\Lambda^*(1 - s, \eta_{T^{\flat}})
=\Lambda^*(s, \eta_{T^{\flat}})$
gives
    \begin{equation}
    \frac{\Lambda^{* \prime}(0, \eta_{T^{\flat}})}
    {\Lambda^*(0, \eta_{T^{\flat}})}
    =
    -2 \opn{h}^{\mrm{CM}}_{\mrm{Fal},F}
    -\frac{1}{2} \log(4 \pi e^{\gamma})    
    \end{equation}
for $F = \Q[X]/(X^2 - \opn{disc} T^{\flat})$,
and for $\gamma$ the Euler--Mascheroni constant.
Our non-Archimedean main theorem
for corank $1$ (Theorem \ref{corank-1-local})
gives
    \begin{equation}\label{eq: whit-der-height-change}
    \frac{W^{\star \prime}_{T^{\flat},p}(-1/2)}{W^{\star}_{T^{\flat},p}(-1/2)}
    =
    -\frac{\delta_{\mrm{Fal},p}(T^{\flat})}{\deg_p(T^{\flat})} \cdot \log p
    \end{equation}
for all primes $p$, using also the comparison
between $\opn{Den}^{\star}$ and $W^{\star}$
and \eqref{equation:non-Arch_Whittaker:normalized_Den}.

Recalling the local-to-global relation
from \eqref{equation:RZ_uniformization:heights:normalized_height},
we find
    \begin{equation}
    \frac{F(T^\flat)}
    {\deg_{\mc{E}^{\vee}_{\Q}}(\mathcal{Z}(T^{\flat})_{\Q})}
    =    
    -2 \opn{h}^{\mrm{CM}}_{\mrm{Fal},F}
    -\frac{1}{2} \log(4 \pi e^{\gamma})
    +
    \sum_p \frac{\widehat{\deg}_{\widehat{\mc{E}}^{\vee}}(\mathcal{Z}(T^{\flat}))_p}{\deg_{\mc{E}^{\vee}_{\Q}}(\mathcal{Z}(T^{\flat})_{\Q})}.
    \end{equation}
By definition of $\widehat{\deg}_{\widehat{\mc{E}}^{\vee}}(\mathcal{Z}(T^\flat))_p$
from \eqref{equation:height_decomp:difference} and subsequent discussion,
we conclude
    \begin{equation}
    F(T^\flat)
    =
    \widehat{\deg}(\widehat{\mc{E}}_{\textup{Pet}}^{\vee} |_{\mathcal{Z}(T^{\flat})})
    -
    \frac{1}{2} \deg_{\mc{E}^{\vee}_{\Q}}(\mathcal{Z}(T^{\flat})_{\Q}) \cdot \log(4 \pi e^{\gamma})=\widehat{\deg}(\widehat{\mc{E}}^{\vee} |_{\mathcal{Z}(T^{\flat})}),
    \end{equation}
    where the last identity follows from (\ref{relation of metrics}).
\end{proof}

Let $T=\textup{diag}\{0,T^{\flat}\}\in\textup{Sym}_3(\mathbb{Q})$ be a rank $2$ matrix where $T^{\flat}\in\textup{Sym}_2(\mathbb{Q})$ is nondegenerate. Let $\msf{y}=\textup{diag}\{\msf{y}_1,\msf{y}^{\flat}\}\in\textup{Sym}_3(\mathbb{R})$
be a positive definite matrix.
By \eqref{equation:Eisenstein:coeff_corank_1}
and \eqref{equation:normalized_global_Whittaker:m=4},
we have
\begin{align}\label{der-rk2-step1}
    &(-32\pi^{2})^{-1}\left(C_T^{3}\right)^{\prime}(\msf{y},0)\\&=- \frac{\textup{d}}{\textup{d}s} \bigg|_{s = 0} \left ( \msf{y}_1^{s/2} W^{\star}_{T^{\flat},\infty}(\msf{y}^{\flat}, s + 1/2)
    \Lambda^*(s + 1, \eta_{T^{\flat}}) \prod_p W^{\star}_{T^{\flat},p}(s + 1/2) \right ).\notag
\end{align}

If $T^\flat$ is not positive definite, it follows from \cite[{Proposition 3.3(ii)}]{GS19}
that there exists a positive number $C$ such that
\begin{equation*}
    \frac{\textup{d}}{\textup{d}t}\left(\left(C_T^{3}\right)^{\prime}(t\msf{y},0)\right)=O(e^{-Ct})\,\,\,\,\textup{as}\,\,\,\,t\rightarrow+\infty.
\end{equation*}
Therefore 
\begin{equation}\label{integral-g-rk2-3}
    \int\limits_{\overline{\mathcal{M}}(\mathbb{C})}\mathfrak{g}^\circ(T,\msf{y})=(16\pi^2)^{-1}\ct\int_{1}^{+\infty}\frac{\textup{d}}{\textup{d}t}\left(\left(C_T^{3}\right)^\prime\left(t\msf{y},0\right)\right)\frac{\textup{d}t}{t^{\rho}}=-(16\pi^2)^{-1}\left(C_T^{3}\right)^\prime\left(\msf{y},0\right).
\end{equation}

If $T^\flat$ is positive definite, (\ref{der-rk2-step1}) implies
\begin{align}\label{derct-rk2-3}
    (-32\pi^{2})^{-1}\left(C_T^{3}\right)^{\prime}(\msf{y},0)=&-\frac{1}{2}\log\msf{y}_1\cdot W_{T^\flat,\infty}^{\star}(\msf{y}^\flat,1/2)\Lambda^\ast(1,\eta_{T^\flat})\prod_p W^{\star}_{T^{\flat},p}(1/2)\\&-\left(W_{T^\flat,\infty}^{\star}\right)^\prime\left(\msf{y}^\flat,1/2\right)\Lambda^\ast(1,\eta_{T^\flat})\prod_p W^{\star}_{T^{\flat},p}(1/2)\label{int-deg3-2-simplify-2-2}\\&- 
    \frac{\textup{d}}{\textup{d}s}\bigg|_{s=0}\Lambda^*( s+1, \eta_{T^{\flat}}) \prod_p W^{\star}_{T^{\flat},p}(s + 1/2) \label{int-deg3-2-simplify-2-3}
\end{align}
where we also used the special value formula \eqref{norW-archi-value}.
Therefore by the functional equation of the product $\Lambda^*(s + 1, \eta_{T^{\flat}}) \prod_p W^{\star}_{T^{\flat},p}(s + 1/2)$ and (\ref{geo-SW-rk2-nor}),
\begin{equation}\label{growth-g-positive-definite-3-2}
    (-32\pi^{2})^{-1}\left(C_T^{3}\right)^{\prime}(t\msf{y},0)=-\frac{1}{2}\log\,t\cdot \deg_{\mc{E}_\mathbb{Q}^{\vee}}(\mc{Z}(T^{\flat})_{\Q})+A+F(t)\,\,\,\,\textup{as}\,\,\,\,t\rightarrow+\infty,
\end{equation}
where $A$ is independent of $t$
and where the function $F(t)$ is obtained by replacing $\msf{y}^{\flat}$ with $t\msf{y}^{\flat}$ in (\ref{int-deg3-2-simplify-2-2}). By Lemma \ref{lemma: growth estimate}, there exists a positive constant $C$ such that
\begin{equation}\label{grwoth-of the tail}
    F(t)=O(t^{-1}),\,\,\,\,F^\prime(t)=O(t^{-1-C})\,\,\,\,\textup{as}\,\,\,\,t\rightarrow+\infty.
\end{equation}
\begin{theorem}
    Let $T\in\textup{Sym}_3(\mathbb{Q})$ be a matrix of rank $2$. For all $\msf{y}\in\textup{Sym}_3(\mathbb{R})_{>0}$,
    \begin{equation}
        \widehat{\deg}[\widehat{\mc{Z}}(T,\msf{y})]=-(32\pi^{2})^{-1}\left(C_T^3\right)^\prime(\msf{y},0).
    \end{equation}
\end{theorem}
\begin{proof}
    Let $\alpha\in\textup{GL}_3(\mathbb{Z})$ be a matrix such that (\ref{reduce-to-diagonal}) holds, i.e., $\tilde{T}\coloneqq{^{t}\alpha}T\alpha=\textup{diag}\{\boldsymbol{0}_1,T^\flat\}$.
    Let $\tilde{\msf{y}}=\diag\{\msf{y}_1,\msf{y}^\flat\}$ be another positive definite matrix as in (\ref{equation:y_decomp_for_current}).

    If $T^\flat$ is not positive definite,
    we may apply \eqref{integral-g-rk2-3} to find
    \begin{align}\label{int-deg3-2-simplify}
        \frac{1}{2}\int\limits_{\overline{\mathcal{M}}(\mathbb{C})}\mathfrak{g}(T,\msf{y})&=\frac{1}{2}\int\limits_{\overline{\mathcal{M}}(\mathbb{C})}\mathfrak{g}^\circ(T,\msf{y})=\frac{1}{2}\int\limits_{\overline{\mathcal{M}}(\mathbb{C})}\mathfrak{g}^\circ(\tilde{T},\alpha^{-1}\msf{y}{^t\alpha^{-1}})\\&=-(32\pi^2)^{-1}\left(C_{\tilde{T}}^{3}\right)^\prime\left(\alpha^{-1}\msf{y}{^t\alpha^{-1}},0\right)\stackrel{(\ref{simplify-C_T})}{=}-(32\pi^2)^{-1}\left(C_T^{3}\right)^\prime\left(\msf{y},0\right).\notag
    \end{align}

    Now assume $T^\flat$ is positive definite.
    In this case, we have
    \begin{align}
        &\frac{1}{2}\int\limits_{\overline{\mathcal{M}} (\mathbb{C})}\mathfrak{g}(T,\msf{y})\notag\\=&\frac{1}{2}\int\limits_{\overline{\mathcal{M}} (\mathbb{C})}\mathfrak{g}^\circ(T,\msf{y})-\frac{1}{2}\log\msf{y}_1\cdot\deg_{\mc{E}^{\vee}_{\mathbb{Q}}}(\mathcal{Z}(T^\flat)_{\mathbb{Q}})\label{final-3-2-1}\\=&\frac{1}{2}\int\limits_{\overline{\mathcal{M}}    (\mathbb{C})}\mathfrak{g}^\circ(\tilde{T},\alpha^{-1}\msf{y}{^t\alpha^{-1}})-\frac{1}{2}\log\msf{y}_1\cdot\deg_{\mc{E}^{\vee}_{\mathbb{Q}}}(\mathcal{Z}(T^\flat)_{\mathbb{Q}})\label{final-3-2-2}
        \\=&-\left(W_{T^\flat,\infty}^{\star}\right)^\prime\left(\msf{y}^\flat,1/2\right)\Lambda^\ast(1,\eta_{T^\flat})\prod_p W^{\star}_{T^{\flat},p}(1/2)-\frac{1}{2}\log\msf{y}_1\cdot\deg_{\mc{E}^{\vee}_{\mathbb{Q}}}(\mathcal{Z}(T^\flat)_{\mathbb{Q}}).\label{final-3-2-3}
        \\=&(-32\pi^{2})^{-1}\left(C_{\tilde{T}}^{3}\right)^{\prime}(\tilde{\msf{y}},0)+ 
    \frac{\textup{d}}{\textup{d}s}\bigg|_{s=0}\Lambda^*( s+1, \eta_{T^{\flat}}) \prod_p W^{\star}_{T^{\flat},p}(s + 1/2)\label{final-3-2-4}
    \\
    =&(-32\pi^{2})^{-1}\left(C_{T}^{3}\right)^{\prime}(\msf{y},0)-
    \frac{\textup{d}}{\textup{d}s}\bigg|_{s=-1/2}\Lambda^*( s+1/2, \eta_{T^{\flat}}) \prod_p W^{\star}_{T^{\flat},p}(s).\label{final-3-2-5}
    \end{align}
    \eqref{growth-g-positive-definite-3-2}.
    For (\ref{final-3-2-1}), we use Definition \ref{green-currents-T2}.
    For (\ref{final-3-2-2}), we use the $\textup{GL}_3(\mathbb{Z})$-invariance of the Green currents.
    For (\ref{final-3-2-3}), we combine Lemma \ref{integral-3}
    with the estimate of the growth (\ref{grwoth-of the tail}) as in \cite[Lemma 5.11]{GS19}.
    For (\ref{final-3-2-4}), we use (\ref{derct-rk2-3}). For the last identity (\ref{final-3-2-5}), we use Lemma \ref{reduce-both-diagonal} and the functional equation of the product $\Lambda^*(s + 1, \eta_{T^{\flat}}) \prod_p W^{\star}_{T^{\flat},p}(s + 1/2)$.

    Therefore by Proposition \ref{ari-deg-rk-2-full},
    \begin{equation*}
        \widehat{\deg}[\widehat{\mc{Z}}(T,\msf{y})]=\widehat{\deg}\left(\widehat{\mc{E}}^{\vee} |_{\mc{Z}(T^{\flat})}\right)+\frac{1}{2}\int\limits_{\overline{\mathcal{M}}(\mathbb{C})}\mathfrak{g}(T,\msf{y})=-(32\pi^2)\left(C_T^3\right)^\prime(\msf{y},0).
    \end{equation*}
\end{proof}

\subsection{The case of corank 2}
Recall from $\S$\ref{ssec:moduli_problems}
that $\overline{M}$ (resp. $M$) denotes the coarse moduli scheme of $\overline{\mc{M}}$ (resp. $\mc{M}$).
In \cite[{Theorem 7.61}]{Burgos2007},
Burgos Gil, Kramer, and K\"uhn calculated the height of Hecke correspondences for $\overline{M}$
with respect to the line bundle $\mc{E}^\vee_{\overline{M}}$
equipped with the Petersson metric $\|\cdot\|_{\textup{Pet}}$,
which we denote by $\widehat{\mc{E}}_{\overline{M},\textup{Pet}}^{\vee}$.
The difference is recorded given in (\ref{relation of metrics}).
A Green current calculation is given in \cite{BerndtKuehn2012b},
and the results of loc. cit. can be expressed as follows.
\begin{lemma}
    For all $T^\flat=(t)\in\textup{Sym}_1(\mathbb{Q})\backslash\{(0)\}$ and $\msf{y}^\flat\in\textup{Sym}_1(\mathbb{R})_{>0}$,
    \begin{align}\label{kuhn-main-rev}
    &\widehat{\deg}\left(\left(\widehat{\mc{E}}_{\overline{M}}^{\vee}\right)^{2}|_{\overline{T_t}}\right)+\frac{1}{2}\int_{M(\mathbb{C})}\mathfrak{g}(T^\flat,\msf{y}^\flat)\wedge\Omega^{2}\\=&-\left((4\pi^2)^{-1}\Lambda^\prime(2)+(48\pi)^{-1}\right)C_{T^\flat}^{1}(\msf{y}^\flat,1)-(48\pi)^{-1}\left(C_{T^\flat}^1\right)^\prime(\msf{y}^\flat,1)\notag.
\end{align}
\end{lemma}
\begin{proof}
 Berndt and K\"uhn
 in \cite{BerndtKuehn2012b}
 use the Hermitian line bundle 
 $\left(\widehat{\mathcal{E}}^{\vee}_{\overline{M},\mrm{Pet}}\right)^{\otimes12}$
 for their height, so their main theorem gives
\begin{align}\label{kuhn-main}
    144\cdot\left(\widehat{\deg}\left(\left(\widehat{\mc{E}}_{\overline{M},\textup{Pet}}^{\vee}\right)^{2}|_{\overline{T_t}}\right)+\frac{1}{2}\int_{M(\mathbb{C})}\mathfrak{g}(T^\flat,\msf{y}^\flat)\wedge\Omega^{2}\right)=-12\left(\psi(s)a_{T^\flat}(\msf{y}^\flat,s)\right)^{\prime}\big|_{s=1}
\end{align}
where $\psi(s)$ is a meromorphic function with the following expansion at $s=1$,
 \begin{equation*}
     \psi(s)=-1+4\left(\frac{\zeta^\prime(-1)}{\zeta(-1)}+\frac{1}{2}\right)(s-1)+O((s-1)^{2}).
 \end{equation*}
 By (\ref{relation of metrics}) and e.g. \cite[Proposition 3.2.2]{Bostetal1994}, the change of height is
 \begin{equation*}
     \widehat{\deg}\left(\left(\widehat{\mc{E}}_{\overline{M},\textup{Pet}}^{\vee}\right)^{2}|_{\overline{T_t}}\right)-\widehat{\deg}\left(\left(\widehat{\mc{E}}_{\overline{M}}^{\vee}\right)^{2}|_{\overline{T_t}}\right)=\log(4\pi e^{\gamma})\cdot\deg_{\mc{E}^\vee_\mathbb{Q}}\left(\left(\overline{T_t}\right)_{\mathbb{Q}}\right)
 \end{equation*}
 By (\ref{rk1-deg}), notice that the morphism $i_1:\mathcal{Z}(T^\flat)\rightarrow \mc{T}_t$ generically has degree 2,
 and compute
 \begin{equation*}
     \deg_{\mc{E}^\vee_\mathbb{Q}}\left(\left(\overline{T_t}\right)_{\mathbb{Q}}\right)=\int\limits_{\overline{T_t}(\mathbb{C})}\Omega=\frac{1}{2}\int\limits_{\vert\mathcal{Z}(T^\flat)(\mathbb{C})\vert}\Omega=2\deg_{\mc{E}^\vee_\mathbb{Q}}\left(\mathcal{Z}(T^\flat)_\mathbb{Q}\right)=(48\pi)^{-1}C_{T^\flat}^1(\msf{y}^\flat,1).
 \end{equation*}
 By (\ref{relation-Eis}), we have
    \begin{align*}
        -12\left(\psi(s)a_{T^\flat}(\msf{y}^\flat,s)\right)^{\prime}\big|_{s=1}=6\pi^{-1}\left(\frac{\zeta^\prime(-1)}{\zeta(-1)}+\frac{1}{2}\right)C_{T^\flat}^1(\msf{y}^\flat,1)-3\pi^{-1}\left(C_{T^\flat}^1\right)^\prime(\msf{y}^\flat,1).
    \end{align*}
    The lemma then follows from the functional equation of $\Lambda(s)$.
\end{proof}
\begin{remark}
    By (\ref{height-rank1-to-hecke}), formula (\ref{kuhn-main-rev}) can be stated in the following way:
    \begin{align}\label{refine-kuhn}
        &\widehat{\deg}\left(\left(\widehat{\mc{E}}^{\vee}\right)^{2}\big|_{\mc{Z}(T^\flat)}\right)+\int_{\mc{M}(\mathbb{C})}\mathfrak{g}(T^\flat,\msf{y}^\flat)\wedge\Omega^2\\
        &=-\left((8\pi^2)^{-1}\Lambda^\prime(2)+(96\pi)^{-1}\right)C_{T^\flat}^{1}(\msf{y}^\flat,1)-(96\pi)^{-1}\left(C_{T^\flat}^1\right)^\prime(\msf{y}^\flat,1)\notag.
    \end{align}
\end{remark}

\begin{theorem}
    Let $T\in\textup{Sym}_3(\mathbb{Q})$ be a matrix of rank $1$. For all $\msf{y}\in\textup{Sym}_3(\mathbb{R})_{>0}$,
    \begin{equation}
        \widehat{\deg}[\widehat{\mc{Z}}(T,\msf{y})]=-(32\pi^{2})^{-1}\left(C_T^3\right)^\prime(\msf{y},0).
    \end{equation}
\end{theorem}
\begin{proof}
    Let $\alpha\in\textup{GL}_3(\mathbb{Z})$ be a matrix such that (\ref{reduce-to-diagonal}) holds, i.e., $\tilde{T}\coloneqq{^{t}\alpha}T\alpha=\textup{diag}\{\boldsymbol{0}_2,T^\flat\}$.
    Let $\tilde{\msf{y}}=\diag\{\msf{y}_1,\msf{y}^\flat\}$ be another positive definite matrix as in (\ref{equation:y_decomp_for_current}).
    
    By Lemma \ref{integral-3}, we have
    \begin{align}\label{int-deg3-simplify}
        \int\limits_{\overline{\mathcal{M}}(\mathbb{C})}\mathfrak{g}^\circ(T,\msf{y})=\int\limits_{\overline{\mathcal{M}}(\mathbb{C})}\mathfrak{g}^\circ(\tilde{T},\alpha^{-1}\msf{y}{^t\alpha^{-1}})=(16\pi^2)^{-1}\ct\int_{1}^{+\infty}\frac{\textup{d}}{\textup{d}t}\left(\left(C_{\tilde{T}}^{3}\right)^\prime\left(t\tilde{\msf{y}},0\right)\right)\frac{\textup{d}t}{t^{\rho}}.
    \end{align}
    By the Definition \ref{green-currents-T2} of $\mathfrak{g}(T,\msf{y})$,
    combining (\ref{diff-3-1}) with \eqref{int-deg3-simplify} and \eqref{integral-rank-ct},
    and using Lemma \ref{lemma:rk1-deg}
    to evaluate $\deg_{\mc{E}^{\vee}_{\Q}}(\mc{Z}(T)_{\Q})$,
    we find
    \begin{align}\label{int-diff-3-1}
        &\frac{1}{2}\int\limits_{\overline{\mathcal{M}}(\mathbb{C})}\mathfrak{g}(T,\msf{y})-\int\limits_{\mathcal{M}(\mathbb{C})}\mathfrak{g}(T^\flat,\msf{y}^\flat)\wedge\Omega^{2}\\=&\frac{1}{2}\int\limits_{\overline{\mathcal{M}}(\mathbb{C})}\mathfrak{g}^\circ(T,\msf{y})-\int\limits_{\mathcal{M}(\mathbb{C})}\mathfrak{g}(T^\flat,\msf{y}^\flat)\wedge\Omega^{2}-\frac{1}{2}\log\det\msf{y}_1\cdot\deg_{\mc{E}^{\vee}_{\mathbb{Q}}}(\mathcal{Z}(T)_{\mathbb{Q}})\notag\\=&-(32\pi^2)^{-1}\frac{\partial f_{\tilde{T}}}{\partial s}(\tilde{\msf{y}},0)-(192\pi)^{-1}\log\det\msf{y}_1\cdot C_{T^\flat}^{1}(\msf{y}^\flat,1),\notag
    \end{align}
    where 
    \begin{align}
       -\frac{\partial f_{\tilde{T}}}{\partial s}(\tilde{\msf{y}},0)=-\left(C_{\tilde{T}}^{3}\right)^\prime(\tilde{\msf{y}},0)+\left(\frac{\pi}{6}\log\det\msf{y}_1+\frac{\pi}{3}+4\Lambda^{\prime}(2)\right)C_{T^\flat}^{1}(\msf{y}^\flat,1)+\frac{\pi}{3}\left(C_{T^\flat}^1\right)^\prime(\msf{y}^\flat,1)
    \end{align}
    as in \eqref{equation:Eisenstein:coeff_corank_1:deriv}.
    By the definition of the arithmetic degree in \eqref{def: arithmetic degree},
    we have
    \begin{equation*}
        \widehat{\deg}[\widehat{\mc{Z}}(T,\msf{y})]=\widehat{\deg}\left(\left(\widehat{\mc{E}}^{\vee}\right)^{2}\big|_{\mc{Z}(T^\flat)}\right)+\frac{1}{2}\int\limits_{\overline{\mathcal{M}}(\mathbb{C})}\mathfrak{g}(T,\msf{y}).
    \end{equation*}
    Combining (\ref{int-diff-3-1}) and (\ref{refine-kuhn}), we get
    \begin{equation*}
        \widehat{\deg}[\widehat{\mc{Z}}(T,\msf{y})]=-(32\pi^{2})^{-1}\left(C_{\tilde{T}}^3\right)^\prime(\tilde{\msf{y}},0)\overset{(\ref{simplify-C_T})}{=}-(32\pi^{2})^{-1}\left(C_T^3\right)^\prime(\msf{y},0).
    \end{equation*}
\end{proof}
\subsection{The case of corank 3}
Let $T=\boldsymbol{0}_3$. The special cycle
is $\mc{Z}(\boldsymbol{0}_3)=\mc{M}$.
It is proved in \cite[Theorem 7.61]{Burgos2007} that
\begin{equation*}
    \widehat{\deg}\left(\left(\widehat{\mc{E}}_{\overline{M},\textup{Pet}}^{\vee}\right)^{3}|_{\overline{M}}\right)=-\frac{1}{2}\left(\frac{1}{2}\zeta(-1)+\zeta^{\prime}(-1)\right)=-(4\pi)^{-1}\Lambda^\prime(2)-48^{-1}+48^{-1}\log(4\pi e^\gamma).
\end{equation*}
By (\ref{relation of metrics}) and \cite[Proposition 3.2.2]{Bostetal1994},
we have
\begin{align}\label{height-rk0}
    \widehat{\deg}\left(\left(\widehat{\mc{E}}_{\overline{M}}^{\vee}\right)^{3}|_{\overline{M}}\right)&=\widehat{\deg}\left(\left(\widehat{\mc{E}}_{\overline{M},\textup{Pet}}^{\vee}\right)^{3}|_{\overline{M}}\right)-\frac{3}{2}\log(4\pi e^\gamma)\cdot\deg_{\mc{E}^\vee_\mathbb{Q}}\left(\overline{M}_\mathbb{Q}\right)\\&=-(4\pi)^{-1}\Lambda^\prime(2)-48^{-1}.\notag
\end{align}
\begin{theorem}
 For all $\msf{y}\in\textup{Sym}_3(\mathbb{R})_{>0}$,
 we have
    \begin{equation}
        \widehat{\deg}[\widehat{\mc{Z}}(\boldsymbol{0}_3,\msf{y})]=-(32\pi^{2})^{-1}\left(C_{\boldsymbol{0}_3}^3\right)^\prime(\msf{y},0).
    \end{equation}
\end{theorem}
\begin{proof}
By Lemma \ref{integral-3} and Lemma \ref{rank=0-CT},
particularly \eqref{diff-3-0} and \eqref{rk=0-ct}, we have
    \begin{align}
        \frac{1}{2}\int\limits_{\overline{\mathcal{M}}(\mathbb{C})}\mathfrak{g}^{\circ}(\boldsymbol{0}_3,\msf{y})=-(32\pi^{2})^{-1}\left(C_{\boldsymbol{0}_3}^{3}\right)^\prime(\msf{y},0)+(16\pi)^{-1}\Lambda^\prime(2)+192^{-1}+576^{-1}\log\det\msf{y}.
    \end{align}
    By the Definition \ref{green-currents-T2} of $\mathfrak{g}(\boldsymbol{0}_3,\msf{y})$,
    we have
    \begin{align}\label{integral-part-rk0}
        \frac{1}{2}\int\limits_{\overline{\mathcal{M}}(\mathbb{C})}\mathfrak{g}(\boldsymbol{0}_3,\msf{y})&=\frac{1}{2}\int\limits_{\overline{\mathcal{M}}(\mathbb{C})}\mathfrak{g}^{\circ}(\boldsymbol{0}_3,\msf{y})-\frac{1}{2}\log\det\msf{y}\cdot\deg_{\mc{E}^{\vee}_{\mathbb{Q}}}(\mathcal{M}_{\mathbb{Q}})\\
        &=-(32\pi^{2})^{-1}\left(C_{\boldsymbol{0}_3}^{3}\right)^\prime(\msf{y},0)+(16\pi)^{-1}\Lambda^\prime(2)+192^{-1}.\notag
    \end{align}
We have
    \begin{align}\label{ht-part-rk0}
        \widehat{\deg}\left(\left(\widehat{\mc{E}}^{\vee}\right)^{3}\big|_{\mc{M}}\right)\overset{\eqref{height-rank0-to-hecke}}{=}4^{-1}\widehat{\deg}\left(\left(\widehat{\mc{E}}_{\overline{M}}^{\vee}\right)^{3}|_{\overline{M}}\right)\overset{(\ref{height-rk0})}{=}-(16\pi)^{-1}\Lambda^\prime(2)-192^{-1}.
    \end{align}
    The theorem now follows from the definition of $\widehat{\deg}[\widehat{\mc{Z}}(\boldsymbol{0}_3,\msf{y})]$
    in \eqref{def: arithmetic degree},
    which is the sum of the two terms \eqref{integral-part-rk0}
    and \eqref{ht-part-rk0}.
\end{proof}

\clearpage

\appendix

\section{Finiteness and boundedness}
In this section, we prove the finiteness and boundedness of some sets arising naturally in $\S$\ref{ssec:Green_current_def} when we define the green currents $\mf{g}^{\circ}(T,\msf{y},\rho)$.

\begin{lemma}\label{finite-bound}
    Let $U\subset\mathbb{D}\simeq\mathbb{H}_1^2$ be an open set with compact closure. Let $L\subset V$ be a sublattice of rank $4$. Let $T\in\textup{Sym}_r(\mathbb{Q})$ be a matrix. The following set is finite for all positive integers $r$:
    \begin{equation}
        S=\{\boldsymbol{x}\in L^{r}: T(\boldsymbol{x})=T\,\,\textup{and}\,\,\mathbb{D}_{\boldsymbol{x}}\cap \overline{U}\neq\varnothing\}.
    \end{equation}
    Moreover, there exists a positive number $C>0$, such that for all elements $\boldsymbol{x}=(x_1,\cdots,x_r)\in \Omega_T(L)\backslash S$, we have
    \begin{equation}
        R(\boldsymbol{x},z)=\sum\limits_{i=1}^rR(x_i,z)>C,\,\,z\in\overline{U}.
    \end{equation}
\end{lemma}
\begin{proof}
    We first prove the claim about the finiteness of the set $S$. Let $\boldsymbol{x}=(x_1,\cdots,x_r)$, the condition $\mathbb{D}_{\boldsymbol{x}}\cap \overline{U}\neq\varnothing$ is
    \begin{equation*}
        \mathbb{D}_{x_1}\cap\cdots\cap\mathbb{D}_{x_r}\cap\overline{U}\neq\varnothing.
    \end{equation*}
    Therefore we can reduce the lemma to the case that $r=1$, i.e., we prove that for all numbers $t\in\mathbb{Q}$, the following set is finite
    \begin{equation*}
        S=\{x\in L: \det(x)=t\,\,\textup{and}\,\,\mathbb{D}_{x}\cap \overline{U}\neq\varnothing\}.
    \end{equation*}
    For $t<0$, the set $S$ is empty. For $t=0$, the set $S$ consists of the single element $\boldsymbol{0}_2$, hence it is finite.
    
    Now we consider the case that $t>0$. By the compactness of $\overline{U}$, we can assume that there exists a number $M>1$ such that for all points $z=(x_1+iy_1,x_2+iy_2)\in\overline{U}$,
    \begin{equation*}
        \vert x_i\vert\leq M,\,\,\,\, M^{-1}\leq\vert y_i\vert\leq M,\,\,\,\,i=1,2.
    \end{equation*}
    Let $x=\left(\begin{smallmatrix}
        a & b\\
        c & d
    \end{smallmatrix}\right)\in L$ be an element in $S$. If $\mathbb{D}_x\cap\overline{U}\neq\varnothing$, we conclude that there exists an element $(z_1,z_2)\in \overline{U}$ such that
    \begin{equation*}
        dz_1z_2-cz_1-bz_2+a=0
    \end{equation*}
    by (\ref{equationDx}). Therefore
    \begin{equation*}
        z_2=\frac{a-cz_1}{b-dz_1},\,\,\,\,z_1=\frac{a-bz_2}{c-dz_2}.
    \end{equation*}
    Let $z_1=x_1+iy_1$ and $z_2=x_2+iy_2$, we get 
    \begin{equation}\label{basic-ineq}
        y_2=\frac{t\cdot y_1}{\vert b-dz_1\vert^2}\leq\frac{t}{d^{2}y_1+y_1^{-1}(b-dx_1)^{2}}.
    \end{equation}
    Therefore $\vert d\vert\leq\sqrt{t(y_1y_2)^{-1}}\leq M\sqrt{t}$. We can also use the fact that $\vert b-dz_1\vert^2\geq (b-dx_1)^{2}$ to deduce that
    \begin{equation*}
        \vert b\vert\leq\sqrt{ty_1y_2^{-1}}+\vert dx_1\vert\leq M\sqrt{t}+M^{2}\sqrt{t}\leq 2M^{2}\sqrt{t}.
    \end{equation*}
    We can get the same estimate for $c$ by using the expression $z_1=\frac{a-bz_2}{c-dz_2}$. Finally,
    \begin{equation*}
        \vert a\vert=\vert dz_1z_2-cz_1-bz_2\vert\leq\left(4\sqrt{2}+2\right)M^{3}\sqrt{t}.
    \end{equation*}
    Therefore all the coordinates of $x$ are bounded by some numbers depending only on the set $U$. Since $x$ belongs to the $\mathbb{Z}$-lattice $L$, there are only finitely many $x$ satisfying this condition.
    \par
    Now we prove the claim about the set $\Omega_T(L)\backslash S$. Notice that the quadratic form $x\in V\mapsto R(x,z)+\det(x)$ is positive definite. By the compactness of $\overline{U}$, there exists a positive number $c>0$ such that
    \begin{equation*}
        R(x,z)+\det(x)>c\cdot\|x\|^2,\,\,\,\,\textup{for}\,\,x\in L\,\,\textup{and}\,\,z\in \overline{U}.
    \end{equation*}
    The set $\Omega_T(L)\backslash S$ can be divided into the following parts:
    \begin{align*}
        S^\pr\coloneqq\{\boldsymbol{x}\in L^r: \sum\limits_{i=1}^r\|x_i\|^2\leq c^{-1}(\vert \mrm{tr}(T)\vert +1)\}\cap \Omega_T(L)\backslash S,\\
        S^{\pr\pr}\coloneqq\{\boldsymbol{x}\in L^r: \sum\limits_{i=1}^r\|x_i\|^2> c^{-1}(\vert \mrm{tr}(T)\vert +1)\}\cap \Omega_T(L)\backslash S.
    \end{align*}
    For an element $\boldsymbol{x}\in S^{\pr\pr}$, we have
    \begin{equation*}
        R(\boldsymbol{x},z)>\vert \mrm{tr}(T)\vert+1-\mrm{tr}(T)\geq1.
    \end{equation*}
    The set $S^{\pr}$ is finite, therefore the number $C^{\pr}\coloneqq\min\limits_{\boldsymbol{x}\in S^\pr, z\in \overline{U}}\left\{R(\boldsymbol{x},z)\right\}$ is positive because $\overline{U}$ is compact. Define
    \begin{equation*}
        C=\min\left\{C^\pr/2,1/2\right\}.
    \end{equation*}
    Then we have $R(\boldsymbol{x},z)>C$.
\end{proof}

\section{Calculations of local densities and the case \texorpdfstring{$p=2$}{p=2}}
    \label{sec:local_den_and_p=2}
\subsection{Gross--Keating invariants}
We first recall the definition of Gross--Keating invariant from \cite[$\S$4]{GK93} and \cite[$\S$1]{Bouw2007Invariants}.
Let $p$ be a prime number. Let $(L,q)$ be a quadratic $\zp$-lattice of rank $n\geq1$. 
We always assume that the quadratic form is nondegenerate. Let $\boldsymbol{e}=(e_1,\cdots,e_n)$ be a basis of $L$ such that
\begin{equation}\label{basis-q}
    q(x_1e_1+\cdots+x_ne_n)=\sum_{i\leq j}b_{ij}(\boldsymbol{e})x_ix_j.
\end{equation}
We write $S(\boldsymbol{e})$ for the set of tuples $\boldsymbol{y} = (y_1,\cdots ,y_n) \in \mathbb{Z}^n$ such that
\begin{equation}\label{GK-condition}
    y_1\leq\cdots\leq y_n,\,\,\,\frac{y_i+y_j}{2}\leq\nu_p\left(b_{ij}(\boldsymbol{e})\right)\,\,\textup{for}\,\,1\leq i\leq j\leq n.
\end{equation}
Let $S(L) = \bigcup S(\boldsymbol{e})$ where the union ranges over all basis $\boldsymbol{e}$. We order tuples $(y_1,\cdots, y_n) \in S$ lexicographically. The Gross--Keating invariants $\textup{GK}(L)=(a_1,\cdots,a_n)$ are the maximum of $(y_1,\cdots, y_n) \in S(L)$. A basis $\boldsymbol{e}$ is called optimal if $(a_1,\cdots,a_n) \in S(\boldsymbol{e})$.
\begin{lemma}\label{gk-inv-rk2}
Let $(L,q)$ be a nondegenerate quadratic $\zp$-lattice of rank $n\geq1$ with $\textup{GK}(L)=(\mu_1,\cdots,\mu_n)$,
    \begin{itemize}
    \item [(a).]  We have
    \begin{equation*}
        \mu_1=\min\limits_{x\in L}\{\nu_p(q(x))\}.
    \end{equation*}
        \item [(b).] If $p$ is odd, the quadratic lattice $L$ is diagonalizable, i.e., there exists a basis $\boldsymbol{e}=(e_1,\cdots,e_n)$ such that $(e_i,e_j)=0$ for $i \neq j$,
            and $\nu_p(q(e_i))\leq\nu_p(q(e_j))$ if $1\leq i\leq j\leq n$. Then we have
        \begin{equation*}
           \mu_i=\nu_p(q(e_i)),\,\,\,1\leq i\leq n.
        \end{equation*}
        \item[(c).] If $p=n=2$ and the lattice $L$ is not diagonalizable, then there exists an integer $e$ and a basis $\boldsymbol{e}=(e_1,e_2)$ such that
        \begin{equation*}
            q(xe_1+ye_2)=2^{e}\cdot\begin{cases}
                x^{2}+xy+y^{2}, &\textup{if $L$ is anisotropic;}\\
                xy, &\textup{if $L$ is isotropic.}
            \end{cases}
        \end{equation*}
        The Gross--Keating invariant of $L$ is $\textup{GK}(L)=(e,e)$.
        \item[(d).] If $p=n=2$ and the lattice $L$ is diagonalizable, then there exists a basis $\boldsymbol{e}=(e_1,e_2)$ such that
        \begin{equation*}
            q(xe_1+ye_2)=u_12^{a_1}x^{2}+u_22^{a_2}y^{2},
        \end{equation*}
        where $u_1,u_2\in\mathbb{Z}_2^{\times}$, $a_1,a_2\in\mathbb{Z}$ and $a_1\leq a_2$. The Gross--Keating invariant of $L$ is given by
        \begin{equation}\label{gk-2-diag}
            \textup{GK}(L)=\begin{cases}
                (a_1,a_2), &\textup{if $a_1\nequiv a_2$ mod $2$;}\\
                (a_1,a_2+1), &\textup{if $a_1\equiv a_2$ mod $2$ and $u_1+u_2\equiv 2$ mod $4$;}\\
                (a_1,a_2+2), &\textup{if $a_1\equiv a_2$ mod $2$ and $u_1+u_2\equiv 0$ mod $4$.}
            \end{cases}
        \end{equation}
    \end{itemize}
\end{lemma}
\begin{proof}
For part $(a)$, let $\boldsymbol{f}=(f_1,\cdots,f_n)$ be an optimal basis of $L$. It's obvious that $\min\limits_{i\leq j}\nu_p(b_{ij}(\boldsymbol{f}))=\min\limits_{x\in L}\{\nu_p(q(x))\}$ because
\begin{equation*}
    b_{ij}(\boldsymbol{f})=\begin{cases}
        q(f_i+f_j)-q(f_i)-q(f_j), &\textup{if $i\neq j$;}\\
        q(f_i), &\textup{if $i=j$.}
    \end{cases}
\end{equation*}
Hence $\min\limits_{i\leq j}\nu_p(b_{ij}(\boldsymbol{f}))\geq\min\limits_{x\in L}\{\nu_p(q(x))\}$. The inequality of the other direction $\min\limits_{i\leq j}\nu_p(b_{ij}(\boldsymbol{f}))\leq\min\limits_{x\in L}\{\nu_p(q(x))\}$ follows from (\ref{basis-q}). If $\mu_1>\min\limits_{x\in L}\{\nu_p(q(x))\}$, let $\leq i\leq j\leq n$ be two integers such that $\nu_p(b_{ij}(\boldsymbol{f}))=\min\limits_{x\in L}\{\nu_p(q(x))\}$, then
\begin{equation*}
    \nu_p(b_{ij}(\boldsymbol{f}))\geq\frac{\mu_i+\mu_j}{2}\geq \mu_1>\min\limits_{x\in L}\{\nu_p(q(x))\},
\end{equation*}
which is a contradiction, hence $\mu_1\leq \min\limits_{x\in L}\{\nu_p(q(x))\}$. It's easy to see that the following element belongs to $S(\boldsymbol{f})$:
\begin{equation*}
    \left(\min\limits_{x\in L}\{\nu_p(q(x))\},\cdots,\min\limits_{x\in L}\{\nu_p(q(x))\}\right),
\end{equation*}
therefore $\mu_1\geq \min\limits_{x\in L}\{\nu_p(q(x))\}$. Then we must have $\mu_1=\min\limits_{x\in L}\{\nu_p(q(x))\}$.
\par
    Part $(b)$ is well-known (see for example \cite[$\S$2]{Bouw2007Invariants}), we are left to prove $(c)$ and $(d)$.
    For $(c)$, the existence of the basis can be found in \cite[Lemma 3.2]{Bouw2007Invariants}. Now we compute $\textup{GK}(L)$. Suppose that $\textup{GK}(L)=(\mu_1,\mu_2)$ where $\mu_1\leq \mu_2$. By $(a)$, we can verify that $\mu_1=e$. It's easy to see that $(e,e)\in S(L)$, hence $\mu_2\geq e$. Moreover, we have $\mu_1+\mu_2\leq 2e$ by \cite[Lemma 1.3(c)]{Bouw2007Invariants}, hence $\mu_2\leq e$. Therefore $\mu_2=e$.
 \par
 For $(d)$, part $(a)$ implies that $\mu_1=a_1$. It's easy to see that $(a_1,a_2)\in S(L)$, hence $\mu_2\geq a_2$. Let $\boldsymbol{f}=(f_1,f_2)$ be an optimal basis such that $[f_1,f_2]=[e_1,e_2]\cdot\left(\begin{smallmatrix}
         a & b\\
         c & d
     \end{smallmatrix}\right)$. Then
 \begin{equation*}
     q(xf_1+yf_2)=\left(a^{2}u_12^{a_1}+c^2u_22^{a_2}\right)x^{2}+\left(abu_12^{a_1+1}+cdu_22^{a_2+1}\right)xy+\left(b^{2}u_12^{a_1}+d^2u_22^{a_2}\right)y^{2}.
 \end{equation*}
 \par
 We first consider the case that $a_1\nequiv a_2$ mod $2$. Let $a_2=a_1+2\lambda-1$ where $\lambda\geq1$. Since $\mu_2\geq a_2$, we have $\nu_2\left(b^{2}u_12^{a_1}+d^2u_22^{a_2}\right)\geq \mu_2\geq a_2$, hence $\nu_2(b^{2}u_12^{a_1})\geq a_2$, which implies that $\nu_2(b)\geq\lambda$. Therefore $d\in\mathbb{Z}_2^\times$ since $\left(\begin{smallmatrix}
         a & b\\
         c & d
     \end{smallmatrix}\right)\in\textup{GL}_2(\mathbb{Z}_2)$, and $\nu_2\left(b^{2}u_12^{a_1}+d^2u_22^{a_2}\right)=a_2$. Since $\boldsymbol{f}$ is an optimal basis, we have $\mu_2\leq\nu_2\left(b^{2}u_12^{a_1}+d^2u_22^{a_2}\right)=a_2$. Therefore $\mu_2=a_2$.
 \par
 Next we consider the case that $a_1\equiv a_2$ mod $2$. Let $a_2=a_1+2\lambda$ where $\lambda\geq0$. By \cite[Lemma 1.3(c)]{Bouw2007Invariants},
 \begin{equation*}
     \mu_1+\mu_2\leq \nu_2\left(u_1u_22^{a_1+a_2+2}\right)=a_1+a_2+2,
 \end{equation*}
 hence $\mu_2\leq a_2+2$.
 \begin{itemize}
     \item If $u_1+u_2\equiv 0\,\,\textup{mod}\,\,4$. Take $a=1, b=2^{\lambda}, c=0$ and $d=1$, we get
     \begin{equation*}
     q(xf_1+yf_2)=u_12^{a_1}x^{2}+u_12^{a_1+1+\lambda}xy+\left(u_1+u_2\right)2^{a_2}y^{2}.
 \end{equation*}
 It's easy to verify that $(a_1,a_2+2)\in S(\boldsymbol{f})$ since $\nu_2(\left(u_1+u_2\right)2^{a_2})\geq a_2+2$ and $\nu_2\left(u_12^{a_1+1+\lambda}\right)=a_1+1+\lambda=(a_1+a_2+2)/2$. Therefore $\mu_2=a_2+2$.
 \item If $u_1+u_2\equiv 2\,\,\textup{mod}\,\,4$. We first prove that $\mu_2\leq a_2+1$. Let's assume the converse that $\mu_2\geq a_2+2$. Since $\nu_2\left(b^{2}u_12^{a_1}+d^2u_22^{a_2}\right)\geq \mu_2\geq a_2+2$, the element $b$ satisfies that $\nu_2(b)\geq\lambda$ because otherwise $\nu_2\left(b^{2}u_12^{a_1}+d^2u_22^{a_2}\right)<a_2$, which is a contradiction. If $\nu_2(b)>\lambda$, the element $d\notin\mathbb{Z}_2^{\times}$ because otherwise $\nu_2\left(b^{2}u_12^{a_1}+d^2u_22^{a_2}\right)=a_2$, which is also a contradiction. Therefore $\nu_2(b)=\lambda$ and $d\in\mathbb{Z}_2^{\times}$. Let $b=2^{\lambda}\cdot b_0$, then
 \begin{equation*}
     b^{2}u_12^{a_1}+d^2u_22^{a_2}=2^{a_2}\cdot\left(u_1b_0^{2}+u_2d^{2}\right)\equiv 2^{a_2}\cdot\left(u_1+u_2\right)\,\,\textup{mod}\,\,2^{a_2+2}.
 \end{equation*}
 Therefore $\mu_2\leq a_2+1$, which is a contradiction, hence $\mu_2\leq a_2+1$. Next we take $a=1, b=2^{\lambda}, c=0$ and $d=1$, then
     \begin{equation*}
     q(xf_1+yf_2)=u_12^{a_1}x^{2}+u_12^{a_1+1+\lambda}xy+\left(u_1+u_2\right)2^{a_2}y^{2}.
 \end{equation*}
 It's easy to verify that $(a_1,a_2+1)\in S(\boldsymbol{f})$ since $\nu_2(\left(u_1+u_2\right)2^{a_2})= a_2+1$ and $\nu_2\left(u_12^{a_1+1+\lambda}\right)=a_1+1+\lambda>(a_1+a_2+1)/2$. Therefore $\mu_2=a_2+1$.
 \end{itemize}
\end{proof}
For a nondegenerate quadratic $\zp$-lattice of rank $n$ with $\textup{GK}(L)=(\mu_1,\cdots,\mu_n)$, define
\begin{equation}
    \vert\textup{GK}(L)\vert=\mu_1+\cdots+\mu_n.
\end{equation}
Let $\boldsymbol{e}=(e_1,\cdots,e_n)$ be a $\zp$-basis of $L$, define the inner product matrix of $L$ with respect to the basis $\boldsymbol{e}$ to be $T(\boldsymbol{e})\coloneqq\left(\frac{1}{2}(e_i,e_j)\right)_{1\leq i, j\leq n}$. Define $\det L=\det T(\boldsymbol{e})\in\mathbb{Q}_p^{\times}/\left(\mathbb{Z}_p^{\times}\right)^{2}$, this number is independent of the choice of the basis $\boldsymbol{e}$.
\begin{lemma}\label{gk-dep-V}
    Let $L$ be a nondegenerate quadratic $\mathbb{Z}_p$-lattice of rank $n$. 
    \begin{itemize}
        \item [(a).] If $p>2$, we have
        \begin{equation}\label{gk-det->2}
            \vert\textup{GK}(L)\vert=\nu_p(\det L).
        \end{equation}
        \item [(b).] If $p=n=2$, denote by $\mathbb{Q}_{L,2}$ the degree $2$ algebra $\mathbb{Q}_{2}\left(\sqrt{\textup{disc}(L)}\right)$ over $\mathbb{Q}_2$. Then
    \begin{equation}\label{gk-det}
        \vert\textup{GK}(L)\vert = 2\cdot\left[\frac{\nu_2(\det L)}{2}\right]+\delta_{L,2},
    \end{equation}
    where $\delta_{L,2}$ is $1$ if $\mathbb{Q}_{L,2}$ is a ramified extension of $\mathbb{Q}_2$, is $2$ otherwise.
    \end{itemize}
\end{lemma}
\begin{proof}
    Part $(a)$ follows from Lemma \ref{gk-inv-rk2} $(b)$. Now we focus on the case that $p=n=2$, if the lattice $L$ is not diagonalizable, we can verify the identity easily by using Lemma \ref{gk-inv-rk2} $(c)$. Now let's assume that the lattice $L$ is diagonalizable with a diagonalized basis as in Lemma \ref{gk-inv-rk2} $(d)$. If $a_1\nequiv a_2$ mod $2$, the algebra $\mathbb{Q}_{L,2}$ is a ramified extension of $\mathbb{Q}_2$, the identity (\ref{gk-det}) can be verified by straightforward computation. 
    \par
    If $a_1\equiv a_2$ mod $2$ with $u_1+u_2\equiv 2$  mod $4$, the mod $8$ reduction of $(u_1,u_2)$ have the following possibilities:
    \begin{align*}
        (u_1,u_2)\equiv (1,1),\,\, (1,5),\,\, (3,3),\,\, (3,7),\,\, (5,1),\,\, (5,5),\,\, (7,3),\,\, (7,7)\,\,\,\textup{mod}\,\,8,
    \end{align*}
    then $\disc\, L\equiv -1, -5$ mod $8$, the corresponding algebra $\mathbb{Q}_{L,2}$ is a ramified extension of $\mathbb{Q}_2$. The identity (\ref{gk-det}) is true because both sides are $a_1+a_2+1$.
    \par
    If $a_1\equiv a_2$ mod $2$ with $u_1+u_2\equiv 0$  mod $4$, the mod $8$ reduction of $(u_1,u_2)$ have the following possibilities:
    \begin{align*}
        (u_1,u_2)\equiv (1,3),\,\, (1,7),\,\,(3,1),\,\, (3,5),\,\, (5,3),\,\, (5,7),\,\, (7,1),\,\, (7,5)\,\,\,\textup{mod}\,\,8,
    \end{align*}
    then $\disc\, L\equiv -3, 1$ mod $8$, the corresponding algebra $\mathbb{Q}_{L,2}$ is either an unramified quadratic extension of $\mathbb{Q}_2$ ($\disc\, L\equiv -3$ mod $8$), or splits $\mathbb{Q}_2\times\mathbb{Q}_2$ ($\disc\, L\equiv 1$ mod $8$). The identity (\ref{gk-det}) is true because both sides are $a_1+a_2+2$.
\end{proof}
\begin{remark}\label{Q_L}
    By the proof of Lemma \ref{gk-dep-V}, we can see that the $\mathbb{Q}_p$-algebra $\mathbb{Q}_{L,p}\coloneqq\mathbb{Q}_{p}\left(\sqrt{\disc\,L}\right)$ is a ramified extension of $\mathbb{Q}_p$ if and only if $a\nequiv b$ mod $2$. Otherwise, the algebra $\mathbb{Q}_{L,p}$ is either an unramified quadratic extension of $\mathbb{Q}_p$, or $\mathbb{Q}_p\times\mathbb{Q}_p$.
\end{remark}

\subsection{Local densities}
The goal of this subsection is to compute $\den^{\star}(X,L)$ defined in \eqref{equ: denstar},
in the case of $F = \Q_p$, by combining the difference formula in \cite{Zhu23}, \cite{zhusingular} and explicit calculations in \cite{Wedhorn}.
\begin{lemma}\label{diff-locden}
    Let $L$ be a nondegenerate integral quadratic $\zp$-lattice of rank $2$. For an element $\varepsilon\in\mathbb{Z}_p^{\times}$ and sufficiently large integers $n$,
    \begin{align*}
        \den\left(H^{k+2},L\obot\langle\varepsilon p^{n}\rangle\right)-&p^{-2k}\cdot\den\left(H^{k+2},L\obot\langle\varepsilon p^{n-2}\rangle\right)\\&=\left(1-p^{-k-1}\right)\left(1+p^{-k}\right)\cdot\den\left(H^{k+1},L\right),\,\,\textup{for all}\,\,k\geq0.
    \end{align*}
\end{lemma}
\begin{proof}
    By \cite[Theorem 7.2.1, Remark 7.2.4]{Zhu23}, for all $k\geq0$ and $n\geq2$,
    \begin{align*}
        &\den\left(H^{k+2},L\obot\langle\varepsilon p^{n}\rangle\right)-p^{-2k}\cdot\den\left(H^{k+2},L\obot\langle\varepsilon p^{n-2}\rangle\right)\\=&\textup{Pden}\left(H^{k+1},\langle\varepsilon p^{n}\rangle\right)\cdot\den\left(H^{k+1}\obot\langle-\varepsilon p^{n}\rangle,L\right)\\=&\left(1-p^{-k-1}\right)\left(1+p^{-k}\right)\cdot\den\left(H^{k+1}\obot\langle-\varepsilon p^{n}\rangle,L\right).
    \end{align*}
    By \cite[Lemma 2.2.4]{zhusingular}, if the integer $n$ is larger than $\abs{\mrm{GK}(L)}$, we have $\den\left(H^{k+1}\obot\langle-\varepsilon p^{n}\rangle,L\right)=\den\left(H^{k+1},L\right)$. Therefore the lemma follows.
\end{proof}
\begin{lemma}\label{explicit-local-den-poly}
    Let $L$ be a nondegenerate integral quadratic $\zp$-lattice of rank $2$ with $\textup{GK}(L)=(a,b)$.
    \begin{itemize}
        \item If $a\nequiv b$ mod $2$,
        \begin{equation}\label{anequivb-poly}
            \den^{\star}(X,L)=X^{-(a+b-1)/2}\cdot\sum\limits_{i=0}^{a}\sum\limits_{j=0}^{(a+b-1)/2-i}p^{i+j}X^{i+2j}.
        \end{equation}
        \item If $a\equiv b$ mod $2$ and $L$ is anisotropic,
        \begin{equation}\label{aequivb-an-poly}
            \den^{\star}(X,L)=X^{-(a+b)/2}\cdot\left(\left(1+X\right)\sum\limits_{i=0}^{a}\sum\limits_{j=0}^{(a+b)/2-i-1}p^{i+j}X^{i+2j}+p^{(a+b)/2}X^{b}\cdot\sum\limits_{i=0}^{a}X^{i}\right).
        \end{equation}
        \item If $a\equiv b$ mod $2$ and $L$ is isotropic,
        \begin{equation}\label{aequivb-iso-poly}
            \den^{\star}(X,L)=X^{-(a+b)/2}\cdot\left(\left(1-X\right)\sum\limits_{i=0}^{a}\sum\limits_{j=0}^{(a+b)/2-i-1}p^{i+j}X^{i+2j}+p^{(a+b)/2}X^{b}\cdot\sum\limits_{i=0}^{a}X^{i}\right).
        \end{equation}
    \end{itemize}
\end{lemma}
\begin{proof}
    We only give the detailed proof for the identity (\ref{aequivb-iso-poly}), the other two follow from the same idea. For an element $\varepsilon\in\mathbb{Z}_p^{\times}$ and an integer $n\geq b+4$, Wedhorn \cite[(2.8)]{Wedhorn} gives a precise formula for the polynomial $\den^{\star}\left(X,L\obot\langle\varepsilon p^{n}\rangle\right)$. Notice that the symbol $\tilde{\xi}=1$ by $\S$2.5 ($p>2$), $\S$2.6 (A2) and (B4) ($p=2$) of loc. cit.. Then
    \begin{align}
        &\frac{\den\left(X,L\obot\langle\varepsilon p^{n}\rangle\right)}{\left(1-p^{-2}X\right)\left(1-p^{-2}X^{2}\right)}-X^{2}\cdot\frac{\den\left(X,L\obot\langle\varepsilon p^{n-2}\rangle\right)}{\left(1-p^{-2}X\right)\left(1-p^{-2}X^{2}\right)}\label{diff-poly}\\
        =&\left(1-X^{2}\right)\sum\limits_{i=0}^{a}\sum\limits_{j=0}^{(a+b)/2-i-1}p^{i+j}X^{i+2j}+p^{(a+b)/2}X^{b}(1+X)\cdot\sum\limits_{i=0}^{a}X^{i}\notag
    \end{align}
    The number $\chi_p(L)=1$ in this case, therefore (\ref{aequivb-iso-poly}) follows from (\ref{diff-poly}) and Lemma \ref{diff-locden}.
\end{proof}

\subsection{Choice of a larger lattice}\label{larger-lattice}
Let $L$ be a nondegenerate quadratic $\zp$-lattice of rank $2$ with $\textup{GK}(L)=(a,b)$. We are going to construct a lattice $L^\pr$ such that 
\begin{itemize}
    \item [(a).]$L\subsetneq L^{\prime}$ and $L^\pr/L\simeq\mathbb{F}_p$.
    \item [(b).] $\textup{GK}(L^{\pr})$ is either $(a,b-2)$, or $(b-2,a)$ if $b-2< a$.
\end{itemize}
The construction is nontrivial only when $p=2$.
\begin{itemize}
    \item [Case 1] We first single out a special case that the lattice $L$ is isotropic with $a=b$. Then there exists a basis $\boldsymbol{e}=(e_1,e_2)$ of $L$ such that $q\left(xe_1+ye_2\right)=p^{a}xy$. Define $L^\pr$ to be the $\zp$-span of $e_1$ and $p^{-1}(e_1+e_2)$, notice that
    \begin{equation*}
        q\left(xe_1+y\frac{e_1+e_2}{p}\right)=p^{a-2}\left(\left(\frac{p}{2}x+y\right)^{2}-\frac{p^{2}}{4}x^{2}\right),
    \end{equation*}
    hence $\textup{GK}(L^{\pr})=(a-2,a)$.
    \item [Case 2] If $p>2$ and $L$ is not isotropic with $a=b$, let $\boldsymbol{e}$ be a basis of $L$ as in Lemma \ref{gk-inv-rk2} $(b)$. Define $L^\pr$ to be the $\zp$-span of $e_1$ and $p^{-1}e_2$. Then $\textup{GK}(L^\pr)$ is either $(a,b-2)$ or $(b-2,a)$ if $b-2\leq a$.
    \item [Case 3] If $p=2$ and $L$ is not diagonalizable, we have $a=b$ by Lemma \ref{gk-inv-rk2} $(c)$. The isotropic case has been done before, let's consider the case that $L$ is anisotropic. Let $\boldsymbol{e}$ be a basis of $L$ as in Lemma \ref{gk-inv-rk2} $(c)$, define $L^\pr$ to be the $\mathbb{Z}_2$-span of $e_1$ and $2^{-1}e_2$. Then 
    \begin{align*}
        q\left(xe_1+y\cdot 2^{-1}e_2\right)=2^{a-2}\cdot\left(4x^{2}+2xy+y^{2}\right)=2^{a-2}\cdot\left(3x^{2}+\left(x+y\right)^{2}\right).
    \end{align*}
    Therefore $\textup{GK}(L^\pr)=(a-2,a)$ by Lemma \ref{gk-inv-rk2} $(d)$.
    \item [Case 4] If $p=2$ and $L$ is diagonalizable, let $\boldsymbol{e}$ be a basis of $L$ as in Lemma \ref{gk-inv-rk2} $(d)$. Then 
    \begin{align*}
        q\left(xe_1+y\cdot e_2\right)=u_12^{a}x^{2}+u_22^{b-\delta}y^{2},
    \end{align*}
    where $\delta=0,1,2$ according to (\ref{gk-2-diag}). Therefore 
    \begin{itemize}
        \item [4-0] For $\delta=0$, we have $a\leq b$ and $a\nequiv b$ mod $2$. Define $L^\pr$ to be the $\mathbb{Z}_2$-span of $e_1$ and $2^{-1}e_2$, then $\textup{GK}(L^{\pr})=(a,b-2)$ if $a\leq b-3$. If $a\geq b-1$, we must have $a=b-1$, then $\textup{GK}(L^{\pr})=(b-2,b-1)=(b-2,a)$.
        \item [4-1] For $\delta=1$, we have $a\leq b-1$ and $a\equiv b-1$ mod $2$. If $a\leq b-3$, define $L^\pr$ to be the $\mathbb{Z}_2$-span of $e_1$ and $2^{-1}e_2$, then $\textup{GK}(L^{\pr})=(a,b-2)$. If $a>b-3$, we must have $a=b-1$, define $L^\pr$ to be the $\mathbb{Z}_2$-span of $e_1$ and $2^{-1}(e_1+e_2)$, then
        \begin{align*}
        q\left(xe_1+y\cdot \frac{e_1+e_2}{2}\right)=2^{a-1}\cdot\left(\frac{2u_1u_2}{u_1+u_2}\cdot x^{2}+\frac{u_1+u_2}{2}\left(\frac{2u_1}{u_1+u_2}x+y\right)^{2}\right).
    \end{align*}
    Notice that
    \begin{equation*}
        \frac{2u_1u_2}{u_1+u_2}+\frac{u_1+u_2}{2}\equiv 2u_1\equiv 2\,\,\pmod{4},
    \end{equation*}
    hence $\textup{GK}(L^{\prime})=(a-1,a)=(b-2,a)$.
        \item [4-2] For $\delta=2$, we have $a\leq b-2$ and $a\equiv b$ mod $2$. If $a\leq b-4$, define $L^\pr$ to be the $\mathbb{Z}_2$-span of $e_1$ and $2^{-1}e_2$, then $\textup{GK}(L^{\pr})=(a,b-2)$. If $a>b-4$, we must have $a=b-2$, define $L^\pr$ to be the $\mathbb{Z}_2$-span of $e_1$ and $2^{-1}(e_1+e_2)$, then
        \begin{align}\label{mod-last-case}
        q\left(xe_1+y\cdot \frac{e_1+e_2}{2}\right)=2^{a}\cdot\left(u_1x^{2}+u_1xy+\frac{u_1+u_2}{4}\cdot y^{2}\right).
    \end{align}
    Now we prove that the quadratic form (\ref{mod-last-case}) is not diagonalizable. Let's assume the converse that there exists a matrix $\left(\begin{smallmatrix}
        a & b\\
        c & d
    \end{smallmatrix}\right)\in\textup{GL}_2(\mathbb{Z}_2)$ such that under the transformation $x\mapsto ax+by$, $y\mapsto cx+dy$, the quadratic form (\ref{mod-last-case}) is diagonalized. The coefficient of the term $xy$ after the transformation is $2abu_1+(ad+bc)u_1+\frac{u_1+u_2}{2}\cdot cd=0$, hence
    \begin{equation*}
        (ad-bc)u_1=-2abu_1-2bcu_1-\frac{u_1+u_2}{2}\cdot cd.
    \end{equation*}
    Notice that the $2$-adic valuation of the right hand side is at least $1$ since $u_1+u_2\equiv0$ mod $4$, while that of the left hand side is $0$, we get a contradiction. Therefore the quadratic form (\ref{mod-last-case}) is not diagonalizable, hence its Gross--Keating invariant is $(a,a)=(a,b-2)$ by Lemma \ref{gk-inv-rk2} $(c)$.
    \end{itemize}
\end{itemize}

\clearpage

\section{The absolute convergence of the genus 3 Green currents}
\label{sec append: genus 3 convergence}
    \subsection{Statement of the main result}

Let \(V_{\infty}=\textup{M}_2(\R)\) and \(L=\textup{M}_2(\Z)\).
For a matrix $T\in\textup{Sym}_3(\mathbb{Q})$
and $\mathsf{y}=a\cdot{^{t}a}\in\textup{Sym}_3(\mathbb{R})_{>0}$,
define the scalar kernel $K_T(t, \msf{y}, z)$ by the sum
\begin{equation}\label{eq:kernel}
 K_T(t,\mathsf{y},z)\cdot\Omega_z^2
 =e^{2\pi t \opn{tr} (T\mathsf{y})}\sum_{\substack{\bsm{x} \in L^3\\T(\boldsymbol{x})=T}}
                  \nu_3(\sqrt t\,\boldsymbol{x}a)_z,\qquad z\in\mathbb{D},\,\,t\ge1.
\end{equation}
Recall that the $(1,1)$-form $\Omega$ is defined by (\ref{first-chern}).
If $T$ is positive definite, we have $K_T(t, \msf{y}, z) = 0$, since $V$ having signature $(2,2)$
implies that the sum is empty.

For $\opn{Re}(\rho) > 0$,
recall that we defined a family of locally $L^1$ forms $\mathfrak{g}^\circ(T,\mathsf{y},\rho)$
on $\mathcal{M}(\mathbb{C})$ in \eqref{green-currents-T},
with continuation as a meromorphic family of currents to $\opn{Re}(\rho) > -1/2$,
as in Lemma \ref{sum-con-current}.
At least for $\opn{Re}(\rho) > 2$, the estimate explained
after \eqref{eq:term} allows interchange of the $\bsm{x}$-sum and the $t$-integral to give
the equivalent formula
\[
 \mathfrak{g}^\circ(T,\mathsf{y},\rho)
 =\left(\int_1^\infty K_T(t,\msf{y},z)t^{-1-\rho}\,dt\right)\Omega^2.
\]

This appendix concerns the extension of $\mathfrak{g}^\circ(T,\mathsf{y},\rho)$
to the compactification $\overline{\mc{M}}(\C)$.
The main result of the appendix is the following theorem.
We use the notation $z_j = x_j + i y_j$ for $z = (z_1, z_2) \in \mbb{D}$.

\begin{theorem}\label{thm:extension}
For \(T\in\Sym_3(\Q)\) and \(\msf{y} \in \Sym_{3}(\R)_{>0}\), the following statements hold.
\begin{enumerate}
\item[(a).] For $\opn{Re}(\rho) > 2$, the top-degree form $\mf{g}^{\circ}(T, \msf{y}, \rho)$
is $L^1$ on $\overline{\mathcal{M}}(\mathbb{C})$, and hence defines a current on $\overline{\mc{M}}(\C)$.
\item[(b).] The family $\mf{g}^{\circ}(T, \msf{y}, \rho)$ has a (unique) meromorphic continuation
to $\opn{Re}(\rho) > -1/2$, valued in currents on \(\overline{\mathcal{M}}(\mathbb{C})\).
The only possible pole is
a simple pole at \(0\). The constant term at $\rho = 0$, denoted
\begin{equation}
    \mathfrak{g}^\circ(T,\mathsf{y})\coloneqq \opn{\underset{\rho=0}{CT}} \mathfrak{g}^\circ(T,\mathsf{y},\rho),
\end{equation}
is also a current on $\overline{\mc{M}}(\C)$,
and is represented by an $L^1$ form on $\overline{\mc{M}}(\C)$ if $T$ has no positive eigenvalues.
Moreover, we have the following.
\begin{enumerate}
    \item [(b1).] If $T$ is not positive semi-definite, then the family $\mf{g}^{\circ}(T, \msf{y}, \rho)$
    extends holomorphically to a family of currents on $\overline{\mc{M}}(\C)$ for all $\rho \in \C$, with
    $\mf{g}^{\circ}(T, \msf{y}, \rho)$ given by an $L^1$ form on $\overline{\mc{M}}(\C)$ for any fixed $\rho$.
    \item [(b2).] If \(T=\boldsymbol{0}_3\), then the residue of $\mathfrak{g}^\circ(T,\mathsf{y},\rho)$ at $\rho=0$
    is \(3[\Omega^2]\), and its constant term $\mathfrak{g}^\circ(T,\mathsf{y})$
    is an \(L^1\) form on $\overline{\mc{M}}(\C)$ satisfying \(\mathfrak{g}^\circ(T,\mathsf{y})
    =(-3\log(y_1 y_2)+O_{\msf{y}}(1)) \Omega^2\) for $y_1 y_2 \geq 1$ and $|x_1|, |x_2| \leq 1/2$,
    where $O_{\msf{y}}(1)$ means a constant depending on $\msf{y}$.
\end{enumerate}
\end{enumerate}
\end{theorem}

For $\rho\in\mathbb{C}$, we also need to study absolute convergence of
\begin{equation}\label{eq:joint}
 J(T,\mathsf{y},\rho)=
 \int_1^\infty\int_\mathcal{M(\mathbb{C})} \abs{K_T(t,\mathsf{y},z) t^{- 1 - \rho}}\,\Omega^2\,
                         \dd t.
\end{equation}
For this analysis, we estimate $K_T(t,\mathsf{y},z)$
for $z \in \mathbb{D}_1 \coloneqq \mathcal{F}_1 \times \mc{F}_1 \subseteq \mathbb{D}$
where $\mathcal{F}_1\subset\mathbb{H}_1$ consists of points $\tau = x+iy$ such that \(\abs{x}\le\tfrac12\) and
\(y\ge\sqrt3/2\).

\begin{theorem}\label{thm:main}
For \(T\in\Sym_3(\Q)\) and $t \geq 1$ and $\mathsf{y} \in \opn{Sym}_{3}(\R)_{>0}$
and $z \in \mbb{D}_1$, we have
\begin{equation}\label{eq:bound kernel}
    \abs{K_T(t,\msf{y},z)} \lesssim_{T, \msf{y}}
    \begin{cases}
    t^2 & \text{if $T \neq \bsm{0}_3$ and is positive semi-definite}
        \\
    e^{-c_{T, \msf{y}} t} & \text{if $T$ has a negative eigenvalue}
        \\
    1 + e^{-c_{T,\msf{y}}t(y_1y_2)^{-1}} & \text{if $T = \bsm{0}_3$}
    \end{cases}
\end{equation}
for some constants $c_{T,\msf{y}} > 0$ depending on $T$ and $\msf{y}$.
Therefore the following statements hold.
\begin{itemize}
    \item[(a).] If $T\neq\boldsymbol{0}_3$ is positive semi-definite,
    we have $J(T,\mathsf{y},\rho) < \infty$ if $ \opn{Re}(\rho)>2$.
    \item[(b).] If $T \neq \bsm{0}_3$ is not positive semi-definite, we have $J(T,\mathsf{y},\rho) < \infty$ for all $\rho\in\mathbb{C}$.
     \item [(c).] If $T=\boldsymbol{0}_3$, we have $J(T,\mathsf{y},\rho) < \infty$
     if and only if $ \opn{Re}(\rho)>0$.
\end{itemize}
\end{theorem}
\begin{proof}
For $T \neq \bsm{0}_3$ positive semi-definite, this
follows from the estimates \eqref{equation:genus_three_convegence:K_T_bound_pos_semi_def_2}
and \eqref{equation:genus_three_convegence:K_T_bound_pos_semi_def_1}.
For $T \neq \bsm{0}_3$ not positive semi-definite, this is the bound
\eqref{eq:indefinite}, which is also valid for $\rank(T) = 1$ as explained 
in \eqref{ssec:genus_three_convergence:rank_1}.
For $T = \bsm{0}_3$, this follows from the estimates
in \eqref{eq:larege t} and
\eqref{eq:zero-small-t}.
\end{proof}

\begin{remark}\label{rem:boundedness of K}
    For fixed $t\geq1$ and $\msf{y}\in\mathrm{Sym}_3(\mathbb{R})_{>0}$, the estimate \eqref{eq:bound kernel} implies that
     \begin{equation*}
         \sup\limits_{z\in\mathcal{M}(\mathbb{C})}\abs{K_T(t,\mathsf{y},z)}<\infty.
     \end{equation*}
\end{remark}

\subsection{Some notations}
Throughout Appendix \ref{sec append: genus 3 convergence},
we write $A \lesssim_{\alpha} B$ to mean $A \leq C_{\alpha} B$ for some constant $C_{\alpha} > 0$
depending on a parameter $\alpha$. The notation $A \lesssim B$ means that the implied constant is absolute,
not depending on any parameters.

For $\boldsymbol{x}=(x_1,x_2,x_3)\in V_\infty^{3}$ where \(x_i=\left(\begin{smallmatrix}a_i&b_i\\c_i&d_i\end{smallmatrix}\right)\),
set \(U(\boldsymbol{x})=(a_i+d_i)_i\), \(V(\boldsymbol{x})=(c_i-b_i)_i\), and
\begin{equation}\label{eq:poly}
 P(U,V)=3-2\pi(\abs U^2+\abs V^2)
       +2\pi^2\bigl(\abs U^2\abs V^2-(U\cdot V)^2\bigr).
\end{equation}
At \(z=(i,i)\), the scalar of \(\nu_3\) relative to \(\Omega_z^2\) can be obtained by combining the explicit formula (\ref{equ: xi33}) and Lemma \ref{form-to-function}
\begin{equation}
    \widetilde\nu_3(\boldsymbol{x})=e^{-\pi\|\boldsymbol{x}\|^2}P\left(U(\boldsymbol{x}),V(\boldsymbol{x})\right)
\end{equation}
where $\|\boldsymbol{x}\|^{2}=\|x_1\|^{2}+\|x_2\|^{2}+\|x_3\|^{2}$.

For a point $z\in\mathbb{D}_1$, let $h_{z_i}=\left(\begin{smallmatrix}
    1 & x_i\\
    0 & 1
\end{smallmatrix}\right)\cdot\left(\begin{smallmatrix}
    y_i^{1/2} & 0 \\
    0 & y_{i}^{-1/2}
\end{smallmatrix}\right)\in\textup{SL}_2(\mathbb{R})$,
and define $h_z=(h_{z_1},h_{z_2}) \in H(\R)^{\circ}$. Define
\begin{equation*}
    \|\boldsymbol{x}\|_z=\nn{h_z^{-1}\boldsymbol{x}}^2
\end{equation*}
Throughout this appendix, we denote $y_1y_2$ by $Y$ and $y_2/y_1$ by $k$. The change of variables gives
\begin{equation}\label{eq:cusp-volume}
 \Omega^2 = -\frac{1}{16 \pi^2} \frac{\dd x_1 \wedge \dd x_2 \wedge \dd Y \wedge \dd k}{Y^2 k}.
\end{equation}
Moreover, the restrictions \(y_j\ge \sqrt{3}/2\) on $\mathbb{D}_1$ imply
\(3/(4Y)\le k\le 4Y/3\) on $\mbb{D}_1$.

Given a tuple $\boldsymbol{x}=(x_1,x_2,x_3)\in V_\infty^{3}$, define $d(\boldsymbol{x})=(d_1,d_2,d_3)\in\mathbb{R}^{3}$. We will split the kernel function $K_T$ into the following two parts:
\begin{equation}\label{eq:decom Kt}
    K_T(t,\mathsf{y},z)=K_T^{d\neq0}(t,\mathsf{y},z)+K_T^{d=0}(t,\mathsf{y},z),
\end{equation}
where the first part $K_T^{d\neq0}$ is the summation
over all those $\boldsymbol{x}\in L^3$ with $d(\boldsymbol{x})\neq\boldsymbol{0}$,
while the second part $K_T^{d=0}$ is the summation of the remaining terms.

\subsection{The Gaussian estimate when \(d\ne0\)}\label{ssec: gauss dneq0}
For $a\in\mathrm{GL}_3(\mathbb{R})$, denote by $\lambda$ the smallest eigenvalue of $\mathsf{y}=a\cdot{^{t}a}$.

If $d(\boldsymbol{x})\neq\boldsymbol{0}$, at least one $d_i$ is not 0,
then $d_1^{2}+d_2^{2}+d_3^{2}\geq1$ since all $d_i$ are integers.
Notice that the lower-right element of $h_z^{-1}\boldsymbol{x}\cdot a$ forms a row vector $\sqrt{Y}(d_1,d_2,d_3)\cdot a$,
\begin{equation}\label{eq:majorant1}
   \|\boldsymbol{x}\cdot a\|_z^2\geq\|\sqrt Y(d_1,d_{2},d_{3})\cdot a\|^{2}\geq\lambda Y(d_1^{2}+d_2^{2}+d_3^{2})\geq\lambda Y.
\end{equation}
There is also a bound on all twelve lattice coordinates. Indeed, on the set $\mathbb{D}_1$,
\[
 \nn{h_{z_j}}^2
 = y_j+y_{j}^{-1}(1+x_j^2)< 3y_j,
\]
hence $\|\boldsymbol{x}\|^{2}=\|h_z\cdot h_z^{-1}\boldsymbol{x}\|^{2}<9Y\|h_z^{-1}\boldsymbol{x}\|^{2} = 9 Y \| \bsm{x} \|^2_z$.
On the other hand, we have
\begin{equation*}
   \|\boldsymbol{x}\cdot a\|_z^{2}=\|h_z^{-1}\boldsymbol{x}\cdot a\|^{2}\geq\lambda\|\boldsymbol{x}\|_z^{2}
\end{equation*}
Consequently,
\begin{equation}\label{eq:majorant2}
 \|\boldsymbol{x}\cdot a\|_z^2\ge(9Y)^{-1}\lambda\|\boldsymbol{x}\|^2.
\end{equation}

For $\boldsymbol{x}\in V_\infty^3$ with \(T(\boldsymbol{x})=T\),
we have
\(R(\boldsymbol{x} \cdot a,z)=\norm{\boldsymbol{x}\cdot a}_z^{2}/2-\opn{tr}(T\mathsf{y})\) by (\ref{maj-pro}).
Since \eqref{eq:poly} has degree four, we have
\begin{equation}\label{eq:term}
 e^{2\pi t \opn{tr}(T\mathsf{y})}\abs{\widetilde\nu_3(\sqrt t\,(h_z^{-1} \bsm{x})\cdot a)}
 \lesssim (1+t\norm{\boldsymbol{x}\cdot a}_z^{2})^2e^{-\pi t(\norm{\boldsymbol{x}\cdot a}_z^{2}-2\opn{tr}(T\mathsf{y}))}.
\end{equation}
Note also that the right-hand side is $\lesssim_{T, \msf{y}, Y} t^2 e^{-c' t \norm{\bsm{x}}^2}$
for some $c' > 0$ depending on $\msf{y}$ and $Y$.
\subsubsection{The unbounded part}
Choose \(Y_0=\max(1,8 \opn{tr}(T\mathsf{y}) /\lambda)\).
For \(Y\ge Y_0\) and an element $\boldsymbol{x}$ such that \(d(\boldsymbol{x})\ne0\), by (\ref{eq:majorant1})
\begin{equation*}
    \norm{\boldsymbol{x}\cdot a}_z^{2}\geq\lambda Y\geq\lambda Y_0\geq 8\opn{tr}(T\mathsf{y})\Longrightarrow\norm{\boldsymbol{x}\cdot a}_z^{2}-2\opn{tr}(T\mathsf{y})\ge3\norm{\boldsymbol{x}\cdot a}_z^{2}/4.
\end{equation*}

We have $(1+x)^{2} \lesssim e^{x/4}$ for all $x\geq0$.
Then use
\eqref{eq:majorant1} and \eqref{eq:majorant2} on separate $e^{-\pi t\norm{\boldsymbol{x}\cdot a}_z^{2}/4}$ factors
to obtain
\begin{align*}
    (1+t\norm{\boldsymbol{x}\cdot a}_z^{2})^2e^{-\pi t(\norm{\boldsymbol{x}\cdot a}_z^{2}-2\opn{tr}(T\mathsf{y}))}
 \lesssim e^{\pi t\norm{\boldsymbol{x}\cdot a}_z^{2}/4}\cdot e^{-3\pi t\norm{\boldsymbol{x}\cdot a}_z^{2}/4}
 \lesssim e^{-\pi\lambda tY/4}e^{-\pi\lambda t\|\boldsymbol{x}\|^2/(36Y)}.
\end{align*}
Summing over all \(L^3\cong\Z^{12}\) by using the inequality \(\sum_{n\in\Z}e^{- \alpha n^2}\le1+\sqrt{\pi/ \alpha}\),
we obtain
\begin{equation}\label{eq:dnonzero}
 \abs{K_T^{d\ne0}(t,\mathsf{y},z)}
 \lesssim_{\lambda} e^{-\pi\lambda tY/4}\bigl(1+(Y/t)^6\bigr)
 \lesssim_{\lambda} e^{-\pi \lambda t Y / 8}.
\end{equation}
The proof shows that \eqref{eq:dnonzero} also holds without the condition $Y \geq Y_0$,
in the special case that $T = \bsm{0}_3$.

\subsubsection{The compact part}\label{ssec: dneq0 compact}
The region \(Y\le Y_0\) in \(\mathbb{D}_1\) is compact because $y_j\leq2\sqrt{3}^{-1}Y_0$. For $t\ge1$,
\begin{equation*}
    (1+t\norm{\boldsymbol{x}\cdot a}_z^{2})^2e^{-\pi t(\norm{\boldsymbol{x}\cdot a}_z^{2}-2\opn{tr}(T\mathsf{y}))}
    \leq t^2(1+\norm{\boldsymbol{x}\cdot a}_z^{2})^2e^{-\pi (\norm{\boldsymbol{x}\cdot a}_z^{2}-2\opn{tr}(T\mathsf{y}))}
    \lesssim t^2 e^{2\pi\opn{tr}(T\mathsf{y})}e^{-\pi \norm{\boldsymbol{x}\cdot a}_z^{2}/2}.
\end{equation*}
The inequality \eqref{eq:majorant2} implies $\|\boldsymbol{x}\cdot a\|_z^2\ge(9Y)^{-1}\lambda\|\boldsymbol{x}\|^2\geq(9Y_0)^{-1}\lambda\|\boldsymbol{x}\|^2$ on the region $Y\leq Y_0$.
The Gaussian lattice estimate implies
\begin{equation}\label{eq:compact}
 \abs{K_T^{d\neq0}(t,\mathsf{y},z)} \lesssim_{T,\msf{y}} t^2
\end{equation}
for $Y \leq Y_0$.
If $T$ is not positive semi-definite, then ${}^t a T a$ has some negative eigenvalue $- \eta$,
and we have the improved bound $\abs{K_T^{d\neq0}(t,\mathsf{y},z)} \lesssim_{T,\msf{y}} t^2 e^{-2 \pi \eta t}$.
Note that the set of eigenvalues of ${}^t a T a$ depends only in $T$ and $\msf{y}$,
and not on the choice of $a$.
If instead $T = \bsm{0}_3$,
we may bound $\abs{\widetilde\nu_3(\sqrt t\,(h_z^{-1} \bsm{x})\cdot a)} \lesssim e^{- \pi \lambda \norm{\bsm{x}}^2 / 18}$
using \eqref{eq:term} and \eqref{eq:majorant2}, since $t \geq 1$, using also $Y \leq Y_0 = 1$ in this case.
Summing over $\bsm{x} \neq 0$ (using the $d \neq 0$ condition)
gives $|K^{d \neq 0}_{\bsm{0}_3}(t, \msf{y}, z)| \lesssim_{\msf{y}} e^{-\pi \lambda t / (18 Y)}$
in the region $Y \leq Y_0$.
This bound on $|K^{d \neq 0}_{\bsm{0}_3}(t, \msf{y}, z)|$ also holds without the condition $Y \leq Y_0$,
if we instead assume $t \geq Y$.

\subsection{The Poisson estimate when \(d=0\)}\label{ssec: poisson estimate}

Let $\boldsymbol{x}\in L^3$ be a tuple such that $d(\boldsymbol{x})=\boldsymbol{0}$. Write \(b(\boldsymbol{x}),c(\boldsymbol{x}),l(\boldsymbol{x})\in\Z^3\) for the upper-right, lower-left, and upper-left rows
of the tuple. Then
\begin{equation}\label{eq:bc}
 T=T(\boldsymbol{x})=-\tfrac12\left({^tb}(\boldsymbol{x})c(\boldsymbol{x})+{^tc}(\boldsymbol{x})b(\boldsymbol{x})\right).
\end{equation}
\begin{lemma}\label{lem: finite pair}
    For fixed $T\in\mathrm{Sym}_3(\mathbb{Q})\neq\boldsymbol{0}_3$, there are only finitely many solution pairs $(b,c)$, where $b,c\in\mathbb{Z}^{3}$, of the equation \eqref{eq:bc}.
\end{lemma}
\begin{proof}
    Let $(b,c)$ be a solution pair to the equation \eqref{eq:bc}.
    Then the function $v\in\mathbb{Q}^{3}\mapsto {^tv}Tv$ is $v \mapsto -(b\cdot v) (c\cdot v)$. The unique factorization of the polynomial ring $\mathbb{Q}[v_1,v_2,v_3]$ determines the two linear factors up to scaling and exchange. Choose primitive integral representatives of those factors $b_0$ and $c_0$, then
\(b=mb_0\), \(c=nc_0\) up to swapping $b_0$ and $c_0$, with \(m,n\in\Z\) and a fixed nonzero product \(mn\) since $T\neq\boldsymbol{0}_3$. There are finitely many possibilities.
\end{proof}

For each fixed pair \((b,c)\) satisfying \eqref{eq:bc}, every vector \(l\in\Z^3\) can appear as the upper-left element of some tuple $\boldsymbol{x}$. For $a\in\mathrm{GL}_3(\mathbb{R})$, we introduce the following notations for simplicity:
\[
 \Lambda_a=\Z^3\cdot a,\quad B=b\cdot a,\quad C=c\cdot a.
\]

For a point $z\in\mathbb{D}$, set $\gamma=-x_1C-x_2B$
(with $x_i = \opn{Re}(z_i)$ here). Then
\[
 U(h_z^{-1}\boldsymbol{x}\cdot a)=\sqrt{Y}^{-1}\left(l(\boldsymbol{x})\cdot a+\gamma\right),\qquad
 V(h_z^{-1}\boldsymbol{x}\cdot a)=\sqrt{k}^{-1}C-\sqrt k B.
\]
For simplicity, let $W=\sqrt{k}^{-1}C+\sqrt k B$. Then
\begin{equation}\label{eq:VW}
 \abs V^2=\abs W^2+4\opn{tr}(T\mathsf{y}).
\end{equation}
Let $f_{b,c}(U)=e^{-\pi t\abs{U}^{2}}P(\sqrt{t}U,\sqrt{t}V)$.
The contribution of this pair $(b,c)$ to $K_T^{d=0}$ is
\begin{equation}\label{eq:pair}
 K_{b,c}(t,\mathsf{y},z)\coloneqq e^{-\pi t\abs W^2}
 \sum_{l\in\Lambda_a}f_{b,c}\left(\frac{l+\gamma}{\sqrt{Y}}\right).
\end{equation}
The Fourier transform of the function $f(U)$ is:
\begin{equation}\label{eq:fourier transform}
    \hat{f}_{b,c}(\eta)=2\pi t^{-5/2}e^{-\pi\abs{\eta}^{2}/t}\left(\abs{\eta}^{2}-\pi t\left(\abs{V}^{2}\abs{\eta}^{2}-(V\cdot\eta)^{2}\right)\right).
\end{equation}
Notice that $\hat{f}(\boldsymbol{0})=0$, hence (\ref{eq:pair}) and Poisson summation implies that
\begin{align}\label{eq:poisson}
 K_{b,c}(t,\mathsf{y},z)
 =&2\pi Y^{5/2}t^{-5/2}\det(\mathsf{y})^{-1/2}e^{-\pi t\abs{W}^{2}}\times\\
 &\sum_{\substack{\xi\in\Lambda_a^\vee\\\xi\ne\boldsymbol{0}}}
 e^{-\pi t^{-1}Y\abs\xi^2}e^{2\pi i\langle\xi,\gamma\rangle}\cdot\left(\abs\xi^2
 -\pi t\bigl(\abs V^2\abs\xi^2-(V\cdot\xi)^2\bigr)\right).\notag
\end{align}

For $q>0$, define $F_a(q)=q^{5/2}
 \sum_{\xi\in\Lambda_a^\vee\backslash\boldsymbol{0}}
          \abs\xi^2e^{-\pi q\abs\xi^2}$.
Taking absolute values in \eqref{eq:poisson}, we find
\begin{equation}\label{eq:master1}
 \abs{K_{b,c}(t,\mathsf{y},z)}
 \lesssim_{\msf{y}} e^{-\pi t\abs W^2}(1+t\abs V^2)F_a(Y/t)
 \lesssim_{\msf{y}} e^{-\pi t\abs W^2}(1+t\abs W^2+4\opn{tr}(T\mathsf{y})t)
\end{equation}
as $\sup_{q>0}F_a(q)<\infty$.
To see $\sup_{q>0}F_a(q)<\infty$, note that
comparison with $\int_{\R^3} |\xi|^2 e^{-\pi q |\xi|^2} \dd \xi$
bounds the sum $\sum_{\xi\in\Lambda_a^\vee\backslash\boldsymbol{0}} \abs\xi^2e^{-\pi q\abs\xi^2}$
by a constant multiple of \(q^{-5/2}\).
For $q \geq 1$, the sum in fact decays exponentially in $q$.
The second inequality in \eqref{eq:master1} follows from \eqref{eq:VW}.
Therefore, we have
\begin{equation}\label{eq:master}
    \abs{K_{b,c}(t,\mathsf{y},z)} \lesssim_{T,\msf{y}} 1+t
\end{equation}
for all $z \in \mbb{D}$ and $t \geq 1$.

We now prove \eqref{eq:bound kernel} in Theorem \ref{thm:main} for $T\neq\boldsymbol{0}_3$.
The estimate for $K_T^{d\neq0}$ follows from \eqref{eq:dnonzero} and \eqref{eq:compact}, while that for $K_T^{d=0}$ follows from Lemma \ref{lem: finite pair} and \eqref{eq:master}. Combining these two estimates yields the desired bound.

\subsection{Proof for matrices of rank \(\geq 2\)}\label{ssec:rank2,3}
In this subsection, we give the proof of the remaining statements of Theorem \ref{thm:main} and \ref{thm:extension} for the matrices $T$ with $\rk(T)=2$ or $3$. This case is much easier than the case $\rk(T)\leq1$ because either $K_T=K_T^{d\neq0}$ or $T$ has a negative eigenvalue.

We first consider the case that $T$ is positive semi-definite.
If $\rk T=3$, then the kernel function $K_T$ is zero,
because there is no tuple $\boldsymbol{x}\in L^3$ with moment matrix $T$
since the signature of $V_\infty$ is $(2,2)$.
If \(\rk T=2\), every tuple $\boldsymbol{x}$ contributing to the sum \eqref{eq:joint} must have \(d(\boldsymbol{x})\ne\boldsymbol{0}\) because
the 3-dimensional subspace of matrices with \(d=0\), with quadratic form \(-bc\), has
positive index one. Hence we have $\abs{K_T(t,\mathsf{y},z)}=\abs{K_T^{d\neq0}(t,\mathsf{y},z)}$.
For the unbounded region \(Y\ge Y_0=\max(1,8 \opn{tr}(T\mathsf{y})/\lambda)\) of $\mathbb{D}_1$,
the estimate function 
\eqref{eq:dnonzero} is integrable jointly in \(z,t\) with the factor
\(t^{-1- \opn{Re}(\rho)}\), for all \(\rho\in\mathbb{C}\), i.e.,
\begin{align}\label{eq:dneq0estimate}
    &\int_1^{\infty}\int_{\mathbb{D}_1\cap\{Y\geq Y_0\}}\abs{K_T(t,\mathsf{y},z)}\Omega^2\frac{\mathrm{d}t}{t^{1+ \opn{Re}(\rho)}}=\int_1^{\infty}\int_{\mathbb{D}_1\cap\{Y\geq Y_0\}}\abs{K_T^{d\ne0}(t,\mathsf{y},z)}\Omega^2\frac{\dd t}{t^{1+ \opn{Re}(\rho)}}
 \\& \lesssim_{\lambda} \int_{1}^{\infty}\int_{Y_0}^{\infty}\int_{3/(4Y)}^{4Y/3}\int_{-1/2}^{1/2}\int_{-1/2}^{1/2}e^{-\pi\lambda tY/4}\bigl(1+(Y/t)^6\bigr) \dd x_1 \dd x_2\frac{ \dd k \dd Y}{Y^2k}\frac{ \dd t}{t^{1+ \opn{Re}(\rho)}}<\infty.\notag
\end{align}
The convergence still holds with
powers of \(\log t\) in the integrand for derivatives in $\rho$.
Recall that the current $\mathfrak{g}^\circ(T,\mathsf{y},\rho)$
on $\mc{M}(\C)$ had meromorphic continuation to $\opn{Re}(\rho) > -1/2$
as in Lemma \ref{sum-con-current},
with a possible simple pole at $\rho=0$ by \cite[Proposition 2.12, (2.91)]{GS19}.
The estimate \eqref{eq:dneq0estimate}
shows that the restriction of $\mathfrak{g}^\circ(T,\mathsf{y},\rho)$
to a sufficiently small neighborhood in $\overline{\mathcal{M}}(\C)$
of the boundary $\overline{\mathcal{M}}(\C) \setminus \mathcal{M}(\C)$
admits a holomorphic continuation to all $\rho \in \C$ and can be represented by
an $L^1$ form in that neighborhood.

For the integral $J(T,\mathsf{y},\rho)$, the estimates \eqref{eq:dnonzero} and \eqref{eq:compact}
give 
    \begin{equation}\label{equation:genus_three_convegence:K_T_bound_pos_semi_def_2}
        \abs{K_T(t,\mathsf{y},z)} \lesssim_{T,\msf{y}} t^2.
    \end{equation}
Since \(\mathcal{M}(\mathbb{C})\) has finite volume, 
we find \(J(T,\mathsf{y},\rho) \lesssim_{T,\msf{y}} \mathrm{vol}(\mathcal{M}(\mathbb{C}))\int_1^\infty t^{1- \opn{Re}(\rho)}\,\mathrm{d}t<\infty\) for \( \opn{Re}(\rho)>2\).

Next we consider the case that $T$ has a negative eigenvalue.
Following the notations of $\S$\ref{ssec: poisson estimate}, for any solution pair $(b,c)$
of the equation \eqref{eq:bc}, we have \({^ta}Ta=-({^tB}C+{^tC}B)/2\).
The matrix ${}^t a T a$ preserves $\opn{span}\{B,C\} \subseteq \mathbb{R}^{3}$
and acts as $0$ on $\opn{span}\{B,C\}^{\perp}$.
Note that $\opn{span}\{B,C\}$ has rank $2$ because ${}^t a T a$ has rank $\geq 2$.
Consider an orthonormal basis $\{u,v\}$ of $\opn{span} \{B,C\}$
with $u = |B|^{-1} B$.
In the basis $\{u, v\}$, the matrix ${^ta}Ta$ acts on $\opn{span}\{B,C\}$
by the matrix
\begin{equation*}
    \begin{pmatrix}
        -B\cdot C & -\frac{1}{2}\abs{B}(C\cdot v)\\
        -\frac{1}{2}\abs{B}(C\cdot v) & 0
    \end{pmatrix}.
\end{equation*}
In this case, the unique negative eigenvalue of ${}^t a T a$ is
$
 -\eta=-\tfrac12(B\cdot C+\abs B\abs C)<0
$. 
Consequently,
\[
 \abs W^2=k\abs B^2+k^{-1}\abs C^2+2B\cdot C
 \ge2(\abs B\abs C+B\cdot C)=4\eta.
\]
In combination with \eqref{eq:master1}, we thus find
$|K_{b,c}(t, \msf{y}, z)| \lesssim_{T, \msf{y}} (1 + t) e^{-2 \pi \eta t}$
and hence $|K_T^{d=0}(t, \msf{y}, z)| \lesssim_{T, \msf{y}} (1 + t) e^{-2 \pi \eta t}$
by Lemma \ref{lem: finite pair}.
Not assuming existence of a solution pair $(b,c)$ to \eqref{eq:bc},
we have $\sup_{Y \geq Y_0} |K^{d \neq 0}_T(t, \msf{y}, z)| \lesssim_{\lambda} e^{- \pi \lambda t/ 8}$.
Combining with the estimate on the compact subset $Y\leq Y_0$ in $\S$\ref{ssec: dneq0 compact}, we get
\begin{equation}\label{eq:indefinite}
 \abs{K_T(t,\mathsf{y},z)}
                 \lesssim_{T,\msf{y}} (1+t^2)e^{-c' t},
\end{equation}
for some $c' > 0$ depending on $T$ and $\msf{y}$.
Therefore, for all $\rho\in\mathbb{C}$, we have
\begin{equation}\label{eq:negative eigen}
    \int_1^\infty\int_{\mathcal{M}(\mathbb{C})} |K_T(t,\mathsf{y},z)|\Omega_z^2\cdot t^{-1- \opn{Re}(\rho)}\,\mathrm{d}t<\infty.
\end{equation}
Therefore \(J(T,\mathsf{y},\rho)<\infty\) for all \(\rho\in\mathbb{C}\). Moreover, 
the convergence \eqref{eq:negative eigen} is still true with
powers of \(\log t\) in the integrand for derivatives in $\rho$. Thus,
for the case that $T$ has a negative eigenvalue, the family of currents $\mf{g}^{\circ}(T, \msf{y}, \rho)$
extends holomorphically to all $\rho \in \C$ and is valued in 
\(L^1\) forms on \(\overline{\mathcal{M}}(\mathbb{C})\).

\subsection{Proof for matrices of rank \(1\)}
\label{ssec:genus_three_convergence:rank_1}
In this subsection, we give the proof of Theorem \ref{thm:main} and \ref{thm:extension} for the matrices $T$ with $\rk(T)=1$.
If the only nonzero eigenvalue of $T$ is negative,
an argument similar to the case of $T$ with negative eigenvalues in $\S$\ref{ssec:rank2,3} still works
and we have a similar bounding result \eqref{eq:indefinite}, hence \eqref{eq:negative eigen}. Therefore we only focus on the case that the only nonzero eigenvalue of $T$ is positive.

We first prove the absolute convergence in Theorem \ref{thm:main}.
Finiteness of pairs $(b,c)$ by Lemma \ref{lem: finite pair}, the bound on $K_{b,c}$ from \eqref{eq:master},
and the bound for $K_T^{d \neq 0}$ from \eqref{eq:dnonzero} and \eqref{eq:compact} imply
    \begin{equation}\label{equation:genus_three_convegence:K_T_bound_pos_semi_def_1}
    \abs{K_T(t,\mathsf{y},z)} \lesssim_{T,y} (1+t^2).
    \end{equation}
As \(\mathcal{M}(\mathbb{C})\) has finite volume,
we find $J(T,\mathsf{y},\rho) \lesssim_{T,\msf{y}} \int_1^\infty (1+t^{2})t^{-1- \opn{Re}(\rho)}  \dd t<\infty$ for $\opn{Re}(\rho)>2$.

Next we consider the more difficult problem of extending the current to the compactification,
as stated in Theorem \ref{thm:extension}.
For the term $K_T^{d\neq0}$ in \eqref{eq:decom Kt},
the estimate \eqref{eq:dnonzero} still works,
so the integral $\int_1^{\infty} K_T^{d \neq 0}(t,\mathsf{y},z) \Omega_z^2 \cdot t^{-1-\rho} \dd t$
converges for all $\rho \in \C$ in the unbounded region $Y \geq Y_0$,
and hence defines family of $L^1$ forms on $\mbb{D}_1 \cap \{ Y \geq Y_0 \}$
which is holomorphic in $\rho$.
Therefore we are left with the term $K_T^{d=0}$.

By Lemma \ref{lem: finite pair}, equation \eqref{eq:bc} has finitely
many solution pairs \((b,c)\). If there are no such pairs there is nothing further
to prove. For each pair \((b,c)\), the two vectors $b$ and $c$ are linearly dependent since otherwise
the rank of $T$ would be $2$.
Then there exists a primitive vector (meaning the gcd of the three coordinates are 1)
$p\in\Z^3$ such that $b=b_0 p$ and $c=c_0 p$, where $b_0,c_0\in\Z\setminus\{0\}$.
The positive semi-definiteness of the matrix $T$ implies $b_0c_0<0$. 

Let $\ell_p=|p\cdot a|^2>0$ and $k_0=-c_0/b_0>0$.
Then $\opn{tr}(T\mathsf{y})=-b_0c_0\ell_p$.
In the notation of $\S$\ref{ssec: poisson estimate}, the vector $W$, as a function in $k$,
\begin{equation}
    W=W(k)=\sqrt{k}^{-1}c_0 p\cdot a+\sqrt{k}b_0 p\cdot a=(\sqrt{k}^{-1}c_0+\sqrt{k}b_0)p\cdot a.
\end{equation}
Hence $\abs{W(k)}^2=-\ell_pb_0c_0(k_0k^{-1}+kk_0^{-1}-2)$.
Consequently, for all $r>0$, we get the following estimate for \(t\ge1\),
\begin{equation}\label{eq:k-gaussian}
 \int_0^\infty e^{-rt|W(k)|^2}\frac{\mathrm{d}k}{k} \lesssim_{r,T,\msf{y}} t^{-1/2}.
\end{equation}

Now we are going to split the
integral $\int_1^{\infty} K_T^{d = 0}(t,\mathsf{y},z) \Omega_z^2 \cdot t^{-1-\rho} \dd t$ at \(t=Y\),
and we study the two parts $1\leq t\leq Y$ and $t\geq Y$ separately.
For the range $1\leq t\leq Y$, let \(q=Y/t\ge1\).
Then the exponential decay of \(F_a(q)\) in $\S$\ref{ssec: poisson estimate}, together
with \eqref{eq:master1}, gives
\begin{equation}\label{eq:rankone-low-t}
 |K_{b,c}(t,\mathsf{y},z)|
 \lesssim_{T, \msf{y}} (1+t) q^{5/2} e^{-c'q}e^{-c't|W(k)|^2}
\end{equation}
where $c'$ depends on $\msf{y}$. 
The factor \(1+t|W|^2+4\opn{tr}(T\mathsf{y})t\) has been absorbed into
the Gaussian $e^{-c' t |W(k)|^2}$.
The formulas
\eqref{eq:cusp-volume} and \eqref{eq:k-gaussian} give
\begin{equation}\label{eq:rankone-low-L1}
 \int_{|x_1|<1/2} \int_{|x_2|<1/2} \int_{k \in \R_{>0}}
 \int_{Y \geq t}|K_{b,c}(t,\mathsf{y},z)|\Omega^2
    \lesssim_{T,\msf{y}} (1+t)t^{-1/2}\int_t^\infty (Y/t)^{5/2} e^{-c' Y/t}
    \frac{\dd Y}{Y^2}
    \lesssim_{\msf{y}} t^{-1/2}.
\end{equation}
It follows that, for $Y \geq 1$, the integral $\int_1^{Y} K_T^{d = 0}(t,\mathsf{y},z) \Omega_z^2 \cdot t^{-1-\rho} \dd t$
converges for $\opn{Re}(\rho) > -1/2$, where it defines a
holomorphic family of \(L^1\)-forms on $\mbb{D}_1$.

The range $t\geq Y$ is considerably more delicate. In the sum on the right-hand side of \eqref{eq:pair},
there are two possibilities for the vector \(l\in\Lambda_a\): it either lies in the subspace spanned
by \(p\cdot a\) or it does not.
Set $v \coloneqq p \cdot a$ and $\ell_p \coloneqq |v|^2$.
Let $\pi_{\perp} \colon \R^3 \ra \{ v \}^{\perp} \subseteq \R^3$ be the orthogonal projection,
and set $\Lambda_{\perp} \coloneqq \pi_{\perp}(\Lambda_a)$.

For each $\lambda \in \Lambda_{\perp}$, choose a lift $\tilde{\lambda} \in \Lambda_a$
with $\pi_{\perp}(\tilde{\lambda}) = \lambda$, and such that $\tilde{0} = 0$.
Define $\theta_{\lambda} \in \R$ by the formula $\tilde{\lambda} = \theta_{\lambda} v + \lambda$.
Then every $l \in \Lambda_a$ has a unique expression
\begin{equation}\label{eq:decom l}
    l = (m + \theta_{\lambda}) v + \lambda \quad\quad m \in \Z, \lambda \in \Lambda_{\perp}.
\end{equation}

Then the function $K_{b,c}$ from \eqref{eq:pair},
as a sum over such $l$, splits into the following two parts:
\begin{equation*}
    K_{b,c}(t,\mathsf{y},z)=e^{-\pi t\abs W^2}
 \sum_{m\in\mathbb{Z}}f_{b,c}\left(\frac{m\cdot(p\cdot a)+\gamma}{\sqrt{Y}}\right)+e^{-\pi t\abs W^2}
 \sum_{\substack{m\in\mathbb{Z}\\\boldsymbol{0}\neq\lambda\in\Lambda_{\perp}}}f_{b,c}\left(\frac{(m + \theta_{\lambda}) \cdot(p\cdot a)+\lambda+\gamma}{\sqrt{Y}}\right).
\end{equation*}
We call the first term the \emph{parallel part} and denote it by $K_{b,c}^{\parallel}$,
since the vector $l$ is parallel to $p\cdot a$.
We call the second term the \emph{nonparallel part} and denote it by $K_{b,c}^{\nparallel}$. 

\subsubsection{Nonparallel part}
We first consider those vectors $l\in\Lambda_a$ not parallel to $p\cdot a$,
i.e., with $\lambda \neq 0$ in the decomposition of $l$ from \eqref{eq:decom l}.
Let $\delta\coloneqq\min\limits_{\boldsymbol{0}\neq\lambda\in\Lambda_{\perp}}\{\abs{\lambda}\}>0$.

By the explicit formula \eqref{eq:poly} in combination 
with \eqref{eq:VW},
\begin{equation}
    \abs{P(\sqrt{t}U,\sqrt{t}V)} \lesssim_{T,\msf{y}} (1+t)(1+t\abs{U}^2)(1+t\abs{W}^2)
\end{equation}
for $t \geq 1$.
The individual summand in $K_{b,c}(t, \msf{y}, z)$
corresponding to $l$ is bounded by a multiple (depending on $T$ and $\msf{y}$) of the function
\begin{equation}\label{eq:rankone-individual}
 (1+t)e^{-r t(\abs{W}^2+\abs{U}^2)}=(1+t)e^{-rt\abs{W}^2-rt\abs{l+\gamma}^2/Y},
\end{equation}
for some $r > 0$,
where we used $\sup\limits_{x\in\mathbb{R}} ((1 + |x|)e^{-r |x|}) <\infty$ for all $r>0$,
to absorb $(1+t\abs{U}^2)(1+t\abs{W}^2)$ into the Gaussian $e^{-\pi t(\abs{U}^2+\abs{W}^2)}$,
up to changing the constant $\pi$.

We next estimate
    \begin{equation*}
        \sum_{\substack{m\in\mathbb{Z}\\\boldsymbol{0}\neq\lambda\in\Lambda_{\perp}}}
        e^{-r t |l + \gamma|^2 /Y}.
    \end{equation*}
By the decomposition \eqref{eq:decom l},
we have $l + \gamma = ( m + \theta_{\lambda} )\cdot(p\cdot a) + \lambda-x_1c_0p\cdot a-x_2b_0p\cdot a$.
We then have
\begin{equation}
    \abs{l+\gamma}^2=\ell_p\abs{m + \theta_{\lambda} -x_1c_0-x_2b_0}^2+\abs{\lambda}^2.
\end{equation}
After summing over \(m\), the function $x_1\mapsto\sum\limits_{m\in\mathbb{Z}}e^{-r t\ell_p\abs{m-x_1 c_0-x_2b_0}^2/Y}$ has period $\abs{c_0}^{-1}$, hence
\begin{align}
    &\int_{-1/2}^{1/2}\sum\limits_{m\in\mathbb{Z}}e^{-r t\ell_p\abs{m + \theta_{\lambda} - x_1 c_0-x_2b_0}^2/Y} \dd x_1
    =\abs{c_0}^{-1}\int_{-|c_0|/2}^{|c_0|/2}\sum\limits_{m\in\mathbb{Z}}e^{-r t\ell_p\abs{m + \theta_{\lambda} - x-x_2b_0}^2/Y} \dd x\\
    =&\int_{-1/2}^{1/2}\sum\limits_{m\in\mathbb{Z}}e^{-r t\ell_p\abs{m + \theta_{\lambda} - x-x_2b_0}^2/Y} \dd x
    =\int_\mathbb{R}e^{-r t\ell_p\abs{-x-x_2b_0}^2/Y} \dd x=\sqrt{\pi\ell_p^{-1}}\sqrt{Y/(rt)}.\notag
\end{align}
For the sum over nonzero $\lambda \in \Lambda_{\perp}$,
we have $\sum\limits_{\boldsymbol{0}\neq \lambda\in\Lambda_{\perp}}e^{-r t\abs{\lambda}^2/Y} < e^{-(r \delta^2/2Y)t}\sum\limits_{\boldsymbol{0}\neq \lambda\in\Lambda_{\perp}}e^{-r \abs{\lambda}^2/2}$ since $t\geq Y$.

Taking the product of these two estimates
and setting $M' \coloneqq \sqrt{\pi r^{-1} \ell_p^{-1}} \cdot \sum\limits_{\boldsymbol{0}\neq \lambda\in\Lambda_{\perp}}e^{-r\abs{\lambda}^2/2}$
gives
\begin{align}\label{eq:estimate nonpara-1}
    \int_{-1/2}^{1/2}\int_{-1/2}^{1/2}\abs{K_{b,c}^{\nparallel}(t,\mathsf{y},z)} \dd x_1 \dd x_2
    \lesssim_{T, \msf{y}} M'(1+t)e^{-r t\abs{W}^2}e^{-(r \delta^2/2Y)t}\cdot\sqrt{Y/t}.
\end{align}
By \eqref{eq:k-gaussian}, integration against $k^{-1} \dd k$ gives
\begin{equation}\label{eq:estimate nonpara-2}
    \int_{0}^{\infty} \int_{-1/2}^{1/2}\int_{-1/2}^{1/2}\abs{K_{b,c}^{\nparallel}(t,\mathsf{y},z)} \dd x_1 \dd x_2 \frac{\dd k}{k}
    \lesssim_{T, \msf{y}}
    M' (1+t)\sqrt{Y/t}\,t^{-1/2} e^{-(r\delta^2/2Y)t}
    \lesssim_{T, \msf{y}} \sqrt{Y} \cdot e^{-(r\delta^2/2Y)t}.
\end{equation}

Using the volume form \eqref{eq:cusp-volume},
we find that
    \begin{equation*}
    \int_{\mbb{D}_1} \int_{\max \{Y, 1\}}^{\infty}
    |K^{\nparallel}_{b,c}(t, \msf{y}, z)| \Omega^2 t^{-1-\opn{Re}(\rho)} \dd t
    \end{equation*}
is bounded by a constant multiple of the integral
\begin{align}\label{eq:nonparallel part}
 \int_{3/4}^\infty Y^{-3/2}
       \int_Y^\infty e^{-(r\delta^2/2Y)t}t^{-1 - \opn{Re}(\rho)}\,\mathrm{d}t\,\mathrm{d}Y
 &=\left(\int_1^\infty e^{-(r\delta^2/2)t}t^{-1 - \opn{Re}(\rho)}\,\mathrm{d}t\right)
       \int_{3/4}^\infty Y^{-3/2 - \opn{Re}(\rho)}\,\mathrm{d}Y.
\end{align}
For $\rho\in\mathbb{C}$ such that \( \opn{Re}(\rho)>-1\), the integral \eqref{eq:nonparallel part} converges. Moreover, 
the convergence of \eqref{eq:nonparallel part} still holds with
powers of \(\log t\) in the integrand for derivatives in $\rho$.

\subsubsection{Parallel part}
It remains to consider the parallel part. For this, it is more convenient to
change from $z_1, z_2$ to the following two coordinates:
\begin{equation}\label{eq:new coordinate}
    w_m=m-c_0z_1-b_0z_2 \quad \text{for} \quad (m\in\mathbb{Z}),\qquad \zeta=z_2.
\end{equation}
In these coordinates, we have $|U|^2 + |W|^2 = Y^{-1} \ell_p |w_m|^2$.
For fixed $m$, the corresponding term in $K_{b,c}^{\parallel}$ is
\begin{equation}\label{eq:para single term}
     e^{-\pi t\abs{W}^{2}}f_{b,c}\left(\frac{m\cdot(p\cdot a)+\gamma}{\sqrt{Y}}\right)=(3-2\pi tY^{-1}\ell_p\abs{w_m}^{2}-8\pi t\opn{tr}(T\mathsf{y}))e^{-\pi tY^{-1}\ell_p\abs{w_m}^{2}}.
\end{equation}
Since $x_1$ and $x_2$ are bounded,
for any $\varepsilon > 0$,
there are only finitely many integers \(m\) such that $\{ |w_m| < \varepsilon \} \cap\mathbb{D}_1\neq\varnothing$.

Let $h\in C^{\infty}(\overline{\mathcal{M}}(\mathbb{C}))$ be a compactly supported function.
We still use $h$ to denote its pullback to $\mathbb{D}_1$.
We are going to analyze the (absolute) convergence of the following integral:
\begin{equation}\label{eq:k-para-f}
    \int_{\mathbb{D}_1}\int_Y^{\infty}K_{b,c}^{\parallel}(t,\mathsf{y},z)h(z)t^{-1-\rho}\mathrm{d}t\cdot\Omega^2_z.
\end{equation}
To do this, we split the integral on $\mathbb{D}_1$ into two parts: the first part is those $z\in\mathbb{D}_1$ such that $\abs{w_m(z)}\geq\varepsilon$ for all the integers $m$, while the second part is the finite union $\bigcup\limits_m\{z\in\mathbb{D}_1:\abs{w_m(z)}<\varepsilon\}$.

On the first part where \(|w_m(z)|\ge\varepsilon\) for all $m$,
similar estimation techniques as the nonparallel case imply that
\begin{equation}
    \abs{\eqref{eq:para single term}} \lesssim_{T, \msf{y}} (1+t) e^{-rt/Y} e^{-rt\abs{w_m}^{2}/Y}
    \lesssim_{T, \msf{y}} (1 + t) e^{-r t / Y} e^{-r t ((m - c_0 x_1 - b_0 x_2)^2 / Y + \ell_p^{-1} |W(k)|^2)}
\end{equation}
for some $r$ (possibly depending on $T, \msf{y}, \varepsilon$),
where the $|w_m(z)|\ge\varepsilon$ condition is used to produce the term $e^{-rt / Y}$.
After summing over all the integers \(m\), the periodization and \(k\)-integration used in the nonparallel case give the analogues
of \eqref{eq:estimate nonpara-1}, \eqref{eq:estimate nonpara-2} and \eqref{eq:nonparallel part}.
Thus the parallel part $K_{b,c}^{\parallel}$ defines an \(L^1\)-form on $\mbb{D}_1$ over the first part holomorphically for
\( \opn{Re}(\rho)>-1\).

On the second part, we consider the integral \eqref{eq:k-para-f} within a tube $\abs{w_m}<\varepsilon$.
For all $m^\prime\neq m$
and for $\varepsilon$ sufficiently small (assume this), we have
$\abs{w_{m^\prime}}>\varepsilon$ on the tube $\abs{w_m}<\varepsilon$, therefore their summation defines an \(L^1\)-form over the first part holomorphically for
\( \opn{Re}(\rho)>-1/2\). Hence we are left to consider the integral of \eqref{eq:para single term} against $h$ over the tube. Using the substitution $t=Y\tau$,
\begin{equation}\label{eq:rankone-transverse-expression}
 \int_{Y}^{\infty}\eqref{eq:para single term}t^{-1-\rho}\mathrm{d}t = Y^{-\rho}\bigl(A_\rho(w_m)-8\pi \opn{tr}(T\mathsf{y}) Y H_\rho(w_m)\bigr),
\end{equation}
where
\begin{align}\label{eq:A-family}
 A_\rho(w_m)
 =\int_1^\infty
       (3-2\pi\ell_p\tau|w_m|^2)e^{-\pi\ell_p\tau|w_m|^2}
                  \tau^{-1-\rho}\,\mathrm{d}\tau,\quad
 H_\rho(w_m)
 =\int_1^\infty e^{-\pi\ell_p\tau|w_m|^2}\tau^{-\rho}\mathrm{d}\tau.
\end{align}

Under the new coordinates $w_m$ and $\zeta$,
we have
    \begin{equation*}
    \Omega^2 = -\frac{c_0^2}{32 \pi^2}
        \frac{\dd w_m \wedge \dd \overline{w}_m \wedge \dd \zeta \wedge \dd \overline{\zeta}}{Y^2}.
    \end{equation*}
Recalling that $Y=y_1y_2$ and $y_1=-c_0^{-1}(b_0y_2+\mathrm{Im}(w_m))$ by \eqref{eq:new coordinate},
we get $y_2^{2} \lesssim_{T} Y \lesssim_{T} y_2^{2}$
on the tube $\abs{w_m}<\varepsilon$, for sufficiently small $\varepsilon$ (relative to $T$).

We first consider the $A_\rho$-part. Notice that for $\rho\in\mathbb{C}$ with $\opn{Re}(\rho)>-1$,
\begin{align}
    \int_{\abs{w_m}<\varepsilon}\abs{A_\rho(w_m) \dd w_m \dd \overline{w}_m} &<\int_1^{\infty}\int_{\mathbb{C}}\abs{(3-2\pi\ell_p\tau|w_m|^2)e^{-\pi\ell_p\tau|w_m|^2}
                  \tau^{-1-\rho} \dd w_m \dd \overline{w}_m} \dd \tau \\
                  & \lesssim_{T,\msf{y}} \int_{1}^{\infty}\tau^{-2-\opn{Re}(\rho)}\mathrm{d}\tau
                  = (1+\opn{Re}(\rho))^{-1},\notag
\end{align}
where the second inequality follows from integrating the Gaussian.
For \(\rho\) in any compact subset
\(K\subseteq\{\operatorname{Re}(\rho)>-1\}\),
we have $Y^{-2-\opn{Re}(\rho)} \lesssim_{T, K} y_2^{-4-2\opn{Re}(\rho)}$.
Then for $\rho \in K$, we have
\begin{align}
    \notag
    \int_{\{\abs{w_m}<\varepsilon\}\times\mathcal{F}_1}\abs{Y^{-\rho}A_\rho(w_m)h(w_m,\zeta)}\Omega^2
    & = \frac{c_0^2}{32 \pi^2} \int_{\mathcal{F}_1}\int_{\abs{w_m}<\varepsilon}\abs{Y^{-\rho}A_\rho(w_m)h(w_m,\zeta) Y^{-2} \dd w_m \dd \overline{w}_m \dd \zeta \dd \overline{\zeta}} \\
    \label{eq:estimate-A}
    & \lesssim_{T, \msf{y}, K} \|h\|_{\infty}(1+\opn{Re}(\rho))^{-1}\int_{1}^{\infty}y_2^{-4-2\opn{Re}(\rho)}\mathrm{d}y_2<\infty.
\end{align}
Moreover, 
the convergence \eqref{eq:estimate-A} still holds with
powers of \(\log Y\) in the integrand for derivatives in $\rho$.

Next we consider the \(H_\rho\)-part.
For $\zeta\in\mathcal{F}_1$ and $\rho\in\mathbb{C}$, define $\phi_{\zeta,\rho}(w_m)=h(w_m,\zeta)Y^{-\rho-1}$.
Let $\xi = \opn{Re}(w_m)$ and $\eta = \opn{Im}(\eta)$,
so $w_m=\xi+i\eta$. We first want to get an estimate for
\begin{equation}
    \|\phi_{\zeta,\rho}\|_{C^{2}}\coloneqq\max\limits_{j+k\leq 2}\sup\limits_{|w_m| \leq \varepsilon}\abs{\partial_{\xi}^j\partial_\eta^k\phi_{\zeta,\rho}(w_m)}.
\end{equation}
In the standard ``cusp coordinates'' associated to $(\infty,\infty)$
on $\overline{\mathcal{M}}(\mathbb{C})$, we have
\begin{equation*}
    q_1=e^{2\pi iz_1}=e^{2\pi ic_0^{-1}(m-b_0\zeta-w_m)} \quad\quad\text{and}\quad\quad
    q_2=e^{2\pi iz_2}=e^{2\pi i\zeta}.
\end{equation*}
The function $h$ is smooth in the coordinates $q_1,q_2$.
Since $\frac{\dd }{\dd \xi}q_1=-2\pi ic_0^{-1}q_1$ and $\frac{\dd }{\dd \eta}q_1
=2\pi c_0^{-1}q_1$,
and since $h$ is compactly supported on $\overline{\mc{M}}(\C)$,
we have $\|h\|_{C^{2}, w}\coloneqq\max\limits_{j+k\leq 2}\sup\limits_{|w_m| \leq \varepsilon}\abs{\partial_{\xi}^j\partial_\eta^kh(w_m,\zeta)}<\infty$.
Notice also that $\partial_\xi Y=0$ and $\partial_\eta Y=-c_0^{-1}y_2$.
Therefore, in any compact subset \(K\subseteq\{\operatorname{Re}(\rho)>-1\}\),
we have
\begin{equation}\label{eq:est-phi}
    \|\phi_{\zeta,\rho}\|_{C^{2}} \lesssim_{T,K} \norm{h}_{C^2,w} \cdot y_2^{-2-2 \opn{Re}(\rho)}.
\end{equation}
Taylor expansion of $\phi_{\zeta,\rho}$ at $w_m=0$ gives
\begin{equation}
    \phi_{\zeta,\rho}(w_m)=\phi_{\zeta,\rho}(0)+\partial_\xi\phi_{\zeta,\rho}(0)\xi+\partial_\eta\phi_{\zeta,\rho}(0)\eta+O(\|\phi_{\zeta,\rho}\|_{C^{2}}\abs{w_m}^{2}),\quad\abs{w_m}<\varepsilon.
\end{equation}
Then
    \begin{equation}\label{eq:heat-taylor}
 \int_{\abs{w_m}<\varepsilon}e^{-\pi\ell_p\tau|w_m|^2}\phi_{\zeta,\rho}(w_m) \frac{i}{2} \dd w_m \dd \overline{w}_m
 =\frac{\phi_{\zeta,
 \rho}(0)(1-e^{-\pi\ell_p\tau\varepsilon^{2}})}{\ell_p\tau}
   +O\!\left(\tau^{-2}\|\phi_{\zeta,\rho}\|_{C^2}\right).
\end{equation}
Therefore
\begin{align}\label{eq: meromorphic part}
    \int_{\{\abs{w_m}<\varepsilon\}}Y^{-1-\rho}H_\rho(w_m)h(w_m,\zeta) \frac{i}{2} \dd w_m \dd \overline{w}_m =&\frac{\phi_{\zeta,
 \rho}(0)}{\ell_p\rho}-\phi_{\zeta,\rho}(0)\ell_p^{-1}\int_{1}^{\infty}e^{-\pi\ell_p\varepsilon^2\tau}\tau^{-1-\rho} \dd \tau\\&+\left(\textup{remainder term}\right),\notag
\end{align}
for $\opn{Re}(\rho) > 0$,
where the (remainder term) is the integration of the term $O\!\left(\tau^{-2}\|\phi_{\zeta,\rho}\|_{C^2}\right)$ in \eqref{eq:heat-taylor} against $\tau$.
This (remainder term) extends holomorphically to $\opn{Re}(\rho) > -1$ by the estimate \eqref{eq:est-phi}.
The second term in \eqref{eq: meromorphic part} is also holomorphic for all $\rho \in \C$
due to the Gaussian term $e^{-\pi\ell_p\varepsilon^2\tau}$ in the integrand.
Moreover, we have $\phi_{\zeta,\rho}(0)=(-c_0b_0^{-1})^{\rho+1}y_2^{-2-2\rho}h(0,\zeta)$,
so $\abs{\phi_{\zeta,\rho}(0)} \lesssim_{T,K,h} y_{2}^{-2 - 2\opn{Re}(\rho)}$
for $\rho \in K$.
Therefore we can abbreviate \eqref{eq: meromorphic part} as
\begin{equation}\label{eq: meromorphic part-2}
     \int_{\{\abs{w_m}<\varepsilon\}}Y^{1-\rho}H_\rho(w_m)h(w_m,\zeta) \frac{i}{2} \dd w_m \dd \overline{w}_m
     =-\frac{(c_0b_0^{-1})^{\rho+1}y_2^{-2-2\rho}h(0,\zeta)}{\ell_p\rho}+t_m(\rho,\zeta; h),
\end{equation}
where in a compact subset \(K\subseteq\{\operatorname{Re}(\rho)>-1\}\),
the function $t_m(\rho,\zeta; h)$ is holomorphic in $\rho$
and we have $|t_m(\rho, \zeta; h)| \lesssim_{T,\msf{y},K,\varepsilon} \norm{h}_{C^2,w} y_2^{-2-2 \opn{Re}(\rho)}$.
Therefore, \eqref{eq: meromorphic part-2} is integrable against $d \zeta \wedge d \overline{\zeta}$
over $\zeta \in \mathcal{F}_1$ when $ \opn{Re}(\rho)>-1/2$.
Since $Y=Y(w_m,\zeta)$ and $H_\rho(w_m)$ are positive, the integral
\begin{equation}\label{eq:k-para-second part}
    \int_{\{\abs{w_m}<\varepsilon\}\times\mathcal{F}_1}\int_Y^{\infty}K_{b,c}^{\parallel}(t,\mathsf{y},z)h(z)t^{-1-\rho} \dd t\cdot\Omega^2_z
\end{equation}
converges absolutely when $ \opn{Re}(\rho)>-1/2$,
and defines a meromorphic function in $\rho$ with a possible simple pole at $\rho=0$.
From the bounds above in terms of $\norm{h}_{C^2,w}$,
the constant term of \eqref{eq:k-para-second part} at $\rho = 0$ is continuous as a function of $h$, i.e. the constant term
defines a $(2,2)$-current.

\subsection{Proof for the zero matrix}
In this subsection, we give the proof of Theorem \ref{thm:main}
and Theorem \ref{thm:extension} for the matrix $T=\boldsymbol{0}_3$.
We focus on the situation with $d = 0$ for the moment.
In that case, equation \eqref{eq:bc} forces \(b=\boldsymbol{0}\) or \(c=\boldsymbol{0}\) because if some \(b_i\ne0\),
the diagonal condition from $T$ gives \(c_i=0\), and the off-diagonal conditions from $T$ then
give \(c_j=0\) for every \(j\neq i\). There are now infinitely many such pairs $(b,c)$, since if \(b=\boldsymbol{0}\), then \(c\) can be any vector in \(\mathbb{Z}^3\), and conversely. Then
\begin{equation}\label{eq:t=0 d=0}
    K_{\boldsymbol{0}_3}^{d=0}(t,\mathsf{y},z)=\sum\limits_{c\in\mathbb{Z}^3}K_{\boldsymbol{0},c}(t,\mathsf{y},z)+\sum\limits_{b\in\mathbb{Z}^3}K_{b,\boldsymbol{0}}(t,\mathsf{y},z)-K_{\boldsymbol{0},\boldsymbol{0}}(t,\mathsf{y},z).
\end{equation}
For $q>0$ and $\tau =x+iy\in\mbb{H}_1$, define
\begin{equation}\label{eq:H}
    H(q,\tau)
 =\sum_{l,\mu\in\Lambda_a}
 e^{-\pi(\abs{l-x\mu}^2+y^2\abs\mu^2)/q}
 P\left(\frac{l-x\mu}{\sqrt q},\frac{y\mu}{\sqrt q}\right),
\end{equation}
here we follow the notation of $\S$\ref{ssec: poisson estimate} 
with $\Lambda_a=\mathbb{Z}^3\cdot a$.
For $z=(z_1,z_2)\in\mbb{H}_1^2=\mathbb{D}$,
a calculation gives
\begin{equation}\label{eq:two families}
    \sum\limits_{c\in\mathbb{Z}^3}K_{\boldsymbol{0},c}(t,\mathsf{y},z)=H(Y/t,z_1),\quad\sum\limits_{b\in\mathbb{Z}^3}K_{b,\boldsymbol{0}}(t,\mathsf{y},z)=H(Y/t,z_2),
    \quad
    K_{\bsm{0}, \bsm{0}}(t, \msf{y}, z) = \sum_{l \in \Lambda_a} e^{-\pi |l|^2 / q} P(l q^{-1/2}, 0).
\end{equation}

For the range $q=Y/t\le1$, for $\tau \in \mc{F}_1$, we have
\begin{equation}\label{eq: estimate single H}
    \abs{l-x\mu}^2+y^2\abs\mu^2 \geq |l|^2 + \tfrac34|\mu|^2 - |l \cdot \mu|
    \geq \tfrac14(\abs{l}^2+\abs\mu^2).
\end{equation}
Since $P$ is a polynomial in its variables, we can absorb \(P\) into the square-root
of the Gaussian term and sum over \(\Lambda_a^2\).
This gives $H(q,\tau) \lesssim_{\msf{y}} 1$ for $q \leq 1$ and $\tau \in \mc{F}_1$.

For $q \geq 1$, we use Poisson summation in \(l\) and the Fourier transform from \eqref{eq:fourier transform}
to give
\begin{align*}
 H(q,\tau)
 =2\pi q^{3/2}\abs{\det a}^{-1}
 \sum_{\substack{\mu\in\Lambda_a\\ \boldsymbol{0}\ne\xi\in\Lambda_a^\vee}}
 e^{-\pi y^2\abs\mu^2/q-\pi q\abs\xi^2}
  e^{-2\pi i x\langle\xi,\mu\rangle}\left(
  q\abs\xi^2
 -\pi y^2\bigl(\abs\mu^2\abs\xi^2-(\mu\cdot\xi)^2\bigr)
 \right).
\end{align*}
For the polynomial in the bracket, a calculation gives $\abs{
  q\abs\xi^2
 -\pi y^2\bigl(\abs\mu^2\abs\xi^2-(\mu\cdot\xi)^2\bigr)
 }\leq \abs{\xi}^{2}q(1+2\pi y^2\abs{\mu}^2 / q)\leq 2\pi\abs{\xi}^{2}q(1+y^2\abs{\mu}^2/q)$.
Hence
\begin{align*}
 \abs{H(q,\tau)}
 \le 4\pi^2 q^{5/2}
 \left ( \sum_{\mu\in\Lambda_a}
  \left(1+y^2\abs\mu^2/q\right)e^{-\pi y^2\abs\mu^2/q}
  \right )
 \left ( \sum_{\boldsymbol{0}\ne\xi\in\Lambda_a^\vee}\abs\xi^2e^{-\pi q\abs\xi^2} \right ).
\end{align*}
We find $\sum_{\mu\in\Lambda_a} \left(1+y^2\abs\mu^2/q\right)e^{-\pi y^2\abs\mu^2/q}
\lesssim_{\msf{y}} 1 + q^{3/2}$ for $y\geq \sqrt{3}/2$,
using the Gaussian integral comparison technique used to bound $F_a(q)$ in \S\ref{ssec: poisson estimate}.
For the $\xi$ sum, let $\delta_a = \min\limits_{\boldsymbol{0}\ne\xi\in\Lambda_a^\vee}\abs{\xi}>0$,
  \begin{equation*}
      \sum_{\boldsymbol{0}\ne\xi\in\Lambda_a^\vee}\abs\xi^2e^{-\pi q\abs\xi^2}
      \leq e^{-\pi (q-1)\delta_a^2 }\sum_{\boldsymbol{0}\ne\xi\in\Lambda_a^\vee}\abs\xi^2e^{-\pi \abs\xi^2}=e^{-\pi \delta_a^2 q}\cdot\left(e^{\pi \delta_a^2 }\sum_{\boldsymbol{0}\ne\xi\in\Lambda_a^\vee}\abs\xi^2e^{-\pi \abs\xi^2}\right).
  \end{equation*}
  Hence for $q\geq1$, we have
  \begin{equation}\label{eq: single H bound}
      \abs{H(q,\tau)} \lesssim_{\msf{y}} q^{4} e^{- \pi \delta_a^2 q}
  \end{equation}
  For all $q > 0$, we thus have
  $\sup\limits_{\tau \in \mc{F}_1}\abs{H(q,\tau)} \lesssim_{\msf{y}} 1$
  using exponential decay of the $\xi$ sum for $q \geq 1$.
  For all $q > 0$, Poisson summation again gives
    \begin{equation*}
        K_{\bsm{0}, \bsm{0}}(t, \msf{y}, z) = 2 \pi q^{5/2} |\det a|^{-1} \sum_{0 \neq \xi \in \Lambda_a^{\vee}}
        |\xi|^2 e^{-\pi q |\xi|^2} = 2 \pi |\det a|^{-1} F_a(q)
    \end{equation*}
  with $F_a(q)$ as used in Section \ref{ssec: poisson estimate}.
  By the estimates in loc. cit., we have
  $|K_{\bsm{0}, \bsm{0}}(t, \msf{y}, z)| \lesssim_{\msf{y}} 1$ (uniformly in $z$).
  The decomposition in \eqref{eq:t=0 d=0}
  now gives $\sup\limits_{z\in\mathbb{D}_1}\abs{K_{\boldsymbol{0}_3}^{d=0}(t,\mathsf{y},z)} \lesssim_{\msf{y}} 1$.

Recall that $\sup\limits_{z\in\mathbb{D}_1}\abs{K_{\boldsymbol{0}_3}^{d\neq0}(t,\mathsf{y},z)} \lesssim_{\msf{y}} 1$
by \eqref{eq:dnonzero} for $Y \geq Y_0$ and Section \ref{ssec: dneq0 compact}
for $Y \leq Y_0$.
We thus have $\sup\limits_{z\in\mathbb{D}_1}\abs{K_{\boldsymbol{0}_3}(t,\mathsf{y},z)}\lesssim_{\msf{y}} 1$,
which proves one direction of Theorem \ref{thm:main} (c) for $T=\boldsymbol{0}_3$,
i.e. that the integral defining $J(\bsm{0}_3, \msf{y}, \rho)$
converges absolutely for $\opn{Re}(\rho) > 0$,
using $\int_{\mc{M}(\C)} \Omega^2 < \infty$.
In particular, $\mf{g}^{\circ}(\bsm{0}_3, \msf{y}, \rho)$
is a holomorphic family of $L^1$ forms on $\overline{\mc{M}}(\C)$
for $\opn{Re}(\rho) > 0$.

Now we explain meromorphic continuation of the $L^1$ form $\mf{g}^{\circ}(\bsm{0}_3, \msf{y}, \rho)$
for $\opn{Re}(\rho) > -1$.
The $l=\mu=0$ term in the summation \eqref{eq:H} is $3$,
while all other terms are $\lesssim e^{-\pi(\abs{l}^2+\abs{\mu}^2)/(8 q)}$ by \eqref{eq: estimate single H}.
Combining with the estimate on $\abs{K_{\boldsymbol{0}_3}^{d\neq0}(t,\mathsf{y},z)}$
from Section \ref{ssec: dneq0 compact}, we get that for $t\geq\max(1,Y)$, we have
\begin{equation}\label{eq:larege t}
    |K_{\boldsymbol{0}_3}(t,\mathsf{y},z)-3|\leq \abs{K_{\boldsymbol{0}_3}^{d\neq0}(t,\mathsf{y},z)}+|K_{\boldsymbol{0}_3}^{d=0}(t,\mathsf{y},z)-3|
        \lesssim_{\msf{y}} e^{-c' t / Y},
\end{equation}
for some $c' > 0$ depending only on $\msf{y}$.
The estimate in \eqref{eq:larege t} also shows the divergence for $\opn{Re}(\rho) \leq 0$
claimed in Theorem \ref{thm:main} (c).

When \(1\le t\le Y\),
we have the exponential estimate
  $|K_{\bsm{0}, \bsm{0}}(t, \msf{y}, z)| \lesssim_{\msf{y}} e^{-c'' q}$
  for some $c'' > 0$ depending only on $\msf{y}$,
  as follows from the exponential estimate on the Gaussian sum from
  Section \ref{ssec: poisson estimate}.
  Combining this with the estimates \eqref{eq: single H bound} and \eqref{eq:dnonzero},
  we have
\begin{equation}\label{eq:zero-small-t}
 |K_{\boldsymbol{0}_3}(t,\mathsf{y},z)|
 \leq |K_{\boldsymbol{0}_3}^{d=0}(t,\mathsf{y},z)|+|K_{\boldsymbol{0}_3}^{d\neq0}(t,\mathsf{y},z)|
 \lesssim_{\msf{y}} q^4 e^{-c''' q}
   +e^{-c'''tY}(1+q^6) \lesssim_{\msf{y}} 1
\end{equation}
for \(1\le t\le Y\), and some constant $c''' > 0$ which depends only on $\msf{y}$.
For any $c_0 > 0$ and $Y \geq 1$,
\begin{align}
    &\int_{1}^{Y}e^{-c_0 tY}(1+q^6)\frac{\mathrm{d}t}{t}\overset{s=tY}{\leq}\int_{Y}^{Y^2}e^{-c_0 s}(1+s^6)\frac{\mathrm{d}s}{s}<\int_{1}^{\infty}e^{-c_0s}(1+s^6)\frac{\mathrm{d}s}{s},\\
    \int_{1}^{Y}&q^{4}e^{-c_0 q}\frac{\mathrm{d}t}{t}\overset{q=Y/t}{\leq}\int_{1}^{Y}q^{4}e^{-c_0 q}\frac{\mathrm{d}q}{q}<\int_{1}^{\infty}q^{4}e^{-c_0 q}\frac{\mathrm{d}q}{q},
    \quad\quad \int_{Y}^{\infty}e^{-c_0 t/Y}\frac{\mathrm{d}t}{t}\leq\int_{1}^{\infty}e^{-c_0 t}\frac{\mathrm{d}t}{t}.\notag
\end{align}
It follows that for $Y\geq1$, we have
\begin{equation}\label{eq:uniform finite}
    \int_1^Y|K_{\boldsymbol{0}_3}(t,\mathsf{y},z)|\,\frac{\mathrm{d}t}{t}\lesssim_{\msf{y}} 1,\qquad
 \int_Y^\infty|K_{\boldsymbol{0}_3}(t,\mathsf{y},z)-3|\,\frac{\mathrm{d}t}{t}\lesssim_{\msf{y}} 1,
\end{equation}
uniformly in $z$.

For $\opn{Re}(\rho) > 0$, we have
\begin{align}\label{eq:mero contin 0}
    \int_{1}^{\infty}K_{\boldsymbol{0}_3}(t,\mathsf{y},z)\frac{\mathrm{d}t}{t^{1+\rho}}=\frac{3}{\rho}+\int_{1}^{\infty}\left(K_{\boldsymbol{0}_3}(t,\mathsf{y},z)-3\right)\frac{\mathrm{d}t}{t^{1+\rho}}
\end{align}
where the integral converges absolutely as already explained in the paragraph before \eqref{eq:larege t},
because $\abs{K_{\boldsymbol{0}_3}(t,\mathsf{y},z)} \lesssim_{\msf{y}} 1$.

The second integral splits into the following two parts:
\begin{equation}
    \int_{1}^{Y}\left(K_{\boldsymbol{0}_3}(t,\mathsf{y},z)-3\right)\frac{\mathrm{d}t}{t^{1+\rho}},\quad\quad\int_{Y}^{\infty}\left(K_{\boldsymbol{0}_3}(t,\mathsf{y},z)-3\right)\frac{\mathrm{d}t}{t^{1+\rho}}.
\end{equation}
The first one converges absolutely for all $\rho\in\mathbb{C}$,
again using $\abs{K_{\boldsymbol{0}_3}(t,\mathsf{y},z)} \lesssim_{\msf{y}} 1$.
The second one also converges absolutely for all $\rho\in\mathbb{C}$ by the estimate \eqref{eq:larege t}.
Therefore, for any fixed $z \in \mbb{D}_1$, the formula \eqref{eq:mero contin 0} gives a meromorphic continuation
of $\mathfrak{g}^{\circ}(\boldsymbol{0}_3,\mathsf{y},\rho)_z$
to $\rho \in \mathbb{C}$, holomorphic away from $\rho = 0$ where there is a simple pole, i.e.
\begin{equation}\label{eq:zero-meromorphic}
 \mathfrak{g}^\circ(\boldsymbol{0}_3,\mathsf{y},\rho)_z
 =\frac{3 \Omega^2_z}{\rho}
  +\left[\left(\int_1^\infty(K_{\boldsymbol{0}_3}(t,\mathsf{y},z)-3)\frac{\mathrm{d}t}{t^{1+\rho}}\right)
               \Omega^2_z \right].
\end{equation}

For any fixed $\rho \neq 0$, the formula \eqref{eq:zero-meromorphic} defines
a measurable $(2,2)$-form $\mathfrak{g}^\circ(\boldsymbol{0}_3,\mathsf{y},\rho)$
on $\overline{\mc{M}}(\C)$.
Next we consider the absolute integrability
of $\mathfrak{g}^{\circ}(\boldsymbol{0}_3,\mathsf{y},\rho)$
and its constant term at $\rho = 0$.
We first estimate the integral of $(K_{\boldsymbol{0}_3}(t,\mathsf{y},z)-3) \Omega^2$ over $\overline{\mc{M}}(\C)$,
before integrating over $t$.
By \eqref{eq:cusp-volume}, we have
\begin{align}
    \int_{\overline{\mathcal{M}}(\mathbb{C})}\abs{K_{\boldsymbol{0}_3}(t,\mathsf{y},z)-3}\Omega^2&\lesssim\int_{3/4}^{\infty}\int_{3/(4Y)}^{4Y/3}\int_{-1/2}^{1/2}\int_{-1/2}^{1/2}\abs{K_{\boldsymbol{0}_3}(t,\mathsf{y},z)-3}\dd x_1\dd x_2\frac{\dd k\dd Y}{Y^2k}\\
    &\lesssim_{\msf{y}} \int_{3/4}^te^{-c' t/Y}\frac{\log (4Y/3)}{Y^2}\dd Y+\int_t^{\infty}\frac{\log (4Y/3)}{Y^2}\dd Y\lesssim_{\msf{y}}  \frac{1 + \log t}{t},\notag
\end{align}
where we used \eqref{eq:larege t} for the range $Y \leq t$
and \eqref{eq:zero-small-t} for the range $t \leq Y$.
Therefore, the difference $\mathfrak{g}^\circ(\boldsymbol{0}_3,\mathsf{y},\rho)-3\rho^{-1}[\Omega^2]$
is an $L^1$ form on $\overline{\mc{M}}(\C)$ for $\opn{Re}(\rho) > -1$,
and admits a holomorphic extension to $\rho = 0$ where its value is also an $L^1$
form $\overline{\mc{M}}(\C)$.

For the constant term at $\rho=0$, and for $Y \geq 1$, we have
\begin{align}\label{eq:zero-finite-part}
 &\opn{\underset{\rho=0}{CT}} \mf{g}^{\circ}(\bsm{0}_3, \msf{y}, \rho)
   = \left ( \int_1^\infty(K_{\boldsymbol{0}_3}(t,\mathsf{y},z)-3)\,\frac{\dd t}{t} \right ) \Omega^2 \\
 &= \left ( -3\log Y+
     \int_1^Y K_{\boldsymbol{0}_3}(t,\mathsf{y},z)\,\frac{\dd t}{t}
     +\int_Y^\infty(K_{\boldsymbol{0}_3}(t,\mathsf{y},z)-3)\,\frac{\dd t}{t} \right ) \Omega^2
   = (-3\log Y + O_{\msf{y}}(1) ) \Omega^2, \nonumber
\end{align}
by \eqref{eq:uniform finite},
where $O_{\msf{y}}(1)$ means a finite number independent of the point $z$,
where the constant could depend on $\msf{y}$.
The absolute integrability of
$\mathfrak{g}^\circ(\boldsymbol{0}_3,\mathsf{y}) \coloneqq \opn{\underset{\rho=0}{CT}} \mf{g}^{\circ}(\bsm{0}_3, \msf{y}, \rho)$
over $\overline{\mc{M}}(\C)$
can also be seen from integrability of $(1+\log Y) \Omega^2$ over $\mbb{D}_1$.

\clearpage

\bibliographystyle{alpha}
\bibliography{reference}

@article{HarrisKudla1991,
  author = {Harris, Michael and Kudla, Stephen S.},
  title = {The central critical value of the triple product $L$-function},
  journal = {Annals of Mathematics},
  volume = {133},
  number = {3},
  pages = {605--672},
  year = {1991},
  doi = {10.2307/2944321},
}

@article{weil1965siegel,
  title={Sur la formule de {S}iegel dans la th{\'e}orie des groupes classiques},
  author={Weil, Andr{\'e}},
  journal={Acta Mathematica},
  volume={113},
  number={1},
  pages={1--87},
  year={1965},
  publisher={Institut Mittag-Leffler},
  doi={10.1007/BF02391774}
}

@article{STEIN_2021, 
  title={ANALYTIC PROPERTIES OF {E}ISENSTEIN SERIES AND STANDARD {$L$}-FUNCTIONS}, 
  volume={244}, 
  DOI={10.1017/nmj.2020.11}, 
  journal={Nagoya Mathematical Journal}, 
  author={Stein, Oliver}, year={2021}, pages={168–203}}

@article{KudlaMillson1986,
  author  = {Kudla, Stephen S. and Millson, John J.},
  title   = {The {T}heta {C}orrespondence and {H}armonic {F}orms. {I}},
  journal = {Mathematische Annalen},
  volume  = {274},
  number  = {3},
  pages   = {353--378},
  year    = {1986},
  doi     = {10.1007/BF01457221}
}

@article{BerndtKuehn2012a,
  author    = {Berndt, Rolf and Kühn, Ulf},
  title     = {On {K}udla’s {G}reen function for signature (2,2), {P}art {I}},
  journal   = {arXiv preprint},
  eprint    = {arXiv:1205.6417},
  year      = {2012},
  note      = {Preprint available at arXiv}
}

@misc{zhusingular,
      title={Arithmetic {S}iegel--{W}eil formula on $\mathcal{X}_0({N})$: singular terms}, 
      author={Baiqing Zhu},
      year={2024},
      eprint={2401.06368},
      archivePrefix={arXiv},
      primaryClass={math.NT},
      url={https://arxiv.org/abs/2401.06368}, 
}

@article{BerndtKuehn2012b,
  author    = {Berndt, Rolf and Kühn, Ulf},
  title     = {On {K}udla’s {G}reen function for signature (2,2), {P}art {II}},
  journal   = {arXiv preprint},
  eprint    = {arXiv:1209.3949},
  year      = {2012},
  note      = {Preprint available at arXiv}
}

@article{KudlaMillson1987,
  author  = {Kudla, Stephen S. and Millson, John J.},
  title   = {The {T}heta {C}orrespondence and {H}armonic {F}orms. {II}},
  journal = {Mathematische Annalen},
  volume  = {277},
  number  = {2},
  pages   = {267--314},
  year    = {1987},
  doi     = {10.1007/BF01457364}
}

@article{Rao1993WeilRepresentation,
  author    = {Ranga Rao, Ramaswamy},
  title     = {On some explicit formulas in the theory of {W}eil representation},
  journal   = {Pacific Journal of Mathematics},
  year      = {1993},
  volume    = {157},
  number    = {2},
  pages     = {335--371}
}

@article{KR90DegeneratePS,
  title={{D}egenerate principal series and invariant distributions},
  author={Stephen S. Kudla and Stephen Rallis},
  journal={Israel Journal of Mathematics},
  year={1990},
  volume={69},
  pages={25-45},
}

@article{BY20,
  title={Arithmetic degrees of special cycles and derivatives of Siegel Eisenstein series},
  author={Bruinier, Jan and Yang, Tonghai},
  journal={Journal of European Math. Soc},
  pages={pp60},
  year={2020}
}

@article{CY,
  title={A reformulation of the {S}iegel series and intersection numbers},
  author={Cho, Sungmun and Yamauchi, Takuya},
  journal={Mathematische Annalen},
  volume={377},
  number={3},
  pages={1757--1826},
  year={2020},
  publisher={Springer}
}

@article {D,
    AUTHOR = {Drinfeld, V. G.},
     TITLE = {Coverings of {$p$}-adic symmetric domains},
   JOURNAL = {Funkcional. Anal. i Prilo\v{z}en.},
    VOLUME = {10},
      YEAR = {1976},
    NUMBER = {2},
     PAGES = {29--40},
      ISSN = {0374-1990},
   MRCLASS = {14L05 (10D05 14G20 12B25)},
  MRNUMBER = {0422290},
MRREVIEWER = {Michiel Hazewinkel},
}

@article {DY1,
    AUTHOR = {Du, Tuoping and Yang, Tonghai},
     TITLE = {Arithmetic {S}iegel-{W}eil formula on {$X_0(N)$}},
   JOURNAL = {Adv. Math.},
  FJOURNAL = {Advances in Mathematics},
    VOLUME = {345},
      YEAR = {2019},
     PAGES = {702--755},
      ISSN = {0001-8708},
   MRCLASS = {11G15 (11F41 14K22)},
  MRNUMBER = {3901671},
MRREVIEWER = {Marc Masdeu},
       DOI = {10.1016/j.aim.2019.01.015},
       URL = {https://doi-org.ezproxy.library.wisc.edu/10.1016/j.aim.2019.01.015},
}

@article {GS19,
    AUTHOR = {Garc\'ia, Luis E. and Sankaran, Siddarth},
     TITLE = {Green forms and the arithmetic {S}iegel--{W}eil formula},
  JOURNAL = {Inventiones Mathematicae},
    VOLUME = {215},
      YEAR = {2019},
    NUMBER = {3},
     PAGES = {863--975},
      ISSN = {0020-9910},
   MRCLASS = {11G18 (14G35)},
  MRNUMBER = {3935034},
       DOI = {10.1007/s00222-018-0839-4},
       URL = {https://doi-org.ezproxy.library.wisc.edu/10.1007/s00222-018-0839-4},
}

@article{G,
  title={On canonical and quasi-canonical liftings},
  author={Gross, Benedict},
  journal={Inventiones Mathematicae},
  volume={84},
  number={2},
  pages={321--326},
  year={1986},
  publisher={Springer}
}

@article {Howard12,
    AUTHOR = {Howard, Benjamin},
     TITLE = {Complex multiplication cycles and {K}udla-{R}apoport divisors},
   JOURNAL = {Ann. of Math. (2)},
  FJOURNAL = {Annals of Mathematics. Second Series},
    VOLUME = {176},
      YEAR = {2012},
    NUMBER = {2},
     PAGES = {1097--1171},
      ISSN = {0003-486X},
   MRCLASS = {11G18 (14G35)},
  MRNUMBER = {2950771},
MRREVIEWER = {Rolf Berndt},
       DOI = {10.4007/annals.2012.176.2.9},
       URL = {https://doi.org/10.4007/annals.2012.176.2.9},
}

@article{Kud97Ann,
  title={Central derivatives of {E}isenstein series and height pairings},
  author={Kudla, Stephen S.},
  journal={Annals of Mathematics},
  volume={146},
  number={3},
  pages={545--646},
  year={1997},
  publisher={JSTOR}
}

@article{Kudla1992,
author = {Stephen S. Kudla and Stephen Rallis and David Soudry},
journal = {Inventiones Mathematicae},
number = {3},
pages = {483-542},
title = {On the degree 5 {L}-function for {Sp}(2).},
url = {http://eudml.org/doc/143976},
volume = {107},
year = {1992},
}

@book{KatzMazur1985,
  author    = {Katz, Nicholas M. and Mazur, Barry},
  title     = {Arithmetic Moduli of Elliptic Curves},
  series    = {Annals of Mathematics Studies},
  volume    = {108},
  publisher = {Princeton University Press},
  address   = {Princeton, NJ},
  year      = {1985},
  pages     = {xiv + 514}
}

@article{NT91,
  author       = {Yukiyoshi Nakkajima and Yuichiro Taguchi},
  title        = {A generalization of the {C}howla--{S}elberg formula},
  journal      = {Journal für die reine und angewandte Mathematik},
  volume       = {419},
  pages        = {119--124},
  year         = {1991},
}

@article{Kudlasplitting,
  title={Splitting metaplectic covers of dual reductive pairs},
  author={Kudla, Stephen },
  journal={Israel Journal of Mathematics},
  volume={87},
  number={1-3},
  pages={361--401},
  year={1994},
  publisher={Springer}
}

@book{KRYbook,
  title={Modular Forms and Special Cycles on Shimura Curves.(AM-161)},
  author={Kudla, Stephen and Rapoport, Michael and Yang, Tonghai},
  volume={161},
  pages={pp361},
  year={2006},
  publisher={Princeton university press}
}

@ARTICLE{KR1,
  author = {Kudla, Stephen and Rapoport, Michael},
  title = {Special cycles on unitary {S}himura varieties {I}. {U}nramified local
	theory},
  journal = {Invent. Math.},
  year = {2011},
  volume = {184},
  pages = {629--682},
  number = {3},
  doi = {10.1007/s00222-010-0298-z},
  fjournal = {Inventiones Mathematicae},
  issn = {0020-9910},
  mrclass = {14G35 (11G18 14L05)},
  mrnumber = {2800697},
  mrreviewer = {Jeffrey D. Achter},
  url = {https://doi.org/10.1007/s00222-010-0298-z}
}

@ARTICLE{KR2,
  author = {Kudla, Stephen and Rapoport, Michael},
  title = {Special cycles on unitary {S}himura varieties {II}: {G}lobal theory},
  journal = {J. Reine Angew. Math.},
  year = {2014},
  volume = {697},
  pages = {91--157},
  doi = {10.1515/crelle-2012-0121},
  fjournal = {Journal f\"{u}r die Reine und Angewandte Mathematik. [Crelle's Journal]},
  issn = {0075-4102},
  mrclass = {11G18 (14G35)},
  mrnumber = {3281653},
  mrreviewer = {Jeffrey D. Achter},
  url = {https://doi.org/10.1515/crelle-2012-0121}
}

@incollection{Gortz2007Arithmetic,
  author       = {Görtz, Ulrich},
  title        = {Arithmetic intersection numbers},
  booktitle    = {Argos seminar on intersections of modular correspondences},
  series       = {Astérisque},
  number       = {312},
  pages        = {15--24},
  year         = {2007},
  publisher    = {Société Mathématique de France},
  address      = {Paris},
  isbn         = {978-2-85629-231-0},
  mrnumber     = {MR2408245},
  zbl          = {1223.11074}
}

@article{Rohrlich1984,
  author  = {Rohrlich, David E.},
  title   = {A modular version of {Jensen}'s formula},
  journal = {Math. Proc. Cambridge Philos. Soc.},
  volume  = {95},
  number  = {1},
  year    = {1984},
  pages   = {15--20},
  doi     = {10.1017/S0305004100061259}
}

@incollection{Deligne85,
  author    = {Deligne, Pierre},
  title     = {Preuve des conjectures de {Tate} et de {Shafarevitch}},
  booktitle = {S{\'e}minaire Bourbaki, volume 1983/84},
  series    = {Ast{\'e}risque},
  number    = {121--122},
  publisher = {Soci{\'e}t{\'e} math{\'e}matique de France},
  year      = {1985},
  pages     = {25--41},
  note      = {Expos{\'e} no.~616},
  mrnumber  = {768952},
  url       = {https://www.numdam.org/item/SB_1983-1984__26__25_0/}
}

@article{Burgos2007, 
  title={COHOMOLOGICAL ARITHMETIC CHOW RINGS}, 
  volume={6}, DOI={10.1017/S1474748007000011}, 
  number={1}, 
  journal={Journal of the Institute of Mathematics of Jussieu}, 
  author={Burgos Gil, Jos\'e Ignacio and Kramer, J\"urg and Kühn, Ulf}, 
  year={2007}, 
  pages={1–172}}

@article{Howe1989Remarks,
  author       = {Howe, Roger},
  title        = {Remarks on classical invariant theory},
  journal      = {Transactions of the American Mathematical Society},
  volume       = {313},
  number       = {2},
  pages        = {539--570},
  year         = {1989},
  doi          = {10.1090/S0002-9947-1989-0986027-X},
  mrnumber     = {986027},
}

@incollection{KudlaProgram2004,
  author    = {Kudla, Stephen S.},
  title     = {Special cycles and derivatives of Eisenstein series},
  booktitle = {Heegner Points and Rankin {$L$}-Series},
  editor    = {Darmon, Henri and Zhang, Shou-Wu},
  series    = {Mathematical Sciences Research Institute Publications},
  volume    = {49},
  publisher = {Cambridge University Press},
  address   = {Cambridge},
  year      = {2004},
  pages     = {243--270},
  doi       = {10.1017/CBO9780511756375.009}
}

@book{BruinierHoward2026,
  author    = {Bruinier, Jan Hendrik and Howard, Benjamin},
  title     = {Arithmetic Volumes of Unitary Shimura Varieties},
  series    = {Memoirs of the American Mathematical Society, no. 1619},
  volume    = {318},
  publisher = {American Mathematical Society},
  address   = {Providence, RI},
  year      = {2026},
  pages     = {v+119},
  doi       = {10.1090/memo/1619}
}

@article{Nagaoka1992,
  author  = {Shoyu Nagaoka},
  title   = {A Note on the Siegel-Eisenstein Series of Weight 2 on {$Sp_2(\mathbb{Z})$}},
  journal = {Manuscripta Mathematica},
  volume  = {77},
  year    = {1992},
  number  = {1},
  pages   = {71--88}
}

@article{Bostetal1994,
  author  = {Bost, Jean-Beno\^{i}t and Gillet, Henri and Soulé, Christophe},
  title   = {Heights of projective varieties and positive Green forms},
  journal = {Journal of the American Mathematical Society},
  year    = {1994},
  volume  = {7},
  pages   = {903--1027},
  doi     = {10.1090/s0894-0347-1994-1260106-x}
}

@article{LZ,
  title={Kudla--{R}apoport cycles and derivatives of local densities},
  author={Li, Chao and Zhang, Wei},
  journal={Journal of the American Mathematical Society},
  volume={35},
  number={3},
  pages={705--797},
  year={2022}
}

@article{LiZhang-orthogonalKR,
  title={On the arithmetic {S}iegel--{W}eil formula for {GS}pin {S}himura varieties},
  author={Li, Chao and Zhang, Wei},
  journal={Inventiones mathematicae},
  volume={228},
  number={3},
  pages={1353--1460},
  year={2022},
  publisher={Springer}
}

@article{Kudla2003,
  author       = {Stephen S. Kudla},
  title        = {Integrals of Borcherds Forms},
  journal      = {Compositio Mathematica},
  volume       = {137},
  number       = {3},
  pages        = {293--349},
  year         = {2003},
  doi          = {10.1023/A:1024127100993},
  url          = {https://doi.org/10.1023/A:1024127100993}
}

@article {KM,
    AUTHOR = {Kudla, Stephen S. and Millson, John J.},
     TITLE = {Intersection numbers of cycles on locally symmetric spaces and
              {F}ourier coefficients of holomorphic modular forms in several
              complex variables},
   JOURNAL = {Inst. Hautes \'{E}tudes Sci. Publ. Math.},
  FJOURNAL = {Institut des Hautes \'{E}tudes Scientifiques. Publications
              Math\'{e}matiques},
    VOLUME = {71},
      YEAR = {1990},
     PAGES = {121--172},
      ISSN = {0073-8301},
   MRCLASS = {11F32 (11F30 11F46 11F67 32N10 32N15)},
  MRNUMBER = {1079646},
       URL = {http://www.numdam.org/item?id=PMIHES_1990__71__121_0},
}

@book{M,
  title={The Crystals Associated to Barsotti-Tate Groups: with Applications to Abelian Schemes},
  author={William Messing},
  series= {Lecture Notes in Mathematics},
  volume={264},
  pages={pp190},
  year={1972},
  publisher={Springer}
}

@article{Faltings1983,
  author       = {Faltings, Gerd},
  title        = {Endlichkeitss\"atze f\"ur abelsche Variet\"aten \"uber Zahlk\"orpern},
  journal      = {Inventiones Mathematicae},
  volume       = {73},
  number       = {3},
  pages        = {349--366},
  year         = {1983},
  note         = {German; “Finiteness theorems for abelian varieties over number fields.”},
}

@incollection{Gortz2007Sum,
  author       = {Görtz, Ulrich},
  title        = {A sum of representation numbers},
  booktitle    = {Argos seminar on intersections of modular correspondences},
  series       = {Astérisque},
  number       = {312},
  pages        = {9--14},
  year         = {2007},
  publisher    = {Société Mathématique de France},
  address      = {Paris},
  isbn         = {978-2-85629-231-0},
  mrclass      = {11F32, 11G15, 11E12, 14K22},
  mrnumber     = {MR2408244},
  zbl          = {1157.14310}
}

@article{I,
  title={A regularized Siegel-Weil formula for unitary groups},
  author={Ichino, Atsushi},
  journal={Mathematische Zeitschrift},
  volume={247},
  number={2},
  pages={241--277},
  year={2004},
  publisher={Springer}
}

@article{J,
  title={Hermitian forms over local fields},
  author={Jacobowitz, Ronald},
  journal={American Journal of Mathematics},
  volume={84},
  number={3},
  pages={441--465},
  year={1962},
  publisher={JSTOR}
}

@ARTICLE{Liu2011,
  author = {Liu, Yifeng},
  title = {Arithmetic theta lifting and {$L$}-derivatives for unitary groups,
	{I}},
  journal = {Algebra Number Theory},
  year = {2011},
  volume = {5},
  pages = {849--921},
  number = {7},
  doi = {10.2140/ant.2011.5.849},
  fjournal = {Algebra \& Number Theory},
  issn = {1937-0652},
  mrclass = {11G18 (11F27 11G50 20G05)},
  mrnumber = {2928563},
  mrreviewer = {Stefano Vigni},
  url = {https://doi.org/10.2140/ant.2011.5.849}
}

@book{L,
  title={Algebraic geometry and arithmetic curves},
  author={Qing, Liu},
  year={2006},
  publisher={Oxford University Press}
}

@book{O,
  title={Introduction to quadratic forms},
  author={O’Meara, Onorato Timothy},
  volume={117},
  year={2013},
  publisher={Springer}
}

@article{P,
  title={On the arithmetic moduli schemes of PEL Shimura varieties},
  author={Pappas, Georgios},
  journal={Journal of Algebraic Geometry},
  volume={9},
  number={3},
  pages={577},
  year={2000},
  publisher={Providence, RI: University Press, c1992-}
}

@misc{gillet2009arithmeticintersectiontheorydelignemumford,
      title={Arithmetic Intersection Theory on {D}eligne--{M}umford Stacks}, 
      author={Henri Gillet},
      year={2009},
      eprint={0904.1424},
      archivePrefix={arXiv},
      primaryClass={math.AG},
      url={https://arxiv.org/abs/0904.1424}, 
}

@article{Chen2024,
      title={Co-rank 1 {A}rithmetic {S}iegel--{W}eil: combined version}, 
      author={Chen, Ryan C.},
    journal={Online},
      year={May 2, 2024},
      eprint={},
}

@inproceedings{Bouw2007Invariants,
  author       = {Bouw, Irene I.},
  title        = {12. {I}nvariants of {T}ernary {Q}uadratic {F}orms},
  booktitle={ARGOS Seminar on Intersections of Modular Correspondences},
  pages={113--137},
  year={2007}
}

@inproceedings{R,
  title={13. {D}eformations of isogenies of formal groups},
  author={Rapoport, Michael},
  booktitle={ARGOS Seminar on Intersections of Modular Correspondences},
  pages={139--170},
  year={2007}
}

@article{SSY,
  title={{A genus two arithmetic Siegel-Weil formula on $X_0(N)$}},
  author={Sankaran, Siddarth and Shi, Yousheng and Yang, Tonghai},
  journal={Transactions of the American Mathematical Society},
  volume={376},
  number={06},
  pages={3995--4041},
  year={2023}
}

@article{GK93,
author = {Gross, Benedict and Keating, Kevin},
journal = {Inventiones Mathematicae},
number = {2},
pages = {225-246},
title = {On the intersection of modular correspondences.},
url = {http://eudml.org/doc/144102},
volume = {112},
year = {1993},
}

@book {RZ96,
    AUTHOR = {Rapoport, Michael and Zink, Thomas},
     TITLE = {Period spaces for {$p$}-divisible groups},
    SERIES = {Annals of Mathematics Studies},
    VOLUME = {141},
 PUBLISHER = {Princeton University Press, Princeton, NJ},
      YEAR = {1996},
     PAGES = {xxii+324},
      ISBN = {0-691-02782-X; 0-691-02781-1},
   MRCLASS = {14G20 (11G18 14F30 14L05 14M15 20G05 20G25)},
  MRNUMBER = {1393439},
MRREVIEWER = {Robert E. Kottwitz},
       DOI = {10.1515/9781400882601},
       URL = {https://doi.org/10.1515/9781400882601},
}

@inproceedings{tate1967,
  title={p-Divisible groups},
  author={Tate, John T},
  booktitle={Proceedings of a conference on Local Fields},
  pages={158--183},
  year={1967},
  organization={Springer}
}

@inproceedings{V,
  title={11. {E}NDOMORPHISMS OF QUASI-CANONICAL LIFTS},
  author={Vollaard, Inken},
  booktitle={ARGOS Seminar on Intersections of Modular Correspondences},
  pages={105--112},
  year={2007}
}

@inproceedings{Vol,
  title={9. {C}anonical and quasi-canonical liftings in the split case},
  author={Meusers, Volker},
  booktitle={ARGOS Seminar on Intersections of Modular Correspondences},
  pages={87--98},
  year={2007}
}

@inproceedings{Wedhorn,
  title={15. {C}alculation of representation densities},
  author={Wedhorn, Torsten},
  booktitle={ARGOS Seminar on Intersections of Modular Correspondences},
  pages={179--190},
  year={2007}
}

@inproceedings{W,
  title={8. {C}ANONICAL AND QUASI-CANONICAL LIFTINGS},
  author={Wewers, Stefan},
  booktitle={ARGOS Seminar on Intersections of Modular Correspondences},
  pages={67--86},
  year={2007}
}

@inproceedings{W2,
  title={14. {A}N ALTERNATIVE APPROACH USING IDEAL BASES},
  author={Wewers, Stefan},
  booktitle={ARGOS Seminar on Intersections of Modular Correspondences},
  pages={171--177},
  year={2007}
}

@article{HM22,
author={Howard, Ben and Madapusi, Keerthi},
  title={Kudla's modularity conjecture on integral models of orthogonal Shimura varieties},
  journal={arXiv preprint arXiv:2211.05108},
  year={2022},
}

@article{Zhu23,
author  = {Baiqing Zhu},
  title   = {Arithmetic {S}iegel--{W}eil formula on $\mathcal{X}_0({N})$},
  journal = {Algebra $\&$ Number Theory},
  volume  = {19},
  number  = {9},
  pages   = {1771--1822},
  year    = {2025},
  doi     = {10.2140/ant.2025.19.1771},
}

@article{Kudla1988,
author = {Kudla, Stephen S. and Rallis, Stephen},
journal = {Journal für die reine und angewandte Mathematik},
pages = {1-68},
title = {On the {W}eil--{S}iegel formula.},
url = {http://eudml.org/doc/153025},
volume = {387},
year = {1988},
}

@article{KudRallis94,
 ISSN = {0003486X, 19398980},
 URL = {http://www.jstor.org/stable/2118540},
 author = {Stephen S. Kudla and Stephen Rallis},
 journal = {Annals of Mathematics},
 number = {1},
 pages = {1--80},
 publisher = {[Annals of Mathematics, Trustees of Princeton University on Behalf of the Annals of Mathematics, Mathematics Department, Princeton University]},
 title = {A Regularized {S}iegel--{W}eil Formula: The First Term Identity},
 urldate = {2025-05-30},
 volume = {140},
 year = {1994}
}

@article {HZ76,
    AUTHOR = {Hirzebruch, Friedrich and Zagier, Don},
     TITLE = {Intersection numbers of curves on {H}ilbert modular surfaces
              and modular forms of {N}ebentypus},
   JOURNAL = {Invent. Math.},
  FJOURNAL = {Inventiones Mathematicae},
    VOLUME = {36},
      YEAR = {1976},
     PAGES = {57--113},
      ISSN = {0020-9910},
   MRCLASS = {10D20 (14J25)},
  MRNUMBER = {453649},
MRREVIEWER = {Hiroshi Saito},
       DOI = {10.1007/BF01390005},
       URL = {https://doi.org/10.1007/BF01390005},
}

@article {KRY04,
    AUTHOR = {Kudla, Stephen S. and Rapoport, Michael and Yang, Tonghai},
     TITLE = {Derivatives of {E}isenstein series and {F}altings heights},
   JOURNAL = {Compos. Math.},
  FJOURNAL = {Compositio Mathematica},
    VOLUME = {140},
      YEAR = {2004},
    NUMBER = {4},
     PAGES = {887--951},
      ISSN = {0010-437X},
   MRCLASS = {11G18 (11F30 11F37 11G50 14G35 14G40)},
  MRNUMBER = {2059224},
MRREVIEWER = {Jens Funke},
       DOI = {10.1112/S0010437X03000459},
       URL = {https://doi.org/10.1112/S0010437X03000459},
}

@book {Mumford85,
    AUTHOR = {Mumford, David},
     TITLE = {Abelian varieties},
    SERIES = {Tata Institute of Fundamental Research Studies in Mathematics},
      NOTE = {With appendices by C. P. Ramanujam and Yuri Manin,
              Reprint of the second (1974) edition},
      YEAR = {1985},
 PUBLISHER = {Oxford University Press},
}

@book {Voight21,
    AUTHOR = {Voight, John},
     TITLE = {Quaternion algebras},
    SERIES = {Graduate Texts in Mathematics},
    VOLUME = {288},
 PUBLISHER = {Springer, Cham},
      YEAR = {[2021] \copyright 2021},
     PAGES = {xxiii+885},
      ISBN = {978-3-030-56692-0; 978-3-030-56694-4},
   MRCLASS = {11R52 (11-02 11E12 11F06 16H05 16U60 20H10)},
  MRNUMBER = {4279905},
MRREVIEWER = {Juliusz Brzezi\'{n}ski},
       DOI = {10.1007/978-3-030-56694-4},
       URL = {https://doi.org/10.1007/978-3-030-56694-4},
}

@article {Ikeda17,
    AUTHOR = {Ikeda, Tamotsu},
     TITLE = {On the functional equation of the {S}iegel series},
   JOURNAL = {J. Number Theory},
  FJOURNAL = {Journal of Number Theory},
    VOLUME = {172},
      YEAR = {2017},
     PAGES = {44--62},
      ISSN = {0022-314X,1096-1658},
   MRCLASS = {11E08 (11E45)},
  MRNUMBER = {3573143},
MRREVIEWER = {Suzanne\ Caulk},
       DOI = {10.1016/j.jnt.2016.08.002},
       URL = {https://doi.org/10.1016/j.jnt.2016.08.002},
}

@article {IK18,
    AUTHOR = {Ikeda, Tamotsu and Katsurada, Hidenori},
     TITLE = {On the {G}ross-{K}eating invariant of a quadratic form over a
              non-archimedean local field},
   JOURNAL = {Amer. J. Math.},
  FJOURNAL = {American Journal of Mathematics},
    VOLUME = {140},
      YEAR = {2018},
    NUMBER = {6},
     PAGES = {1521--1565},
      ISSN = {0002-9327,1080-6377},
   MRCLASS = {11E08 (11E95)},
  MRNUMBER = {3884638},
MRREVIEWER = {Detlev\ W.\ Hoffmann},
       DOI = {10.1353/ajm.2018.0046},
       URL = {https://doi.org/10.1353/ajm.2018.0046},
}

@article {Shimura82,
    AUTHOR = {Shimura, Goro},
     TITLE = {Confluent hypergeometric functions on tube domains},
   JOURNAL = {Math. Ann.},
  FJOURNAL = {Mathematische Annalen},
    VOLUME = {260},
      YEAR = {1982},
    NUMBER = {3},
     PAGES = {269--302},
      ISSN = {0025-5831,1432-1807},
   MRCLASS = {32N10 (10D20)},
  MRNUMBER = {669297},
MRREVIEWER = {A.\ B.\ Venkov},
       DOI = {10.1007/BF01461465},
       URL = {https://doi.org/10.1007/BF01461465},
}

@incollection {BF10,
    AUTHOR = {Bruinier, Jan Hendrik and Funke, Jens},
     TITLE = {On the injectivity of the {K}udla-{M}illson lift and
              surjectivity of the {B}orcherds lift},
 BOOKTITLE = {Moonshine: the first quarter century and beyond},
    SERIES = {London Math. Soc. Lecture Note Ser.},
    VOLUME = {372},
     PAGES = {12--39},
 PUBLISHER = {Cambridge Univ. Press, Cambridge},
      YEAR = {2010},
      ISBN = {978-0-521-10664-1},
   MRCLASS = {11F55 (11F32)},
  MRNUMBER = {2681774},
MRREVIEWER = {Rainer\ Schulze-Pillot},
}

@misc{CLZ,
  title={Faltings heights and the subleading terms of adjoint $L$-functions},
  author={Chen, Ryan and Lu, Weixiao and Zhang, Wei},
  note={In preparation},
  }

@inproceedings {DR73,
    AUTHOR = {Deligne, Pierre and Rapoport, Michael},
     TITLE = {Les sch\'{e}mas de modules de courbes elliptiques},
 BOOKTITLE = {Modular functions of one variable, {II} ({P}roc. {I}nternat.
              {S}ummer {S}chool, {U}niv. {A}ntwerp, {A}ntwerp, 1972)},
    SERIES = {Lecture Notes in Math., Vol. 349},
     PAGES = {143--316},
 PUBLISHER = {Springer, Berlin-New York},
      YEAR = {1973},
   MRCLASS = {14K10 (10D05)},
  MRNUMBER = {337993},
MRREVIEWER = {T.\ Oda},
}

@misc{Hain14,
      title={Lectures on Moduli Spaces of Elliptic Curves}, 
      author={Richard Hain},
      year={2014},
      eprint={0812.1803},
      archivePrefix={arXiv},
      primaryClass={math.AG},
      url={https://arxiv.org/abs/0812.1803}, 
}

\end{document}